\documentclass[10pt,electronic]{amsart}        
\usepackage{baskervald}
\usepackage[swedish, english]{babel}

\usepackage[latin1]{inputenc}
\usepackage{accents}
\usepackage{adjustbox}
\usepackage{url}
\usepackage[hmargin=3.3cm,vmargin=3.5cm]{geometry} 
\usepackage{graphicx}                         
\usepackage{amsmath}

\usepackage[pdfusetitle]{hyperref}
\usepackage[noabbrev,capitalize]{cleveref}
\hypersetup{
	colorlinks,
	linkcolor={red!50!black},
	citecolor={blue!50!black},
	urlcolor={blue!80!black}
}

\usepackage{amsthm}                      
\usepackage{amssymb} 
\usepackage{mathrsfs}
\usepackage{stmaryrd}
\usepackage{mathtools}
\usepackage[shortlabels]{enumitem}
\usepackage{tikz-cd}
\usetikzlibrary{babel}
\DeclareFontFamily{U}{wncy}{}
\DeclareFontShape{U}{wncy}{m}{n}{<->wncyr10}{}
\DeclareSymbolFont{mcy}{U}{wncy}{m}{n}
\DeclareMathSymbol{\sha}{\mathord}{mcy}{"58} 
\usepackage{euscript}
\renewcommand{\mathcal}{\EuScript}
\swapnumbers
\theoremstyle{plain}                                              \makeatletter
\def\swappedhead#1#2#3{%
	\thmnumber{\@upn{\the\thm@headfont#2\@ifnotempty{#1}{.~}}}%
	\thmname{#1}%
	\thmnote{ {\the\thm@notefont(#3)}}}
\makeatother
\newtheorem{thm}{Theorem}[subsection]

\newtheorem{lem}[thm]{Lemma}
\newtheorem{prop}[thm]{Proposition}
\newtheorem{cor}[thm]{Corollary}
\newtheorem{conj}[thm]{Conjecture}

\theoremstyle{definition}
\newtheorem{examplex}[thm]{Example}
\newtheorem{rem}[thm]{Remark}
\newtheorem{defn}[thm]{Definition}
\newtheorem{construction}[thm]{Construction}

\newtheorem{para}[thm]{}

\DeclareMathOperator{\Tr}{Tr}
\newcommand{\F}{\mathbf{F}}
\newcommand{\Q}{\mathbf{Q}}
\newcommand{\C}{\mathbf{C}}
\newcommand{\R}{\mathbf{R}}
\newcommand{\N}{\mathbf{N}}
\newcommand{\V}{\mathbb V}
\newcommand{\Z}{\mathbf{Z}}

\newcommand{\Sym}{\mathrm{Sym}}
\newcommand{\map}{\mathrm{map}}

\newcommand{\GL}{\mathrm{GL}}
\newcommand{\Sp}{\mathrm{Sp}}
\newcommand{\KN}{\mathrm{KN}}
\newcommand{\length}{l}
\newcommand{\et}{\mathrm{\acute{e}t}}
\newcommand{\sing}{\mathrm{sing}}
\newcommand{\M}{\mathcal M}
\renewcommand{\H}{\mathcal H}
\DeclareMathOperator{\Spec}{Spec}
\newcommand{\Fq}{\mathbf F_q}
\newcommand{\Fqbar}{\overline{\mathbf{F}}_q}
\renewcommand{\V}{\mathbb V}
\renewcommand{\aa}{\alpha}
\newcommand{\Conf}{\mathsf{Fin}}
\newcommand{\Exp}{\mathbf{Exp}}
\newcommand{\Log}{\mathbf{Log}}
\newcommand{\Hur}{\mathsf{Hur}}
\newcommand{\antishriek}{\text{!`}}

\newcommand{\Ana}{\mathsf{Ana}}
\newcommand{\Poly}{\mathsf{Poly}}

\newcommand{\tr}{\operatorname{tr}}

\title{Hyperelliptic curves, the scanning map, and moments of families of quadratic $L$-functions}
\author{Jonas Bergstr\"om}
\author{Adrian Diaconu} 
\author{Dan Petersen}
\author{Craig Westerland}
\address{Matematiska institutionen \\ Stockholms Universitet \\ 106 91 Stockholm \\ Sweden}
\email{jonasb@math.su.se, dan.petersen@math.su.se}

\address{School of Mathematics \\
127 Vincent Hall \\ 206 Church St.\ SE \\ Minneapolis MN 55455 \\ USA }
\email{cad@umn.edu, cwesterl@umn.edu}

\begin{document}

 \maketitle

\begin{abstract}
    We compute the stable homology of the braid group with coefficients in any Schur functor applied to the integral reduced Burau representation. This may be considered as a hyperelliptic analogue of the Mumford conjecture (Madsen--Weiss theorem) with twisted coefficients. We relate the result to the function field case of conjectures of Conrey--Farmer--Keating--Rubinstein--Snaith on moments of families of quadratic $L$-functions. Combined with a recent homological stability theorem of Miller--Patzt--Petersen--Randal-Williams, our homological calculations confirm the Conrey--Farmer--Keating--Rubinstein--Snaith predictions for all large enough prime powers $q$.
\end{abstract}

\tableofcontents 

\section{Introduction}

\subsection{Summary of results}

 \begin{para}Let $q$ be an odd prime power. For $d \in \mathbf F_q[t]$ monic and squarefree, let $L(s,\chi_d)$ denote the $L$-function attached to the Galois representation on the first cohomology of the hyperelliptic curve with equation $y^2=d(x)$. It is interesting to study the distribution of the central values $L(\frac 1 2, \chi_d)$ as $d$ varies; our concern will be the  \emph{moments} of the distribution. A precise prediction for the asymptotics of the moments of this family was given by Conrey--Farmer--Keating--Rubinstein--Snaith \cite{CFKRS}, as a special case of a conjecture applying to many different families of $L$-functions. \end{para}
 
 \begin{conj}[Conrey--Farmer--Keating--Rubinstein--Snaith]\label{CFKRS conjecture} For fixed $q$ and $r$ there exists an explicit polynomial $Q_r(n)$ of degree $r(r+1)/2$ such that
\[ 
\frac{1}{q^{2g+1}}\sum_{\substack{d \text{\emph{ monic  \&  sq.\ free}} \\ \deg d = 2g+1}}
L(\tfrac{1}{2}, \chi_{d})^{r} 
=\,  Q_{r}(2g+1) \, + \, o(1).
\] 
\end{conj}
 
 \begin{para}The goal of this paper is to prove \cref{CFKRS conjecture} in the regime where $q$ is much larger than $r$, using methods of homotopy theory. The strategy will be to translate the conjecture to a problem about homological stability, using the Grothendieck--Lefschetz trace formula. In this paper we determine completely the stable homology, and show that the answer matches \cref{CFKRS conjecture} perfectly. The companion paper \cite{MPPRW} proves an improved stable range for the relevant homological stability problem. Together, this implies \cref{CFKRS conjecture} for $q$ sufficiently large. \end{para}

\begin{para}Let $\beta_n$ denote the Artin braid group on $n$ strands. We denote by $V_n$ the \emph{reduced Burau representation specialized at $t=-1$}, also known as the \emph{integral} reduced Burau representation,  $\beta_n \to \GL_{n-1}(\Z)$. This representation arises naturally in diverse problems in algebraic topology, geometric group theory and algebraic geometry, and has been extensively studied by many authors from different perspectives.  
	\end{para}\begin{para}
	It follows from a general homological stability theorem of Randal-Williams and Wahl \cite{randalwilliamswahl} that the embedding $\beta_n \hookrightarrow \beta_{n+1}$ induces isomorphisms on homology in a range: specifically, 
	$$ H_i(\beta_n,V_n^{\otimes k}) \longrightarrow H_i(\beta_{n+1},V_{n+1}^{\otimes k})$$
	is an isomorphism for $i \leq \tfrac {n-2-k} 2$. We denote by $H_i(\beta_\infty,V^{\otimes k})$ the stable homology. All of our main results in this paper concern the \emph{rationalization} $V_n \otimes_\Z \Q$, and to ease notation we denote the rationalization simply by $V_n$. The main goals of this paper are:
	\begin{enumerate}
		\item To calculate the stable homology of each of the local systems $V_n^{\otimes k}$, with its natural action by the symmetric group $\Sigma_k$. Equivalently, we calculate the stable homology of $\beta_n$ with coefficients in any Schur functor $S^\lambda(V_n)$. 
		\item To determine the Galois action on stable homology. The Galois action arises from an interpretation as the homology of an algebro-geometrically defined moduli space of hyperelliptic curves. The moduli interpretation also shows that our result is a hyperelliptic analogue of the Mumford conjecture (the Madsen--Weiss theorem \cite{madsenweiss}), with twisted coefficients. 
		\item To relate the result to the function field version of the conjectures of Conrey--Farmer--Keating--Rubinstein--Snaith \cite{CFKRS} on asymptotics of moments of families of quadratic $L$-functions. In particular, we show that the improved stable range in Randal-Williams--Wahl's theorem recently obtained in \cite{MPPRW} implies the conjectured asymptotics for moments in the rational function field case. 
	\end{enumerate} 
Let us first state the topological result. The Poincar\'e series of the stable homologies can be packaged into a generating series as follows:
\end{para}

\begin{thm}\label{thmA}There is an equality
	$$ \sum_i \sum_\lambda \dim H_i(\beta_\infty,S^\lambda(V)) (-z)^i s_{\lambda'} =\Exp (z^{-1}\Log(z + \sum_{j\geq 0} h_{2j}z^j)-1).$$ 
\end{thm}

\begin{para}
	Let us briefly comment on the notation here. The equality takes place in a completion of $\Lambda[z]$, where $\Lambda$ is the ring of symmetric functions in infinitely many variables. We denote by $s_\lambda$ the Schur polynomial associated to the partition $\lambda$; they form a basis for $\Lambda$ as a free abelian group. We denote by $\lambda'$ the conjugate partition of $\lambda$. We denote by $h_r$ the $r$th complete homogeneous symmetric polynomial. By $\Exp$ and $\Log$ we mean the plethystic exponential and plethystic logarithm introduced by Getzler and Kapranov \cite{getzlerkapranov}. These are universal operations acting on any complete $\lambda$-ring, and in particular they act on the completion of $\Lambda[z]$. 
\end{para}

\begin{rem}From the theorem one may in particular algorithmically calculate an expression for the generating series $ \sum_i\dim H_i(\beta_\infty,S^\lambda(V)) (-z)^i$ for any fixed $\lambda$, for example using SAGE's functionality for calculations with symmetric functions. For each $\lambda$ the result is a rational function of $z$ with all its poles on the unit circle. In particular, the stable homology grows polynomially in $i$, for any fixed $\lambda$. 
	We have compiled tables of these rational functions for small values of $\lambda$, as well as tables of Betti numbers for small values of $i$; these tables are accessible on GitHub \cite{tablesjonas}.
\end{rem}

\begin{para}
	As we briefly mentioned and will explain later, these homology groups admit an algebro-geometric interpretation, from which it follows that the homology is ``motivic''. In particular, it admits a mixed Hodge structure, and the structure of an $\ell$-adic Galois representation after tensoring with $\Q_\ell$ for any $\ell$. We show that the stable homology is a pure Tate motive, and we determine the weights in each degree. 
\end{para}

\begin{thm}\label{thmB}
	The stable homology $H_i(\beta_\infty,V^{\otimes k})$ is pure Tate of weight $-2i+k$. 
\end{thm}

\begin{para}That there should be a link between these results and analytic number theory over function fields can in fact be seen already from \cref{thmA} and \cref{thmB}. Indeed, according to \cref{thmB} the eigenvalues of $\mathrm{Frob}_q$ on $H_i(\beta_\infty,V^{\otimes k})$ are $q^{-i+k/2}$, and so we obtain a generating series for the ``stable traces of Frobenius on homology'' from \cref{thmA} (where trace is taken in the graded sense) by making the substitutions $z \leadsto q^{-1}$ and $h_{j} \leadsto q^{j/2}h_j$. Since elements of $\Lambda$ are symmetric polynomials in infinitely many variables $x_1,x_2,x_3,\ldots$, the resulting generating series can be expanded as a power series in these variables. As we will explain, this results --- up to a factor of $(1-q^{-1})$ --- in the following expression:
	\begin{align*} R(x_1,x_2,\ldots) &= \prod_{n>0} \left(1 + \frac{1}{1+q^{-n}} \sum_{j>0} h_{2j}(x_1^n,x_2^n,\ldots) \right)^{\frac 1 n \sum_{d \mid n} \mu(n/d)q^{d}} \\
	&= \prod_{P \text{ monic irreducible}} \left(1 + \frac{1}{1+\vert P \vert^{-1}} \sum_{j>0}\sum_{j_1+\ldots+j_r=2j} \prod_{d} x_d^{j_d \deg P}  \right) \end{align*}
where $\mu$ denotes the M\"obius function, and the product runs over monic irreducible polynomials in $\Fq[t]$, of which there are $\frac 1 n \sum_{d \mid n} \mu(n/d)q^{d}$ of degree $n$, and $\vert P\vert = q^{\deg P}$. 
This function $R$ is the ``arithmetic factor'' appearing in the function field version of the moment conjectures of \cite{CFKRS}, see \cite[Section 4]{AK}.
\end{para}

\begin{para}
The calculations of this paper do not by themselves prove the expected asymptotics of moments. To obtain an asymptotic formula for moments, one needs to control the range of homological stability for the braid group, with coefficients in an irreducible representation of the symplectic group.
\end{para}

\begin{defn}\label{def:USB}We say that a function $\theta$ is a \emph{uniform stability bound} if all stabilization maps $$ H_i(\beta_n,V_\lambda) \longrightarrow H_i(\beta_{n+1},V_\lambda)$$
	are isomorphisms for $i \leq \theta(n)$, for all $\lambda$. Here $V_\lambda$ denotes the irreducible representation of the symplectic group of highest weight $\lambda$; when $n$ is even, $V_\lambda$ is a representation of the ``odd symplectic group'' of Gelfand--Zelevinsky \cite{gelfandzelevinskyodd}. 
\end{defn}

\begin{thm} \label{thmC} Let $q$ be an odd prime power, and $\theta$ a uniform stability bound. Then one has the asymptotic formula for the moments of quadratic $L$-functions
\[ 
\frac{1}{q^{2g+1}}\sum_{\substack{d \text{\emph{ monic  \&  sq.\ free}} \\ \deg d = 2g+1}}
L(\tfrac{1}{2}, \chi_{d})^{r} 
=\,  Q_{r}(2g+1) \, + \, O\!\left(4^{g(r + 1)} q^{-{\theta} (2g+1) \slash 2}\right)
\] 
where $Q_{r}$ is an explicit polynomial of degree $r(r + 1)\slash 2$, and the implied constant depending upon {$\theta$} in the $O$-symbol is explicitly computable. \end{thm}

\begin{para}When this paper was first posted as a preprint, the existence of a nontrivial uniform stability bound was posed as a conjecture. In fact, we conjecture that there exists a uniform stability bound of the form $\theta(n)=\tfrac 1 4 n - c$, where $c\geq 0$ is an absolute constant. The existence of a nontrivial uniform stability bound was subsequently proven in \cite{MPPRW}. Note that \cref{thmC} and the following \cref{MPPRW} together imply \cref{CFKRS conjecture} for $q > 4^{34(r+1)}$. 
\end{para}

\begin{thm}[Miller--Patzt--Petersen--Randal-Williams]\label{MPPRW}
    The function $\theta(n)=(n-35)/34$ is a uniform stability bound.
\end{thm}

\begin{para}\label{not uniform for polynomial}
    The reader may note that \cref{def:USB} involves coefficients in $V_\lambda$, whereas \cref{thmA} involved coefficients in $S^\lambda(V)$. This is a subtle and important point. In fact, the stabilization maps $ H_i(\beta_n,S^\lambda(V)) \longrightarrow H_i(\beta_{n+1},S^\lambda(V))$ \emph{cannot} have a stable range independent of $\lambda$, even for $i=0$. The point is that the branching from $\mathrm{GL}(2g)$ to $\mathrm{Sp}(2g)$ only becomes independent of $g$ in the regime $g > \length(\lambda)$, by Littlewood's stable branching rule.  
\end{para}

\begin{para}The link between homological stability and moments of families of quadratic $L$-functions is mediated through the \emph{Grothendieck--Lefschetz trace formula}, applied to the cohomology of the moduli space $\H_{g,1,0}$ of hyperelliptic curves with a marked Weierstrass point, with coefficients in a symplectic local system. Indeed, the left-hand side of the asymptotic formula in \cref{thmC} is a sum over the $\Fq$-points of $\H_{g,1,0}$, and if we replace each value $\tfrac 1 2$ with a parameter, say $\tfrac 1 2 + it_j$, $j=1,\ldots,r$, then the coefficients of the resulting series are symmetric polynomials in the Frobenius eigenvalues of the corresponding hyperelliptic curve. 
\end{para}

\begin{para}
    The error term occurring in \cref{thmC} is in a sense extremely crude: it arises by bounding the Betti numbers $\dim H^k(\beta_n,V_\lambda)$ by the number of cells in the Fuks stratification of configuration space, multiplied by $\dim V_\lambda$. Nevertheless, it gives a good bound for sufficiently large $q$, because of the factor $q^{- \theta(2g+1) \slash 2}$, arising from the Deligne bounds on Frobenius eigenvalues.
\end{para}

\begin{para}
More general than the moment conjecture of \cite{CFKRS} is the \emph{ratios conjecture} \cite{CFZ}. It is an insight of Wang \cite{Wang} that the methods of this paper can be applied to also derive certain asymptotics for ratios. One obtains a power-saving error term for sufficiently large $q$, as long as all $L$-function arguments in the denominator have real part $\geq \frac 1 2 + q^{-\delta}$, for an explicit $\delta>0$. See \cite{Wang} for a detailed treatment.
\end{para}

\begin{para}
    Methods of algebraic topology have long been leveraged to attack problems in number theory over function fields. In particular, homological stability for Hurwitz spaces has been applied to asymptotic questions in arithmetic statistics in \cite{EVW,ETW17,EL}. For prior geometric and cohomological approaches to moments of families of $L$-functions in the function field setting, see \cite{saw1,saw2,saw3}.
\end{para}

\subsection{Asymptotics for moments of quadratic \texorpdfstring{$L$}{L}-functions} 

\begin{para}
Understanding the value distribution of the {R}iemann zeta-function $\zeta(s)$ on the critical line $\Re(s) = \frac{1}{2}$ and the distribution of central values in families of $L$-functions are problems of fundamental importance in analytic number theory. The general behavior of the values $\zeta(\tfrac{1}{2} + it)$ is described by a fundamental theorem of {S}elberg, which asserts that $\log \zeta(\tfrac{1}{2} + it)$ is distributed like a complex {G}aussian with prescribed mean and variance. 
More refined information, such as understanding the large deviations range in {S}elberg's theorem, can be obtained by studying the asymptotics of \emph{moments}. 
\end{para}
\begin{para}
    The study of moments of the {R}iemann zeta-function began with the work of {H}ardy and {L}ittlewood \cite{Ha-Li}, who obtained the asymptotic formula for the second moment on the critical line, 
\[
\int_{0}^{T}|\zeta(\tfrac{1}{2} + it)|^{2}\, dt \sim T\log T    \qquad \text{(as $T\to \infty$)}.     
\] 
A few years later, {I}ngham \cite{Ing} obtained the asymptotic formula for the fourth moment 
\[
\int_{0}^{T}|\zeta(\tfrac{1}{2} + it)|^{4}\, dt \sim \frac{1}{2\pi^2}T \log^4 T  \qquad \text{(as $T\to \infty$)}.
\] 
At present, no asymptotic formula for higher moments of the {R}iemann zeta-function is known, but it is conjectured that 
\[ 
\int_{0}^{T}|\zeta(\tfrac{1}{2} + it)|^{2r}\, dt \sim C_{r}T\log^{r^2} T \qquad \text{(for $r\ge 3$)}.
\] 
The value of the constant $C_{r}$ was conjectured by {K}eating and {S}naith \cite{Ke-Sn1}, using random matrix theory. For a more accurate conjecture, where the full main terms in the asymptotic formula are predicted, see \cite{CFKRS}. Perhaps the main motivation for studying the moments of $\zeta(s)$ is that they capture information about the large values of 
$|\zeta(\tfrac{1}{2} + it)|$. For instance, the {L}indel\"of {H}ypothesis 
 that $|\zeta(\tfrac{1}{2} + it)| \ll_{\varepsilon} (1 + |t|)^{\varepsilon}$ --- a well-known consequence of the Riemann Hypothesis --- is equivalent to the bound 
\[ 
\int_{0}^{T}|\zeta(\tfrac{1}{2} + it)|^{2r}\, dt \ll_{r, \varepsilon} T^{1 + \varepsilon} 
\] 
for all $r \in \mathbf{N}.$ 
\end{para}

\begin{para}The values $\zeta(\frac 1 2 + it)$ are the central $L$-values of a continuous family of automorphic representations: for any $t \in \R$ we may consider the one-dimensional representation $|\cdot|^{it}$, and its associated $L$-function is $s \mapsto \zeta(s+it)$. One may similarly consider the moments of different natural discrete families of $L$-functions; this represents a central topic in analytic number theory. 
    \end{para}

\begin{para}

Besides the continuous family $\{\zeta(\tfrac{1}{2} + it) : t \in \R\}$, there is yet the symplectic family 
\[
\mathcal{F}_{\mathcal{D}} =\,  \{\chi_{d} : \text{$d\in \Z$ non-zero and square-free}\}
\] 
of real {\it primitive} {D}irichlet characters, which received considerable attention over the past forty years or so. The full moment conjecture of {C}onrey, {F}armer, {K}eating, {R}ubinstein, and {S}naith \cite{CFKRS} predicts the asymptotic formula for this family: 
\[ 
\sum_{\substack{\chi_{d} \in \mathcal{F}_{\mathcal{D}} \\ |d| < D}} L(\tfrac{1}{2}, \chi_{d})^{r}
\sim  DQ(\log D) \qquad \text{(as $D\to \infty$)}
\] 
for some polynomial $Q(x)$ of degree $r (r + 1)\slash 2$; for an explicit description of this polynomial, see \cite{GHRR}. The leading coefficient of $Q(x)$ has been previously obtained by {K}eating and {S}naith \cite{Ke-Sn2}, using random matrix theory calculations. As in the case of the {R}iemann zeta-function, the above asymptotic formula is known only for the first few moments. {J}utila \cite{Jut} established asymptotics for the first two moments, and the third moment was evaluated by {S}oundararajan \cite{Sound}. Subsequent asymptotics for the third moment, with better error terms, were obtained in \cite{DGH} and \cite{Young}. Assuming the generalized {R}iemann {H}ypothesis, {S}hen \cite{Sh} obtained an asymptotic formula for the fourth moment, using the method developed by {S}oundararajan and {Y}oung in \cite{SY}. More recently, Shen and Stucky \cite{SheSt24} unconditionally established an asymptotic formula with four main terms, following a method of Li \cite{XLi}, which had previously been used to establish an asymptotic formula for the second moment of quadratic twists of modular $L$-functions. 
\end{para}
\begin{para}
Diaconu--Goldfeld--Hoffstein \cite{DGH} conjectured the existence of additional terms in the asymptotic formula of the $r$th moment when $r \ge 3$. This expectation was later confirmed in \cite{Di-Wh}, where an asymptotic formula with an additional term of size $D^{\frac{3}{4}}$ has been established for a {\it smoothed} version of the third moment. When $r \ge 4$, we have the following: 
\end{para}

\begin{conj} \label{conj: Moment-conjecture-over-rationals} Let $N \ge 1$ be an integer, and take $(N + 1)^{-1} < \Theta < N^{-1}$. Then 
	\begin{equation*} 
		\sum_{\substack{\chi_{d} \in \mathcal{F}_{\mathcal{D}} \\ |d| < D}} L(\tfrac{1}{2}, \chi_{d})^{r} 
		= \sum_{n = 1}^{N} D^{(\frac{1}{2} + \frac{1}{2n})}Q_{n, r}(\log D)
		\, + \, O\big(D^{(1 + \Theta)\slash 2}\big)
	\end{equation*} 
as $D\to \infty$, for some polynomials $Q_{n, r}(x)$. 
\end{conj} 

\begin{para}

The \emph{general} moment problem for \emph{suitable} families $\mathcal{F}$ of automorphic representations asks for asymptotics of the form 
\[
\frac{1}{D^{*}} \!\sum_{\substack{f \in \mathcal{F} \\ c(f) \le D}} W(L(\tfrac{1}{2}, f))^{r}  \sim g_{r} 
\frac{a_{r}}{\Gamma(1 + B(r))}(\log D^{A})^{B(r)} \qquad \text{(as $D\to \infty$)}
\] 
where $r \in \mathbf{N},$ and the $L$-functions associated to $f \in \mathcal{F}$ are normalized to have functional equations as $s \to 1 - s$ (so $L(\tfrac{1}{2}, f)$ is the central value). Here we think of $\mathcal{F}$ as being partially ordered by the ``conductor'' $c(f)$, with $D^{*}$ the number of elements with $c(f) \le D.$ The function $W(z)$ equals $z$ or $|z|^{2}$, depending on the {\it symmetry} type of $\mathcal{F}$, introduced by {K}atz and {S}arnak \cite{KaS}. The constants $g_{r}$ and $B(r)$ depend also only on the symmetry type of $\mathcal{F}$, and the constant $a_{r}$, usually referred to as the \emph{arithmetic factor}, depends on the family in a natural way. Finally, the constant $A$ depends both on the symmetry type and the functional equation satisfied by the elements in $\mathcal{F}.$ 

\end{para}

\begin{para}An important feature of the family $\mathcal{F}_{\mathcal{D}}$, compared to $\{\zeta(\tfrac{1}{2} + it) : t \in \R\}$, is that its moments admit a perfect function field analogue. In this respect, most of the existing results are in the context of rational function fields over finite fields of odd characteristic. 
Prior to this work, asymptotics were known only for the first four moments, largely by the work of {F}lorea \cite{Florea1, Florea2, Florea3}. Previously, {H}offstein and {R}osen \cite{HR} obtained an asymptotic formula for the first moment, albeit with a weaker error term. Following the ``recipe'' developed in \cite{CFKRS}, {A}ndrade and {K}eating \cite{AK} obtained the corresponding conjectural asymptotics for all moments (see also \cite{RW}), and the analogue of \cref{conj: Moment-conjecture-over-rationals} over rational function fields in \cite{DT} is: \end{para}
\begin{conj}Let $D$, $N$ and $r \geq 4$ be integers. Take $(N + 1)^{-1} < \Theta < N^{-1}$. Then \label{conj: Moment-conjecture-over-ff}
\begin{equation*}  
\sum_{\substack{d \text{ \emph{monic \& sq.\ free}} \\ \deg d = D}}
L(\tfrac{1}{2}, \chi_{d})^{r} 
= \sum_{n \le N} Q_{r, n}(D, q)q^{(\frac{1}{2} 
+ \frac{1}{2n})D} + \, 
O\big(q^{D (1 + \Theta)\slash 2}\big).
\end{equation*} 
     The coefficients $Q_{r, n}(D, q)$ are computable, and the implied constant in the $O$-symbol depends on $\Theta, q$ and $r$, but not on $D$. 
\end{conj}\begin{para}When $r = 4$, a weighted version of \cref{conj: Moment-conjecture-over-ff} was established in \cite{DPP}, and when $r = 3$, an asymptotic formula analogous to that in \cite{Di-Wh}, with a secondary term of size $q^{\frac{3}{4}D}$, was proved in \cite{Dia}. 
\end{para} 

\begin{para}\label{traces-complete-asympt}
More generally, one can formulate a precise variant of \cref{conj: Moment-conjecture-over-ff} for shifted moments; see \cite{DT}. By applying a result of the second-named author with Pa\c sol in \cite{DV}, together with the Weyl integration formula for the unitary symplectic group of dimension $2r$, one can interpret the shifted moment conjecture when $D = 2g + 1$ in terms of asymptotics for the traces of Frobenius on the \'etale (co)homology with coefficients in symplectic local systems of the moduli space of smooth hyperelliptic curves of genus $g$ with a marked Weierstrass point. More precisely, if $\tr_{\lambda}(g)$ denotes the alternating sum of the traces of Frobenius on these homology groups, then one is led to the conjectural asymptotic formula:
\[
\tr_{\lambda}(g)
=
\sum_{n \le N} \mathcal{T}_{n}(\lambda^{\dag})\,
+ \, O\big(q^{((2g + 1)(\Theta-1))\slash 2}\big)
\]
for all partitions
$
\lambda = (\lambda_{1} \ge \cdots \ge \lambda_{g} \ge 0)
$,
with $\lambda_{1} \le r$. The terms in the right-hand side of the asymptotic formula depend on the partition
$
\lambda^{\dag} =
(g - \lambda_{r}' \ge \cdots \ge g - \lambda_{1}')
$,
where the nonzero integers among $\lambda_{j}'$, $1 \le j \le r$, are the parts of the conjugate partition $\lambda'$ of $\lambda$, and $(N + 1)^{-1} < \Theta < N^{-1}$ is the same as in \cref{conj: Moment-conjecture-over-ff}. All terms involved are described in more detail in \cref{predictions}. Conversely, this conjectural asymptotic formula for $\tr_{\lambda}(g)$ implies \cref{conj: Moment-conjecture-over-ff}, but with a weaker error term, since one must sum over all partitions whose Young diagrams fit inside a rectangle of $g$ rows and $r$ columns, which introduces an additional polynomial factor in $g$.
\end{para}


\begin{para}
The main motivation of our results comes from an attempt to explain the presence of each term $\mathcal{T}_{n}(\lambda^{\dag})$ in the asymptotic formula of $\tr_{\lambda}(g)$ from the topological and algebro-geometric perspectives. It turns out that the asymptotics of $\tr_{\lambda}(g)$ depend on different regimes of the ratio 
$0 \le \frac{|\lambda|}{g} \le r$ where terms $\mathcal{T}_{n}(\lambda^{\dag})$, for specific $n$, dominate all the other terms. This observation points towards the existence of some {\it higher-order} homological stability statements, which, in turn, would explain the significant changes in the behavior of the traces $\tr_{\lambda}(g)$ in specific regimes of $\lambda$. 
In this paper, we will focus on analyzing the first term $\mathcal{T}_{1}(\lambda^{\dag})$, and determine its precise connection with the homological stability satisfied by the sequence of braid groups. This term dominates all the other terms for small partitions, but as the weight of $\lambda$ increases, the size of $\mathcal{T}_{1}(\lambda^{\dag})$ decreases. Ultimately, one would like to show that the traces 
$\tr_{\lambda}(g)
$ 
decrease, as $|\lambda|$ grows with $g$. As we will see in the proof of Theorem \ref{thmC}, the {\it uniform} homological stability assumption will be precisely used to obtain an explicit estimate for these traces.
\end{para}

\begin{para} A number-theoretically inclined reader may want to start reading Sections \ref{point counting} and \ref{predictions}, which discuss the number-theoretic results in this paper. 
In particular, in Subsection \ref{subsec: genfunc} we study certain generating functions for the traces of 
{F}robenius on the \'etale homology with 
twisted coefficients $S^\lambda(\mathbb V)$ 
of the moduli stacks $\H_g^{1,0}$ and 
$\H_g^{0,1}$ of smooth hyperelliptic curves of 
genus $g$, and show that the limit as 
$g \to \infty$ of each of these functions exists 
and it matches, after some simple substitutions, 
the right-hand side of the equality in Theorem \ref{thmA}; here the limit exists in the sense that the coefficients of the generating functions converge. This result will be deduced from an asymptotic formula for the sum over all monic square-free polynomials $d \in \F_{q}[x]$ of a fixed degree of the elementary symmetric polynomial 
$
e_{\lambda} = e_{ \lambda_{1}}e_{ \lambda_{2}} \cdots
$ 
in the {F}robenius eigenvalues attached to the hyperelliptic curve with affine equation $y^{2} = d(x)$. In Subsection \ref{subsec: moments from hom stab} we prove a slightly more general version of Theorem \ref{thmC}. 

In the concluding Section \ref{predictions}, by assuming the asymptotic formula of \cref{conj: Moment-conjecture-over-ff}, we will determine the asymptotic behaviour of the traces $\tr_{\lambda}(g)$ as the genus $g$ grows, 
for all partitions $\lambda$ whose Young diagrams 
fit inside a $g \times r$ rectangle, for a fixed positive integer $r$. We then analyze the first two terms $\mathcal{T}_{1}(\lambda^{\dag})$ and $\mathcal{T}_{2}(\lambda^{\dag})$ of the asymptotic formula of $\tr_{\lambda}(g)$, determine the precise connection of $\mathcal{T}_{1}(\lambda^{\dag})$ with the homological stability satisfied by the sequence of braid groups, and show that the second term $\mathcal{T}_{2}(\lambda^{\dag})$ dominates the first term for partitions $\lambda$ such that $g - \lambda_{i}' = O(1)$ for $i = 1, 2, 3$, and 
$\lambda_{i}' = O(1)$ for $i \ge 4$. In particular, if $\lambda_{i} = 3$ for all $i$, then we have unconditionally that 
\[
\tr_{\lambda}(g) \sim \mathcal{T}_{2}(\lambda^{\dag})
\] 
as $g \to \infty$. 
\end{para}

\begin{para}
    Given that our results provide a cohomological interpretation of the first term $\mathcal T_1(\lambda^\dag)$, it is natural to ask whether in addition the higher terms $\mathcal T_n(\lambda^\dag)$, $n \geq 2$, can be understood similarly. Since $\mathcal T_1$ is described by the \emph{stable homology}, a natural speculation is that the higher order terms correspond to \emph{secondary stability}, and higher-order stability, in the sense of Galatius--Kupers--Randal-Williams \cite{gkrw-secondary}. 
\end{para}

\subsection{Braid groups and moduli spaces}

\begin{para}\label{big diagram}As already stated, this paper aims in particular to study the stable homology of braid groups with certain algebraic coefficients. The family of all groups involved can be assembled in the following diagram:\[ \begin{tikzcd}
    \ldots \arrow[r]& \beta_{2g+1} \arrow[r]\arrow[d]&\arrow[r]\arrow[d] \beta_{2g+2} & \arrow[r]\arrow[d]\beta_{2g+3} &\arrow[r]\arrow[d] \beta_{2g+4} & \ldots \\
   \ldots \arrow[r]& \mathrm{Mod}_g^1 \arrow[r]\arrow[d]&\arrow[r]\arrow[d] \mathrm{Mod}_g^2 &\arrow[r]\arrow[d] \mathrm{Mod}_{g+1}^1 &\arrow[r]\arrow[d] \mathrm{Mod}_{g+1}^2 & \ldots  \\
  \ldots \arrow[r]& \arrow[r] \Sp(2g,\Z) &\arrow[r] \Sp(2g+1,\Z) &\arrow[r] \Sp(2g+2,\Z) & \arrow[r]\Sp(2g+3,\Z)& \ldots 
\end{tikzcd}\]
Let us define our notation:
\begin{itemize}
    \item We denote by $\beta_n$ Artin's braid group on $n$ strands, and $\beta_n \to \beta_{n+1}$ is the natural embedding which adds another strand to the braid.
    \item We denote by $\mathrm{Mod}_g^n$ the mapping class group of a compact oriented surface of genus $g$ with $n$ boundary components. The map $\mathrm{Mod}_g^1 \to \mathrm{Mod}_g^2$ glues on a pair of pants along its waist, and the map $\mathrm{Mod}_g^2 \to \mathrm{Mod}_{g+1}^1$ glues on a pair of pants along its two legs. 
    \item We denote by $\Sp(2g)$ the usual symplectic group of $2g \times 2g$ matrices. The group $\Sp(2g-1)$ is the \emph{odd symplectic group} of Gelfand--Zelevinsky \cite{gelfandzelevinskyodd}; $\Sp(2g-1)$ is the subgroup of $\Sp(2g)$ fixing a vector in the standard $2g$-dimensional symplectic vector space (a unimodular vector, when we take integer coefficients). These are not the same as the odd symplectic groups defined by Proctor \cite{proctor}, although they surject onto Proctor's groups.
    \item The map $\beta_{2g+1} \to \mathrm{Mod}_g^1$ is defined as follows: consider the configuration space of $2g+1$ unordered points in a disk. For each such configuration there is a unique branched double cover of the disk, ramified exactly over each point of the configuration. This double cover is a genus $g$ surface with one boundary component, and in this way we obtain a map from the moduli space of configurations on the disk to the moduli space of genus $g$ surfaces; on fundamental groups, it induces $\beta_{2g+1} \to \mathrm{Mod}_g^1$. The map $\beta_{2g+2} \to \mathrm{Mod}_g^2$ is defined exactly the same way, but when the number of branch points is even the monodromy along the boundary of the disk is trivial, and we obtain two boundary components on the double cover. 
    \item The map $\mathrm{Mod}_g^1 \to \Sp(2g,\Z)$ associates to a mapping class of the surface $S_g$ its action on the first cohomology $H^1(S_g,\Z)$. The composite $ \mathrm{Mod}_g^2 \to \mathrm{Mod}_{g+1}^1 \to \Sp(2g+2,\Z)$ lands in the subgroup consisting of symplectic transformations fixing the class of the curve given by either of the two bottom hems of the pair of pants, which is precisely $\Sp(2g+1,\Z)$.  
\end{itemize}\end{para}

\begin{para}For each $n$ there is a natural map $\Sp(n) \to \GL(n)$: for $n$ even it is the usual inclusion, and for $n$ odd it is defined by assigning to a symplectic vector space $V$ and a vector $v \in V$ the symplectic complement of the line spanned by $v$.     The composite map $\beta_{n+1} \to \GL(n,\Z)$ is the integral reduced Burau representation. The maps from the braid groups to the mapping class groups, and the reduced Burau representation, have been studied by many people over the years, including \cite{birmanhilden,a'campo,songtillmann,segaltillmann,bandboyland,bmp,chen-burau,bianchi-braid,harrvistrupwahl,bloomquist}. It seems in some situations more natural to think of the Burau representation as taking values in the symplectic group, in particular as the map $\Sp(n) \to \GL(n)$ fails to be injective for odd $n$. 
 \end{para}

 \begin{rem}   
    In general, the (reduced) Burau representation depends on a parameter $t$, and the \emph{integral} (reduced) Burau representation denotes its specialization as $t=-1$. A similar topological interpretation of the representation exists more generally when $t$ is specialized to a primitive $r$th root of unity, except we take an $r$-fold cyclic branched cover of the disk, and instead of taking $H^1$ of the covering surface we pick out one isotypical component of it under the action by the group of deck transformations. The case $t=-1$ is called the ``integral'' Burau representation since it is the only nontrivial root of unity for which the matrices have integer coefficients. 
    
    \end{rem}
    
\begin{para}While the homomorphism from the mapping class group to $\Sp(n,\Z)$ is surjective for all $n$, the Burau representation is not: it turns out that there is a natural embedding of the symmetric group into the mod $2$ symplectic group, $\Sigma_{n+1} \to \Sp(n,\Z/2)$, in such a way that the projection $\beta_{n+1} \to \Sp(n,\Z/2)$ coincides with the standard surjection $\beta_{n+1} \to \Sigma_{n+1}$, and the image of the integral reduced Burau representation in $\Sp(n,\Z)$ is the preimage of $\Sigma_{n+1}$ under the reduction map to $\Sp(n,\Z/2)$. When $n$ is even, this is a theorem of A'Campo \cite{a'campo}; for the case of odd symplectic groups see \cite{bloomquist}. In particular, the image of the Burau representation has finite index.
\end{para}

\begin{para}
    Each row in the diagram of \S\ref{big diagram} satisfies \emph{homological stability}. Homological stability of braid groups is due to Arnold \cite{arnoldbraid2}. The case of mapping class groups is due to Harer \cite{harerstability}, see also \cite{boldsen,wahlstability}. There has been extensive study of homological stability of groups $G_n(A)$ where $G_n$ is a family of classical groups and $A$ is a ring, following the original work of Quillen for $G_n=\mathrm{GL}_n$ in connection with algebraic $K$-theory. The case of symplectic groups over a Dedekind domain is originally due to Charney \cite{charney}. Note though, that Charney only considers the usual even symplectic groups; homological stability for the family of odd and even symplectic groups together is considered in \cite{sierrawahl,MPPRW}. For a unified perspective on these homological stability results, see \cite{randalwilliamswahl}. As explained in recent work of Harr--Vistrup--Wahl \cite{harrvistrupwahl}, homological stability for the system of mapping class groups alternating between one and two boundary components is particularly interesting.
    \end{para}\begin{para}\label{stable homologies}
    In all three cases, more is true than homological stability: there is a \emph{monoidal structure}, which makes the disjoint union of classifying spaces of the groups involved into an $E_1$-algebra. By the group-completion theorem \cite{mcduffsegal} the stable homology can be identified with the homology of the group-completion, and this group-completion can in the three respective cases be explicitly identified:
    \begin{itemize}
        \item Classifying spaces of braid groups group-complete to $\Omega^2 S^2$, by a theorem of Segal \cite{segalconfiguration}, see also \cite[p. 226]{cohenladamay}.
        \item Classifying spaces of mapping class groups group-complete to $\Omega^\infty \mathrm{MTSO}(2)$, by the Madsen--Weiss theorem \cite{madsenweiss}.
        \item Classifying spaces of symplectic groups group-complete to $\Omega^\infty K\Sp(\Z)$, where $K\Sp(\Z)$ denotes the symplectic $K$-theory spectrum of $\Z$. This is one of many ways of \emph{defining} higher algebraic $K$-theory.
    \end{itemize}
\end{para}

\begin{para}The descriptions in the preceding paragraph can be used as a tool to compute the stable (co)homology of the groups involved. Notably, this is the case for the Madsen--Weiss theorem, which was used to resolve the \emph{Mumford conjecture} \cite{mumfordtowards}: rationally and in the stable range, the cohomology of $\mathrm{Mod}_g^1$ coincides with the polynomial algebra $\Q[\kappa_1,\kappa_2,\ldots]$ on the Morita--Miller--Mumford classes, with $|\kappa_i| = 2i$.  The cohomology of $\beta_n$ is less interesting rationally: for every $n>1$, the circle action on the configuration space of $n$ points in the plane induces an isomorphism $H^\bullet(\beta_n,\Q) \cong H^\bullet(S^1,\Q)$ \cite{arnoldbraid2}. The stable rational cohomology of $\Sp(2g,\Z)$ was calculated by Borel \cite{borelstablereal1}: it is a polynomial algebra $\Q[\lambda_1,\lambda_3,\lambda_5,\ldots]$ with $\vert \lambda_i \vert = 2i$. 
\end{para}

\begin{para}\label{stable twisted par}
    We may instead consider stable cohomology with \emph{twisted} coefficients. To a partition $\lambda = (\lambda_1\geq \lambda_2 \geq \ldots \geq \lambda_g \geq 0)$ we may associate an irreducible representation $V_\lambda$ of the symplectic group $\Sp(2g)$, and by padding with zeroes we may consider $V_\lambda$ as a system of representations of all larger (even) symplectic groups. (It would take us too far afield to work out the representation theory of odd symplectic groups in this paper.) The stable cohomology of $\Sp(2g,\Z)$ with  coefficients in $V_\lambda$ was also computed by Borel \cite{borelstablereal}, and the stable cohomology vanishes for all $\lambda \neq 0$. The stable cohomology of $\mathrm{Mod}_g^1$ with $V_\lambda$ coefficients was determined by Looijenga \cite{looijengastable}, conditional on the Mumford conjecture. (Looijenga works with closed surfaces, but his arguments extend straightforwardly to the case of surfaces with boundary.) Given these results, it is natural to also try to compute the stable cohomology of braid groups with coefficients in $V_\lambda$, and our \cref{thmA} accomplishes this.
\end{para}


\begin{para}
    We can also give these maps algebro-geometric meaning. Let $\M_g^n$ denote the moduli space parametrizing smooth genus $g$ algebraic curves with $n$ distinct ordered marked points and a nonzero tangent vector at each marking. Its fundamental group is $\mathrm{Mod}_g^n$. We may consider the closed subspace $\H_g^{1,0} \subset \M_g^1$ parametrizing \emph{hyperelliptic curves} marked at a \emph{Weierstrass point}. We have $\H_g^{1,0} \cong \mathrm{Conf}_{2g+1}(\mathbb A^1)/\mathbb G_a$. There is a diagram of moduli spaces
    $$ \H_g^{1,0}\to \M_g^1 \to \mathcal A_g$$
    where the second map is the \emph{Torelli morphism} from the moduli space of curves to the moduli space of principally polarized abelian varieties, assigning to a curve its Jacobian. All three spaces are $K(\pi,1)$, and on fundamental groups the morphisms correspond to the maps 
    $$ \beta_{2g+1}\to\mathrm{Mod}_g^1\to \Sp(2g,\Z).$$
    In particular, the cohomology of these groups with finite and adic coefficients carries an action of $\mathrm{Gal}(\overline{\Q}/\Q)$, and the rational cohomology has a canonical mixed Hodge structure. 
\end{para}

\begin{para}
    Also the maps $\beta_{2g+2}\to\mathrm{Mod}_g^2\to \Sp(2g+1,\Z)$ admit an interpretation in terms of algebraic geometry, but this is slightly more subtle. Since we will not need this, we allow ourselves to merely sketch the construction. If $A$ is an abelian variety, then extensions $$1 \to \mathbb G_m \to B \to A \to 0$$ are naturally parametrized by the dual abelian variety $A^\vee$. In families, this means in particular that the universal principally polarized abelian variety $\mathcal X_g \to \mathcal A_g$ carries a natural family of $(g+1)$-dimensional \emph{semiabelian} varieties of torus rank $1$. This defines a degree $2$ morphism $\mathcal X_g \to \widetilde{ \mathcal A}_{g+1}$ onto the locus of rank $1$ degenerations. Here $\widetilde{ \mathcal A}_{g+1}$ denotes a toroidal compactification of $\mathcal A_{g+1}$. It is well known that although such a toroidal compactification depends on choices of highly noncanonical data, the locus of rank $\leq 1$ degenerations is independent of choices. The rank $1$ locus is a divisor in $\widetilde{ \mathcal A}_{g+1}$, and we may pull back its normal bundle to $\mathcal X_g$. The complement of the zero-section in this bundle is a $K(\pi,1)$ for the group $\Sp(2g+1,\Z)$. Now we may first consider the space $\M_{g,2}$ of smooth genus $g$ curves with two distinct markings. We have a diagram
    \[\begin{tikzcd}
        \M_{g,2} \arrow[r,"\nu"]\arrow[d]&\arrow[d] \overline{\M}_{g+1} \\
        \mathcal X_g \arrow[r]& \widetilde{\mathcal A}_{g+1}
    \end{tikzcd}\]
    where the horizontal arrows both are double covers of codimension $1$ boundary strata, the map $\M_g\to\mathcal X_g$ assigns to $(C,p,q)$ the pair of $\mathrm{Jac}(C)$ and the line bundle $\mathcal O_C(p-q)$, and $\overline{\M}_{g+1}\to\widetilde{\mathcal A}_{g+1}$ is the extended Torelli map. (We suppose that our toroidal compactification has been chosen so that the Torelli map extends, e.g.\ the second Voronoi compactification.) Now from the data of a nonzero tangent vector at the points $p$ and $q$ we obtain a nonzero normal vector to the stratum at the point $\nu(C,p,q)$, which maps to a nonzero normal vector to the rank $1$ divisor on $\widetilde{\mathcal A}_{g+1}$. Hence $\M_g^2$ maps into the space considered just above, producing the map $\mathrm{Mod}_g^2 \to \Sp(2g+1,\Z)$. Moreover, on $\M_g^2$ we may consider the locus $\H_g^{0,1}$ of hyperelliptic curves for which the two markings are conjugate under the hyperelliptic involution. This locus is isomorphic to $\mathrm{Conf}_{2g+2}(\mathbb A^1)/\mathbb G_a$, and $\H_g^{0,1}\to\M_g^2$ induces the map $\beta_{2g+2}\to\mathrm{Mod}_g^2$. 
\end{para}

\begin{para}
    As a consequence of our algebro-geometric discussion, we may think of our calculation as determining the \emph{stable homology of moduli spaces of hyperelliptic curves with symplectic coefficients}. This interpretation is also what connects the topological results to the questions of moments of families of quadratic $L$-functions over function fields.
\end{para}

\begin{para}
    The \emph{horizontal} maps in the diagram of \S\ref{big diagram} are much harder to interpret in terms of algebraic geometry, except for the map $\Sp(2g,\Z)\to\Sp(2g+2,\Z)$ which can be understood via the map $\mathcal A_g \to \mathcal A_{g+1}$ given by taking the product with some fixed elliptic curve. But there is for example no morphism of varieties $\mathrm{Conf}_n(\mathbb A^1) \to \mathrm{Conf}_{n+1}(\mathbb A^1)$ implementing the natural embedding of braid groups. The stabilization maps can however be understood by means of \emph{logarithmic algebraic geometry}. Informally, one chooses a suitable (partial) compactification of the spaces involved in such a way that the boundary is a normal crossing divisor. This divisor then defines a logarithmic structure, whose effect on the analytification of the variety is to take a real oriented blow-up of the boundary; in particular, this blow-up has the same homotopy type as the interior. But now the stabilization maps can be understood as inclusions of boundary strata, and are logarithmically meaningful. We elaborate on this in the body of the paper; the important consequence is that the stable homology has a well-defined Galois action. 
\end{para}

\begin{para}
    One consequence of \cref{thmB}, which identifies the Galois action on the stable homology of $\beta_n$ with symplectic coefficients, is that the map $ H_\bullet(\beta_{2g+1},V_\lambda) \to H_\bullet(\mathrm{Mod}_g^1,V_\lambda) $ is zero in the stable range. Indeed, the source and target have different weights. This fits philosophically with earlier results of Song--Tillmann \cite{songtillmann} and Bianchi \cite{bianchi-braid}, who proved the analogous vanishing statement for homology with arbitrary constant coefficients (including integrally), and for coefficients in the standard local system $V$ (integrally), respectively. But this vanishing statement for $V_\lambda$ coefficients is not, strictly speaking, new. It is known that the stable cohomologies $H^\bullet(\mathrm{Mod}_g^1,V_\lambda)$, for all $\lambda$, are generated under natural tensor operations by the class of the Ceresa cycle in $H^1(\mathrm{Mod}_g^1,\wedge^3(V))$, see \cite{kawazumimorita} and \cite[Section 12]{petersentavakolyin}. But the Ceresa cycle vanishes on the hyperelliptic locus.
\end{para}

\begin{para}
    Algebraic geometers are often more interested in the moduli spaces $\M_g$ and $\H_g$ of \emph{closed} surfaces. In case of closed surfaces the notion of homological stability is more slippery, as there are no natural maps between the respective moduli spaces to induce the stabilization isomorphisms. For homology with constant coefficients, Harer proved that all maps
    $$ \begin{tikzcd}
    H_i(\mathrm{Mod}_g^1,\Z) \arrow[d]\arrow[r]& H_i(\mathrm{Mod}_{g+1}^1,\Z) \arrow[d]\\
    H_i(\mathrm{Mod}_g,\Z) & H_i(\mathrm{Mod}_{g+1},\Z)\end{tikzcd}
    $$are isomorphisms in a stable range, giving a canonical identification $H_i(\mathrm{Mod}_g,\Z) \cong H_i(\mathrm{Mod}_{g+1},\Z)$ stably. For homology with twisted coefficients, this is not true (although see \cref{standard coefficient system}). Nevertheless, it is still the case 
that $H_i(\mathrm{Mod}_g,S^\lambda(V)) \cong H_i(\mathrm{Mod}_{g+1},S^\lambda(V))$ stably, but less canonically. In Section \ref{closed surfaces} we prove analogous results for moduli spaces of closed hyperelliptic surfaces, and calculate explicitly the stable homology.
\end{para}

\subsection{Outline of arguments}

\begin{para}Since the paper is quite long, it may be helpful to sketch the arguments leading to \cref{thmA} and \cref{thmB}. Our calculation of the stable homology of $\beta_n$ with symplectic coefficients is modelled on earlier calculations for the mapping class group. As mentioned in \S\ref{stable twisted par}, the homology of $\mathrm{Mod}_g^1$ with symplectic coefficients was first determined by Looijenga. Unfortunately his methods do not appear to generalize at all to the hyperelliptic case. However, the result was later re-proved by Randal-Williams \cite{randalwilliamstwisted} by a completely different method, which turns out to admit a hyperelliptic analogue. 
\end{para}

\begin{para}
    The starting point of Randal-Williams is a topological moduli space $M_g^1(X)$ of surfaces with boundary, equipped with a continuous map to a target space $X$, taking the boundary to the basepoint of $X$.  This space fits in a fiber sequence
    \begin{equation}
         \map_\ast(S_g^1/\partial S_g^1,X) \to M_g^1(X) \to M_g^1, \label{fiber seq}
    \end{equation}
    where $S_g^1$ denotes a ``reference surface'' of genus $g$ with one boundary component.\footnote{We write $M_g^1$ rather than $\M_g^1$ as we shall attempt to distinguish notationally --- here, and in the body of the paper --- between the topological moduli space $M_g^1=B\mathrm{Diff}_\partial(S_g^1)$ and the algebro-geometric moduli space $\M_g^1$, which is a stack over $\mathrm{Spec}(\Z)$.} Randal-Williams chooses $X$ to be an Eilenberg--MacLane space $K(A,n)$ for a $\Q$-vector space $A$. Then $$\map_\ast(S_g^1/\partial S_g^1,X) \simeq K(A \otimes V,n-1) \times K(A,n-2),$$ where $V=H^1(S_g^1,\Q)$. The second factor $K(A,n-2)$ is mostly a nuisance, and we will simply pretend that it does not exist in this introduction. Then the homology of the fiber of \eqref{fiber seq} is the free graded-commutative algebra on $A \otimes V$, placed in degree $n-1$. By the Cauchy identity we have $\mathrm{Sym}(A \otimes V) = \bigoplus_\lambda S^\lambda(A) \otimes S^\lambda(V)$. One can show that the Serre spectral sequence for the fibration \eqref{fiber seq} degenerates immediately. The conclusion is then that 
    $$ H_\bullet(M_g^1(K(A,n)),\Q) \cong \bigoplus_\lambda S^\lambda(A) \otimes H_{\bullet-\vert\lambda\vert (n-1)}(M_g^1,S^\lambda(V)),$$
  if $n$ is odd. (If $n$ is even the homology of the fiber is instead an exterior algebra which changes the formula slightly.) What is important in the formula is the following:
  \begin{enumerate}
      \item The functor $A \mapsto H_\bullet(M_g^1(K(A,n)),\Q)$ is an \emph{analytic functor} of $A$, i.e.\ an infinite direct sum of Schur functors.
      \item The \emph{coefficients} of this analytic functor are exactly the cohomology groups we are trying to determine.
  \end{enumerate}
  It follows that if we can somehow determine the coefficients of this analytic functor by any other means, then we can immediately read off the homologies $H_\bullet(M_g^1,S^\lambda(V))$. But stably we \emph{can} determine these coefficients, by the results of Galatius--Madsen--Tillmann--Weiss \cite{gmtw}: the stable homology of $M_g^1(X)$ is the homology of the infinite loop space $\Omega_0^\infty (\mathrm{MTSO}(2) \wedge X_+)$, and it is an easy exercise to compute the rational homology of this space as an analytic functor in $A$, when $X=K(A,n)$. 
\end{para}

\begin{para}
    The idea is now to repeat the exactly analogous strategy for moduli spaces of hyperelliptic surfaces. Although this is not essential, it will be slightly more natural to suppose that the target space $X$ carries a $C_2$-action, and define $H_g^1(X)$ to be the moduli space of hyperelliptic surface equipped with a $C_2$-equivariant map to $X$ taking the boundary to the basepoint, where $C_2$ acts on the hyperelliptic surface by the hyperelliptic involution. The first parts of Randal-Williams's arguments carry over essentially verbatim to show that $A \mapsto H_\bullet(H_g^1(K(A,n)),\Q)$ is an analytic functor of $A$, and that its coefficients are the homology groups we are trying to compute. Thus all what remains to carry out the argument is to determine the analogue of Galatius--Madsen--Tillmann--Weiss in the hyperelliptic case. 
\end{para}

\begin{para}
    For hyperelliptic surfaces \emph{without} the map to a target space, the analogue of the Madsen--Weiss theorem is Segal's theorem, mentioned in \S\ref{stable homologies}, that the classifying spaces of braid groups (i.e.\ moduli spaces of hyperelliptic surfaces with boundary) group-complete to $\Omega^2 S^2$. Thus our task is to understand how Segal's theorem changes when the hyperelliptic curves are equipped with a map to a target space. 
\end{para}

\begin{para}One method of proving Segal's theorem goes via the \emph{scanning map}.    Let us briefly explain what it does. Consider a configuration of points in $\R^2$, and imagine it as dots on a piece of paper, placed in a document scanner (like you would find in any office).\footnote{The word ``scanning'' was coined by Segal, who  however describes the procedure not in terms of an actual scanner but a microscope \cite{segalrationalfunctions}.} The scanner will sweep vertically across the paper, and at a given time the scanner only sees those dots which lie along a thin horizontal slice of the paper. Mathematically we can model this as follows. Let $Y$ be the configuration space parametrizing distinct unordered points on the horizontal strip $\R \times (-\epsilon,\epsilon)$, topologized in such a way that points can appear and disappear along the top and bottom of the strip. We take the basepoint of $Y$ to be the empty configuration. For each $n$, we then have a natural map $\mathrm{Conf}_n(\R^2) \to \map_c(\R,Y)$ --- where $\map_c$ denotes the space of maps constant at the basepoint outside a compact set --- defined as follows: for any configuration $z \in \mathrm{Conf}_n(\R^2)$ we obtain a map $\R \to Y$ sending $t \in \R$ to the configuration $z \cap (\R \times (t-\epsilon,t+\epsilon))$. Now the point is that $\map_c(\R,Y) \simeq \Omega Y$, and that one can prove that when we take the disjoint union over all $n$ then the map to $\Omega Y$ turns out to be a group-completion. In particular, the maps $\mathrm{Conf}_n(\R^2) \to \Omega Y$ become more and more connected as $n$ grows, and the stable homology of braid groups can be identified with the homology of $\Omega Y$.
\end{para}

\begin{para}
    The scanning map described in the preceding paragraph can be \emph{iterated}: a further delooping, i.e.\ a delooping of $Y$, is obtained as the space of configurations of points in a small square $(-\epsilon,\epsilon)^2$, where points are allowed to  create and annihilate along all sides. To finish the proof of Segal's theorem, one shows that the latter space has the homotopy type of $S^2$. This is done by a process of ``zooming in'' on a configuration, which has the effect of deformation retracting down to the subspace of configurations of at most one point in the square $(-\epsilon,\epsilon)^2$. But the latter subspace is homeomorphic to the one-point compactification of $(-\epsilon,\epsilon)^2$.
\end{para}

\begin{para}
    If one carries out this argument while keeping track of the data of an equivariant map to a target space $X$, it works out very similarly all the way to the final step, where we identify the homotopy type of the subspace of configurations of at most one point in $(-\epsilon,\epsilon)^2$. When $X=K(A,n)$, one finds (after localizing away from the prime $2$) that this subspace has the homotopy type of the wedge sum $S^2 \vee K(A,n)/\langle \iota\rangle$, where $\iota$ is the involution induced by multiplication by $-1$ on $A$. What remains is to compute the rational homology of the double loop space $\Omega^2_0 (S^2 \vee K(A,n)/\langle\iota\rangle)$, treated as an analytic functor of $A$. 
\end{para}

\begin{para}
    We can now explain why computing the rational homology of $\Omega^2_0 (S^2 \vee K(A,n)/\langle\iota\rangle)$ leads to the formula $$\Exp(z^{-1}\Log(z+\sum_{j\geq 0} h_{2j}z^j)-1)$$ given in \cref{thmA}. The space $W=S^2 \vee K(A,n)/\langle\iota\rangle$ is formal, and $H^\bullet(W,\Q)$ is a Koszul quadratic algebra. Hence its quadratic dual algebra is $H_\bullet(\Omega W,\Q)$, by a result of Berglund \cite{berglundkoszulspaces}, and $z+\sum_{j\geq 0} h_{2j}z^j$ is nothing but the coefficients of the analytic functor of this quadratic dual. Now the rational homology of a loop space is the symmetric algebra on the rational homotopy, and taking the symmetric algebra corresponds to the functor $\Exp$. Hence $\Log(z+\sum_{j\geq 0} h_{2j}z^j)$ gives the rational homotopy groups of $\Omega W$. Dividing by $z$ shifts homotopy groups down one degree, and exponentiating again gives the homology of the double loop space. The subtraction by $1$ corresponds to the fact that we only want the homology of the base component. 
\end{para}

\begin{para}
    This type of quadratic duality argument is also what allows us to prove \cref{thmB}. We want to understand the Galois action on $H_\bullet(\Omega^2 W,\Q_\ell)$. But this space has a Galois-equivariant Gerstenhaber algebra structure, and it then suffices to compute the Galois action on Gerstenhaber indecomposables. Quadratic duality for Gerstenhaber algebras shows that the space of indecomposables is dual to the space of indecomposables in $H^\bullet(W,\Q_\ell)$. But the indecomposables in $H^\bullet(S^2 \vee K(A,n)/\langle\iota\rangle,\Q_\ell)$ are the fundamental class of $S^2$ in degree $2$, and the space $\mathrm{Sym}^2(A^\vee)$ in degree $2n$. Keeping track of Koszul signs and degree shifts in the quadratic duality, and then translating further from the homology of $\Omega^2 W$ to the twisted homology of braid groups, one sees that it suffices to compute the (obvious) Galois action on the two groups
    $$ H_0(\beta_\infty,\Q_\ell)\cong \Q_\ell \qquad \text{and} \qquad H_0(\beta_\infty,\wedge^2 (V)) \cong \Q_\ell(-1),$$
    and all other classes in the stable homology with symplectic coefficients are obtained from these two by means of the operations of the little disk operad. Here $\Q_\ell(-1) \subset \wedge^2(V)$ is spanned by the class of the symplectic form.
\end{para}

\begin{rem}There are now multiple different proofs of the Madsen--Weiss theorem, including the more general Galatius--Madsen--Tillmann--Weiss theorem \cite{gmtw} and the Cobordism Hypothesis \cite{luriecobordism}. Interestingly, the proof of \cite{grw} (see also \cite{hatcher-madsenweiss}) interprets the Madsen--Weiss isomorphism as an instance of a scanning map, making the parallel between this story and the Madsen--Weiss theorem even closer. To see the parallel, note first that the proof of Segal's theorem sketched above works just as well for configurations of points in $\R^n$ for any $n$ (including $n=\infty$, corresponding to the Barratt--Priddy--Quillen theorem), in which case one finds that the monoid of configuration spaces group-completes to $\Omega^n S^n$; in particular, we get an $n$-fold loop space since there are $n$ different coordinate directions along which one can ``scan''. Now to prove Madsen--Weiss one may attempt scanning, but not inputting configuration spaces of points (moduli spaces of embedded $0$-manifolds), but instead moduli spaces of embedded \emph{surfaces}. Applied to the space of surfaces in $\R^\infty$, which is a classifying space of the mapping class groups, the scanning map should then optimistically produce an infinite loop space. It turns out that this is the case, and that this is a viable strategy for proving the Madsen--Weiss theorem.
\end{rem}



\subsection{Acknowledgements}
We are grateful to Alexander Berglund, Andrea Bianchi, Geoffroy Horel, Sander Kupers, Sam Nariman,  Sheila Sundaram, and Nathalie Wahl for useful conversations. Our treatment of logarithmic geometry was clarified by discussions with Aaron Landesman and Oliver Lindstr\"om. We would like to particularly thank Jeremy Miller, Peter Patzt, and Oscar Randal-Williams. The second author was partially supported by the CNCS-UEFISCDI grant PN-IIIP4-ID-PCE-2020-2498, and a Simons Fellowship in Mathematics with award number 681716. He would also like to thank the Department of Mathematics at Stockholm University for the hospitality and the excellent working conditions during his visit in 2021. The third author was supported by the grant ERC-2017-STG 759082 and a Wallenberg Academy Fellowship.  Part of this work was carried out at Institut Mittag-Leffler in Djursholm, Sweden during the semester programs ``Moduli and algebraic cycles'' (2021) and ``Higher algebraic structures in algebra, topology and geometry'' (2022), supported by the Swedish Research Council under grant no. 2016-06596.

\subsection{Degree conventions} Gradings are homological by default. If a homologically graded object and a cohomologically graded object appear in the same formula, then the cohomological grading is interpreted homologically via $C^\bullet = C_{-\bullet}$. We denote by $[n]$ the functor that raises homological degree by $n$, and lowers cohomological degree by $n$; for example, $\Z[n]$ denotes the abelian group $\Z$ considered as a chain complex concentrated in degree $n$, or a cochain complex concentrated in degree $-n$.

\section{Symmetric function formalism}

\begin{para}In this section we recall the theory of symmetric functions and analytic functors. We work over a field of characteristic zero, which remains fixed throughout this section. The canonical reference for symmetric functions is the book of Macdonald \cite{macdonald}. The plethystic exponential and logarithm was introduced by Getzler and Kapranov \cite{getzlerkapranov}. 
\end{para}
\subsection{Polynomial and analytic functors}
\label{polynomial and analytic functors}
\begin{defn}
	By a \emph{symmetric sequence} we mean a sequence $V = \{V(n)\}_{n \geq 0}$, where each $V(n)$ is a finite dimensional representation of the symmetric group $\Sigma_n$. We say that $V(n)$ is the \emph{arity $n$ component} of $V$. We say that $V$ is \emph{bounded} if $V(n) = 0$ for $n \gg 0$. 
\end{defn}

\begin{defn}By an \emph{analytic functor} we mean an endofunctor of the category of vector spaces of the form
$$ A \mapsto \bigoplus_{n \geq 0} V(n) \otimes_{\Sigma_n} A^{\otimes n},$$
where $V$ is a symmetric sequence. The representations $V(n)$ are called the \emph{Taylor coefficients} of the analytic functor. We denote by $\Ana$ the category of analytic functors; it is equivalent to the category of symmetric sequences. An analytic functor corresponding to a bounded symmetric sequence is called a \emph{polynomial functor}, and we denote the category of polynomial functors by $\Poly$. \end{defn}

\begin{para}
	If $F$ and $G$ are analytic functors, then so is their pointwise tensor product, defined by 
$$ (F \otimes G)(A) = F(A) \otimes G(A).$$
The corresponding tensor product of symmetric sequences is the Day convolution product:
\[ (V \otimes W)(n) = \bigoplus_{p+q=n} \mathrm{Ind}_{\Sigma_p \times \Sigma_q}^{\Sigma_n} V(p) \otimes W(q).\]
The tensor product of two polynomial functors is a polynomial functor.
\end{para}\begin{para}\label{composition product}
If $F$ and $G$ are polynomial functors, then so is the composition $F \circ G$. The corresponding monoidal product on bounded symmetric sequences is the \emph{composition product}
$$ (V \circ W)(n) = \bigoplus_{k \geq 0} V(k) \otimes_{\Sigma_k} (W^{\otimes k})(n),$$
where $W^{\otimes k}$ denotes the $k$-fold Day convolution, and $(W^{\otimes k})(n)$ thus carries commuting actions of $\Sigma_n$ and $\Sigma_k$. 
\end{para}\begin{para}\label{plethysm undefined}
If $F$ and $G$ are analytic functors, then $F \circ G$ is not necessarily analytic by our definition, since its Taylor coefficients are not necessarily finite dimensional. However, if $G(0) = 0$, then $F \circ G$ is analytic. One may compare this with the fact that the composition $f \circ g$ of two formal power series is defined only if $g$ has vanishing constant term. 
\end{para}\begin{para}
The irreducible representations of $\Sigma_n$ are conventionally indexed by partitions of $n$. We denote by $\sigma_\lambda$ the irreducible representation of $\Sigma_n$ corresponding to the partition $\lambda$. We may consider it as a symmetric sequence concentrated in arity $n$. The corresponding polynomial functor is called a \emph{Schur functor}, and denoted $S^\lambda(-)$. 
For example, the Schur functor associated to the trivial representation of $\Sigma_n$ is the $n$th symmetric power functor $\Sym^n$, and the sign representation corresponds to the exterior power $\wedge^n$. 
\end{para}

\begin{examplex}\label{example of tca}
	Fix a finite dimensional vector space $W$. We consider the tensor algebra on $W$ as a symmetric sequence whose arity $n$ component is $W^{\otimes n}$. It corresponds to the analytic functor 
	$$ A \mapsto \mathrm{Sym}(A \otimes W),$$
	where $\mathrm{Sym}(-)$ denotes the symmetric algebra on a vector space. 
\end{examplex}

\subsection{The ring of symmetric functions}

\begin{defn}Define
$$ \Lambda := \varprojlim \Z[x_1,\ldots,x_n]^{\Sigma_n},$$
where the transition maps $\Z[x_1,\ldots,x_n]^{\Sigma_n} \to \Z[x_1,\ldots,x_{n-1}]^{\Sigma_{n-1}}$ are given by setting the variable $x_n$ to zero, and the inverse limit is taken in the category of graded rings, where each $x_i$ has degree $1$. We call $\Lambda$ the \emph{ring of symmetric functions}.
\end{defn}\begin{para}
The ring of symmetric functions has several standard sets of generators. Write
$$ h_r = \sum_{i_1 \leq \ldots \leq i_r } \prod_{j=1}^r x_{i_j}, \qquad \qquad e_r = \sum_{i_1 < \ldots < i_r } \prod_{j=1}^r x_{i_j} $$
for the $r$th \emph{complete} and \emph{elementary} symmetric functions, respectively. The ring homomorphisms
$$ \Z[h_1,h_2,h_3,\ldots ] \to \Lambda \qquad \qquad \Z[e_1,e_2,e_3,\ldots] \to \Lambda$$
are both isomorphisms. \end{para}

\begin{para}We have defined $\Lambda$ as a graded ring; we will call the grading the  \emph{arity} grading. We write $\Lambda = \bigoplus_{n \geq 0} \Lambda_n$, and refer to $\Lambda_n$ as the \emph{arity $n$ component} of $\Lambda$. We may make an identification of $\Lambda_n$ with the degree $n$ part of  $\Z[x_1,\ldots,x_n]^{\Sigma_n}$.
\end{para}\begin{para}
There is a standard isomorphism of rings 
$$ \Lambda \cong K_0(\Poly).$$
The addition and multiplication on $K_0(\Poly)$ are induced by the pointwise direct sum and tensor product of polynomial functors. This isomorphism takes $h_r$ to the class of the symmetric power functor $\mathrm{Sym}^r(-)$, and $e_r$ to the class of the exterior power functor $\wedge^r(-)$. For each $n$, this isomorphism restricts to an isomorphism $\Lambda_n \cong R(\Sigma_n)$ between the arity $n$ component of $\Lambda$ and the ring of virtual representations of $\Sigma_n$. For a polynomial functor $F$ (or bounded symmetric sequence $V$), we denote by $[F]$ (resp. $[V]$) its class in the Grothendieck group $\Lambda$. \end{para}
\begin{para}

 The symmetric function associated to the Schur functor $S^\lambda(-)$ is called a \emph{Schur polynomial}, and will be denoted $s_\lambda$. The Schur polynomials form a basis for $\Lambda$ as a free abelian group. An explicit formula for $s_\lambda$ is given by the Jacobi bialternant formula: for $\lambda = (\lambda_1 \geq \lambda_2\geq \ldots \geq \lambda_n \geq 0)$, one has
$$ s_\lambda = \frac{\det \left( x_i^{\lambda_j+n-j} \right)_{1\leq i,j \leq n}}{\det \left( x_i^{n-j} \right)_{1\leq i,j \leq n}} \in \Z[x_1,\ldots,x_n]^{\Sigma_n}.$$
Note that the denominator is the Vandermonde determinant $\prod_{i<j}(x_i-x_j)$. 
One may also define $s_\lambda$ as follows: choose $n \geq \length(\lambda)$, and let $s_\lambda$ be the unique symmetric function such that $s_\lambda(x_1,\ldots,x_n)$ is the character of the irreducible representation of $\mathrm{GL}(n)$ of highest weight $\lambda$, i.e.\ $S^\lambda$ applied to the defining $n$-dimensional representation of $\mathrm{GL}(n)$. The Jacobi bialternant formula is then the Weyl character formula.  
\end{para}

\begin{para}
    We denote by $\omega : \Lambda \to \Lambda$ the unique ring automorphism sending $e_n$ to $h_n$. It is an involution. Under the isomorphism $R(\Sigma_n) \cong \Lambda_n$, $\omega$ corresponds to tensoring with the sign representation of $\Sigma_n$. We have $\omega (s_\lambda) = s_{\lambda'}$.
\end{para}

\begin{para}
We denote by $\langle -, -\rangle$ the inner product on $\Lambda$, uniquely defined by the condition that 
$$ \langle s_\lambda, s_\mu \rangle = \begin{cases}1 & \lambda = \mu \\ 0 & \lambda \neq \mu. \end{cases} $$
In other words, we take the inner product induced from the standard inner product on $R(\Sigma_n)$, for which the classes of irreducible representations form an orthonormal basis. 
\end{para}
\begin{para}
    For $f \in \Lambda$ we write $f^\perp$ for the adjoint of multiplication by $f$ with respect to the inner product $\langle -, -\rangle$. We have $s_\lambda^\perp s_\mu = 0$ unless $\lambda \subseteq \mu$, and in this case we use the notation $s_{\mu/\lambda} := s_\lambda^\perp s_\mu$. These are the \emph{skew Schur functions}, indexed by skew diagrams $\mu/\lambda$. 
\end{para}

\begin{para}
We define a \emph{comultiplication} on $\Lambda$ by $\Lambda \ni f(x_1,x_2,\ldots) \mapsto f(x_1,x_2,\ldots,y_1,y_2,\ldots) \in \Lambda \otimes \Lambda$. Alternatively, it is defined by the formula
$$ \Delta( h_n ) = \sum_{p+q=n} h_p \otimes h_q;$$
equivalently, $ \Delta( e_n ) = \sum_{p+q=n} e_p \otimes e_q$. With this comultiplication, $\Lambda$ becomes a commutative and cocommutative Hopf algebra. In terms of representations of the symmetric groups, the coproduct is the sum over all $p+q=n$ of the maps 
$$ \mathrm{Res}^{\Sigma_n}_{\Sigma_p \times \Sigma_q} : R(\Sigma_n) \to R(\Sigma_p \times \Sigma_q) \cong R(\Sigma_p) \otimes R(\Sigma_q)$$ 
obtained from the inclusions $\Sigma_p \times \Sigma_q \to \Sigma_n$. As the product in $\Lambda \cong K_0(\Poly)$ corresponded to the Day convolution (induction), it follows from Frobenius reciprocity that the coproduct is the adjoint of the product with respect to the inner product. 

\end{para}\begin{para}

Each ring of virtual representations $R(\Sigma_n)$ is itself a Hopf algebra, and carries its own multiplication and comultiplication. Applying these aritywise to $\Lambda$ gives another, different, product and coproduct, which are called the \emph{inner multiplication} and \emph{inner comultiplication}, which we denote $\ast \colon \Lambda \otimes \Lambda \to \Lambda$ and $\delta \colon \Lambda \to \Lambda \otimes \Lambda$. 
\end{para}\begin{para}\label{plethysm identities}
A further important operation on symmetric functions is the \emph{plethysm}, which we denote by $\circ$. Plethysm is a binary operation $\Lambda \times \Lambda \to \Lambda$ and is characterized by the following algebraic identities:
\begin{enumerate}
	\item Plethysm is associative.
	\item The element $h_1$ is a two-sided identity for plethysm. 
	\item For all $f \in \Lambda$, $- \circ f$ is a ring homomorphism $\Lambda \to \Lambda$. 
	\item For all $f, g, h \in \Lambda$, $f \circ (g + h) = \sum (f_{(1)}^\Delta \circ g)(f_{(2)}^\Delta \circ h)$. 
	\item For all $f, g, h \in \Lambda$, $f \circ (gh) = \sum (f_{(1)}^\delta \circ g)(f_{(2)}^\delta \circ h)$. 
\end{enumerate}
In properties (4) and (5) we are using Sweedler notation for the two coproducts $\Delta$ and $\delta$: 
\[ \Delta(f) = \sum f_{(1)}^\Delta \otimes f_{(2)}^\Delta, \qquad \delta(f) = \sum f_{(1)}^\delta \otimes f_{(2)}^\delta. \]
The structure constants for the coproducts $\Delta$ and $\delta$ are the Littlewood--Richardson coefficients $\{c_{\mu,\nu}^\lambda\}$ and the Kronecker coefficients $\{g_{\mu,\nu}^\lambda\}$, respectively. Hence (4) and (5) may also be stated as $s_\lambda \circ (g+h) = \sum c_{\mu,\nu}^\lambda (s_\mu \circ g)(s_\nu \circ h)$ and $s_\lambda \circ (gh) = \sum g_{\mu,\nu}^\lambda (s_\mu \circ g)(s_\nu \circ h)$.
\end{para}

\begin{para}The importance of plethysm is that it corresponds to the composition of polynomial functors. If $F, G \in \Poly$ are polynomial functors, then their composition $F \circ G$ is again a polynomial functor, and 
	$$ [F \circ G] = [F] \circ [G]$$
	in $\Lambda \cong K_0(\Poly)$. 
\end{para}

\begin{para}\label{adams op}

We denote by $p_n$ the $n$th power sum, which is the element
$$ p_n = x_1^n + x_2^n + \ldots \in \Lambda. $$
The elements $p_n$ do not generate $\Lambda$ as a ring, as for example $h_2 = \tfrac{p_1^2+p_2} 2$ and $e_2 = \tfrac{p_1^2-p_2} 2$, but the map
$$ \Q[p_1,p_2,\ldots ] \to \Lambda \otimes \Q$$
is an isomorphism. The power sums are useful for carrying out calculations involving plethysm, as they satisfy the simple relation
\[ p_n \circ f(p_1,p_2,\ldots) = f(p_n,p_{2n},\ldots).\]  
This relation implies in particular that $p_n \circ (-)$ is a ring homomorphism $\Lambda \to \Lambda$, for every $n$. We write $\psi_n(x)$ instead of $p_n \circ x$, and we refer to the ring homomorphism $\psi_n(-)$ as the $n$th \emph{Adams operation}. 
\end{para}

\subsection{Plethystic exponential and logarithm}

\begin{para}
We denote by $\widehat{\Lambda}$ the completion of $\Lambda$ with respect to the arity grading; that is,
$$ \widehat{\Lambda} = \prod_{n=0}^\infty \Lambda_n. $$
In the same way that $\Lambda \cong K_0(\Poly)$, we have $\widehat{\Lambda} \cong K_0(\Ana)$. If $F$ is an analytic functor, then we denote by $[F]$ its class in $\widehat \Lambda$, and we refer to $[F]$ as the \emph{Taylor series} of the analytic functor $F$. We write $\widehat{\Lambda}_{0}$ for $\prod_{n=1}^\infty \Lambda_n \subset \widehat{\Lambda}$. \end{para}\begin{para}\label{plethysm undefined 2}

As in \S\ref{plethysm undefined}, plethysm does not extend to a well-defined operation on all of $\widehat \Lambda$; the plethysm $f \circ g$ is meaningful only if $f$ and $g$ are both in $\Lambda$, or if the arity $0$ component of $g$ vanishes. When defined, the plethysm operation on $\widehat{\Lambda}$ satisfies the same algebraic identities as the plethysm on $\Lambda$. If $F$ and $G$ are analytic functors and $G(0)=0$, so that $F \circ G$ is again an analytic functor, then there is an identity of Taylor series
$$ [F \circ G] = [F] \circ [G].$$
\end{para}

\begin{rem}
	
	A fruitful analogy to keep in mind is that $\Lambda$ is analogous to the ring of polynomials $\Z[t]$, in that both are rings equipped with an additional ``composition'' operation: $\Lambda$ has the plethysm operation, and $\Z[t]$ has composition of polynomials. More formally, $\Lambda$ and $\Z[t]$ are canonical examples of \emph{plethories} \cite{tallwraith,borgerwieland}. Under this analogy, the completion $\widehat \Lambda$ is analogous to the formal power series ring $\Z[[t]]$, and $\widehat{\Lambda}_0$ to the ideal $t \Z[[t]]$. The algebraic structure on $\widehat{\Lambda}$ and $\Z[[t]]$ is that of a \emph{formal plethory} \cite{bauerformal}. In fact, these are two of the most basic formal plethories one encounters in topology, as $K^0(BU) \cong \widehat{\Lambda}$ and $K^0(BU(1)) \cong \Z[[t]]$. \end{rem}

\begin{para}
	From \S\ref{plethysm undefined 2} we see in particular that  $(\widehat\Lambda_{0}, \circ, h_1)$ is a monoid under plethysm. It is natural to ask which elements of this monoid are invertible. The analogous question for power series has a well known answer: given a commutative ring $A$, the set of power series 
	$ tA[[t]] = \{ \sum_{n \geq 1} a_n t^n : a_n \in A\} $
	form a monoid under composition, and such a series is invertible under composition if and only if $a_1 \in A^\times$. (Indeed, one simply writes down the upper triangular system of equations giving the coefficients of the compositional inverse.) In the same way, one can show the following proposition. 
	\end{para}
	\begin{prop}An element of $\widehat\Lambda_{0}$ is invertible with respect to plethysm if and only if the coefficient before $h_1$ is $\pm 1$.
	\end{prop}

\begin{defn}\label{defn of L}
	Define special elements of $\widehat{\Lambda}$ by
	$$ E = \sum_{r \geq 0} h_r = \exp \Big( \sum_{k \geq 1} \frac{p_k}{k} \Big)$$
	(where the second equality is a form of Newton's identitites) and 
	$$ L = \sum_{k \geq 1} \frac{\mu(k)} k \log(1+p_k).$$
	By the preceding proposition, $E-1$ and $L$ are both invertible under plethysm. In fact more is true.
\end{defn}

\begin{prop}\label{inverse identity}
	In the ring $\widehat{\Lambda}$ one has
	$$ (E-1) \circ L =  h_1 = L \circ (E-1).$$ 
\end{prop}

\begin{proof}It suffices to prove the first identity; an invertible element of a monoid has the same right and left inverse. Note that
	\[
		\Big(\sum_{k\geq 1} \frac{p_k}{k} \Big) \circ \Big(\sum_{l \geq 1} \frac{\mu(l)} l \log(1+p_l)\Big) = \sum_{k,l \geq 1} \frac{\mu(l)}{kl}\log(1+p_{kl}) = \log (1+p_1),
	\]
	using that $\sum_{d \mid n} \mu(d) = 0$ for $n>1$ in the second equality. Exponentiate this to get $E \circ L=1+p_1$, which is equivalent to the proposition.	\end{proof}

\begin{rem}The degree $n$ component of $L$ is $\frac{(-1)^{n}}{n}\sum_{d\mid n} {\mu(d)} (-p_d)^{n/d}$. Let $\mathsf{Lie}(n)$ denote the representation of $S_n$ spanned by the set of multilinear Lie words in $n$ symbols. Then the character of $\mathsf{Lie}(n)$ is given by $\frac{1}{n}\sum_{d\mid n} {\mu(d)} p_d^{n/d}$ \cite[Proposition 4]{joyalanalyticfunctors}. This is not a coincidence: if $P$ is a Koszul operad and $P^\antishriek$ is its Koszul dual cooperad, then the analytic functors corresponding to $P$ and $P^\antishriek$ are compositional inverses, since the Koszul complex $P \circ P^\antishriek$ with the twisting differential is acyclic. Now $(E-1)$ is the Taylor series associated to the non-unital commutative operad, so its compositional inverse $L$ must be the Taylor series of the Koszul dual, which is the shifted Lie cooperad. The shift accounts for the sign differences between $L$ and the character of $\mathsf{Lie}$. The calculation of the character of $\mathsf{Lie}(n)$ may be seen as an equivariant refinement of Witt's formula for the graded dimensions of the free Lie algebra \cite[Corollary 5.3.5]{lothaire}. 
\end{rem}

\begin{defn} For $x \in \widehat\Lambda_{0}$, define the \emph{plethystic exponential}
	$$ \Exp(x) = E \circ x,$$
	and \emph{plethystic logarithm}
	$$ \Log(1+x) = L \circ x.$$
\end{defn}

\begin{prop}The map $\Exp : \widehat\Lambda_{0}  \to 1+\widehat\Lambda_{0}$ is a bijection, with inverse $\Log : 1+\widehat\Lambda_{0} \to \widehat\Lambda_{0}$.
\end{prop}

\begin{proof}Let $x \in \widehat\Lambda_{0}$. Then $\Exp(\Log(1+x)) = E \circ L \circ x = (1+h_1) \circ x =1 + x$, using \cref{inverse identity}. One shows $\Log(\Exp(x))=x$ in the same way. \end{proof}

\begin{prop}
	The plethystic exponential and logarithm satisfy the identities
	$$\Exp(x+y) = \Exp(x)\Exp(y) \qquad \text{and} \qquad \Log(ab) = \Log(a)+\Log(b).$$ 
\end{prop}

\begin{proof}
	It suffices to prove the first identity. To do this, use the identity \S\ref{plethysm identities}(4) and $\Delta(h_n) = \sum_{p+q=n} h_p \otimes h_q$. 
\end{proof}

\begin{prop}\label{lemma:pleth id}For $x \in \widehat{\Lambda}$, $y \in \widehat{\Lambda}_{0}$, there is an identity
\[\Exp\big(\,x\, \Log(1+y)\big) = \prod_{n=1}^\infty (1+\psi_n(y))^{\frac{1}{n} \sum_{d \mid n} \mu(n/d)  \psi_d(x)}
	\]
	where $\psi_n$ denotes the $n$th Adams operation (\S\ref{adams op}), and each factor on the right-hand side is interpreted by expanding the binomial series. 
\end{prop}
\begin{proof} As in the proof of \cref{inverse identity}, we observe that  
\begin{align*}
	\Big(\sum_{k\geq 1} \frac{p_k}{k} \Big) \circ \Big(x \sum_{l \geq 1} \frac{\mu(l)} l \log(1+\psi_l(y))\Big) = & \sum_{k,l \geq 1} \frac{\mu(l)\psi_k(x)}{kl}\log(1+\psi_{kl}(y)) \\
	= & \sum_{n=1}^\infty \frac{1}{n} \sum_{d \mid n} \mu(n/ d) \psi_d(x) \log(1+\psi_n(y)),
\end{align*}
from which the result follows after exponentiating. \end{proof}

\subsection{Graded plethystic algebra}
\label{graded plethystic algebra}
\newcommand{\gr}{\mathsf{gr}}

\begin{para}The definitions of Section \ref{polynomial and analytic functors} have evident analogues for graded vector spaces. By an \emph{analytic functor} between graded vector spaces we mean a functor of the form
	$$ A \mapsto \bigoplus_{n \geq 0} V(n) \otimes_{\Sigma_n} A^{\otimes n}$$
where instead each $V(n)$ is a bounded below and degreewise finite-dimensional vector space. It is a \emph{polynomial functor} if $V(n)=0$ for $n \gg 0$. We write $\Ana^{\gr}$ and $\Poly^\gr$ for the categories of analytic and polynomial functors between graded vector spaces. As before, these are equivalent to categories of symmetric sequences and bounded symmetric sequences, respectively, in the category of degreewise finite-dimensional bounded below graded vector spaces.
\end{para}

\begin{para}\label{K0 iso} Define $\Lambda^\gr = \Lambda\otimes \Z(\!(z)\!)$. There is an isomorphism 
	$$ K_0(\Poly^\gr) \cong \Lambda^\gr$$
	generalizing the isomorphism $K_0(\Poly) \cong \Lambda$, defined as follows: if $(V(n))_{n \geq 0}$ is a symmetric sequence in graded vector spaces, with $V(n) = \bigoplus_{k \in \Z} V(n)_k$, then its class in the Grothendieck group is
	$$ [V] = \sum_n \sum_k [V(n)_k] \cdot (-z)^k $$
	where $[V(n)_k]$ denotes the class of the representation $V(n)_k$ in $\Lambda_n$. 
\end{para}

\begin{para}
	The plethysm operation $\circ$ on $\Lambda$ extends uniquely to an operation on $\Lambda^\gr$ satisfying the identities of \S\ref{plethysm identities}, that plethysm is $\Z(\!(z)\!)$-linear in the first variable, that 
	$$ h_n \circ z = z^n$$
	for all $n \geq 0$, and continuity for the $z$-adic topology. 
\end{para}

\begin{para}\label{composition law}
	The composition of two polynomial functors of graded vector spaces is again a polynomial functor. For $F, G \in \Poly^\gr$ there is an identity
	$$ [F \circ G] = [F] \circ [G]$$
	in $\Lambda^\gr = K_0(\Poly^\gr)$. 
\end{para}
\begin{para}
We write 
$$ \widehat \Lambda^\gr = \prod_{n=0}^\infty \Lambda_n \otimes \Z(\!(z)\!).$$
(The notation $\widehat{\Lambda^\gr}$ would be more logical, but leads to unpleasant typesetting.) The isomorphism $K_0(\Poly^\gr) \cong \Lambda^\gr$ extends to an isomorphism $K_0(\Ana^\gr) \cong \widehat{\Lambda}^\gr$. 
\end{para}
\begin{para}

For $F$ and $G$ analytic functors between ungraded vector spaces, we saw that $F \circ G$ was analytic if $G(0)=0$. In the graded setting, we may relax this condition: $F \circ G$ is again analytic when $G(0)$ is concentrated in strictly positive degrees. We write
$$ \widehat \Lambda^\gr_0 = z \Z[[z]] \oplus \prod_{n=1}^\infty \Lambda_n \otimes \Z(\!(z)\!) \subset \widehat \Lambda^\gr.$$
The plethysm operation on $\Lambda^\gr$ extends to a well-defined operation
$$ \circ : \widehat\Lambda^\gr \times \widehat\Lambda^\gr_0 \to \widehat\Lambda^\gr $$
satisfying the same  algebraic identities as before. If $F$ and $G$ are analytic functors and $G(0)$ is concentrated in positive degrees, then 
$$ [F \circ G] = [F] \circ [G].$$
\end{para}
\begin{para}We continue to denote by $\Exp : \widehat \Lambda^\gr_0 \to 1+\widehat \Lambda^\gr_0$ and $\Log : 1+\widehat \Lambda^\gr_0 \to \widehat \Lambda^\gr_0$ the evident extensions of the plethystic exponential and logarithm. 
\end{para}

\begin{para}\label{sign rule 1}We have implicitly assumed in this section that the Koszul sign rule is applied in the category of graded vector spaces; that is, if $V$ and $W$ are graded vector spaces, then the swap isomorphism $V \otimes W \cong W \otimes V$ is given by the formula$$ v \otimes w \mapsto (-1)^{\vert v \vert \vert w \vert} w \otimes v .$$
	In particular, the equality $[F \circ G] = [F] \circ [G]$ of \S\ref{composition law} is only valid when the Koszul sign rule is in effect. The reason that the formulas are sensitive to this convention is that the composition product of symmetric sequences (\S\ref{composition product}) makes use of the natural action of $\Sigma_k$ on $W^{\otimes k}$, which therefore depends on the choice of symmetric monoidal structure. If we had instead considered the category of graded vector spaces to have the ``ungraded'' symmetric monoidal structure, then the results of this section (including the equality $[F \circ G] = [F] \circ [G]$) would remain valid if the isomorphism of \cref{K0 iso} were instead defined by the formula 
		$$ [V] = \sum_n \sum_k [V(n)_k] \cdot z^k. $$
		That is, the difference between the Koszul sign rule and the ungraded symmetry is given by a global change of variables $z \mapsto -z$. Informally, terms appearing with a minus sign are treated as living in ``odd'' degree; removing the Koszul rule amounts to treating all degrees as being even, hence removing all minus signs. 
\end{para}

\subsection{A primer on \texorpdfstring{$\lambda$}{lambda}-rings}

\begin{defn}\label{lambda-ring definition}
	Let $R$ be a commutative ring. A \emph{$\lambda$-ring structure on $R$} consists of a map
	\begin{align*}
		\Lambda \times R &\to R \\
		(f,x) &\mapsto f \circ x 
	\end{align*}
	satisfying the axioms
	\begin{enumerate}[label=$\Lambda$\arabic*.]
		\item For all $f, g \in \Lambda$ and $x \in R$, $(f \circ g) \circ x = f \circ (g \circ x)$.
		\item For all $x \in R$, $h_1 \circ x = x$. 
		\item For all $x \in R$, $- \circ x$ is a ring homomorphism $\Lambda \to R$.  
		\item For all $x, y \in R$ and $f \in \Lambda$, $f \circ (x+y) = \sum (f^\Delta_{(1)} \circ x)(f^\Delta_{(2)} \circ y)$. 
		\item For all $x, y \in R$ and $f \in \Lambda$, $f \circ (xy) = \sum (f_{(1)}^\delta \circ x)(f_{(2)}^\delta \circ y)$.
	\end{enumerate}
 Compare with \S\ref{plethysm identities}.
\end{defn}

\begin{rem}The most common definition of a $\lambda$-ring is in terms of operations $\lambda_n \colon R \to R$. This definition is related to ours via $\lambda_n(x) = e_n \circ x$. We prefer the definition given here since it relates more directly to the role that $\lambda$-rings play in this paper: we want to think of $\lambda$-rings as rings on which $\Lambda$ acts by plethystic substitution, rather than rings equipped with a notion of exterior powers. Indeed, axioms $\Lambda$1 and $\Lambda$2 say precisely that $R$ is a left module over the monoid $(\Lambda, \circ, h_1)$, and the remaining axioms say that the action of $\Lambda$ on $R$ satisfies the same compatibilities with addition and multiplication as the usual plethysm operation on $\Lambda$.
\end{rem}

\begin{rem}
    More generally, for any plethory $P$ one can define a notion of $P$-ring, such that when $P=\Lambda$ one recovers the notion of $\lambda$-ring.
\end{rem}

\begin{rem}Some authors make a distinction between ``$\lambda$-rings'' and ``special $\lambda$-rings''. The previous definition describes what these authors would call a ``special'' $\lambda$-ring. We will only be interested in special $\lambda$-rings, and drop the extra adjective for brevity. \end{rem}

\begin{examplex}Let us mention some standard examples of $\lambda$-rings: 
	\begin{enumerate}[label=(\arabic*a)]
		\item The ring of integers $\Z$ has a unique $\lambda$-ring structure, for which $e_n \circ x = \binom x n$.
		\item The universal example of a $\lambda$-ring is $\Lambda$, acting on itself by plethysm.  
		\item The polynomial ring $\Z[z]$ has a standard $\lambda$-ring structure characterized by $h_n \circ z = z^n$. 
		\item Let $X$ be a topological space. The topological $K$-theory $K(X)$ has a natural $\lambda$-ring-structure.
		\item Let $G$ be a group. The ring of virtual representations $R(G)$ is naturally a $\lambda$-ring.
		\end{enumerate}
Each of these examples arises as the Grothendieck group of a category. These categories are, respectively:
	\begin{enumerate}[label=(\arabic*b)]
		\item The category $\mathsf{Vect}$ of finite dimensional $\Q$-vector spaces. 
		\item The category $\mathsf{Poly}$ of polynomial functors $\mathsf{Vect}\to\mathsf{Vect}$.
		\item The category of finite dimensional nonnegatively graded  vector spaces. If $V = \bigoplus_{n\geq 0} V_n$ is such a vector space, its class $[V]$ in the Grothendieck group is the polynomial $\sum_{n \geq 0} (-1)^n (\dim V_n) z^n$.
		\item The category of finite rank vector bundles over $X$.
		\item The category of finite dimensional representations of $G$.
	\end{enumerate}
\end{examplex}

\begin{para}
	In each of the preceding examples, the $\lambda$-ring structure on $K_0$ arises from the fact that Schur functors are defined in the respective category; the $\lambda$-ring structure is characterized by $[S^\lambda(M)] = s_\lambda \circ [M]$, for $M$ an object of the respective category. More generally, the formula for Schur functors
	$$ S^\lambda(M) := s_\lambda \otimes_{\Sigma_n} M^{\otimes n}, $$
	for $\lambda \vdash n$, is meaningful for $M$ in any symmetric monoidal $\Q$-linear pseudo-abelian category $\mathsf R$. Baez--Moeller--Trimble  \cite{baezmoellertrimble} refer to this notion as a \emph{2-rig}, and Getzler \cite{mixedhodge} calls it a \emph{Karoubian rring} [sic]. If $\mathsf R$ is any $2$-rig then there is a natural structure of $\lambda$-ring on $K_0(\mathsf R)$.
\end{para}

\begin{para}\label{sign rule 2}
	We can now also explain why the ``recipe'' of \S\ref{sign rule 1} works. Consider the $2$-rig $\mathsf R$  of degreewise finite dimensional bounded below graded vector spaces. The point is then that there are isomorphisms of $\lambda$-rings $K_0(\mathsf R) \cong \Z(\!(z)\!)$ --- where $\Z(\!(z)\!)$ carries its standard $\lambda$-ring structure with $\sigma_n(z)=z^n$ for all $n \geq 0$ --- for \emph{both} natural symmetric monoidal structures on $\mathsf R$, and the two isomorphisms differ by $z \mapsto -z$. 
 
 Moreover, for any 2-rig $\mathsf R$, we may consider the category $\Poly_{\mathsf R}$ of polynomial functors $\mathsf R \to \mathsf R$ (bounded symmetric sequences in $\mathsf R$). This category is again a $2$-rig, and there is an isomorphism of $\lambda$-rings
	$$ K_0(\mathsf{Poly}_\mathsf{R}) \cong K_0(\Poly) \otimes K_0(\mathsf R).$$ 
	This is even an isomorphism of plethories, using that if $P$ is a plethory and $A$ is a $P$-ring, then $P \otimes A$ is naturally a plethory. When $\mathsf R$ is the category of degreewise finite dimensional bounded below graded vector spaces, we recover the $\lambda$-ring structure of $\Lambda^\gr$.
\end{para}

\begin{para}
	The plethystic logarithm and exponential can be defined more generally for \emph{complete filtered} $\lambda$-rings. 
\end{para}

\begin{defn}
	A \emph{graded $\lambda$-ring} is a $\lambda$-ring $R$ equipped with a decomposition $R=\bigoplus_{p\in\Z}R_p$ such that\begin{enumerate}
		\item $R_p \cdot R_q \subseteq R_{p+q} $,
		\item if $f \in \Lambda_n$ and $x \in R_p$, then $f \circ x \in R_{np}$. 
	\end{enumerate}
\end{defn}

\begin{defn}\label{defn:filtered lambda}
	A \emph{filtered $\lambda$-ring} is a $\lambda$-ring $R$ equipped with a descending filtration $\ldots \supseteq F^p R \supseteq F^{p+1}R \supseteq \ldots$ as a module, such that:
	\begin{enumerate}
		\item $F^p R \cdot F^q R \subseteq F^{p+q}R $,
		\item if $f \in \Lambda_n$ and $x \in F^pR$, then $f \circ x \in F^{np}R$. 
	\end{enumerate}
Note that every graded $\lambda$-ring is filtered by $F^p R= \bigoplus_{q\geq p}R_q$. 
\end{defn}\begin{defn}The \emph{completion} of a filtered $\lambda$-ring $R$ is the filtered $\lambda$-ring $\widehat R = \varprojlim R/F^p R$.
A filtered $\lambda$-ring $R$ is said to be \emph{complete} if the natural map $R \to \widehat R$ is an isomorphism.  
\end{defn}

\begin{prop}
If $R$ is a complete filtered $\lambda$-ring, then the plethystic action $\Lambda \times R\to R$ extends to an action $\widehat \Lambda \times F^1 R \to R$. In particular, $\Exp$ is a bijection $F^1 R \to 1 + F^1 R$, with inverse $\Log$. 
\end{prop}

\begin{examplex}
	The $\lambda$-ring $\Lambda$ is graded by the usual arity grading $\Lambda = \bigoplus_n \Lambda_n$. Its completion is $\widehat \Lambda$.  
\end{examplex}

\begin{examplex}
	The $\lambda$-ring $\Lambda^\gr = \Lambda \otimes \Z(\!(z)\!)$ has two natural gradings: one by arity, and one by the degree with respect to $z$. The ring $\widehat{\Lambda}^\gr$ is obtained by completing with respect to both gradings.
\end{examplex}

\begin{examplex}
	Define an increasing filtration on $\Lambda$ by $F_p \Lambda = \mathrm{span} \{s_\lambda : \length(\lambda) \leq p\}$. The corresponding descending filtration $F^p\Lambda=F_{-p}\Lambda$ makes $\Lambda$ into a filtered $\lambda$-ring. Indeed, condition (1) in \cref{defn:filtered lambda} is satisfied by \cite[Lemma 8.3]{bergstromminabe}. Condition (2) follows as well from this, and the representation-theoretic interpretation of plethysm: if $x \in F_{p} \Lambda$, then $x$ corresponds to a virtual representation all of whose constituents are indexed by partitions of length $\leq p$, and if $f \in \Lambda_n$ then $f \circ x$ is a sum of retracts of $n$-fold tensor products of such representations. 
\end{examplex}

\begin{examplex}
	Applying the conjugation automorphism $\omega$, we see that $F_p\Lambda = \mathrm{span}\{s_\lambda : \lambda_1 \leq p\}$ also defines a filtered $\lambda$-ring structure on $\Lambda$. 
\end{examplex}

\begin{prop}\label{plehystic ineq 0} Let $R$ be a $\lambda$-ring equipped with both a grading and a filtration. For any real number $c$ consider 
	$$ R(c) := \bigoplus_{\substack{p,n \\ p \geq cn}} F^p R_n.$$
	Then $R(c)$ is a sub-$\lambda$-ring of $R$. 
\end{prop}

\begin{proof}
	Immediate from the definitions. 
\end{proof}

\begin{para} We give two examples illustrating the preceding proposition, by considering the gradings and filtrations just defined. 
\end{para}

\begin{cor}\label{plethystic ineq 1}
	Let $f, g \in \Lambda$. Suppose that every Schur polynomial $s_\lambda$ occurring with nonzero coefficient in $f$ or $g$ satisfies the inequality $\length(\lambda) \leq c\vert\lambda\vert$, for some constant $c$. Then $f+g$ and $fg$ have the same property, and so does $h \circ f$, for any $h \in \Lambda$.  
\end{cor}

\begin{cor}\label{plethystic ineq 2}
	Let $f, g \in \Lambda^\gr$. Suppose that every monomial $z^k s_\lambda$ occurring with nonzero coefficient in $f$ or $g$ satisfies the inequality $\lambda_1 \leq ck$, for some constant $c$. Then $f+g$ and $fg$ have the same property, and so does $h \circ f$, for any $h \in \Lambda$.  
\end{cor}

\subsection{Schur--Weyl duality, Weyl's construction, and symplectic Schur functions}

\begin{para}
    The symmetric function identity
\begin{equation} \label{eq:cauchy identity}
    \prod_{i,j} (1-x_i y_j)^{-1} = \sum_\lambda s_\lambda(x)s_\lambda(y) \in (\Lambda \otimes \Lambda)^\wedge
\end{equation} 
is called the \emph{Cauchy identity}. One conceptual interpretation of the Cauchy identity is that it is a decategorification of the decomposition of the symmetric algebra of a tensor product:
$$ \mathrm{Sym}(V \otimes W) \cong \bigoplus_\lambda S^\lambda(V) \otimes S^\lambda(W).$$
Indeed, if $S \in \mathrm{GL}(V)$ has eigenvalues $x_1,x_2,\ldots$ and $T \in \mathrm{GL}(W)$ has eigenvalues $y_1,y_2,\ldots$, then it is clear that the trace of $S \times T$ acting on $\mathrm{Sym}(V \otimes W)$ is the left-hand side of \eqref{eq:cauchy identity}, and the trace on $\bigoplus_\lambda S^\lambda(V) \otimes S^\lambda(W)$ is the right-hand side of \eqref{eq:cauchy identity}. The Cauchy identity can also be seen as an avatar of the Peter--Weyl theorem, or the Robinson--Schensted--Knuth correspondence. 
\end{para}

\begin{para}\label{dual cauchy}
    Applying the automorphism $\omega$ to one of the two factors of $\Lambda$ gives the \emph{dual Cauchy identity}:
    \[ \prod_{i,j} (1+x_i y_j) = \sum_\lambda s_\lambda(x)s_{\lambda'}(y).\]
\end{para}

\begin{para}

The group algebra $\Q[\Sigma_n]$ has commuting left and right actions by $\Sigma_n$. Since $\Q$ is a splitting field for $\Sigma_n$, the Peter--Weyl theorem for finite groups tells us how to decompose $\Q[\Sigma_n]$ into irreducibles under these commuting left and right actions: there is a canonical isomorphism
$$ \Q[\Sigma_n] \cong \bigoplus_{\lambda \vdash n}  \sigma_\lambda \otimes \sigma_\lambda^\vee.$$
Every representation of $\Sigma_n$ is self-dual, and we could have written the summands above instead as $\sigma_\lambda \otimes \sigma_\lambda$, but we prefer to keep the dual to signify handedness: we are considering $\sigma_\lambda$ as a left $\Sigma_n$-module, and $\sigma_\lambda^\vee$ as a right $\Sigma_n$-module. 
\end{para}\begin{para}
If $V$ is any vector space (or graded vector space) then the tensor power $V^{\otimes n}$ carries commuting left and right actions by $\GL(V)$ and $\Sigma_n$. The decomposition of the group algebra of $\Q[\Sigma_n]$ implies a decomposition of $V^{\otimes n}$ into Schur functors applied to $V$:
$$ V^{\otimes n} \cong V^{\otimes n} \otimes_{\Sigma_n} \Q[\Sigma_n] \cong \bigoplus_{\lambda \vdash n} S^\lambda(V) \otimes \sigma_\lambda^\vee.$$
In particular, the decomposition of $V^{\otimes n}$ into irreducible representations of $\Sigma_n$ completely determines the decomposition into irreducible representations of $\GL(V)$, and vice versa. This statement, and the above explicit decomposition of $V^{\otimes n}$, is known as \emph{Schur--Weyl duality}. 

\end{para}\begin{para}\label{weyl construction}
Now suppose that $V$ is a symplectic vector space of dimension $2g$. Contracting with the symplectic form defines a map $V \otimes V \to \Q$. We define 
$$ V^{\langle n \rangle} := \mathrm{Ker}\Big(V^{\otimes n} \to \bigoplus_{\binom n 2} V^{\otimes (n-2)}\Big).$$ 
Clearly the subspace $V^{\langle n \rangle} \subset V^{\otimes n}$ is $\mathrm{Sp}(V)$-invariant. \emph{Weyl's construction of the irreducible representations of the symplectic group} can be formulated as the assertion that there are unique subrepresentations $V_\lambda \subset S^\lambda(V)$ of $\mathrm{Sp}(V)$ such that 
$$ V^{\langle n \rangle} = \bigoplus_{\lambda \vdash n} V_\lambda \otimes \sigma_\lambda^\vee \subset \bigoplus_{\lambda \vdash n} S^\lambda(V) \otimes \sigma_\lambda^\vee = V^{\otimes n},$$
that $V_\lambda \neq 0$ if and only if $g \geq \length(\lambda)$, and that the set of representations $\{V_\lambda\}$ for $\lambda = (\lambda_1 \geq \lambda_2 \geq \ldots \geq \lambda_g \geq 0)$ form a complete set of pairwise distinct irreducible representations of $\mathrm{Sp}(V)$. 

\end{para}
\begin{defn}\label{symplectic cauchy}We define the family of \emph{symplectic Schur polynomials} $s_{\langle\lambda\rangle} \in \Lambda$, where $\lambda$ is a partition, as the unique family of elements satisfying the equation  $$ \Exp(e_2(y))\sum_\lambda s_{\langle \lambda \rangle} (x)s_\lambda(y) = \sum_\lambda s_\lambda(x)s_\lambda(y).$$

\end{defn}

\begin{para}The functions $s_{\langle \lambda\rangle}$ bear the same relationship to the representations $V_\lambda$ of $\Sp(V)$ as $s_\lambda$ does to the representations $S^\lambda(V)$ of $\GL(V)$. More precisely, the character of the irreducible representation $V_\lambda$ of $\Sp(2g)$ of highest weight $\lambda$ is given by 
$$s_{\langle\lambda\rangle}(x_1,x_1^{-1},\ldots,x_g,x_g^{-1}).$$
We refer to the latter as a \emph{symplectic Schur function}. Thus we attempt to make a distinction between the symplectic Schur polynomial $s_{\langle\lambda\rangle}$ (a polynomial, symmetric in its arguments) and the symplectic Schur function $s_{\langle\lambda\rangle}(x_1^\pm,\ldots,x_g^\pm)$ (a Laurent polynomial, with hyperoctahedral symmetry in its arguments). 
\end{para}

\begin{para}An explicit bialternant-type formula for the symplectic Schur functions is given by the Weyl character formula:\label{wcf}
	\[ s_{\langle\lambda\rangle}(x_1^\pm,\ldots,x_g^\pm) = \frac{\det\big(x^{\lambda_j + g-j+1}_i - x^{-(\lambda_j + g-j+1)}_i\big)_{1 \leq i,j \leq g}}{\det\big(x^{g-j+1}_i - x^{-(g-j+1)}_i\big)_{1 \leq i,j \leq g}}\]
\end{para}

\begin{para}It is clear from \cref{symplectic cauchy} that for all $\lambda$, $$s_\lambda = s_{\langle\lambda\rangle} + \text{lower order terms},$$ where the lower order terms are sums of symplectic Schur polynomials $s_{\langle\mu\rangle}$ with $\vert \mu \vert < \vert \lambda \vert$ (and $\vert \mu \vert \equiv \vert \lambda \vert \pmod 2$). This is also clear from Weyl's construction, and the representation-theoretic interpretation of $s_{\langle\lambda\rangle}$. In particular, the symplectic Schur polynomials are indeed a basis of $\Lambda$ as an abelian group. 
\end{para}

\begin{para}\label{no symplectic cauchy} It is interesting to specialize the ``symplectic Cauchy identity''
$$ \prod_{i,j} (1-x_i y_j)^{-1} = \Exp(e_2(y)) \sum_\lambda s_{\langle \lambda\rangle}(x) s_\lambda(y)$$
to finitely many variables. Suppose we set $x_1=t_1$, $x_2=t_1^{-1}$, $x_3=t_2$, \ldots, $x_{2g}=t_g^{-1}$, and $x_i=0$ for $i>2g$. Then for $\length(\lambda)\leq g$, $s_{\langle\lambda\rangle}(x)$ specializes to the symplectic Schur function $s_{\langle\lambda\rangle}(t_1^\pm,\ldots,t_g^\pm)$. Unfortunately, the symplectic Schur polynomials $s_{\langle\lambda\rangle}(x)$ for $\length(\lambda)>g$ specialize in a rather complicated and nontrivial way under this substitution. However, if we specialize the $x$-alphabet as above and we set the variables $y_j=0$ for $j>g$, then we will have $s_\lambda(y)=0$ whenever $\length(\lambda)>g$, and so all terms involving $s_{\langle\lambda\rangle}(x)$ for $\length(\lambda)>g$ vanish. Thus in this case we obtain a useful version of the symplectic Cauchy identity also in finitely many variables: whenever $r \leq g$, we have
$$ \prod_{i=1}^g \prod_{j=1}^r (1-t_i y_j)^{-1} (1-t_i^{-1} y_j)^{-1} = \prod_{1\leq i<j\leq r} (1-y_iy_j)^{-1} \sum_\lambda s_{\langle \lambda\rangle}(t_1^\pm,\ldots, t_g^\pm) s_\lambda(y_1,\ldots, y_r).$$
See also \cite{Su90}. Here we used that $\Exp(e_2(y)) = \prod_{i<j} (1-y_iy_j)^{-1}$; this follows by combining the formulas in \S 5, Example 4(b), and \S 8, Example 6(b), from \cite[Chapter 1]{macdonald}.    By contrast, the ``dual symplectic Cauchy identity''
    \begin{align*}
         \prod_{i,j} (1+x_i y_j) & = \Exp(h_2(y)) \sum_\lambda s_{\langle \lambda\rangle}(x) s_{\lambda'}(y) \\
         & = \prod_{i\leq j} (1-y_i y_j)^{-1} \sum_\lambda s_{\langle \lambda\rangle}(x) s_{\lambda'}(y)
    \end{align*}
    does not admit any simple specialization to finitely many variables in terms of the symplectic Schur functions.
\end{para}

\section{Recollections on moduli of surfaces}

\begin{para}
	In this section we recall background on the homotopical study of moduli spaces of Riemann surfaces.
\end{para}

\subsection{Moduli of surfaces}\label{subsection-moduli of surfaces}
\newcommand{\Diff}{\mathrm{Diff}}

\begin{para}\label{provisional definition}We want to define first of all a space $M_g^n$, which will be the moduli space of smooth oriented genus $g$ surfaces with $n$ boundary components. Let us first give a provisional definition of the space $M_g^n$, and then explain why it is not completely ideal for our purposes. Choose for any $g$ and $n$ a ``reference surface'' $S_{g,n}$, which is a compact smooth oriented surface of genus $g$ with $n$ boundary components. We define
$$ M_g^n := B\Diff_\partial(S_{g,n}),$$
the classifying space of the topological group of self-diffeomorphisms of $S_{g,n}$ restricting to the identity on some collar neighborhood of the boundary components. When $n=0$ we should impose the condition that diffeomorphisms are orientation-preserving.\end{para}

\begin{rem}\label{ag interpretation} The space $M_g^n$ may be thought of algebro-geometrically: it is homotopy equivalent to the analytification of the moduli stack over $\Z$ parametrizing smooth compact curves of genus $g$ equipped with $n$ distinct ordered marked points, and a nonzero tangent vector at each marking. This interpretation will not play any role until the later sections of the paper. 
\end{rem}

\begin{para}\label{desiderata}Let us now indicate some reasons why this definition of $M_g^n$ is not optimal. We would for example like to say that the disjoint union $$\coprod_{g \geq 0} M_g^1$$ carries the structure of a framed $E_2$-algebra in a geometrically natural way: given a little disk with $k$ parametrized holes, and $k$ surfaces with parametrized boundary, we glue the surfaces to the respective holes to get a new surface with a parametrized boundary component. This framed $E_2$-structure was first considered (although perhaps not rigorously justified) by Miller \cite{miller}. More generally, the spaces $\{M_g^n\}$ ought to form a topological modular operad \cite{getzlerkapranov}, meaning that there are gluing operations
$$ M_{g_1}^{n_1+1} \times M_{g_2}^{n_2+1} \to M_{g_1+g_2}^{n_1+n_2}$$
and $$ M_{g}^{n+2} \to M_{g+1}^n$$
satisfying suitable associativity constraints. This modular operad (or rather a closely related PROP) was first considered by Segal \cite{segalCFT}. The framed $E_2$-structure just mentioned should be an instance of this modular operad structure: if $\{P(g,n)\}$ is any topological modular operad, then $\coprod_{g \geq 0} P(g,1)$ forms an algebra over the cyclic operad $\{P(0,n)\}$, via the gluing maps
$$ P(0,n+1) \times P(g_1,1) \times \ldots \times P(g_n,1) \to P(g_1+\ldots+g_n,1).$$
 This should then recover the framed $E_2$-structure on $\coprod_{g \geq 0} M_g^1$ by an identification of the cyclic operad $\{M_0^n\}$ with the framed $E_2$-operad (cf. \cite{kimurastasheffvoronov2}). 
\end{para}
\begin{para}
However, none of the desiderata of \S\ref{desiderata} are satisfied by the spaces $M_g^n$ as defined in \S\ref{provisional definition}: indeed, the gluing map $ M_{g_1}^{n_1+1} \times M_{g_2}^{n_2+1} \to M_{g_1+g_2}^{n_1+n_2}$ depends on the \emph{noncanonical} choice of a diffeomorphism between $S_{g_1+g_2,n_1+n_2}$ and the surface obtained by gluing $S_{g_1,n_1+1}$ and $S_{g_2,n_2+1}$ along a boundary component, and similarly for the other gluing operations considered. This completely destroys the associativity of gluing of surfaces. \end{para}

\begin{para}\label{slightly better definition}
A slightly better definition of the space $M_g^n$ is as follows. Consider the following topologically enriched groupoid $D_{g,n}$: the objects of $D_{g,n}$ are smooth oriented compact surfaces of genus $g$ with $n$ ordered boundary components, each of which is equipped with a germ of a parametrized collar neighborhood. Morphisms are spaces of diffeomorphisms preserving the collars. We then define
$ M_g^n := B D_{g,n}$
to be its classifying space. This definition has the advantage of not relying on any arbitrary choice. 
\end{para}

\begin{para}\label{stack of surface bundles}
    In a different direction, one may replace the topological space with the \emph{differentiable stack} $\mathfrak M_g^n$ parametrizing smooth oriented genus $g$ surface bundles with  $n$ collared boundary components. The stack $\mathfrak M_g^n$ is isomorphic to the classifying stack of $\mathrm{Diff}_\partial(S_{g,n})$, but does not depend on the choice of a reference surface. 
\end{para}

\begin{para}The definitions of \S\ref{slightly better definition} or \S\ref{stack of surface bundles} are philosophically ``correct'', but still have some problems. First of all, the category $D_{g,n}$ is not small, so its classifying space has a proper class of points. Secondly, the collection $\{D_{g,n}\}$ does not form a modular operad in a $1$-category, but in a $2$-category (or rather a $(2,1)$-category) --- that is, gluing of surfaces is not strictly associative, only associative up to a canonical natural equivalence.\footnote{This has nothing to do with topology: even on the level of the underlying sets we do not have associativity. Compare the fact that $(X\sqcup Y) \sqcup Z$ is not strictly equal to $X \sqcup (Y \sqcup Z)$, for sets $X$, $Y$ and $Z$.} This means that the classifying spaces $\{BD_{g,n}\}$ only form a topological modular operad up to coherent homotopy. The same issue arises with \S\ref{stack of surface bundles}, since stacks only form a $(2,1)$-category. One can get around explicitly working with 2-categories by defining a modular operad as in Costello \cite{costello-moduli}, as a strong symmetric monoidal functor out of a category of graphs; via the (monoidal) Grothendieck construction \cite{monoidalgrothendieck}, a modular operad in $\mathsf{Cat}$ or $\mathsf{Grpd}$ can then be defined as a certain fibered category over a category of graphs. An analogous construction works for operads in stacks. However, we will in practice instead give ourselves the freedom of replacing the spaces $M_g^n$ by another convenient model as we please. For example, in order for $\coprod_{g \geq 0} M_g^1$ to be an algebra over the standard framed little disk operad, one may define $M_g^1$ to be the space of smoothly embedded genus $g$ surfaces in $\mathbb D \times \R^\infty$ (where $\mathbb D$ denotes the standard unit disk), with boundary $\partial \mathbb D \times \{0\}$, such that the surface is tangent to infinite order to the standard disk $\mathbb D \times \{0\}$ along the boundary. Another natural way to obtain a genuine topological modular operad, due to Tillmann \cite{tillmann-detects}, is to fix once and for all an ``atomic'' pair of pants $P$, and replace the category $D_{g,n}$ with the full subcategory whose objects are those finitely many surfaces that can be built by gluing $P$ to itself in all possible ways. This collection $\{D_{g,n}\}$ is a modular operad in the $1$-category of small categories and functors (but no natural transformations), and taking classifying spaces yields a genuine topological modular operad. If we want to include also surfaces for which $2g-2+n \leq 0$ then we need to include in addition an ``atomic'' disk as a building block. Many such workarounds have been considered in the literature, too many to mention them all.
\end{para}

\begin{rem}
Many further variations in the above set-up are possible. We could have replaced diffeomorphisms with homeomorphisms throughout,  without changing the situation: the natural homomorphism $ \mathrm{Diff}_\partial(S_{g,n}) \to \mathrm{Homeo}_\partial(S_{g,n}) $ is a homotopy equivalence. (This would have the advantage of excising all references to collars from the discussions, and in \S\ref{stack of surface bundles} one could work with topological stacks, which are perhaps less exotic than differentiable stacks.) Perhaps more importantly, the connected components of $\mathrm{Diff}_\partial(S_{g,n})$ are contractible, except for a few low-genus exceptions. In fact, these diffeomorphism groups have contractible connected components precisely when the corresponding algebro-geometric moduli stack is of Deligne--Mumford type. In these cases we could therefore have replaced $\mathrm{Diff}_\partial(S_{g,n})$ with $\pi_0\,\mathrm{Diff}_\partial(S_{g,n})=\mathrm{Mod}_g^n$, the \emph{mapping class group}. These facts are classical in surface topology and are due to work of many people, in particular Teichm\"uller \cite{teichmuller}, Earle--Eells \cite{earle-eells}, and Hamstrom \cite{hamstrom}. For a textbook account, see Farb--Margalit \cite[Chapter 1]{farbmargalit}. 
\end{rem}

\subsection{Stable homology and Mumford's conjecture}

\begin{para}\label{stabilization is E2 multiplication}
	If we fix an arbitrary point of $M_1^2$, then repeated composition with this element in the modular operad $\{M_g^n\}$ defines a sequence of maps
	$$ \ldots \to M_g^1 \to M_{g+1}^1 \to M_{g+2}^1 \to \ldots $$
	Equivalently, these maps may be considered as the result of repeatedly multiplying by a fixed element of $M_1^1$ in the framed $E_2$-algebra $\coprod_{g \geq 0} M_g^1$. These are precisely the stabilization maps considered by Harer \cite{harerstability}, who proved the following fundamental theorem:
\end{para}

\begin{thm}[Harer]\label{harer}
	The maps $H_i(M_g^1,\Z) \to H_i(M_{g+1}^1,\Z)$ are isomorphisms for $i\leq \tfrac{2g-2}{3}$.
\end{thm}

\begin{rem}The stable range given here was first obtained by Boldsen \cite{boldsen}. Harer's original argument yielded a stable range of slope $\tfrac{1}{3}$.\end{rem}

\begin{para}We set $M_\infty^1 := \operatorname{hocolim}_g M_g^1$. After Harer's stability theorem, one wants to compute the homology of $M_\infty^1$, i.e.\ the stable homology of $M_g^1$. Harer's stability theorem implies that the hypothesis of the group-completion theorem of McDuff--Segal \cite{mcduffsegal} are satisfied, which shows that this problem is equivalent to computing the homology of the group-completion of the $E_1$-algebra $\coprod_{g \geq 0} M_g^1$. More formally, the group-completion $\Omega B (\coprod_{g \geq 0} M_g^1)$ is weakly equivalent to $\Z \times (M_\infty^1)^+$, where $(-)^+$ denotes Quillen's plus construction. Thus $M_\infty^1$ is homology equivalent to the base component of the loop space $\Omega B (\coprod_{g \geq 0} M_g^1)$. 
\end{para}

\begin{para}\label{calculate rat hom}
	Since $\coprod_{g \geq 0} M_g^1$ is an $E_2$-algebra, its group-completion is a double loop space. Tillmann \cite{tillmann-stable} showed, surprisingly, that it is in fact an \emph{infinite} loop space. 
 This infinite loop space was subsequently completely identified by Madsen and Weiss \cite{madsenweiss}, who proved an equivalence
	$$ \Omega B (\coprod_{g \geq 0} M_g^1) \simeq \Omega^\infty \mathrm{MTSO}(2),$$
	where $\mathrm{MTSO}(2)$ is the Madsen--Tillmann spectrum: explicitly, this is the Thom spectrum associated to the virtual vector bundle $-\gamma$, where $\gamma$ is the tautological bundle over $B\mathrm{SO}(2)$. From this equivalence, Madsen and Weiss in particular calculated the rational homology of $M_\infty^1$: 
	\begin{itemize}
		\item Rationally, the homology of a connected loop space is the symmetric algebra on its rational homotopy groups.
		\item Rationally, the homotopy groups of the Thom spectrum of a vector bundle coincide with the homology groups of the corresponding Thom space.
		\item The homology of the Thom space of an oriented vector bundle agrees with the homology of the base space, shifted by the rank of the bundle. 
	\end{itemize}
Hence rationally the calculation reduces to knowing the homology of $B \mathrm{SO}(2) \simeq \C \mathbb P^\infty$: one has that $H_\bullet(\Omega^\infty_0 \mathrm{MTSO}(2),\Q) \cong \Sym\, \pi_{\bullet>0}^\Q \mathrm{MTSO}(2)$, and $\pi_\bullet^\Q \mathrm{MTSO}(2)$ is one-dimensional in degrees $-2$, $0$, $2$, $4$, ... In terms of cohomology, it follows that $H^\bullet(M_\infty^1,\Q)$ is a polynomial algebra with one generator in each positive even degree, proving the Mumford conjecture.
\end{para}

\begin{rem}We focused above on the case of one boundary component, but a more precise version of Harer's theorem and a more careful argument shows that $(M_\infty^n)^+ \simeq \Omega^\infty_0 \mathrm{MTSO}(2)$ for any $n \geq 0$. In particular, the stable homology of $M_g^n$ is independent of $n$. 
\end{rem}

\begin{rem}A very different proof of Mumford's conjecture, which does \emph{not} prove the Madsen--Weiss theorem as an intermediate step, was recently given by Bianchi \cite{bianchi1,bianchi2,bianchi3,bianchi4}. See also \cite{daspetersen}. \end{rem}

\subsection{Surfaces in a background space}

\begin{para} Galatius--Madsen--Tillmann--Weiss \cite{gmtw} have proved a more general form of the Madsen--Weiss theorem, in which all surfaces involved are equipped with a \emph{tangential structure}. We will only need to consider the simplest type of tangential structure on a surface: the datum of a continuous map to some fixed background space. The moduli of surfaces in a background space was studied by Cohen--Madsen \cite{backgroundspace}. \end{para}

\begin{para}\label{background space definition}Let $X$ be a based space, which we assume to be simply connected. We define $M_g^n(X)$, the space of $n$-holed genus $g$ surfaces equipped with a map to $X$, as the space
	$$ \mathrm{map}_\ast(S_{g,n}/\partial S_{g,n},X) /\!/ \Diff_\partial(S_{g,n}),$$
	where $\mathrm{map}_\ast(S_{g,n}/\partial S_{g,n},X)$ denotes the space of continuous maps $S_{g,n} \to X$ taking the boundary to the basepoint, equipped with the compact-open topology. Just as discussed in \S\S\ref{desiderata}--\ref{slightly better definition} this definition is not optimal, as it depends on the choice of a reference surface $S_{g,n}$. As in \S\ref{slightly better definition} a ``coordinate-free'' construction can be obtained by consider the topologically enriched functor $D_{g,n} \to \mathsf{Top}$,  which takes a surface $S$ with boundary to the mapping space $\mathrm{map}_\ast(S/\partial S,X)$. A better definition of $M_{g}^n(X)$ would then be as the homotopy colimit of this enriched functor, computed e.g.\ using the bar construction. 
\end{para}

\begin{para}The space $M_g^n(X)$ is typically disconnected. If $S \to X$ is a map from an oriented surface, then the image of the fundamental class of $S$ gives a well-defined element of $H_2(X,\Z)$, and the resulting function $\pi_0(M_g^n(X)) \to H_2(X,\Z)$ is a bijection. 
\end{para}

\begin{para}\label{fiber sequence}
	There is for every $g$ and $n$ a fiber sequence
	$$ \mathrm{map}_\ast(S_{g,n}/\partial S_{g,n},X) \longrightarrow M_g^n(X) \longrightarrow M_g^n. $$
	We denote by a subscript $0$ the base component, which parametrizes nullhomotopic maps to $X$. We also have 
	$$ \mathrm{map}_\ast(S_{g,n}/\partial S_{g,n},X)_0 \longrightarrow M_g^n(X)_0 \longrightarrow M_g^n. $$
\end{para}

\begin{para}
	The spaces $\{M_g^n(X)\}$ form a topological modular operad, just as the spaces $\{M_g^n\}$, and the inclusion of the constant maps $M_g^n \to M_g^n(X)$ and the forgetful map $M_g^n(X) \to M_g^n$ that forgets the data of the map to $X$, are both morphisms of topological modular operads. This implies in particular that the disjoint union $\coprod_{g \geq 0} M_g^1(X)$ is a framed $E_2$-algebra, just like $\coprod_{g \geq 0} M_g^1$. The group-completion of this $E_2$-algebra is determined as a special case of the more general version of the Madsen--Weiss theorem of \cite{gmtw}.
\end{para}

\begin{thm}[Galatius--Madsen--Tillmann--Weiss] The group-completion of $\, \coprod_{g \geq 0} M_g^1(X)$ is the infinite loop space $\Omega^\infty (\mathrm{MTSO}(2) \wedge X_+)$. 
\end{thm}

\begin{para}\label{rational homology with background space}The rational homology of this infinite loop space can be computed in exactly the same fashion as in \S\ref{calculate rat hom}, cf.\ \cite{backgroundspace}. The result is that the base component $M_\infty^1(X)_0$, parametrizing nullhomotopic maps into $X$, satisfies
	$$ H_\bullet(M_\infty^1(X)_0,\Q) \cong H_\bullet(\Omega^\infty_0(\mathrm{MTSO}(2) \wedge X_+),\Q) \cong \Sym \, \tau_{>0}\big( \pi_\bullet^\Q \mathrm{MTSO}(2) \otimes {H}_\bullet(X,\Q)\big).$$
	Here $\tau_{>0}(-)$ denotes truncation into positive degrees of a graded vector space. 
\end{para}

\subsection{Twisted coefficients}

\begin{para}\label{twisted1}Harer's stability theorem is part of a much larger framework of homological stability theorems, whose history we will not attempt to survey here. Since the early days of homological stability, it was noticed that it is not only desirable to know that the homology stabilizes with trivial coefficients but also for homology with certain twisted coefficients. One natural class of twisted coefficients for homological stability results are the \emph{polynomial coefficient systems}, a notion going back to work of Dwyer \cite{dwyertwisted}. Randal-Williams and Wahl \cite{randalwilliamswahl} (see also Krannich \cite{krannichtopologicalmoduli}) have made precise the idea that in ``natural situations'' where homological stability is satisfied, one also always has homological stability for polynomial coefficients. 
\end{para}

\begin{para}\label{twisted2}
	We will not give a general definition of a polynomial coefficient system here. In the case of the spaces $M_g^1$, a coefficient system consists in particular of the data of a local system $A{(g)}$ on $M_g^1$ for all $g$, and for each $g$ a map $A{(g)} \to s^{\ast} A{(g+1)}$, where $s \colon M_g^1 \to M_{g+1}^1$ is the stabilization. Hence we get in particular maps on homology $H_i(M_g^1,A{(g)}) \to H_i(M_{g+1}^1,A{(g+1)})$. The definition of polynomiality is inductive: the coefficient system is of degree $-1$ if all $A{(g)}$ are zero, and it is of degree $d$ if $A{(g)} \to s^{\ast} A{(g+1)}$ is injective and its cokernel is of degree $d-1$. All our examples will in fact be \emph{split} polynomial coefficient systems, in which case one obtains a better stable range. See \cite[Section 4]{randalwilliamswahl} for precise definitions. To ease notations, we omit the argument $(g)$, writing the stabilization maps as $H_i(M_g^1,A) \to H_i(M_{g+1}^1,A)$ and denoting by the same symbol $A$ all local systems occurring in the coefficient system. Homological stability of $M_g^1$ with polynomial coefficients was first proven by Ivanov \cite{ivanovtwisted}.  The stable range in the following theorem is due to Boldsen \cite{boldsen}.
\end{para}

\begin{thm}[Ivanov]\label{ivanov theorem}
	Let $A$ be a split polynomial coefficient system of degree $d$ on the spaces $\{M_g^1\}$. Then $H_i(M_g^1,A) \to H_{i}(M_{g+1}^1,A)$ is an isomorphism for $i\leq\frac{2g-d-2}{3}$.
\end{thm}

\begin{para}\label{definition of V}
	 The group $\Diff_\partial(S_{g,1})$ acts on the cohomology group $H^1(S_{g,1}/\partial S_{g,1},\Z)$, which defines a local system of rank $2g$ on the classifying space $B\Diff_\partial(S_{g,1}) \simeq M_g^1$. This collection of local systems assemble to a split polynomial coefficient system of degree $1$. Its $d$-fold tensor power is a natural example of a split polynomial coefficient system of degree $d$. We will only care about the rational homology of these coefficient systems. We denote by the symbol $V$ the rationalization  $H^1(S_{g,1}/\partial S_{g,1},\Q)$, and we call it the \emph{standard coefficient system} on $M_g^1$.  If $S^\lambda$ is a Schur functor, then $S^\lambda(V)$ is a split polynomial coefficient system of degree $\vert \lambda \vert$. The first cohomology of a surface naturally carries a symplectic form, so the monodromy of the local system $V$ lies in the symplectic group $\Sp(2g,\Z)$. Any partition $\lambda$ defines a system of representations of all symplectic groups, as explained in \S\ref{weyl construction}, and we denote by $V_\lambda$ the corresponding coefficient system. (For example, $S^{1,1}(V) = \wedge^2 V \cong V_{1,1} \oplus V_{0}$.) For every $\lambda$, $V_\lambda$ is a split polynomial coefficient system of degree $\vert \lambda \vert$. 
\end{para}

\begin{rem}\label{uniform remark}For the particular polynomial coefficient systems $V_\lambda$, the paper \cite{MPPRW} proves a \emph{uniform} version of \cref{ivanov theorem}, with slope $\frac 1 4$ but no dependence on $d$. Such a result does not hold for the coefficient systems $S^\lambda(V)$.
\end{rem}

\begin{rem}We could equally well have defined $V$ as the cohomology group $H^1(S_{g,n},\Q)$, via the natural isomorphism $H^1(S_{g,1}/\partial S_{g,1},\Q) \to H^1(S_{g,1},\Q)$. But taking cohomology relative to the boundary furnishes naturally the correct functoriality: when $S_{g,1}$ is embedded into $S_{g+1,1}$, we get a collapse map in the opposite direction, $S_{g+1,1}/\partial S_{g+1,1} \to S_{g,1}/\partial{S_{g,1}}$.
\end{rem}

\begin{rem}\label{target}Another natural example of a coefficient system is given by the following construction. Via the fiber sequences of \S\ref{fiber sequence}, we have for any simply connected space $X$ and $d \geq 0$ a local system $H_d(\mathrm{map}_\ast(S_{g,1}/\partial S_{g,1},X),\Z)$ on $M_g^1$. These local systems assemble to a split polynomial coefficient system of degree $\leq d$ (cf. the results of \cite[Section 2]{backgroundspace}). From this result, combined with Ivanov's \cref{ivanov theorem}, it follows that the stabilization maps $ H_i(M_g^1(X),\Z) \to H_i(M_{g+1}^1(X),\Z)$ are isomorphisms in the same range as in the case $X=\{\ast\}$. Indeed, one applies the Serre spectral sequence to the fiber sequence of \S\ref{fiber sequence}; stabilization induces an isomorphism on the $E_2$-page of the spectral sequence in the indicated range, hence also on the abutments. 
	
%
%
\end{rem}

\begin{para}
	The stable (rational) cohomology of $M_g^1$ with coefficients in $S^\lambda(V)$ was first calculated by Looijenga \cite{looijengastable}. One may equivalently formulate Looijenga's result as a calculation of the stable cohomology with coefficients in $V_\lambda$, using the stable branching formula from $\GL(2g)$ to $\Sp(2g)$ to decompose the representation $S^\lambda(V)$ as a direct sum of $V_\lambda$'s. More precisely, Looijenga calculated $H^\bullet(M_g^1,V_\lambda)$ as a free module over $H^\bullet(M_g^1,\Q)$, which after the proof of the Mumford conjecture gives a complete calculation of the stable cohomology. Looijenga's argument was formulated for closed surfaces, but a small variation of his arguments (producing a slightly different answer) works equally well for surfaces with boundary. Looijenga's argument does not seem to generalize to the hyperelliptic setting in any natural way, however. 
\end{para}

\subsection{An argument of Randal-Williams}

\begin{para}
	A completely different way to calculate the stable homology of $M_g^1$ with symplectic coefficients was later given by Randal-Williams \cite[Appendix B]{randalwilliamstwisted}. Whereas Looijenga's argument treats the Mumford conjecture as a black box, the approach of Randal-Williams uses the Madsen--Weiss theorem in its strong form, and in particular the freedom to work with surfaces in a background space. The idea is as follows:\begin{enumerate}
		\item Choose the background space to be an Eilenberg--MacLane space $KA$ of a graded $\Q$-vector space $A$.
		\item Treat the homology of $M_\infty^1(KA)$ as an \emph{analytic functor} of $A$. Try to determine this analytic functor in two different ways. 
		\item The first way is to study the Serre spectral sequence of the fibration $M_g^1(KA) \to M_g^1$. From this one sees that the Taylor coefficients of this analytic functor are given by the homology groups $H_\bullet(M_g^1,S^\lambda(V))$. 
		\item The second way is to use Madsen--Weiss with tangential structure, in which case one can immediately read off the stable homology. 
	\end{enumerate}
Equating the expressions obtained from (3) and (4) recovers Looijenga's theorem. Since this is the argument we will generalize to hyperelliptic surfaces, we will explain Randal-Williams's argument in more detail for the reader's convenience. See also \cite[Section 3]{kupersrandalwilliams-torelli} for a similar application of this technique.
\end{para}

%
%
%

\begin{para}If $A$ is any nonnegatively graded chain complex of abelian groups, then by the Dold--Kan correspondence it may be considered as a simplicial abelian group. We will denote by $KA$ its geometric realization. More generally, if $A$ is an arbitrary chain complex, then we define $KA := K(\tau_{\geq 0} A)$. A space of this form will be called a \emph{generalized Eilenberg--MacLane space}. Up to homotopy, generalized Eilenberg--MacLane spaces are just products of usual Eilenberg--MacLane spaces, and $\pi_n(KA,\ast) \cong H_n(A)$. 
\end{para}

\begin{lem}\label{lem1}The space of maps into any generalized Eilenberg--MacLane space is a generalized Eilenberg--MacLane space. More precisely, for any based space $S$ and chain complex $A$ there is a natural weak equivalence $$\mathrm{map}_\ast(S,KA) \simeq K(\widetilde C^\bullet(S,A)).$$ \end{lem}

\begin{proof}A highbrow argument is as follows. The construction $A \mapsto KA$ defines a functor from $D(\Z)$ (the derived $\infty$-category of abelian groups) to $\mathrm{Spaces}_\ast$ (the $\infty$-category of pointed spaces). As such it turns out to have a natural universal property: it is the right adjoint of the functor of singular chains $\widetilde C_\bullet(-,\Z)$, as one sees by considering the following two adjunctions:
	$$ \begin{tikzcd}
		\mathrm{Spaces}_\ast \arrow[rr,bend left, shift left, dotted, "{\widetilde{C}_\bullet(-,\Z)}"]\arrow[r,shift left,"\Sigma^\infty"]& \arrow[l,shift left,"\Omega^\infty"]\arrow[r,shift left,"-\otimes H\Z"]\mathrm{Spectra} &  H\Z\mathrm{-Mod} \simeq D(\Z)\arrow[l,shift left,"H"] \arrow[ll,shift left, bend left,dotted,"K"]\end{tikzcd}$$
	Then the proposition is a reformulation of adjointness, since the right-hand side is nothing but the space of maps $\widetilde C_\bullet(S,\Z) \to A$ in $D(\Z)$. Indeed, mapping spaces in $D(\Z)$ are defined by taking the internal (derived) Hom of chain complexes and applying the functor $K$. \end{proof}

\begin{rem}We could have written \cref{lem1} as $\mathrm{map}_\ast(S,KA) \simeq K(\widetilde H^\bullet(S,A))$, at the expense of deleting the adjective ``natural'': every chain complex over $\Z$ is quasi-isomorphic to its homology, but not naturally. \end{rem}

\begin{rem}Let us make explicit the degree conventions in \cref{lem1}. Here $K$ is a functor defined on chain complexes, and it depends only on the truncation of a chain complex into nonnegative degrees. Meanwhile, the cochain complex $\widetilde C^\bullet(S,\Z)$ is considered as a chain complex concentrated in nonpositive degree. If $A = \Z[N]$ is the chain complex consisting of $\Z$ placed in degree $N$, the result shows that
	$$ \mathrm{map}_\ast(S,K(\Z,N)) \simeq \prod_{n=0}^N K(\widetilde H^{N-n}(S,\Z),n).$$\end{rem}

\begin{thm}\label{first expression}Suppose that $A$ is a graded vector space concentrated in degree $2$. There is a functorial isomorphism 
	$$ H_\bullet(M_g^1(KA)_0,\Q) \cong \bigoplus_\lambda S^\lambda(A) \otimes H_\bullet(M_g^1,S^\lambda(V[-1])).$$
\end{thm}

\begin{proof}
	We study the Serre spectral sequence associated with the fibration$$\map_\ast(S_{g,1}/\partial S_{g,1},KA)_0 \to M_g^1(KA)_0 \to M_g^n.$$
	By the previous lemma, we have that $$\map_\ast(S_{g,1}/\partial S_{g,1},KA) \simeq K(A \otimes V[-1]) \times K(A[-2]),$$
	where as before we denote by $V = H^1(S_{g,1}/\partial S_{g,1},\Q)$. Recalling that $A$ is concentrated in degree $2$ we see that $\map_\ast(S_{g,1}/\partial S_{g,1},KA)_0 \simeq K(A \otimes V[-1])$. The Serre spectral sequence degenerates immediately, since the differentials are $\GL(A)$-equivariant, and $\GL(A)$ acts  on the homology of the fibers with different weights in different degrees. It follows that
	\begin{align*} H_\bullet(M_g^1(KA)_0,\Q) & \cong H_\bullet(M_g^1,H_\bullet(K(A \otimes V[-1]))) \\
		 & \cong H_\bullet(M_g^1, \Sym (A \otimes V[-1])) \\
		 & \cong \bigoplus_\lambda S^\lambda(A) \otimes H_\bullet(M_g^1, S^\lambda(V[-1])),
	\end{align*}
where in the last step we used the Cauchy identity $\Sym(M \otimes N) \cong \bigoplus_\lambda S^\lambda(M) \otimes S^\lambda(N)$. 
\end{proof}

\begin{prop}\label{second expression}Let $A$ be a graded vector space concentrated in degree $2$. The analytic functor $A \mapsto H_\bullet(M_\infty^1(KA)_0,\Q)$ has the Taylor series
	$$ \Exp(z^{-2}(\Exp(z^2 + h_1) - 1 - z^2 - h_1 )) $$
\end{prop}

\begin{proof}We use the expression for $H_\bullet(M_\infty^1(KA)_0,\Q)$ from \S\ref{rational homology with background space}. 	The Taylor series of the analytic functor $A \mapsto H_\bullet(KA,\Q) \cong \mathrm{Sym}(A)$ is $\Exp(h_1)$, and then the Taylor series of the functor	$A \mapsto \pi_\bullet^\Q \mathrm{MTSO}(2) \otimes {H}_\bullet(KA,\Q)$ is $\tfrac{z^{-2}}{1-z^2} \Exp(h_1) = z^{-2}\Exp(z^2+h_1)$. Truncating this into positive degrees corresponds to subtracting three terms (keeping in mind that $A$ is concentrated in degree $2$, so $z^{-2}h_1$ corresponds to a class in homological degree $0$), and then we take again the exponential. 
\end{proof}

\begin{para}
	Massaging the resulting identity into a slightly more convenient form gives the following formula.
\end{para}

\begin{thm}
	$$\sum_\lambda \sum_k \dim H_k(M_\infty^1,V_\lambda) (-z)^k s_{\lambda} = \Exp(-e_2 + z^{-2}(\Exp(z^2 - zh_1) - 1 - z^2 + zh_1 )). $$
\end{thm}

\begin{proof}Equating the expressions from \cref{first expression} and \cref{second expression} shows that 
	$$\sum_\lambda \sum_k \dim H_k(M_\infty^1,S^\lambda(V[-1]) (-z)^k s_\lambda = \Exp(z^{-2}(\Exp(z^2 + h_1) - 1 - z^2 - h_1 )). $$
	Now $S^\lambda(V[-1]) = S^{\lambda'} (V) [-\vert \lambda \vert]$. Thus we can replace $V[-1]$ by $V$ on the left-hand side, at the expense of replacing interchanging $s_\lambda$ with $s_{\lambda'}$ (i.e. tensoring with the sign representation in each arity) and shifting the arity $r$ component in degrees by $r$. But these two changes combined amount to making the substitution $h_1 \leadsto -z h_1$ on the right-hand side. Finally,  changing $S^\lambda(V)$ to $V_\lambda$ on the left-hand side corresponds to multiplying the right-hand side with $\Exp(-e_2)$, cf.\  \cref{symplectic cauchy}.
\end{proof}

\begin{rem}
	We chose to work with $A$ concentrated in degree $2$, but it would also have been possible to derive the same final result by choosing instead $A$ to be $2$-connected (as Randal-Williams does). In this case the spaces $\mathrm{map}_\ast(S_{g,1}/\partial S_{g,1},KA)$ and $M_g^1(KA)$ would be connected, and in the proof of \cref{first expression} we would therefore pick up an additional factor of $K(A[-2])$. In the proof of \cref{second expression} we would also pick up an additional term, since $z^{-2}h_1$ would then correspond to a class in positive homological degrees. These two effects exactly cancel. 	The argument would however become more subtle at one point. In the proof of \cref{first expression} we used a weight argument to deduce degeneration of the Serre spectral sequence, and this weight argument breaks down if $A$ is $2$-connected. (At this point there is in fact a small gap in Randal-Williams's argument.) It is still true that the spectral sequence degenerates, but a more complicated argument is needed. Indeed, one can argue that there is a quasi-isomorphism
	$$ C_\bullet(\mathrm{map}_\ast(S_{g,1}/\partial S_{g,1},KA),\Q) \simeq H_\bullet(\mathrm{map}_\ast(S_{g,1}/\partial S_{g,1},KA),\Q)$$
	where both sides are considered as homologically locally constant complexes of sheaves of abelian groups on $M_g^1$, or equivalently, as parametrized $H\Z$-module spectra on $M_g^1$. To prove this quasi-isomorphism, one argues that 
	\begin{align*}
		C_\bullet(\mathrm{map}_\ast(S_{g,1}/\partial S_{g,1},KA),\Q) & \simeq C_\bullet(K(\widetilde C^\bullet(S_{g,1}/\partial S_{g,1},\Q) \otimes A)) \\
		& \simeq \Sym (\widetilde C^\bullet(S_{g,1}/\partial S_{g,1},\Q) \otimes A) \\
		& \simeq \Sym (\widetilde H^\bullet(S_{g,1}/\partial S_{g,1},\Q) \otimes A). 
	\end{align*}
The first two equivalences use respectively a parametrized version of \cref{lem1} and a parametrized version of the calculation of the rational homology of Eilenberg--MacLane spaces. The third equivalence can be deduced by using a section of the universal surface bundle over $M_g^1$ to construct a splitting. 
\end{rem}

\begin{rem}\label{standard coefficient system}In the above we only considered the stable homology $H_\bullet(M_\infty^n,S^\lambda V)$ in the case $n=1$. But in the same way as $H_\bullet(M_\infty^n)$ is independent of $n$, so is the stable homology with twisted coefficients, suitably interpreted. Indeed: for any $n \geq 0$, one has $M_\infty^n(X)^+ \simeq \Omega^\infty_0 (\mathrm{MTSO}(2) \wedge X_+)$. For any $n$ one also has a fiber sequence
	$$ \mathrm{map}_\ast(S_{g,n}/\partial S_{g,n},X) \to M_g^n(X) \to M_g^n, $$
	and when $X=KA$ the Serre spectral sequence degenerates. One may now run the same argument for any $n \geq 0$. Let us denote by 
	$$ V := \tau_{\leq 0} \widetilde H^{\bullet+1}(S_{g,n}/\partial S_{g,n},\Q).$$
 Note that if $n > 0$ then $V$ is a local system of rank $2g-1+n$ on $M_g^n$ (in particular, for $n>1$ it is \emph{not} the pullback of the local system $V$ on $M_g^1$). When $n=0$, $V$ is a complex of local systems with vanishing differential, consisting of the standard rank $2g$ local system --- which we shall denote $V'$ --- placed in degree $0$, and the trivial local system of rank $1$ in cohomological degree $-1$. One finds that the homology $H_\bullet(M_\infty^n, S^\lambda(V))$ is independent of $n$, for all $n \geq 0$. In particular, we may write down a generating series for cohomology of $M_g$ with symplectic coefficients, recovering the case originally considered by Looijenga \cite{looijengastable}. Let us write $V_\lambda'$ for the local system on $M_g$ whose pullback to $M_g^1$ is $V_\lambda$. The considerations just explained show that the analytic functor $A \mapsto H_\bullet(M_\infty^n,\Sym(A \otimes V))$ does not depend on $n$. But we also have $$H_\bullet(M_\infty,\Sym(A \otimes V)) \cong H_\bullet(M_\infty,\Sym(A \otimes (V' \oplus \Q[1]))) \cong H_\bullet(M_\infty,\Sym(A \otimes V')) \otimes \Sym(A[1]),$$
 which shows that the Taylor series of the two analytic functors $A \mapsto H_\bullet(M_\infty,\Sym(A \otimes V))$ and $A \mapsto H_\bullet(M_\infty,\Sym(A \otimes V'))$ differ by multiplication by $\Exp(-zh_1)$, which is the Taylor series of $A \mapsto \Sym(A[1])$. We then have the identity:
\end{rem}

\begin{thm}
	$$\sum_\lambda \sum_k \dim H_k(M_\infty,V_\lambda') (-z)^k s_{\lambda} = \Exp(zh_1 -e_2 + z^{-2}(\Exp(z^2 - zh_1) - 1 - z^2 + zh_1 )). $$
\end{thm}

\section{Moduli of hyperelliptic surfaces}

\subsection{Braids and hyperelliptic surfaces}

\begin{para}\label{branched def}
	Let $f: M \to N$ be a proper map between surfaces, possibly with boundary. We say that $f$ is a \emph{branched cover} if there exists a finite subset $z$ in the interior of $ N$ such that $f^{-1}(N\setminus z) \to N \setminus z$ is a finite-sheeted covering space, and for every point $x \in f^{-1}(z)$, there exist local coordinates around $x$ and $f(x)$ such that the map $f$ is diffeomorphic to the map $z \mapsto z^r$, $\C \to \C$, for some $r \in \Z_{>0}$. The points $x$ with $r>1$ are called \emph{ramification points}, and the corresponding points $f(x)$ are called \emph{branch points}.
\end{para}


\begin{defn}\label{defn hyperelliptic}A \emph{hyperelliptic surface} is an oriented compact surface $S$, possibly with boundary, equipped with an involution $\iota$ such that $S \to S/\langle\iota\rangle$ is a branched cover, and the surface $S/\langle\iota\rangle$ is a sphere with a finite number of holes. The involution $\iota$ is called the \emph{hyperelliptic involution}.
\end{defn}

\begin{para}Fixed points of the hyperelliptic involution are called \emph{Weierstrass points}. A closed hyperelliptic surface of genus $g$ always has exactly $2g+2$ Weierstrass points, by consideration of Euler characteristic. \end{para}

\begin{para}
	Compact oriented surfaces are classified up to diffeomorphism by their genus and their number of boundary components. The same is not true for hyperelliptic surfaces; a hyperelliptic surface can have two distinct types of boundary component. A boundary component of a hyperelliptic surface is called a \emph{Weierstrass boundary} if the hyperelliptic involution maps this boundary component to itself. The non-Weierstrass boundary components are interchanged pairwise by the hyperelliptic involution, and we refer to each such pair as a \emph{conjugate boundary pair}. 
\end{para}

\begin{para}\label{hyp definition}Let $S_{g,n,m}$ be a hyperelliptic surface of genus $g$ with $n$ Weierstrass boundary components and $m$ conjugate  boundary pair components. (So $S_{g,n,m}$ has $n+2m$ holes.) We denote by $\mathrm{Hyp}_\partial(S_{g,n,m})$ the subgroup of $\Diff_\partial(S_{g,n,m})$ consisting of diffeomorphisms which commute with the hyperelliptic involution, and define
	$$ H_g^{n,m} = B\mathrm{Hyp}_\partial(S_{g,n,m}).$$
	We call $H_g^{n,m}$ the moduli of hyperelliptic surfaces.  
\end{para}

\begin{para}\label{coordinate free hyperelliptic}
	As in \S\S\ref{desiderata}--\ref{slightly better definition} one may want to give a coordinate-free variant of this definition. We consider instead the topologically enriched groupoid whose objects are branched double covers $S \to S_0$ where $S_0$ is a genus $0$ surface with $n+m$ boundaries, and the surface $S$ has genus $g$, and comes with $n$ ordered Weierstrass boundaries and $m$ ordered conjugate  boundary pairs. All boundaries have a germ of a parametrized collar neighborhood with respect to which the branched cover is in a standard form. By this we mean that each Weierstrass boundary is parametrized in such a way that on a collar neighborhood $S^1 \times [0,\varepsilon)$ the hyperelliptic involution restricts to $(z,t) \mapsto (-z,t)$, and similarly for the conjugate boundary pairs, i.e.\ the restriction of the hyperelliptic involution to the collars is given by permutation of the components, and the identity on a standard collar. Morphisms in this category are commuting squares\[
	\begin{tikzcd}
		S \arrow[d]\arrow[r,"\cong"]&\arrow[d] T \\
		S_0 \arrow[r,"\cong"]& T_0
\end{tikzcd}\]
where the horizontal maps are diffeomorphisms preserving the germs of collar neighborhoods, topologized by the compact-open topology. If $n=m=0$ we need to impose the condition that the horizontal maps are orientation-preserving. The classifying space of this category is $H_g^{n,m}$.
\end{para}

\begin{rem}
	Just as the space $M_g^n$ (see \cref{ag interpretation}), the space $H_g^{n,m}$ admits an algebro-geometric interpretation. Indeed, it is homotopy equivalent to the moduli stack parametrizing smooth compact genus $g$ hyperelliptic curves, equipped with $n$ distinct ordered Weierstrass points, $m$ distinct ordered non-Weierstrass points, none of which are conjugate under the hyperelliptic involution, and at each of the $(n+m)$ marked points a nonzero tangent vector. 
\end{rem}

\begin{rem}
	As in the case of the spaces $M_g^n$, we could have replaced diffeomorphisms with homeomorphisms or mapping classes in these definitions, and it would not have made a difference. 
	One subtlety when working with mapping classes rather than genuine diffeomorphisms or homeomorphisms is the following: for a topological group $G$ and $g \in G$, there is a natural map
	$$ \pi_0 \, C_G(g) \longrightarrow C_{\pi_0(G)}(g)$$
	which is not a bijection in general. Hence it is not a priori clear whether $\pi_0\,\mathrm{Hyp}_\partial(S_{g,n,m})$ can be recovered as the centralizer of the hyperelliptic involution inside the mapping class group; however, this turns out to be the case, and is a result of Birman and Hilden \cite{birmanhilden}. 
	
\end{rem}

\begin{para}There is an evident map $H_g^{n,m} \to M_g^{n+2m}$. In the description of \S\ref{coordinate free hyperelliptic}, it is induced by the functor which takes $(S\to S_0)$ to the surface $S$, and forgets the data of $S_0$. Algebro-geometrically it corresponds to the inclusion of the moduli space of hyperelliptic curves into the moduli space of all curves. There is another almost equally evident map $H_g^{n,m} \to M_{0,(2g+2-n)}^{n+m}$, which instead takes  $(S\to S_0)$ to the surface $S_0$, together with the data of the set of branch points. Since a closed hyperelliptic surfaces has $2g+2$ branch points, a hyperelliptic surface with $n$ Weierstrass boundary components must have $2g+2-n$ branch points. We denote here by $M_{g,(k)}^n$ the moduli space of surfaces of genus $g$ with $n$ parametrized boundary components and $k$ unordered punctures. 
	\end{para}

\begin{para}\label{braid1}
	Let us be more explicit about the map $H_g^{n,m} \to M_{0,(2g+2-n)}^{n+m}$. When $n=m=0$, there is a short exact sequence 
	$$ 1 \to \{\pm 1\} \to \mathrm{Hyp}^+(S_g) \to \Diff^+(S^2 \setminus \{z_1,\ldots,z_{2g+2}\}) \to 1.$$
	Here the ``$+$'' superscripts indicate orientation-preserving diffeomorphisms, and we identify $\Diff^+(S^2 \setminus \{z_1,\ldots,z_{2g+2}\})$ with the subgroup of $\Diff^+(S^2)$ of diffeomorphisms fixing $\{z_1,\ldots,z_{2g+2}\}$ setwise. On $\pi_0$ it exhibits the hyperelliptic mapping class group as a $\Z/2$ central extension of the mapping class group of a $(2g+2)$-times punctured sphere. The short exact sequence expresses that a diffeomorphism of the base admits exactly two lifts to the double cover, interchanged by the hyperelliptic involution. When $(n,m)$ is $(1,0)$ or $(0,1)$, every diffeomorphism of the base (which is now a sphere with a hole, i.e.\ a disk) admits a \emph{unique} lift to the double cover which preserves the parametrization of the boundary, and we obtain isomorphisms
	$$ \mathrm{Hyp}_\partial(S_{g,1,0}) \cong \Diff_\partial(\mathbb D \setminus \{z_1,\ldots,z_{2g+1}\} )$$
	and
	$$ \mathrm{Hyp}_\partial(S_{g,0,1}) \cong \Diff_\partial(\mathbb D \setminus \{z_1,\ldots,z_{2g+2}\} ).$$
	The connected components of these topological groups are contractible, and the mapping class group $\pi_0 \,\Diff_\partial(\mathbb D \setminus \{z_1,\ldots,z_{k}\})$ is nothing but the Artin braid group on $k$ strands, $\beta_k$. Thus the hyperelliptic mapping class groups for surfaces with one boundary component, or one conjugate boundary pair, are identified with the usual braid groups. 
\end{para}

\begin{para}
    We denote similarly $H_{g,n,m}$ the corresponding moduli spaces of hyperelliptic surfaces with punctures, rather than boundary components. The projection $H_g^{n,m}\to H_{g,n,m}$ is an $(n+m)$-fold oriented circle bundle. If $n=m=0$ we simply write $H_g$.  The following result is standard.
\end{para}

\begin{prop}For all $g$, the spaces $H_{g,1,0}$, $H_{g,0,1}$ and $H_g$ have the rational cohomology of a point.     
\end{prop}

\begin{proof}
    Arnold \cite{arnoldbraid2} proved that the braid group has the rational cohomology of a circle, with a generator in degree $1$ induced by the rotation action. Since $H_{g,1,0}$ and $H_{g,0,1}$ are quotients of classifying spaces of braid groups $\beta_{2g+1}$ and $\beta_{2g+2}$ by rotation, they both have the rational cohomology of a point. 

    Now consider the space $H_g[2]$ parametrizing closed hyperelliptic surfaces with a total order on their set of Weierstrass points. There are homotopy equivalences $H_g \cong H_g[2]/\Sigma_{2g+2}$ and $H_{g,1,0} \cong H_g[2]/\Sigma_{2g+1}$. In particular, the transfer map identifies $H^\bullet(H_g,\Q)$ and $H^\bullet(H_{g,1,0},\Q)$ with the $\Sigma_{2g+2}$-invariants, resp.\ the $\Sigma_{2g+1}$-invariants, inside $H^\bullet(H_g[2],\Q)$. In particular, $H^\bullet(H_g,\Q)$ injects into $ H^\bullet(H_{g,1,0},\Q)$. 
\end{proof}

\begin{prop}\label{gysin}
    Let $F$ be a $\Q$-local system on $H_{g,1,0}$. Then $H^\bullet(H_g^{1,0},F) \cong H^\bullet(H_{g,1,0},F) \oplus H^{\bullet-1}(H_{g,1,0},F)$. Similarly, $H^\bullet(H_g^{0,1},F) \cong H^\bullet(H_{g,0,1},F) \oplus H^{\bullet-1}(H_{g,0,1},F)$ for a $\Q$-local system on $H_{g,0,1}$.
\end{prop}

\begin{proof}
    This follows from the Gysin sequence, since the preceding proposition implies that the circle bundles $H_g^{1,0}\to H_{g,1,0}$ and $H_g^{0,1}\to H_{g,0,1}$ have vanishing Euler class rationally. 
\end{proof}

	\subsection{Algebraic and operadic structure}
	
	\begin{para}\label{braid2}
		Denote by $\mathrm{Conf}_k(\mathbb D)$ the configuration space of $k$ distinct unordered points in the interior of the disk $\mathbb D$. From \S\ref{braid1} there are homotopy equivalences $H_g^{1,0} \simeq \mathrm{Conf}_{2g+1}(\mathbb D)$ and $H_g^{0,1} \simeq \mathrm{Conf}_{2g+2}(\mathbb D)$. (Indeed, $B\Diff_\partial(\mathbb D \setminus \{z_1,\ldots,z_k\}) \simeq \mathrm{Conf}_k(\mathbb D)$, since the contractible topological group $\Diff_\partial(\mathbb D)$ acts transitively on $\mathrm{Conf}_k(\mathbb D)$ with stabilizer $\Diff_\partial(\mathbb D \setminus \{z_1,\ldots,z_{k}\})$.) In particular, the standard framed $E_2$-structure on configurations of points in the disk defines for us a structure of framed $E_2$-algebra on 
		$$ \coprod_{g \geq 0} H_g^{1,0} \sqcup H_g^{0,1}.$$
		Stabilization is given by multiplication in this $E_2$-algebra  by a fixed element of $H_0^{1,0}$, and produces a sequence of maps
		$$\ldots \to H_g^{1,0} \to H_g^{0,1}\to H_{g+1}^{1,0} \to H_{g+1}^{0,1} \to \ldots$$ 
		which in terms of moduli spaces of surfaces is given by gluing on a pair of pants, alternating between gluing along the waist or the two cuffs. Given that the hyperelliptic mapping class groups for surfaces with boundary are nothing but the braid groups, we may now identify the following two classical theorems of Arnold \cite{arnoldbraid2} and Segal \cite{segalconfiguration} as the ``hyperelliptic analogues'' of Harer's stability theorem (\ref{harer}), and the Madsen--Weiss theorem (\cref{calculate rat hom}), respectively. 
	\end{para}
	
	\begin{thm}[Arnold]
		The stabilization maps $H_i(\mathrm{Conf}_n(\mathbb D),\Z) \to H_{i}(\mathrm{Conf}_{n+1}(\mathbb D),\Z)$ are isomorphisms for $i\leq\tfrac{n-2} 2$.
	\end{thm}
	
	\begin{thm}[Segal]
		The group-completion of $\, \coprod_{n \geq 0} \mathrm{Conf}_n(\mathbb D)$ is weakly equivalent to $\Omega^2 S^2$. 
	\end{thm}
	
	\begin{para}
	It is natural to ask whether the framed $E_2$-structures on $ \coprod_{g \geq 0} H_g^{1,0} \sqcup H_g^{0,1}$ and on $\coprod_{g \geq 0} M_g^1$ are somehow compatible with one another. The short answer is that they are not. A correct statement is that the maps $H_g^{1,0} \to M_g^1$ and $H_g^{0,1} \to M_g^2$ assemble to a homomorphism of \emph{$E_1$-algebras}
	$$ \coprod_{g \geq 0} H_g^{1,0} \sqcup H_g^{0,1} \longrightarrow \coprod_{g \geq 0} M_g^1 \sqcup M_g^2.$$
	This is not, however, a morphism of $E_2$-algebras: in fact, the right-hand side is not even an $E_2$-algebra in any natural way. Let us describe the $E_1$-algebra structure on the right-hand side. Consider the following monoidal category $\mathscr C$: objects are oriented compact surfaces $S$ equipped with an oriented embedding $I \sqcup I \hookrightarrow \partial S$, such that each boundary component of $S$ contains at least one interval. Morphisms are isotopy classes of diffeomorphisms restricting to the identity on a neighborhood of the embedded intervals. Write the unit interval $I = I_L \cup I_R$ as the union of its left and right half. The monoidal structure is given by boundary connect sum: $S\, \natural\, S'$ is the surface obtained by sewing $(I_R \sqcup I_R) \subset S$ to $(I_L \sqcup I_L) \subset S'$. The classifying space of $\mathscr C$ is homotopy equivalent to $ \coprod_{g \geq 0} M_g^1 \sqcup M_g^2$, and the monoidal structure on $\mathscr C$ defines the $E_1$-algebra structure on this disjoint union. The entirely analogous category with only one embedded boundary interval was considered in \cite[Section 5.6]{randalwilliamswahl}, where it was shown to be braided monoidal in a natural way. But there is no natural braiding on the category $\mathscr C$. If one restricts to hyperelliptic surfaces and imposes the condition that the two boundary intervals are interchanged under the hyperelliptic involution, then a natural braiding does exist, which explains the $E_2$-structure on $\coprod_{g \geq 0} H_g^{1,0} \sqcup H_g^{0,1}$. But the braiding uses nontrivially the hyperelliptic involution. The category $\mathscr C$ and its monoidal structure  has been studied in very recent work of Harr--Vistrup--Wahl \cite{harrvistrupwahl} for completely different reasons; namely, it turns out that applying the general homological stability machinery of \cite{randalwilliamswahl,krannichtopologicalmoduli} to this particular category reproduces Harer stability with optimal slope.
	\end{para}
	
	\begin{para}In the situation of the previous paragraph, one may restrict attention to the subcategory of $\mathscr C$ in which the two intervals lie on different boundary components. On this category one can construct a braiding as in \cite{randalwilliamswahl}, which defines an $E_2$-structure. In fact, if one restricts attention only to surfaces with two components, and braids with an even number of strands, then one obtains a map
		$$ \coprod_{g \geq 0} H_g^{0,1} \longrightarrow \coprod_{g \geq 0} M_g^2$$
		which is naturally a homomorphism of framed $E_2$-algebras. A very geometric way of describing this homomorphism, and the $E_2$-structure on the right-hand side, was given by Segal and Tillmann \cite{segaltillmann} --- see Figure 3 of their paper. If Miller's $E_2$-structure on $\coprod_{g \geq 0} M_g^1$ is called ``pair of pants''-multiplication, then this construction should logically be called ``pair of pair of pants''-multiplication, but that would be silly. 
	\end{para}

\begin{para}\label{hyp operad}
	As noted in \cref{subsection-moduli of surfaces}, the spaces $\{M_g^n\}$ form a modular operad. For the moduli of hyperelliptic surfaces there are two types of natural gluing maps:
	$$ H_{g_1}^{n_1+1,m_1} \times H_{g_2}^{n_2+1,m_2} \longrightarrow H_{g_1+g_2}^{n_1+n_2,m_1+m_2}$$
	and 
	$$ H_{g_1}^{n_1,m_1+1} \times H_{g_2}^{n_2,m_2+1} \longrightarrow H_{g_1+g_2+1}^{n_1+n_2,m_1+m_2}.$$
	Self-gluings are not allowed, however, as they would raise the genus of the curve being covered. Abstractly one would say that the spaces $\{H_g^{n,m}\}$ form a two-colored $\N$-graded cyclic operad for which one of the gluing operations is homogeneous of degree $0$, and the other is homogeneous of degree $1$. 	The framed $E_2$-structure on $\coprod_{g \geq 0} H_g^{0,1}$ can be understood in these operadic terms, too. We may consider a hyperelliptic surface of genus $-1$ to be a surface of the form $S^2 \sqcup S^2$, with the involution switching the two components. This surface is hyperelliptic according to \cref{defn hyperelliptic}, and the fact that its genus is $-1$ can be seen from the fact that the hyperelliptic involution has $2g+2$ fixed points. Then the collection $\{H_{-1}^{0,m}\}$ is a cyclic suboperad of $\{H_g^{n,m}\}$, and it acts naturally on $\coprod_{g \geq 0} H_g^{0,1}$. Now the operad $\{H_{-1}^{0,m}\}$ parametrizes pairs of genus zero surfaces with $m$ holes, with an involution switching the two surfaces, and for each conjugate pair of boundary components one of the two boundaries is chosen as distinguished. It contains the framed little disk operad as the suboperad parametrizing surfaces with all distinguished boundaries on the same component, which is a roundabout way of describing the same action as described in the previous paragraph.
\end{para}

\subsection{Twisted coefficients}

\begin{para}
	Recall from \S\ref{standard coefficient system} the standard local system $V$ on the space $M_g^n$, of rank $2g-1+n$ if $n>0$. We denote by the same name the pullback of this local system to the spaces $H_g^{n,m}$ along the natural map $H_g^{n,m} \to M_g^{n+2m}$. 
\end{para}

\begin{para}
	In particular, the local system $V$ from the previous paragraph defines a sequence of representations of all braid groups, via the identifications $H_g^{1,0} \simeq B\beta_{2g+1}$ and $H_g^{0,1} \simeq B\beta_{2g+2}$ discussed in \S\S\ref{braid1}--\ref{braid2}. For each $n$ we get a representation of $\beta_n$ of rank $n-1$. This representation is classically called the \emph{integral reduced Burau representation}. 
\end{para}

\begin{para}
	As mentioned in \S\S\ref{twisted1}--\ref{twisted2}, Randal-Williams--Wahl have set up a general framework for proving homological stability for sequences of groups, including with certain polynomial coefficient systems. One example of a sequence of groups discussed in their paper is the sequence of all braid groups, and an example of a polynomial coefficient system in case of the braid groups, is precisely the reduced Burau representation. The reduced Burau representation $V$ is split polynomial of degree $1$, and for any Schur functor $S^\lambda$, the sequence of representations $S^\lambda V$ is a split polynomial coefficient system of degree $\vert \lambda \vert$. See \cite[Examples 4.3 and 4.15, and Theorem D]{randalwilliamswahl}. From their general homological stability result we see in particular the following:
\end{para}

\begin{thm}[Randal-Williams--Wahl] The maps $H_i(\beta_n,S^\lambda V) \to H_i(\beta_{n+1},S^\lambda V)$ are isomorphisms in the range $i< \tfrac{n-2-\vert\lambda\vert}{2}$. 
\end{thm}

\begin{rem}As in \cref{uniform remark} one can also define polynomial coefficient systems $V_\lambda$, which satisfy homological stability with a uniform bound by \cite{MPPRW}. We refer to \cite[Section 4]{MPPRW} for how to define these coefficient systems.
\end{rem}

\begin{prop}\label{odd vanishing}
    If $n$ is odd and $\vert\lambda\vert$ is odd, then $H_\bullet(\beta_n,S^\lambda V)$ vanishes (rationally).
\end{prop}

\begin{proof}By \cref{gysin} it suffices to show that $H_\bullet(H_{g,1,0},S^\lambda V)=0$ for all $g$. This is the ``center kills'' argument: the hyperelliptic involution is a central element of the hyperelliptic mapping class group, and it acts on $S^\lambda V$ as multiplication by $(-1)^{\vert\lambda\vert}$. By functoriality of group homology it acts as multiplication by $(-1)^{\vert\lambda\vert}$ on group homology, but it also necessarily acts trivially, so the homology vanishes. \end{proof}

\begin{rem}The same argument shows that $H^\bullet(H_g,S^\lambda(V))=0$ for $\vert\lambda\vert$ odd. However, the homology of $S^\lambda(V)$ does not necessarily vanish on $H_{g,0,1}$, except of course in the stable range: the point is that a hyperelliptic surface marked at a Weierstrass point always has an automorphism given by the hyperelliptic involution, but a surface marked at a non-Weierstrass point does not.
\end{rem}

\subsection{Hyperelliptic surfaces in a background space}\label{hyp target space}

\begin{para}We will now define spaces $H_g^{n,m}(X)$ by analogy with the spaces $M_g^n(X)$. In the hyperelliptic setting, it seems most natural to suppose that $X$ is a space with an action of $\Sigma_2$ (the symmetric group on two elements), and we are parametrizing hyperelliptic curves with a $\Sigma_2$-equivariant map to $X$. We remark that considering equivariant maps gives a strictly more general set-up: for any target space $Y$, arbitrary continuous maps to $Y$ are in bijection with $\Sigma_2$-equivariant maps to $Y \times Y$ (where the $\Sigma_2$-action on $Y \times Y$ is by swapping the factors).
\end{para}

\begin{para}Let $X$ be a based space with a $\Sigma_2$-action. We define $H_g^{n,m}(X)$, the space of hyperelliptic genus $g$ surfaces equipped with an equivariant map to $X$, as the space
	$$ \mathrm{map}_\ast^{\Sigma_2}(S_{g,n,m}/\partial S_{g,n,m},X) /\!/ \mathrm{Hyp}_\partial(S_{g,n,m}),$$
	using the notation of \S\ref{hyp definition}, and where $\mathrm{map}_\ast^{\Sigma_2}(-)$ denotes the $\Sigma_2$-equivariant based mapping space. As in \S\ref{background space definition} it is not hard to formulate a definition which does not make use of the choice of a reference surface.
\end{para}

\begin{rem}Our spaces of hyperelliptic surfaces with an equivariant map to a target space are similar to, but distinct from, the \emph{configuration-mapping spaces} of \cite{evw2}. The configuration-mapping spaces parametrize the data of a configuration $z$ of points (say in a disk $D$), and a continuous map $D \setminus z \to Y$ for some fixed target space $Y$. If $Y$ is a topological quotient stack $[X/C_2]$, then this is the same data as an \'etale double cover of $D \setminus z$ equipped with an equivariant map to $X$. The spaces we consider here parametrize instead the datum of a \emph{branched} double cover of $D$ with branching along $z$, equipped with an equivariant map to $X$. Since an \'etale double cover of $D \setminus z$ extends uniquely to a branched double cover of $D$, the only significant difference is that we in addition specify a continuous extension of the map to $X$ across the ramification points. 
\end{rem}

\begin{para}
	The space $\coprod_{g \geq 0} H_g^{1,0}(X) \sqcup H_g^{0,1}(X)$ is an $E_2$-algebra, although the $E_2$-structure is somewhat subtle. There is a less subtle $E_2$-subalgebra $\coprod_{g \geq 0} H_g^{0,1}(X)$ with a map of $E_2$-algebras to $\coprod_{g \geq 0} M_g^2(X)$. 
\end{para}

\begin{para}
	As in \ref{fiber sequence} there is a fibration sequence
	$$ \mathrm{map}_\ast^{\Sigma_2}(S_{g,n,m}/\partial S_{g,n,m},X) \to H_g^{n,m}(X) \to H_g^{n,m}.$$ 
	As in \cref{target} it turns out that for simply connected $X$, the $d$th homology of the fiber defines a coefficient system of degree $\leq d$, so that these spaces satisfy homological stability as $g \to \infty$. Since we will not need this statement in full generality, we omit the proof; the result will be clear in the case of concern to us, where $X$ is an Eilenberg--MacLane space.  
\end{para}

\begin{thm}\label{thm1}
	Let $A$ be a $1$-connected chain complex. Consider $KA$ as a $\Sigma_2$-space via the action of $\Sigma_2$ on $A$ by multiplication by $-1$. The Serre spectral sequence for the fibration $H_g^{1,0}(KA) \to H_g^{1,0}$ degenerates for any $g$, and there is an isomorphism \[ H_\bullet(H_g^{1,0}(KA),\Q) \cong \bigoplus_\lambda S^\lambda(A) \otimes H_\bullet(H_g^{1,0},S^\lambda(V[-1])).\]
\end{thm}

\begin{proof}
	The fiber of $H_g^{1,0}(KA) \to H_g^{1,0}$ is the space $\mathrm{map}_\ast^{\Sigma_2}(S_{g,1,0}/\partial S_{g,1,0},KA)$. Now we have
	$$ \mathrm{map}_\ast(S_{g,1,0}/\partial S_{g,1,0},KA) \simeq K(A \otimes V[-1]) \times K(A[-2]),$$
	just as in the proof of \cref{first expression}, where $V=H^1(S_{g,1,0}/\partial S_{g,1,0})$. Now we take $\Sigma_2$-fixed points: since $\Sigma_2$ acts as multiplication by $-1$ on $A$ and $V$, the action by $\Sigma_2$ fixes the first factor in the product and has the origin as its only fixed point in the second factor. It follows that the homology of the fibers is $\mathrm{Sym}(A \otimes V[-1])$. Thus we are precisely in the situation of \cref{first expression}, again there can be no nontrivial $\mathrm{GL}(A)$-equivariant differential, and the result follows. 
\end{proof}

\begin{rem}
	The same argument would work for $H_g^{0,1}$, to show similarly that $$H_\bullet(H_g^{0,1}(KA),\Q) \cong \bigoplus_\lambda S^\lambda(A) \otimes H_\bullet(H_g^{0,1},S^\lambda(V[-1])),$$ where $V$ now denotes the rank $2g+1$ local system given by $H^1(S_{g,0,1}/\partial S_{g,0,1})$. For larger values of $n$ and $m$ the calculation changes since it is no longer true that the hyperelliptic involution acts as multiplication by $-1$ on $H^1(S_{g,n,m}/\partial S_{g,n,m})$.
\end{rem}

\begin{para}
	To calculate the stable homology of $H_g^{1,0}$ (equivalently, the stable homology of $H_g^{0,1}$) with coefficients in $S^\lambda(V)$ it is therefore enough to compute the homology of $H_g^{1,0}(KA)$ as an analytic functor of $A$ in a stable range. We will carry this out in the next two sections of the paper. Let us already now state what the outcome will be:
\end{para}

\begin{thm}[Main theorem of Section \ref{scanning section}]\label{thm2}The group-completion of $\coprod_{g \geq 0} H_g^{0,1}(KA)$ is a certain explicit double loop space $\Omega^2 W(A)$.
\end{thm}

\begin{thm}[Main theorem of Section \ref{koszul section}]\label{thm3}
The analytic functor $A \mapsto H_\bullet(\Omega_0^2 W(A),\Q)$ has the Taylor series
	$$ \Exp( z^{-1} \Log(z + \sum_{r \geq 0} z^{-r}h_{2r})-1).$$
\end{thm}

\begin{para}
	We consequently obtain the following result, from which  \cref{thmA} from the introduction follows.
\end{para}

\begin{thm}\label{plethystic formula for poincare series}There is an identity 
	$$\sum_k \sum_\lambda \dim H_k(H_\infty^{1,0},V_\lambda) (-z)^k s_{\lambda'}  = \Exp( z^{-1} \Log(z + \sum_{r \geq 0} z^{r}h_{2r})-1-h_2)$$
\end{thm}

\begin{proof}
	From Theorem \ref{thm1} we should compute the stable homology of $H_g^{1,0}(KA)$, equivalently of $H_g^{0,1}(KA)$. Since we know homological stability, the group-completion theorem \cite{mcduffsegal} and Theorem \ref{thm2} identifies this with the homology of $\Omega^2_0 W(A)$. Combining Theorems \ref{thm1} and \ref{thm3} then gives the identity 
	$$\sum_k \sum_\lambda \dim H_k(H_\infty^{1,0},S^\lambda(V[-1]))  (-z)^k s_{\lambda}= \Exp( z^{-1} \Log(z + \sum_{r \geq 0} z^{-r}h_{2r})-1).$$
	Now $S^\lambda(V[-1]) = S^{\lambda'} (V) [-\vert \lambda \vert]$. Shifting the arity $r$ component in degrees by $r$ (which amounts to multiplying $h_{2r}$ by $z^{2r}$) yields 
	$$\sum_k \sum_\lambda \dim H_k(H_\infty^{1,0},S^\lambda(V))  (-z)^k s_{\lambda'}= \Exp( z^{-1} \Log(z + \sum_{r \geq 0} z^{r}h_{2r})-1). $$
The identity $\Exp(h_2(y))\sum_\lambda s_{\langle\lambda\rangle}(x)s_{\lambda'}(y) = \sum_\lambda s_\lambda(x)s_{\lambda'}(y)$ shows that substituting $S^\lambda(V)$ with $V_\lambda$ amounts to multiplying the right-hand side with $\Exp(-h_2)$. 
\end{proof}

\section{The scanning argument}\label{scanning section}

\begin{para}
	The goal of this section is to derive an expression for the group-completion of the $E_2$-algebra given by $\coprod_{g \geq 0} H_g^{0,1}(KA)$, when $A$ is a positively graded chain complex. Let us begin by stating the result in the specific situation relevant to this paper. \end{para}

\begin{construction}\label{def of W}
	Let $A$ be a chain complex. We denote by $\iota$ the involution of $KA$ induced by the automorphism $x \mapsto -x$ of $A$. The inclusion of the origin in $KA$ defines a map of homotopy quotients $K(\Z/2,1) \to KA/\!/\langle \iota \rangle$. Consider also the map $S^2 \to K(\Z/2,2)$ classified by the nontrivial element of $H^2(S^2,\Z/2)$, and let $W$ denote its homotopy fiber. The space $W$ fits in a fiber sequence
	$$ K(\Z/2,1) \to W \to S^2. $$ 
	We denote by $W(A)$ the homotopy pushout
	$$ W(A) := KA/\!/\langle \iota \rangle\!\!\coprod^h_{K(\Z/2,1)} \!\!W.$$
\end{construction}

\begin{thm}\label{scanning theorem}
	Let $A$ be a chain complex concentrated in degree $\geq 1$ with trivial $2$-torsion. The $2$-fold delooping of the $E_2$-algebra $\coprod_{g \geq 0} H_g^{0,1}(KA)$ is weakly equivalent to $W(A)$. In particular, the group-completion of $\coprod_{g \geq 0} H_g^{0,1}(KA)$ is $\Omega^2 W(A)$.
\end{thm}

\begin{para}When $A=0$ there is a weak equivalence $H_g^{0,1}(KA)\simeq \mathrm{Conf}_{2g+2}(\mathbb D)$, and the theorem asserts that the $2$-fold delooping of the classifying space of the collection of  braid groups on an \emph{even} number of strands is the space $W$. As a sanity check, this is compatible with Segal's result that the $2$-fold delooping of $\coprod_{n \geq 0} \mathrm{Conf}_n(\mathbb D)$ is weakly equivalent to $S^2$. Indeed, the ``short exact sequence'' of topological monoids
	$$ \coprod_{\substack{n \geq 0 \\ n \text{ even}}}  \mathrm{Conf}_n(\mathbb D) \to \coprod_{n \geq 0} \mathrm{Conf}_n(\mathbb D)  \to \Z/2 $$
	deloops twice to a fiber sequence 
	$$ W \to S^2 \to K(\Z/2,2),$$
	which matches our definition of $W$. As another sanity check, we have $\Omega^2_0 S^2 \simeq \Omega^2_0 W$, so for the stable homology it makes no difference if we take all braids or only those with an even number of strands. 
\end{para}	
	
	\begin{para}\cref{scanning theorem} will be deduced from a more general result which applies not only to hyperelliptic curves but to branched $G$-covers of a disk, for an arbitrary finite group $G$, with local monodromies constrained to lie in some conjugation-invariant subset $c \subseteq G$; moreover, the target space will be allowed to be any connected based space with a $G$-action, not necessarily an Eilenberg--MacLane space. 	In order to make the arguments easier to follow, we begin by presenting a form of the scanning argument for ordinary configuration spaces of points, i.e.\ a proof that the $2$-fold delooping of $\coprod_{n \geq 0} \mathrm{Conf}_n(\mathbb D)$ is $S^2$. After this we explain how to modify the argument to deal instead with {Hurwitz spaces}. Our arguments draw significantly from \cite{mcduff} and \cite{hatcher-madsenweiss}. The proof of \cref{scanning theorem} will take up the rest of this section of the paper. We are particularly grateful to Andrea Bianchi for a careful and critical reading of this section, and for pointing out some oversights in a preliminary version of it. 
\end{para}

\subsection{The functor of configuration spaces}

\begin{para}Let $X$ be a locally compact Hausdorff space, $A \subset X$ a subspace. Denote by $\Conf_{X;A}$ the set of all finite subsets of $X$ which are disjoint from $A$. We think of these as configurations of points in $X$ avoiding $A$, and we want to topologize $\Conf_{X;A}$ so that points are allowed to vanish at infinity along the directions $X$ is not compact. Let us give two precise definitions of this topology (and leave it to the reader to verify that the definitions are in fact equivalent). We remark that the essential case is defining the topology when $A=\varnothing$, as one may note that $\Conf_{X;A} \subseteq \Conf_{X;\varnothing}$ has the subspace topology.
\end{para}

\begin{defn}
	A basis for the topology on $\Conf_{X;A}$ is indexed by pairs $(K,\mathfrak U)$, where $K \subseteq X$ is compact and $\mathfrak U$ is a finite family of disjoint open subsets of $K$. The corresponding open set is
	$$ B(K,\mathfrak U) := \{z \in \Conf_{X;A} : \#(z \cap U)=1 \text{ for all } U \in \mathfrak U,  \,\text{ and } \, z \cap (K \setminus \bigcup_{U \in \mathfrak U} U) = \varnothing\}.$$
\end{defn}

\begin{defn}\label{dirac measure}We identify a point of $\Conf_{X;A}$ with a sum of Dirac measures on $X$ in the evident way. The topology on $\Conf_{X;A}$ is then the topology inherited from the weak-$\ast$ topology on the space of Radon measures on $X$. Explicitly, if $f : X \to \R$ is a continuous function with compact support, define a function $\Conf_{X;A} \to \R$ by 
	$$ z \mapsto \sum_\alpha f(z_\alpha).$$
	We give $\Conf_{X;A}$ the coarsest topology for which all these functions are continuous.
\end{defn}

\begin{examplex}
	If $X$ is compact, then $\Conf_{X;A} \cong \coprod_{k \geq 0} \mathrm{Conf}_k(X \setminus A)$. 
\end{examplex}

\begin{rem}A closely related construction was introduced by McDuff \cite{mcduff}. She considered a compact manifold-with-boundary $M$, and defined a space of configurations of particles on $M$ which can be annihilated or created along $\partial M$. This space is similar to, but not quite the same as, $\Conf_{M \setminus \partial M;\varnothing}$. Let us consider a compact Hausdorff space $X$ and a closed subspace $L$. The set of finite subsets of $X \setminus L$ can be given three different natural topologies, all of which capture the informal idea that points can annihilate along $L$:
	\begin{enumerate}[(i)]
		\item The quotient topology from $\coprod_{n \geq 0} \mathrm{Conf}_n(X)$, where two configurations are identified if their intersection with $X \setminus L$ are equal.
		\item The subspace topology from the infinite symmetric product $\mathrm{SP}^\infty(X/L)$.
		\item The topology of $\Conf_{X \setminus L;\varnothing}$. 
	\end{enumerate}
In general, all three topologies are different: (i) is finer than (ii), which is finer than (iii). Consider for example $X=[0,1]$, $L=\{0\}$. In this case, the sequence of configurations $z_n = \{\tfrac{1}n, \tfrac{1}{n+1}\}$ converges to the empty configuration in the topologies (ii) and (iii), but it does not converge\footnote{This should probably be considered a defect of the topology (i). To see that the sequence does not converge, consider the function $\phi : \coprod_{n \geq 0} \mathrm{Conf}_n([0,1]) \to \R$ given by $\phi(z) = \sum_{s,t \in z} \frac{st}{|s-t|}$, the sum taken over $s \neq t$. Then $\phi$ descends to a continuous function on the quotient space, $\phi(z_n)=1$ for all $n$, but $\phi(\varnothing)=0$.} in the topology (i). The sequence of configurations $z_n = \{\tfrac 1 {n^2}, \tfrac 2 {n^2},\ldots, \tfrac 1 n\}$ converges to the empty configuration in the topology (iii), but does not converge in the topologies (i) and (ii). The topology inherited from the infinite symmetric product can be given a measure-theoretic interpretation as in \cref{dirac measure}: we obtain the topology (ii) if in \cref{dirac measure} we replace compactly supported test functions with continuous functions vanishing at infinity. 
\end{rem}

\begin{para}For later purposes it will be useful to describe $\Conf_{X;A}$ as a (set-valued) topological stack, i.e.\ to characterize it in terms of the functor it represents. As in Noohi \cite{noohi-mapping} we will consider topological stacks as sheaves on the site of \emph{compactly generated} topological spaces. Recall that a topological space $X$ is compactly generated if a subset $V \subset X$ is open whenever for all continuous maps $h:K \to X$ from a compact Hausdorff space, $h^{-1}(V)$ is open in $K$. We write $k\mathsf{Top}$ for the category of compactly generated spaces. \end{para}

\begin{para}Note that for any map of sets $f:S \to \Conf_{X;A}$, there is an ``incidence correspondence'' $\Gamma_f \subset S \times X$, defined by
	$$ \Gamma_f = \{(s,x) : f(s) \ni x\}.$$
	The notation reflects that we think of $\Gamma_f$ as the graph of a multi-valued function. 
\end{para}
\begin{prop}
	Let $S$ be a compactly generated topological space. A map of sets $f:S \to \Conf_{X;A}$ is continuous if and only if the incidence correspondence $\Gamma_f$ is closed in $S \times X$, and the projection to the first coordinate $\Gamma_f \to S$ is a local homeomorphism. That is, $\Conf_{X;A}$ represents the functor $k\mathsf{Top}^{\mathrm{op}} \to \mathsf{Set}$ which assigns to $S$ the set of closed subsets $\Gamma \subset S \times X$ such that $\Gamma \cap (S \times A)=\varnothing$, and $\Gamma \to S$ is a local homeomorphism with finite fibers. 
\end{prop}

\begin{proof}It suffices to prove this when $A$ is empty. We set $A=\varnothing$, and omit $A$ from the notation.
	
	We first show that the universal incidence correspondence $\Gamma \subset \Conf_X\times X$ is closed, and that $\Gamma \to \Conf_X$ is a local homeomorphism. Let $(z,x) \notin \Gamma$. Choose $x \in V \subseteq K$ with $V$ open, $K$ compact, and $K \cap z = \varnothing$. Then $B(K,\varnothing) \times V$ is an open neighborhood of $(z,x)$ disjoint from $\Gamma$. One also checks that for any $U \subseteq K \subseteq X$ with $U$ open and $K$ compact, $\Gamma \cap (B(K,\{U\}) \times U) \to B(K,\{U\})$ is a homeomorphism (and these open sets cover $\Gamma$). Consequently, for any topological space $S$ with a continuous map $S \to \Conf_X$, we get an incidence correspondence satisfying the conditions of the proposition.

	For the converse, we need that $S$ is compactly generated. So suppose given a map of sets $f : S \to \Conf_X$  such that $\Gamma_f \subset S \times X$ is closed and $\Gamma_f \to S$ is a local homeomorphism. Take a basic open set $B(K,\mathfrak U)$, with $\mathfrak U = \{U_i\}_{i=1}^n$. We check that $f^{-1}B(K,\mathfrak U)$ is open in $S$.	We show first that 
	the set $N = \{s \in S : \# (f(s) \cap K) \leq n\}$ is open in $S$. Since $S$ is compactly generated, we may assume $S$ itself to be compact Hausdorff. Take $s \in N$. Choose a neighborhood $V$ of $f(s)$ in $X$, such that each component of $\Gamma_f \cap (S \times V)$ maps homeomorphically to a neighborhood $W$ of $s$. Then $\Gamma_f \cap (S \times (K \setminus V))$ is compact, and in particular its image in $S$ is closed, and does not contain $s$ by construction. Removing this closed set from $W$ produces a neighborhood of $s$ contained in $N$. But then we note that $f^{-1}B(K,\mathfrak U)$ is the intersection of $N$ with the $n$ open sets given by the projections of $\Gamma_f \cap (S \times U_i)$ onto the first factor. 
\end{proof}

\begin{rem}
	The product of a compactly generated space with a locally compact space is compactly generated, so there is no ambiguity in the proposition regarding what topology to put on $S \times X$. 
\end{rem}

\begin{para}If $A$ is clear from context, we may omit it from notation and simply write $\Conf_{X}$ for $\Conf_{X;A}$. \emph{In the rest of this section of the paper, $X$ will always be a manifold and $A$ its boundary, and we consistently omit the boundary from the notation.} Disallowing points on the boundary is natural, since we will consider ``gluing'' operations on configuration spaces induced by gluing manifolds along their boundary. Such an operation could potentially introduce collisions if points are allowed to lie on the boundary. We want to emphasize in particular one difference in set-up here from the one in \cite{mcduff} (and many other places). McDuff considers spaces of particles on $M$ which can be annihilated and created along $\partial M$; we consider particles which can be annihilated and created ``at infinity'', whereas $\partial M$ simply acts as a ``wall''. For example, $\Conf_{[0,1) \times (0,1)}$ represents particles moving in the interior of a square, which are allowed to vanish along three of the four sides.
\end{para}

\subsection{Scanning for configuration spaces}

\begin{para}For any manifold $M$ there is a natural $E_1$-algebra structure on $\Conf_{M \times [0,1]}$ defined by ``placing configurations next to each other''. Let us suppose that $M$ is connected. If $M$ is noncompact, then $\Conf_{M \times [0,1]}$ is connected, but when $M$ is compact we have 
	$$ \Conf_{M \times [0,1]} \cong \coprod_{k \geq 0} \mathrm{Conf}_k(M^\circ \times (0,1)).$$
	In particular, this $E_1$-algebra is group-like for noncompact $M$, but not otherwise. The ``scanning map'', following terminology introduced by Segal, describes its group-completion.\end{para}

\begin{para}\label{scanning}In order to describe the scanning map, it will be convenient to replace $\Conf_{M \times [0,1]}$ with a strictly associative ``Moore'' variant. Fix $\epsilon > 0$, and consider the space which parametrizes pairs of $t \in [0,\infty)$, and a configuration $z \in \Conf_{M \times [0,t]}$, such that the configurations are empty above the intervals $[0,\epsilon)$ and $(t-\epsilon, t]$. This space is a topological monoid in an evident manner, and we denote this monoid $C$. There is a weak equivalence of $E_1$-algebras $C \simeq \Conf_{M \times [0,1]}$, in much the same way as the space of Moore loops is homotopy equivalent to the usual loop space. There is a natural map of topological monoids $C \to \Omega' \Conf_{M \times (-\epsilon,\epsilon)}$, where $\Omega'$ denotes the Moore loop space. It takes a configuration $z$ parametrized over $[0,t]$ to the Moore loop of length $t$ in $\Conf_{M \times (-\epsilon,\epsilon)}$, whose value at time $t_0$ is the part of the configuration which lies over the interval $(t_0-\epsilon,t_0+\epsilon)$. This is the \emph{scanning map} -- we are sweeping over a configuration using a scanner with field of vision of size $\epsilon$. \end{para}

\begin{thm}[Segal, McDuff] \label{conf} The scanning map $C \to \Omega' \Conf_{M \times (-\epsilon,\epsilon)}$ is a group-completion. That is, the group-completion of the $E_1$-algebra $\Conf_{M \times [0,1]}$ is weakly equivalent to $\Omega \Conf_{M \times \R}$. 
\end{thm}	


\begin{para}Before giving the proof of \cref{conf} we want to recall explicitly the bar construction. If $M$ is a well-pointed topological monoid, we define the \emph{bar construction} $BM$ as a quotient
	$$ \Big( \coprod_{p \geq 0} \Delta^p \times M^p\Big) \!\Big/\! \sim$$
	where $\sim$ denotes the equivalence relation generated by the following identifications:
	\begin{itemize}
		\item $((0,w_1,\ldots, w_p), (m_1,\ldots,m_p)) \sim ((w_1,\ldots,w_p),(m_2,\ldots,m_p))$,
		\item $((w_0,\ldots, w_{i-1},0,w_{i+1},\ldots,w_p), (m_1,\ldots,m_p)) \sim $ \\ \hphantom{hello} $ ((w_0,\ldots,w_{i-1},w_{i+1},\ldots,w_p),(m_1,\ldots,m_i \cdot m_{i+1},\ldots,m_p))$,
		\item $((w_0,\ldots,w_{p-1},0), (m_1,\ldots,m_p)) \sim ((w_0,\ldots,w_{p-1}),(m_1,\ldots,m_{p-1}))$,
			\item $((w_0,\ldots,w_p),(m_1,\ldots,m_{i-1},e,m_{i+1},\ldots, m_p)) \sim$ \\ \hphantom{hello} $(((w_0,\ldots,w_{i-1}+w_{i},\ldots,w_p), (m_1,\ldots,m_{i-1},m_{i+1},\ldots,m_p)). $ 
	\end{itemize}
	Here the $(w_i)$ are coordinates on the $p$-simplex: $0 \leq w_i \leq 1$, $\sum w_i=1$. The space $BM$ is, in an evident way, the geometric realization of a simplicial space. The bar construction is often instead defined as the geometric realization of the corresponding semisimplicial space (i.e.\ omitting the fourth identification in the definition of $\sim$). Then the construction given here would be called the reduced bar construction. For a well-pointed topological monoid the reduced and unreduced bar constructions are weakly equivalent \cite[Appendix 2]{segalconfiguration}.
\end{para}

\begin{para}For any connected based space, $X$ there is a natural weak equivalence $B\Omega' X \to X$, which is defined as follows. Given a point $(w_0,\ldots,w_p) \in \Delta^p$ and a $p$-tuple of Moore loops in $X$
	$$ \gamma_1 : [t_0,t_1]\to X, \qquad \gamma_2 : [t_1,t_2] \to X, \qquad \ldots \qquad \gamma_p : [t_{p-1},t_p] \to X,$$ 
	where $t_0 \leq t_1 \leq \ldots \leq t_p$, first concatenate these loops to a loop $\gamma : [t_0,t_p] \to X$. The image of this pair in $X$ is then the evaluation $\gamma(\sum w_i t_i) \in X$. 
\end{para}

\begin{proof}[Proof of Theorem \ref{conf}]We prove that the scanning map $C \to \Omega' \Conf_{M \times (-\epsilon,\epsilon)}$ of \S\ref{scanning} is a group-completion. Equivalently, the composition
	$$ BC \to B\Omega' \Conf_{M \times (-\epsilon,\epsilon)} \to \Conf_{M \times (-\epsilon,\epsilon)} $$ is a weak equivalence. 
	Let us denote this composition $\sigma$. 
	
	Consider first a based map $f: S^q \to \Conf_{M \times (-\epsilon,\epsilon)}$. We want to find a lift $g : S^q \to BC$ such that $f$ is homotopic to $\sigma \circ g$. 
	
	Choose (using compactness) a finite open cover $S^q = \bigcup_{\lambda \in \Lambda} U_\lambda$, such that there is a positive real $r>0$ and for each $\lambda$ an element $t_\lambda \in (-\epsilon,\epsilon)$, with the following property:
	\emph{for any $x \in U_\lambda$, the part of the configuration $f(x)$ above time-coordinate $(t_\lambda-r,t_\lambda+r)$ is empty.} Choose also a partition of unity $w_\lambda$ subordinate to the cover. 
	
	Consider $x \in U_{\lambda_0} \cap \ldots \cap U_{\lambda_k}$ (and assume it lies in no other $U_\lambda$). We assume that 
	$$ t_{\lambda_0} < \ldots < t_{\lambda_k}.$$
	By taking the configurations above the intervals $[t_{\lambda_0},t_{\lambda_1}]$, \ldots, $[t_{\lambda_{k-1}},t_{\lambda_k}]$, respectively, and stretching them all by a factor $\tfrac{\epsilon}{r}$, we get a $k$-tuple of elements of our monoid $C$. We define $g(x) \in BC$ to be the element defined by this $k$-tuple of elements of $C$, together with $(w_{\lambda_0}(x),w_{\lambda_1}(x),\ldots,w_{\lambda_k}(x)) \in \Delta^k$. 
	
	One checks that this is compatible with the identifications in the bar construction, and glues together well along overlaps for our open cover. In order that $g$ is a based map, we may insist that the base-point of $S^q$ lies only in one of the sets $U_\lambda$ (if this is not the case, just delete the base point from all but one of the open sets in the cover).
	
	The configuration given by $\sigma(g(x))$ is a ``zoomed in'' version of $f(x)$. More precisely, $\sigma(g(x))$ is the part of the configuration $f(x)$ which lies over a small interval of radius $r$, centered at the point $\sum w_\lambda(x) t_\lambda$. In particular, $\sigma \circ g$ is evidently homotopic to $f$.

	For injectivity on homotopy groups, consider instead a map $g : S^q \to BC$ and a nullhomotopy $f: D^{q+1} \to \Conf_{M \times (-\epsilon,\epsilon)}$ of $\sigma \circ g$. The same construction as above allows us to lift $f$ up to homotopy to a map $f' : D^{q+1} \to BC$. We should construct a homotopy between $g$ and $f' \vert_{S^q}$. Consider an element $x \in S^q$, and let us suppose that $x \in U_{\lambda_1} \cap \ldots \cap U_{\lambda_k}$ (and $x$ lies in no other $U_\lambda$). Now $g(x)$ consists of a $p$-tuple of elements of $C$ over intervals $[t_0,t_1]$,\ldots, $[t_{p-1},t_p]$, and a weight $w \in \Delta^p$. First, we continuously stretch all of these intervals by a factor $\tfrac{\epsilon}{r}$. After stretching, they can be nontrivially written as a product in $C$, with the time-parameters $t_{\lambda_i}$, $i=1,\ldots,k$, providing the breakpoints where we can split up the intervals. This gives us instead an equivalent point of the bar construction, given by $(p+k)$-tuple of elements of $C$ and a weight vector in $\Delta^{p+k}$, with $k$ zeroes inserted in the weight vector. Then continuously change the weights from the given weight vector $w$ to the one specified by the partition of unity $w_\lambda$, lowering the ``old'' weights to zero and raising the ``new'' weights from zero to their correct value $w_{\lambda_i}(x)$, $i=1,\ldots,k$. \end{proof}

\begin{rem}
	In the final paragraph of the preceding proof, it is important that we work with the reduced bar construction as opposed to the unreduced bar construction. In the construction we have continuously varying time-parameters $t_0,\ldots,t_p$ and $t_{\lambda_1},\ldots,t_{\lambda_k}$, and nothing in the construction prevents an overlap $t_i = t_{\lambda_j}$ for some $i,j$, which allows two such time-coordinates to ``swap places''. This has the effect of changing the slot in which we insert a zero coordinate on the simplex, and in the unreduced bar construction this would introduce a discontinuity in the final homotopy where the ``old'' weights are lowered and the ``new'' weights are increased. This is exactly solved by passing to the reduced bar construction, since an overlap $t_i = t_{\lambda_j}$ corresponds to one of the factors being the base-point of the monoid $C$. 
\end{rem}

\begin{prop}\label{stretching}Let $B^n$ be the open unit ball in $\R^n$. Then 	$\Conf_{B^n} \simeq S^n$.
\end{prop}

\begin{proof}First, define a continuous function $\phi : \Conf_{B^n} \to (0,1]$ as follows: if $z \subset B^n$ is a finite subset with at least two elements, then $\phi(z)$ is the \emph{second} smallest distance to the origin among the points in the configuration. (Note that this number is necessarily strictly positive, unlike the smallest distance.) If $z$ has cardinality $\leq 1$, then $\phi(z) = 1$. In formulas, $\phi(z) = \sup \{ r \in [0,1] : |z \cap (rB^n)| \leq 1 \}$.
	
	Then consider the continuous function $F: \Conf_{B^n} \to \Conf_{B^n}$ which rescales a configuration by a factor $1/\phi(z)$, and in particular forgets about  any points which lie outside the unit ball after rescaling. One verifies that the function $F$ defines a deformation retraction down to the subspace of $\Conf_{B^n}$ consisting of configurations of cardinality $\leq 1$. 
	
	But the subspace of $\Conf_{B^n}$ consisting of configurations of cardinality $\leq 1$ is in fact homeomorphic to an $n$-sphere. Indeed, the subspace of configurations of cardinality $1$ is just $B^n$ itself, and adding in the empty configuration amounts to forming the one-point compactification. \end{proof}

\begin{cor} The group-completion of the $E_n$-algebra $\Conf_{[0,1]^n} = \coprod_{k \geq 0} \mathrm{Conf}_k((0,1)^n)$ is $\Omega^n S^n$. \end{cor}

\begin{proof}
	Apply \cref{conf} iteratively $n$ times to see that $\Conf_{[0,1]^n}$ group-completes to $\Omega^n \Conf_{\R^n}$. \cref{stretching} identifies this with $\Omega^n S^n$. 
\end{proof}

\subsection{Hurwitz spaces}

\begin{para}
	We now want to generalize the results of the preceding subsection from configuration spaces to \emph{Hurwitz spaces}. Although we only need these results for branched double covers, we will consider more generally an arbitrary finite group $G$ and a space of branched $G$-covers $E\to M$. Moreover, as explained in \cref{hyp target space}, we should work with spaces parametrizing in addition the datum of a $G$-equivariant map $E \to X$, for some $G$-space $X$. 
\end{para}

\begin{para}
	Let us for the rest of this section fix the following data: a finite group $G$, a conjugation-invariant subset $c \subseteq G$ specifying ``allowed monodromies'' which we assume to generate $G$, and a based $G$-space $X$, which we assume to be connected. Since $G$, $c$ and $X$ will be fixed throughout the arguments, we consistently omit them from notation, writing $\Hur_M$ where something like $\Hur_G^c(M,X)$ would be more logical. 
\end{para}

\begin{para}Let $M$ be an oriented 2-manifold, possibly with boundary. Let us first informally describe the space $\Hur_M$ of interest to us. A point of $\Hur_M$ will consists of a finite  subset $z$ of $M^\circ$, a branched $G$-cover $E \to M$ with branch locus $z$ such that all local monodromies are in the set $c$, and a $G$-equivariant map $E \to X$. Moreover, since we will want to glue surfaces along their boundaries, we insist also that the $G$-torsor $\partial E \to \partial M$ comes with a chosen trivialization, and the map $\partial E \to X$ is constant at the basepoint. As in the case of $\Conf_M$ we want to topologize this in such a way that branch points can vanish at infinity. To make this precise, we describe $\Hur_M$ as a topological stack. 
\end{para}

\begin{defn}\label{hurwitzdef} Define $\Hur_M$ to be the functor $k\mathsf{Top} \to \mathsf{Grpd}$ assigning to a compactly generated space $S$ the following groupoid: an object consists of a closed subset $\Gamma \subset S \times M$ such that $\Gamma \cap (S \times \partial M) = \varnothing$ and $\Gamma \to S$ is a local homeomorphism with finite fibers, a branched $G$-cover $E \to S \times M$ with branch locus $\Gamma$ and all local monodromies in the set $c$, a trivialization $\partial E \cong G \times S \times \partial M$ of the $G$-torsor $\partial E \to S \times \partial M$, and a $G$-equivariant map $E \to X$ taking $\partial E$ to the basepoint of $X$. Morphisms in the groupoid are isomorphisms of this data.  
\end{defn}

\begin{prop}
	The functor $\Hur_M$ is a paratopological stack in the sense of Noohi. 
\end{prop}

\begin{proof}
		There is an evident forgetful map $\Hur_M\to\Conf_M$, and $\Conf_M$ is certainly a topological stack, being represented by an actual topological space. To show that $\Hur_M$ is a paratopological stack, we ``build'' it from $\Conf_M$ in simple steps: start with the universal incidence correspondence $\Gamma \subset \Conf_M \times M$, consider the mapping stack $\mathrm{map}(\Conf_M \times M \setminus \Gamma, BG)$ (i.e.\ the stack of unbranched $G$-covers of $\Conf_M \times M \setminus \Gamma$), take the open and closed substack where local monodromies are in $c$, observe that the universal $G$-torsor extends uniquely to a branched cover of $\Conf_M \times M$, form the mapping stack from the universal branched $G$-cover to $X$, and take $G$-fixed points. 
\end{proof}

\begin{rem}Presumably $\Hur_M$ is actually a topological stack, but we will not need this, and the argument just given only proves existence as a paratopological stack: Noohi only proves existence of mapping stacks in this generality \cite{noohi-mapping}.  \end{rem}

\begin{rem}
	We will consistently describe stacks as sheaves of groupoids rather than (more correctly) as categories fibered in groupoids.
\end{rem}

\begin{para}
	Describing $\Hur_M$ as a stack is not just a convenient way of specifying the topology on $\Hur_M$. In the case that $\partial M = \varnothing$, the functor $\Hur_M$ is simply not represented by a topological space, due to the presence of objects with nontrivial automorphisms, and we are obligated to treat it as a stack. When $\partial M \neq \varnothing$, the datum of a trivialization of $E \to M$ along $\partial M$ rigidifies the moduli problem and makes $\Hur_M$ representable by a space.  
\end{para}

\begin{para}Consider a $1$-manifold $L$. Then $\Hur_{L \times [0,1]}$ is an $E_1$-algebra in the same way as $\Conf_{M \times [0,1]}$. Again, we replace it with a strictly associative version. Fix $\epsilon > 0$ and let $H$ be the topological monoid parametrizing the data of $t \in [0,\infty)$, and a point of $\Hur_{L \times [0,t]}$, which we demand to be {standard} over the intervals $[0,\epsilon)$ and $(t-\epsilon,t]$. As before, there is a scanning map $H \to \Omega' \Hur_{L \times (-\epsilon,\epsilon)}$. \end{para}

\begin{thm}\label{hurw}Suppose that $L$ is an interval (closed, open or half-open). The scanning map $H \to \Omega' \Hur_{L \times (-\epsilon,\epsilon)}$ is the group-completion of $H$. That is, the group-completion of the $E_1$-algebra $\Hur_{L \times [0,1]}$ is weakly equivalent to $\Omega \Hur_{L \times \R}$. 
\end{thm}

\begin{proof}[Proof in the case that $L$ is an open interval] As in the proof of \cref{conf}, we want to show that $\sigma: BH \to \Hur_{L \times (-\epsilon,\epsilon)}$ is a weak equivalence. So take a based map $f : S^q \to \Hur_{L \times (-\epsilon,\epsilon)}$, giving a branched $G$-cover of $S^q \times L \times (-\epsilon,\epsilon)$, branched over a closed subset $\Gamma_f$.  Choose as before a finite open cover $S^q = \bigcup_{\lambda \in \Lambda} U_\lambda$, such that there is a positive real $r>0$ and for each $\lambda$ an element $t_\lambda \in (-\epsilon,\epsilon)$, with the following property:
	for each $\lambda$, $$\Gamma_f \cap (U_\lambda \times L \times (t_\lambda-r,t_\lambda+r)) = \varnothing.$$
	Without loss of generality, each $U_\lambda$ is contractible. Then for each $\lambda$ we get a \emph{trivial} $G$-torsor over $U_\lambda \times L \times (t_\lambda-r,t_\lambda+r)$, since we have an unbranched cover of a contractible space. We choose a trivialization for each $\lambda$. 
	
	Now we need to care for the map to the target space $X$. For each $\lambda$ we get an equivariant map
	$$G \times U_\lambda \times L \times (t_\lambda-r,t_\lambda+r)\to X,$$
	and since $X$ is connected it is necessarily equivariantly nullhomotopic. Plugging in this nullhomotopy in the middle of the interval $(t_\lambda-r,t_\lambda+r)$ we may arrange that the map to $X$ is constant over the smaller interval of radius $r/2$, say. To get this nullhomotopy to be defined on all of $S^q$, we apply the shrinking lemma: there exists an open refinement by subsets $U_\lambda' \subset U_\lambda$, and for each $\lambda$ a bump function on $S^q$ which is $1$ on $U_\lambda'$ and with support in $U_\lambda$. Taking a nullhomotopy ``mollified'' by this bump function, we finally arrange that the map to the target space is constant above $U_\lambda' \times L \times (t_\lambda - \tfrac{r}{2},t_\lambda + \tfrac r 2)$.  We then take $r/2$ to be our new $r$, and $U_\lambda'$ to be our new $U_\lambda$. These homotopies for different $\lambda$ may be done independently of each other, since without loss of generality the family of intervals $(t_\lambda-r,t_\lambda+r)$ are disjoint. 
	
	After all this finagling, we are in a situation exactly like the one for ordinary configurations, and by exactly the same arguments as in the proof of \cref{conf} the map $f : S^q \to \Hur_{L \times (-\epsilon,\epsilon)}$ lifts up to homotopy to a map $g : S^q \to BH$. 
	
	Suppose instead that we have a map $g : S^q \to BH$ and a nullhomotopy $f : D^{q+1} \to \Hur_{L \times (-\epsilon,\epsilon)}$ of $\sigma \circ g$. We carry out the same procedure as just described to homotope $f$ to a map $f_{(1)}$ for which there is an open cover $U_\lambda$ of $D^{q+1}$, such that we can choose a trivialization of the $G$-torsor over $U_\lambda \times L \times (t_\lambda - r,t_\lambda+r)$ for all $\lambda$, and the map to the target space is constant over this region. The construction ``lifts'' in an evident way to a homotopy between $g$ and a new map $g_{(1)}$, in such a way that $f_{(1)}$ is instead a nullhomotopy of $\sigma \circ g_{(1)}$. Then by the same construction as earlier, $f_{(1)}$ lifts up to homotopy to a map $D^{q+1} \to BH$, and the construction in the final paragraph of the proof of \cref{conf} shows that this lifting can be homotoped to a nullhomotopy of $g_{(1)}$. 
\end{proof}

\begin{para}\label{explanation problem}The argument just given works equally well when $L$ is a half-open interval, but it is also less interesting as $\Hur_{L \times [0,1]}$ is contractible in this case. However, it needs to be modified in the case that $L=[a,b]$ is a \emph{closed} interval. The problem arises in the first paragraph of the proof, when arguing that we can choose a trivialization of the $G$-torsor over $U_\lambda \times L \times (t_\lambda-r,t_\lambda+r)$ for each $\lambda$. When $L$ has boundary, this $G$-torsor already comes with two trivializations; one over $U_\lambda \times \{a\} \times (t_\lambda-r,t_\lambda+r)$ and one over $U_\lambda \times \{b\} \times (t_\lambda-r,t_\lambda+r)$. In general, there is no reason for these two trivializations to be compatible with each other, which they would need to be for the rest of the argument to carry through. Let us now explain how to tweak the set-up to treat the case that $L$ is closed. 
\end{para}

\begin{defn}Let $M$ be an oriented surface, $J \subset \partial M$ an embedded interval which is closed in $M$. We write $\Hur_M^J$ for the stack assigning to a compactly generated space $S$ the groupoid of branched $G$-covers $p: E \to S \times M$ just as in \cref{hurwitzdef}, except we only insist on a trivialization $$p^{-1}(S \times (\partial M \setminus J^\circ)) \cong G \times S \times (\partial M \setminus J^\circ),$$ and the equivariant map $E \to X$ is only required to take $p^{-1}(S \times (\partial M \setminus J^\circ))$ to the basepoint of $X$. \end{defn}
\begin{para}
	There is an evident restriction map $\Hur_M^J \to \mathrm{map}_\ast(J/\partial J,[X/G])$, where $[X/G]$ denotes the quotient stack of $G$ acting on $X$. Indeed, an $S$-point of the right-hand side consists of a $G$-torsor over $S \times J$ with a trivialization over $S \times \partial J$, and a $G$-equivariant map from the total space to $X$ taking the inverse image of $S \times \partial J$ to the basepoint of $X$. 
\end{para}

\begin{lem} \label{fibration lemma} The restriction map $\Hur_M^J \to \mathrm{map}_\ast( J/\partial J,[X/G])$ is a fibration. \end{lem}

\begin{proof}Let $S$ be a compactly generated space, and consider a lifting problem\[\begin{tikzcd}
		S \arrow[d]\arrow[r]&\arrow[d] \Hur_M^J \\
		I \times S \arrow[ur, dotted]\arrow[r]&  \mathrm{map}_\ast( J/\partial J,[X/G])
	\end{tikzcd}\]
	The map $S \to \Hur_{M}^J$ furnishes in particular a branched $G$-cover $p:E \to S \times M$; then $I \times E$ is naturally a branched $G$-cover of $I \times S \times M$, which is correctly trivialized to yield the first piece of data to define a lifting in the diagram. To fill in the diagram, we need moreover a $G$-equivariant map $I \times E \to X$ which agrees with the specified values on $\{0\} \times E$ and $I \times p^{-1}(S \times  J)$; that is, we should argue that the inclusion $p^{-1}(S \times {J}) \hookrightarrow E$ is a $G$-cofibration (satisfies the equivariant homotopy extension property). To see this, one observes that $p^{-1}(S \times  J)$ admits a neighborhood on which $G$ acts freely, and taking the quotient by $G$ on this neighborhood gives back an open neighborhood of $S \times  J$ in $S \times M$. But   $S \times  J \hookrightarrow S \times M$ is an ordinary cofibration.\end{proof}

\begin{proof}[Proof of \cref{hurw} in case $L={[a,b]}$ is a closed interval]
	A consequence of \cref{fibration lemma} is that there is a diagram of fiber sequences
	\[\begin{tikzcd}
		\Hur_{L \times [0,1]} \arrow[d]\arrow[r]&\arrow[d]\arrow[r] \Hur_{L \times [0,1]}^{\{a\} \times [0,1]} & \Omega [X/G] \arrow[d,"\simeq"]\\
		\Omega \Hur_{L \times (-\epsilon,\epsilon)}\arrow[r] & \Omega \Hur_{L \times (-\epsilon,\epsilon)}^{\{a\} \times (-\epsilon,\epsilon)} \arrow[r]& \Omega [X/G]. 
	\end{tikzcd}\]
	Indeed, the top row is obtained directly from \cref{fibration lemma} with $M= L\times [0,1]$ and $J= \{a\}\times[0,1]$. The bottom row is instead obtained by once looping the fiber sequence given by \cref{fibration lemma} for $M=L \times (-\epsilon,\epsilon)$, $J=\{a\}\times (-\epsilon,\epsilon)$. Note in particular that $\partial J=\varnothing$ so $\mathrm{map}_\ast( J/\partial J,[X/G]) \simeq [X/G]$ in this case. The three vertical maps are given by the respective scanning maps. Now the entire diagram deloops to 
	\[\begin{tikzcd}
		B\Hur_{L \times [0,1]} \arrow[d]\arrow[r]&\arrow[d]\arrow[r] B\Hur_{L \times [0,1]}^{\{a\} \times [0,1]} &  {[X/G]} \arrow[d,"\simeq"]\\
		\Hur_{L \times (-\epsilon,\epsilon)}\arrow[r] &  \Hur_{L \times (-\epsilon,\epsilon)}^{\{a\} \times (-\epsilon,\epsilon)} \arrow[r]&  {[X/G]}.
	\end{tikzcd}\]
	By an argument exactly like the one in the proof given immediately after \cref{hurw}, the middle vertical map is an equivalence. Indeed, the issue explained in \S\ref{explanation problem} is no longer a problem, precisely because in the space $\Hur_{L \times [0,1]}^{\{a\} \times [0,1]}$ we no longer insist on a trivialization along $\{a\} \times [0,1]$. Thus, the leftmost vertical map is a weak equivalence, too. 
\end{proof}

\subsection{Local Hurwitz space}

\begin{para}
	Just as in \cref{stretching}, let us now analyze the topology of the Hurwitz space $\Hur_M$ when $M=B^2$ is an open disk. We will explicitly determine the homotopy type of the paratopological stack $\Hur_{B^2}$. All the reader needs to know about the homotopy theory of stacks is that one can consistently assign a homotopy type to a paratopological stack \cite{noohi-homotopy}, and that the homotopy type of a quotient stack $[Y/\Gamma]$ is the homotopy quotient $Y/\!/\Gamma$. 
\end{para}

\begin{construction}Introduce the notation
	$$ X(c) := \coprod_{g \in c} X^g $$
	for the disjoint union of the sets of $g$-fixed points, for $g \in c$. We consider $X(c)$ itself as a $G$-space in the natural way: multiplication by an element $h$ maps $X^g$ to $X^{hgh^{-1}}$.  
\end{construction}

\begin{construction}\label{evaluation}A point of $\mathrm{map}(S^1,[X/G])$ can be identified with a $G$-torsor $E\to S^1$ and a $G$-equivariant map $E\to X$. We write $\mathrm{map}(S^1,[X/G])_0$ for the substack on which the corresponding map $E\to X$ is locally constant. There is an isomorphism of stacks
	$$ \mathrm{map}(S^1,[X/G])_0 \cong \big[\big(\coprod_{g \in G} X^g\big)/ G\big] = [X(G) / G].$$
	For any conjugation-invariant $c \subseteq G$ we therefore get a natural ``evaluation'' map
	$$ S^1 \times [X(c)/ G] \hookrightarrow S^1 \times [X(G)/ G] \cong S^1 \times \mathrm{map}(S^1,[X/G])_0 \to [X/G].$$
\end{construction}

\begin{thm}\label{local hurwitz space}
	There is a homotopy pushout square 
	\[\begin{tikzcd}
		S^1 \times X(c)/\!/G \arrow[r]\arrow[d]& \arrow[d]X/\!/G \\
		X(c)/\!/G \arrow[r]& \Hur_{B^2}
	\end{tikzcd}\]
	in which the top horizontal arrow is the evaluation map of \cref{evaluation}, and the left vertical arrow is the projection to the second factor. 
\end{thm}
\begin{proof} The Hurwitz space deformation retracts down to the substack on which there is at most one branch point, by the same argument as in \cref{stretching}. To analyze this substack, we cover it by two open pieces.

	Firstly, there is the locus where there is no branch point at the origin. We claim that this substack is homotopy equivalent to $X/\!/G$. Indeed, by an argument much like that of \cref{stretching}, but using the \emph{smallest} distance to the origin instead of the second-smallest distance, this space deformation retracts down to the subspace where there are no branch points at all. Thus, we get a trivial $G$-torsor over a disk and an equivariant map from the total space to $X$, which is the same as a map from a disk to $[X/G]$. 
	
	Secondly, there is the open locus where there is exactly one branch point on the disk. We claim that this substack is homotopy equivalent to $X(c)/\!/G$. We may for this calculation replace the open disk with a closed disk, and insist that there is a unique branch point on the interior. If we add the auxilliary datum of a trivialization of the $G$-torsor over some fixed point on the boundary of the disk, then we get a space of the homotopy type of $X(c)$. Indeed, note first that after such a choice, the monodromy around the branch point becomes a well-defined element $g \in c$. Moreover, the branched cover is a finite union of disks, and an equivariant map to $X$ from this union is determined up to homotopy by where it sends the center of one of the disks: this will be an element of $X^g$, yielding the space $X(c) = \coprod_{g \in c} X^g$. But then forgetting the auxiliary datum of a trivialization over a point on the boundary amounts to taking the stack quotient by $G$.   
	
	These open subsets are glued along their intersection, whose homotopy type is $S^1 \times X(c)/\!/G$, by an argument much like the previous paragraph. \end{proof}

\begin{para}\label{clutching}We remind the reader of the ``clutching construction'', which says that rank $d$ vector bundles on a sphere $S^n$ are classified by maps $S^{n-1} \to \mathrm{GL}(d)$. The same construction works not just for vector bundles but for fibrations: for any space $F$, fibrations over $S^n$ with fiber $F$ are in bijection with maps $S^{n-1} \to \mathrm{hAut}(F)$, up to homotopy. Explicitly, a map $S^{n-1} \to \mathrm{hAut}(F)$ gives an automorphism of the fibration $S^{n-1} \times F \to S^{n-1}$, and the fibration over $S^n$ associated to such a clutching function can be described as the homotopy pushout 
$$ F \coprod^h_{ S^{n-1} \times F } F,$$
where one of the maps is the projection to the second factor, and the other is the projection to the second factor precomposed with the chosen automorphism of $S^{n-1} \times F$. \end{para}

\begin{proof}[Proof of \cref{scanning theorem}]
	Let $X=KA$, $G=\Sigma_2$, and let $c$ consist only of the non-identity element. Then there is a weak equivalence of $E_2$-algebras
 $$\Hur_{[0,1]^2} \simeq \coprod_{g \geq 0} H_g^{0,1}(KA). $$
	By two applications of \cref{hurw}, the $2$-fold delooping of this $E_2$-algebra is $\Hur_{B^2}$. By \cref{local hurwitz space}, the double delooping therefore has the homotopy type of the pushout
	$$ KA/\!/\langle \iota \rangle \coprod^h_{S^1 \times B\Sigma_2} B\Sigma_2,$$
	using that $(KA)^{\Sigma_2} = \ast$. The map $S^1 \times B\Sigma_2 \to KA/\!/\langle \iota \rangle$ factors through the map $B\Sigma_2 \to KA/\!/\langle \iota\rangle$ induced by the inclusion of the $\Sigma_2$-fixed points in $KA$, so we may rewrite the pushout as 
	$$ KA/\!/\langle \iota \rangle \coprod^h_{B\Sigma_2 } B\Sigma_2 \coprod^h_{S^1 \times B\Sigma_2} B\Sigma_2$$
	and to finish the proof of \cref{scanning theorem} we should only identify the pushout $B\Sigma_2 \coprod^h_{S^1 \times B\Sigma_2} B\Sigma_2$ with the space $W$ defined in \cref{def of W}. But the latter pushout is an instance of the clutching construction described in \S\ref{clutching}, in the case $n=2$ and $F=B\Sigma_2$. Indeed, one leg of the pushout is projection to the second factor, and the other leg is the ``evaluation'' map of \cref{evaluation}. It is not hard to check that the evaluation can be identified with the precomposition of projection to the second factor with the unique nontrivial automorphism of the fibration $S^1 \times B\Sigma_2 \to S^1$, so that the pushout may be identified with the unique nontrivial fibration over $S^2$ with fiber $B\Sigma_2$ --- but this is exactly the space $W$ defined in \cref{def of W}.  
\end{proof}

\begin{rem}
	More generally, the pushout of \cref{local hurwitz space} can always be rewritten in the form 
	$$ \Hur_{B^2} \simeq X/\!/G \coprod^h_{X(c)/\!/G} W(c),$$
	where $W(c)$ is the total space of a fibration over $S^2$ with fiber $X(c)/\!/G$. This fibration is described by the clutching function $S^1 \to \mathrm{hAut}(X(c)/\!/G)$ given by identifying $X(c)/\!/G$ with a subspace of $\mathrm{map}(S^1,X/\!/G)$ via \cref{evaluation}, on which $S^1$ acts by rotation. 
\end{rem}

\section{Double loop space homology}\label{koszul section}

\begin{para}
	In the preceding section of the paper, we showed that the group-completion of the monoid $\coprod_{g \geq 0} H_g^{0,1}(KA)$ is the double loop space $\Omega^2 \, W(A)$, if $A$ is a connected chain complex without $2$-torsion. In particular, the group-completion theorem identifies the stable homology of $H_g^{0,1}(KA)$ with the homology of the space $\Omega^2_0 \, W(A)$. The goal of this section is to  calculate the rational homology $H_\bullet(\Omega^2_0 \, W(A),\Q)$ as an explicit analytic functor of $A$. This result will appear as \cref{calculation of analytic functor}.
\end{para}

\begin{para}The proof of \cref{calculation of analytic functor} is an application of Koszul duality and rational homotopy theory. To fix conventions and for the reader's convenience, we begin by briefly recalling the notions of quadratic algebras and coalgebras, and Koszul duality of such algebras. We work over $\Q$, although all results here are valid more generally over any base field. The algebraic results recalled here are all quite standard. The original reference is the work of Priddy \cite{priddy}; for a textbook account, see Loday--Vallette \cite[Chapter 3]{lodayvallette}. We then discuss the notion of a \emph{Koszul space}, introduced by Berglund \cite{berglundkoszulspaces}, and prove that $W(A)$ is an example of a Koszul space. After all this preparation, we state and prove \cref{calculation of analytic functor}.
 \end{para}

\subsection{Quadratic algebras}

\begin{para}
	Let $V$ be a graded vector space, and $R \subseteq V \otimes V$ a subspace. We denote by $A(V,R)$ the associative algebra $\mathsf T V/\langle R\rangle$, where $\mathsf TV$ denotes the tensor algebra on $V$ and  $\langle R \rangle$ is the two-sided ideal generated by $R \subseteq V^{\otimes 2} \subset \mathsf TV$. An associative algebra of this form will be called a \emph{quadratic algebra}. Being quadratic is a property of an associative algebra, not extra structure, as such a quadratic presentation is unique if it exists; $V$ can be recovered as the space of indecomposables in the algebra. 
\end{para}\begin{para}
A quadratic algebra $A(V,R)$ has two natural gradings: one arising from the grading of $V$, which we refer to as the \emph{homological degree}, and one given by putting $V^{\otimes n}$ in degree $n$, which we refer to as the \emph{length grading}, or the \emph{tensor length}.  If $A$ is a quadratic algebra, then we denote by $A_i(n) \subset A$ the subspace of homological degree $i$ and tensor length $n$. We say that $A$ is of \emph{finite type} if $\dim A_i(n)<\infty$ for all $i$ and $n$.
The \emph{Hilbert--Poincar\'e series} of a quadratic algebra $A$ of finite type is defined by 
$$ A(t,z) = \sum \dim(A_i(n))t^n (-z)^i.  $$
The fact that $(-z)^i$ appears in the formula rather than $z^i$ is due to the fact that the Koszul sign rule is enforced with respect to the homological grading, cf.\ the discussion in \S\ref{sign rule 1}. Similarly, $t^n$ appears rather than $(-t)^n$, since the Koszul sign rule is {not} enforced with respect to the length grading.  
\end{para}\begin{para}
To the same data $(V,R)$ we also associate a \emph{quadratic coalgebra} $C(V,R)$. Explictly, $C(V,R)$ is the subcoalgebra of the tensor coalgebra $\mathsf T^c V$ which in tensor length $n$ is given by the intersection $\cap_{i+2+j=n} V^{\otimes i} \otimes R \otimes V^{\otimes j}$. Again, being quadratic is a property of a coalgebra, not additional structure. Quadratic coalgebras are similarly bigraded, and an identical formula defines the Hilbert--Poincar\'e series of a quadratic coalgebra of finite type. 

\end{para}\begin{para}
If $A = A(V,R)$ is a quadratic algebra, then we define its \emph{quadratic dual} coalgebra to be the quadratic coalgebra $A^\antishriek = C(V[-1],R[-2])$. Similarly, the quadratic dual algebra of $C = C(V,R)$ is $C^\antishriek = A(V[1],R[2])$.
\end{para}

\subsection{Bar-cobar duality and Koszul duality}

\begin{para}
	We denote by $B$ and $\Omega$ the bar and cobar functors \cite[Section 2.2]{lodayvallette}. The bar construction $B$ is a functor from augmented dg algebras to coaugmented conilpotent dg coalgebras, and the cobar functor is its left adjoint. The important fact about the bar and cobar constructions is that they are very close to being inverses of each other up to quasi-isomorphism. Recall that a coaugmented coalgebra $C$ is $n$-connected if the coaugmentation coideal $\overline C$ vanishes in degree $\leq n$. 
	\end{para}

\begin{thm}\label{barcobarduality}The bar and cobar functors preserve quasi-isomorphisms in the following sense:
	\begin{enumerate}[(i)]
		\item If $A \to A'$ is a quasi-isomorphism of dg algebras, then $BA\to BA'$ is a quasi-isomorphism.
		\item If $C \to C'$ is a quasi-isomorphism of $1$-connected dg coalgebras, then $\Omega C \to \Omega C'$ is a quasi-isomorphism.
		\item The counit map $\Omega BA \to A$ is a quasi-isomorphism.
		\item The unit map $C \to B \Omega C$ is a quasi-isomorphism. 
	\end{enumerate}
\end{thm}

\begin{para}The bar and cobar constructions predate the notion of Koszul duality. The original motivation for introducing the cobar construction was in order to formulate the following theorem \cite{adams}:\end{para}

\begin{thm}[Adams]\label{adams}Let $X$ be a $1$-connected CW complex. Then $\Omega C_\bullet(X,\Z) \simeq C_\bullet(\Omega' X,\Z)$, where $C_\bullet(\Omega' X,\Z)$ is considered as a dg algebra using the concatenation of Moore loops. 
\end{thm}

\begin{rem}
	The conclusion of the theorem holds also for $X$ not necessarily simply connected \cite{riverazeinalian}.
\end{rem}

\begin{defn}For any quadratic algebra $A$ there is a natural morphism $A^\antishriek \to BA$, and for a quadratic coalgebra $C$ there is a natural morphism $ \Omega C\to C^\antishriek$. A quadratic algebra $A$ is said to be \emph{Koszul} if $A^\antishriek \to BA$ is a quasi-isomorphism. Similarly, a quadratic coalgebra $C$ is Koszul if $\Omega C \to C^\antishriek$ is a quasi-isomorphism. A quadratic algebra $A$ is Koszul if and only if $A^\antishriek$ is Koszul; we say that $A$ and $A^\antishriek$ are \emph{Koszul dual}. 
	\end{defn}

\begin{para}\label{twisting morphism}Suppose we are given a coaugmented coalgebra $C$, an augmented algebra $A$, and a homomorphism of algebras $\tau \colon \Omega C \to A$ (equivalently, a morphism of coalgebras $C \to BA$), i.e.\ what is called a \emph{twisting morphism} from $C$ to $A$. To this data, one can associate naturally a differential on the tensor product $C \otimes A$. We write $C \otimes_\tau A$ for the tensor product equipped with this twisting differential.  If $\tau$ is a quasi-isomorphism, then $C \otimes_\tau A \simeq \Q$. In particular, applying this to the canonical quasi-isomorphism $\tau : \Omega C \to C^\antishriek$ for a Koszul coalgebra $C$, taking alternating sums of dimensions in $C \otimes_\tau C^\antishriek$ gives a relation between the Hilbert--Poincar\'e series of $C$ and $C^\antishriek$, which reads as follows: 
\end{para}

\begin{thm}\label{functional equation}Let $C$ be a Koszul coalgebra of finite type. Then $$C(t,z)C^\text{\emph{\antishriek}}(\tfrac t z,z) = 1.$$ \end{thm}

\begin{rem}
	This result is most often stated without the additional homological grading, which corresponds to setting $z=-1$, yielding $C(t)C^\antishriek(-t)=1$.
\end{rem}

\begin{para}
	It will be important for us that the construction of \S\ref{twisting morphism} is \emph{natural}, and that therefore so is \cref{functional equation}. For example, if a group $G$ acts on the coalgebra $C$ and the algebra $A$, and the morphism $\tau \colon \Omega C \to A$ is $G$-equivariant, then $C \otimes_\tau A$ is a chain complex of $G$-modules. Thus, we obtain a refined form of \cref{functional equation} in which the coefficients of $C(t,z)$ and $C^\antishriek(\tfrac t z , z)$ are taken in $R(G)$, the Grothendieck ring of representations of $G$. Even more abstractly, we will be interested in quadratic (co)algebras $A(V,R)$ ($C(V,R)$) in which $V$ and $R$ are both \emph{analytic functors}. We may think of this either as having a quadratic algebra in the category of analytic functors, or as having an analytic functor taking values in the category of quadratic algebras. Either way, we obtain a refined form of \cref{functional equation} in which the coefficients of $C(t,z)$ and $C^\antishriek(\tfrac t z , z)$ are taken in $\widehat \Lambda$, the Grothendieck ring of analytic functors.

\end{para}

\subsection{Koszul spaces in general, and the space \texorpdfstring{$W(A)$}{W(A)} in particular}

\begin{defn}
	A simply connected based space $X$ of finite $\Q$-type is said to be \emph{Koszul} if $X$ is formal in the sense of rational homotopy theory, and $H^\bullet(X,\Q)$ is a Koszul algebra.
\end{defn}

\begin{examplex}\label{koszul examples}
	The following are simple examples of Koszul spaces:\begin{enumerate}
		\item Spheres,
		\item Eilenberg--MacLane spaces,
		\item Wedges of Koszul spaces,
		\item Products of Koszul spaces.
	\end{enumerate}
\end{examplex}

\begin{para}
	All we will need about Koszul spaces is the following simple proposition; see Berglund \cite[Theorem 4]{berglundkoszulspaces} for a more general statement. 
\end{para}
\begin{prop}\label{berglund theorem}
	Let $X$ be a Koszul space. Then $H_\bullet(\Omega X,\Q)$, with the algebra structure given by concatenation of loops, is the quadratic dual of the Koszul coalgebra $H_\bullet(X,\Q)$.
\end{prop}

\begin{proof}By Adams' \cref{adams} we have $C_\bullet(\Omega X,\Z) \simeq \Omega C_\bullet(X,\Z)$. Now $X$ is formal over $\Q$, so $C^\bullet(X,\Q) \simeq H^\bullet(X,\Q)$, and for $X$ of finite $\Q$-type we may dualize this to an equivalence of coalgebras $C_\bullet(X,\Q) \simeq H_\bullet(X,\Q)$. Hence, $C_\bullet(\Omega X,\Q) \simeq \Omega C_\bullet(X,\Q) \simeq \Omega H_\bullet(X,\Q)$ using \cref{barcobarduality} and that $X$ is simply connected. Since $H_\bullet(X,\Q)$ is in addition a Koszul coalgebra, $\Omega H_\bullet(X,\Q)$ is equivalent to the quadratic dual of $H_\bullet(X,\Q)$. 
	\end{proof}

\begin{para}We will now prove that the space $W(A)$ is Koszul. \end{para}

\begin{lem}Let $A$ be a $1$-connected chain complex. The space $W(A)$ is simply connected. \end{lem}

\begin{proof}Apply the Seifert--van Kampen theorem to the covering of $W(A)$ by $W$ (simply connected) and $KA/\!/\langle \iota \rangle$ (fundamental group $\Sigma_2$), with intersection $B\Sigma_2$.
\end{proof}

\begin{lem}\label{first reduction for WA} Let $A$ be a $1$-connected chain complex. After localizing away from the prime $2$, there is a homotopy equivalence $W(A) \simeq S^2 \vee KA/\langle \iota \rangle $. 
	
\end{lem}

\begin{proof}
	There is a map $W(A) \to S^2 \vee KA/\langle \iota \rangle $ defined by the projections $W \to S^2$, $B\Sigma_2 \to \ast$, and $KA/\!/\langle \iota \rangle \to KA/\langle \iota \rangle$. Each of these maps is an isomorphism on homology with $\Z[\tfrac 1 2]$-coefficients, and therefore so is $W(A) \to S^2 \vee KA/\langle \iota \rangle $ by Mayer--Vietoris. Both source and target are simply connected, so the result follows by Whitehead's theorem. Simple connectedness of $KA/\langle \iota \rangle$ follows from \cite{armstrongorbit}. 
\end{proof}

\begin{para}We will now prove formality of the space $KA/\langle\iota\rangle$, using a criterion due to Sullivan \cite[Theorem 12.7]{sullivaninfinitesimal}. Let $V = \bigoplus_n V^n$ be a graded $\Q$-vector space. For $q \in \Q^\times$ we define a linear map $\phi_q : V \to V$ by $\phi_q(x) = q^n x$, if $\vert x \vert = n$. Every map of this form is called a \emph{grading automorphism} of $V$. \end{para}

\begin{thm}[Sullivan] Let $C$ be a nilpotent commutative dg algebra over $\Q$. The following are equivalent:
	\begin{enumerate}
		\item $C$ is formal. 
		\item Every automorphism of $H(C)$ lifts to an automorphism of $C$.
		\item Some grading automorphism $\phi_q$ of $H(C)$, where $q\neq \pm 1$, lifts to an automorphism of $C$. 
	\end{enumerate}
\end{thm}

\begin{prop}\label{formality} The space $KA/\langle \iota \rangle$ is formal in the sense of rational homotopy theory. 
\end{prop}

\begin{proof} Let the group $\Q^\times$ act on the graded vector space $A$ by grading automorphisms. It induces an action of $\Q^\times$ on the space $KA$, and thus also an action of $\Q^\times$ on $H^\bullet(KA,\Q)$. The action on $H^\bullet(KA,\Q) = \mathrm{Sym}(A^\vee)$ is again by grading automorphisms. The action of $\Q^\times$ on $A$ commutes with the involution $\iota$, so the action descends to an action of $\Q^\times$ on $KA/\langle \iota \rangle$, and since $$H^\bullet(KA/\langle\iota\rangle,\Q) \subset H^\bullet(KA,\Q)$$ it again acts by grading automorphism. Thus, every grading automorphism of $H^\bullet(KA/\langle\iota\rangle,\Q)$ lifts to an automorphism of the \emph{actual topological space} $KA/\langle \iota \rangle$. In particular, every grading automorphism certainly lifts to an automorphism of the cochain algebra. 
\end{proof}

\begin{prop}\label{koszulity} The cohomology ring $H^\bullet(KA/\langle \iota\rangle,\Q)$ is a Koszul algebra. 
\end{prop}

\begin{proof}The cohomology ring $H^\bullet(KA,\Q)$ is the symmetric algebra $\Sym(A^\vee) = \bigoplus_{d \geq 0} \Sym^d(A^\vee)$. Depending on the parity of $A$ this is either a polynomial ring or an exterior algebra (or a tensor product of the two). Either way, it is certainly a Koszul algebra. Now $H^\bullet(KA/\langle\iota\rangle,\Q)$ is the subalgebra $\bigoplus_{d \geq 0} \Sym^{2d}(A^\vee)$, i.e.\ the second Veronese subring of $\Sym(A^\vee)$. But Veronese subalgebras of Koszul algebras are Koszul, by a theorem of Backelin and Fr\"oberg \cite{backelinfroberg}; see also \cite[Chapter 3, Proposition 2.2]{polishchukpositselski}.
\end{proof}

\begin{cor}
	The space $W(A)$ is Koszul.
\end{cor}

\begin{proof}
	By \cref{first reduction for WA} and \cref{koszul examples} it suffices to show that $KA/\langle \iota\rangle$ is a Koszul space. This is shown in \cref{formality} and \cref{koszulity}.
\end{proof}

\subsection{The homology of \texorpdfstring{$\Omega^2_0 W(A)$}{Omega2 W(A)}}

\begin{thm}\label{calculation of analytic functor}
	The assignment $A \mapsto H_\bullet(\Omega^2_0 W(A),\Q)$ is an analytic functor from $1$-connected graded vector spaces to graded vector spaces. As such, its Taylor series is given by 
	\[\Exp(z^{-1} \Log(z + \sum_{r \geq 0} z^{-r} h_{2r}) - 1).\]
\end{thm}

\begin{proof}The calculation proceeds in multiple steps.

	(i) The homology $H_\bullet(W(A) , \Q)$ is a Koszul coalgebra. As such, it has a Hilbert--Poincar\'e series in the ring $\Z[[t,z]]$, where the variable $t$ corresponds to length grading and the variable $z$ corresponds to homological degree. But $A \mapsto H_\bullet(W(A) , \Q)$ is in addition an analytic functor of $A$, and as such it has a Taylor series taking values in $ \Lambda \widehat\otimes{\Z[[t,z]]}$. (This is the bigraded version of the construction for graded vector spaces explained in Section \ref{graded plethystic algebra}.) Its Taylor series is
$$ F(t,z) = t z^2 + \sum_{r\geq 0} t^r h_{2r}. $$
Indeed, the Hilbert--Poincar\'e series of $H_\bullet(S^2,\Q)$ is $1+tz^2$, which accounts for the first term, since $W(A) \simeq_\Q S^2 \vee KA/\langle \iota\rangle$. The Hilbert--Poincar\'e series of $H_\bullet(KA,\Q) \cong \mathrm{Sym}(A)$ is $\Exp(t h_1) = \sum_{r \geq 0} t^rh_{r}$. Multiplication by $-1$ in $A$ acts on $H_\bullet(KA,\Q)$ via $h_r \mapsto (-1)^r h_r$, so the quotient coalgebra of coinvariants $H_\bullet(KA/\langle \iota \rangle,\Q)$ has Hilbert--Poincar\'e series $\sum_{r\geq 0} t^r h_{2r}$, the shift in the length grading arising from the fact that $H_\bullet(KA/\langle \iota \rangle,\Q)$ is instead cogenerated by $\mathrm{Sym}^2(A)$.

(ii) Let $G(t,z)$ be the Taylor series of the analytic functor
$$ A \mapsto H_\bullet(\Omega W(A),\Q).$$
Using \cref{berglund theorem} and \cref{functional equation},  this Taylor series is given by
$$ G(t,z) = \frac{1}{F(\tfrac{t}{z},z)} = \frac{1}{ t z + \sum_{r\geq 0} t^r z^{-r} h_{2r}}.$$
The additional length grading will no longer play a role, and we thus set $t=1$, giving the Taylor series
$$ G(z) = \frac{1}{ z + \sum_{r\geq 0} z^{-r} h_{2r}}.$$

(iii) For a simply connected space $Y$, $H_\bullet(\Omega Y,\Q)$ is the symmetric algebra on $\pi_\bullet^\Q(\Omega Y)$.
It follows that the analytic functor 
$$ A \mapsto \pi_\bullet^\Q\, \Omega W(A)$$
has the Taylor series
$$ \Log \, G(z) = -\Log(z + \sum_{r \geq 0} z^{-r} h_{2r}).$$

(iv) The analytic functor 
$$ A \mapsto \pi_\bullet^\Q \, \Omega^2 W(A)$$
differs from the previous one only by a degree shift, which corresponds to dividing by $-z$ (the minus sign accounting for the Koszul sign rule, in accordance with the discussion in \S\ref{sign rule 1}). Thus, its Taylor series is $$z^{-1}\Log(z + \sum_{r \geq 0} z^{-r} h_{2r}).$$ Now 
$\Log(1+x) = x + \text{(higher order terms)}$, where the higher order terms are all products of Adams operations applied to $x$; hence the above Taylor series is of the form $1 + \text{(higher order terms)}$, with the term $1$ corresponding to the nontriviality of $\pi_0 \, \Omega^2W(A) \cong \Z$. Note that all terms involving $z^{-r}h_{2r}$ contribute only in positive homological degrees, despite the negative powers of $z$: $A$ is concentrated in degree $\geq 2$, so $h_{2r}$ is concentrated in degrees at least $4r$. For the homotopy groups of the base component we therefore see that $A \mapsto \pi_\bullet^\Q \, \Omega_0^2 W(A)$ has Taylor series 
$z^{-1}\Log(z + \sum_{r \geq 0} z^{-r} h_{2r}) - 1,$
so that the analytic functor
$$ A \mapsto H_\bullet(\Omega^2_0 W(A),\Q)$$
has the Taylor series
$$ \Exp(z^{-1} \Log(z + \sum_{r \geq 0} z^{-r} h_{2r}) - 1) $$
as claimed.
\end{proof}

\section{Symmetric function calculations}

\begin{para}
    In this section we derive, by direct computations with symmetric functions, some consequences of the formula proven in \cref{plethystic formula for poincare series},
    \begin{equation}\label{expression for stable homology}
    \sum \dim H_k(H_\infty^{0,1},V_\lambda)(-z)^k s_{\lambda'} = \Exp(z^{-1}\Log(z+\sum_{k \geq 0}z^kh_{2k}) -1- h_2).
    \end{equation}  
\end{para}

\begin{thm}If $k<\vert\lambda\vert/4$ then $H_k(H_\infty^{0,1},V_\lambda)=0$.\label{1/4}\end{thm}

\begin{proof}
We must show that every monomial $z^k s_\lambda$ occurring nontrivially in the right-hand side of \eqref{expression for stable homology} satisfies the inequality $k \geq \vert\lambda\vert/4$. By \cref{plehystic ineq 0} it suffices to prove this for $z^{-1}\Log(z+\sum_{k \geq 0}z^kh_{2k}) -1- h_2$. Now since every monomial $z^k s_\lambda$ occurring nontrivially in the sum $z+\sum_{k>0} z^kh_{2k}$ satisfies the inequality $k<\vert\lambda\vert/2$, the same is true for the monomials occurring in $\Log(z+\sum_{k\geq 0} z^kh_{2k})$,
again by \cref{plehystic ineq 0}. Hence, all terms of arity $\geq 4$ in $z^{-1}\Log(z+\sum_{k \geq 0}z^kh_{2k}) -1- h_2$ must satisfy the desired inequality. To check the terms of low arity it suffices to compute:
$$ z^{-1}\Log(t+\sum_{k\geq 0} z^k h_2k)-1-h_2 = -z  - z h_2 + \text{ higher order terms},$$
where the higher order terms are of arity $\geq 4$ or degree $\geq 2$ with respect to $z$. The result follows.     
\end{proof}

\begin{para}The preceding theorem is sharp. It is interesting to consider the equality case, which is attained by applying $\Exp$ to those terms $z^k s_\lambda$ inside the expansion of 
$z^{-1}\Log(z+\sum z^k h_{2k}) - 1 - h_2$
for which $k=|\lambda|/4$. The argument in the proof also shows that the inequality is strict for terms of arity $> 4$, so it suffices to calculate the low-degree terms, where one finds the unique term $-z s_{2,2}$ which attains the equality. The consequence is that we can write (set $z=1$):
\end{para}

\begin{prop}\label{equality case}
    $\sum_{k \geq 0} \sum_{|\lambda|=4k} \dim H_k(H^{0,1}_\infty,V_\lambda) (-1)^k s_\lambda = \Exp(-s_{2,2}).$ 
\end{prop}

\begin{para}
    An argument similar to the proof of \cref{1/4} shows the following result, which confirms Conjecture 13.11 of \cite{bergstrom09}. As stated in \cite{bergstrom09}, this conjecture was observed and verified computationally for $\vert \lambda \vert \leq 30$ (on the level of Euler characteristics). We remark that \cite[Conjecture 13.11]{bergstrom09} was stated in terms of the moduli space $H_\infty$ of closed hyperelliptic surfaces, rather than $H_\infty^{0,1}$. Combining the below argument with the results of \cref{closed surfaces} confirms the conjecture in its original form.
\end{para}

\begin{thm}\label{jonas vanishing}If $\lambda_1>\vert\lambda\vert/2$ then $H_\bullet(H_\infty^{0,1},V_\lambda)=0$.\end{thm}

\begin{proof}We must show that every Schur polynomial $s_\lambda$ occurring nontrivially in the right-hand side of \eqref{expression for stable homology} satisfies the inequality $\length(\lambda) \leq \vert\lambda\vert/2$. By \cref{plethystic ineq 1} it suffices to prove this for $z^{-1}\Log(z+\sum_{k \geq 0}z^kh_{2k}) -1- h_2$. By another application of \cref{plethystic ineq 1} it suffices to prove this for $z+\sum_{k \geq 0}z^kh_{2k}$. But this is clear.\end{proof}
    
\begin{para}
    The same strategy proves also a further vanishing theorem, \cref{1/2}. To prove \cref{1/2} we first need a bound on $\lambda_1$ for partitions $\lambda$ such that $s_\lambda$ occurs in $\Log(\sum_{k\geq 0} h_{2k})$. We are grateful to Sheila Sundaram who explained the following argument for how to obtain such a bound. 
    %
\end{para}

\begin{thm}[Calderbank--Hanlon--Robinson]\label{prop:chr1} 
	$$ \Log(\sum_{k\geq 0} h_{2k}) = \sum_{m >0} (-1)^m \beta_{2m},$$
	where $\beta_{2m}$ denotes the character of $\Sigma_{2m}$ acting on the top homology group of the poset of partitions of $\{1,\ldots,2m\}$ with all blocks of even size. 
\end{thm}
\begin{para}Indeed, set $k=2$ in \cite[Corollary 4.7]{chr}. See also \cite[Example 1.6(iii)]{Sundaram}. The poset in question is Cohen--Macaulay, so the top homology group is its only nontrivial (reduced) homology group. The following proposition is explained in the start of Section 6 of \cite{chr}. 
\end{para}
\begin{thm}[Calderbank--Hanlon--Robinson]\label{prop:chr2} For $m \geq 2$, 
	$$\frac{\partial \beta_{2m}}{\partial p_1} = s_{\delta_{m+1}/\delta_{m-1}},$$
 where $\delta_m$ denotes the staircase shape $(m-1,m-2,\ldots,1)$ and $\delta_{m+1}/\delta_{m-1}$ denotes a skew diagram.
\end{thm}

\begin{para}\label{para:observation}
A consequence is that if $\lambda \vdash 2m$ is such that $s_\lambda$ occurs with nonzero coefficient in $\beta_{2m}$, then $\lambda_1 \leq m$, except for $\beta_2=s_2$. Indeed, the skew diagram $\delta_{m+1}/\delta_{m-1}$ has exactly $m$ columns, so \cref{prop:chr2} shows that if $s_\mu$ occurs with nonzero coefficient in $\frac\partial{\partial p_1} \beta_{2m}$ then $\mu_1 \leq m$. Pieri's rule therefore shows that if $s_\lambda$ occurs with nonzero coefficient in $\beta_{2m}$ then $\lambda$ satisfies the following condition: all partitions obtained from $\lambda$ by removing a box have their largest entry bounded by $m = \vert\lambda\vert/2$. But now if $\lambda$ has more than one entry, then removing a box from the last entry shows $\lambda_1 \leq m$. If instead $\lambda=(2m)$ has only one entry, then the condition is satisfied only for $\lambda = (2) \vdash 2$. \end{para}

\begin{lem}\label{lem: taylor}
	Let $f, g \in \widehat\Lambda_0^\gr$. Then
	$$ f \circ (z+g) = \sum_{n \geq 0} z^n f^{(n)} \circ g $$
	where $f^{(n)} = h_n^\perp f$. 
\end{lem}
\begin{proof}Note that 
	$f \circ (z+g) = \sum_\lambda (s_\lambda \circ z) (s_\lambda^\perp f \circ g) = \sum_{n \geq 0} z^n (h_n^\perp f) \circ g$, since $s_\lambda \circ z = z^n$ if $\lambda=(n)$, and $s_\lambda \circ z=0$ otherwise.  
\end{proof}

\begin{lem}
	Suppose that $z^j s_\lambda$ occurs with nonzero coefficient in $\Log(z+\sum_{k \geq 0} h_{2k}z^k)$. Then, with the sole exception of the monomial $z h_2$, one has the inequality $j-1 \geq\lambda_1/2$. 
\end{lem}

\begin{proof}Let $L$ be defined as in \cref{defn of L}. We have 
	\begin{align*}
		 \Log(z+\sum_{k \geq 0} h_{2k}z^k) & = L \circ (z + \sum_{k>0} h_{2k}z^k) & \text{by definition of $\Log$}\\  & = \Log(\sum_{k \geq 0} h_{2k}z^k) + \sum_{n >0 } z^n L^{(n)} \circ (\sum_{k>0} h_{2k}z^k) & \text{by \cref{lem: taylor}}\\ &= \sum_{m >0} (-z)^m \beta_{2m} + \sum_{n >0 } z^n L^{(n)} \circ (\sum_{k>0} h_{2k}z^k) & \text{by \cref{prop:chr1}}
	\end{align*}
	The fact that the first term satisfies the conclusion of the lemma is in \S\ref{para:observation}. As for the second, since every monomial $z^j s_\lambda$ occurring in $\sum_{k>0} h_{2k}z^k$ satisfies the inequality $j \geq \lambda_1/2$, the same will be true for $L^{(n)} \circ (\sum_{k>0} h_{2k}z^k)$, by \cref{plethystic ineq 2}. 
\end{proof}

\begin{thm}If $k<\length(\lambda)/2$ then $H_k(H_\infty^{0,1},V_\lambda)=0$.\label{1/2} \end{thm}

\begin{proof} The theorem is equivalent to the assertion that if $z^k s_\lambda$ occurs with nonzero coefficient in the right-hand side of \eqref{expression for stable homology}, then $k \geq \lambda_1/2$. By \cref{plethystic ineq 2} it suffices to prove this inequality for monomials occurring in $z^{-1}\Log(z+\sum_{k\geq 0} h_{2k}z^k) -1 - h_2$. But this is exactly the preceding lemma. 
\end{proof}

\begin{rem}Theorems \ref{jonas vanishing} and \ref{1/2} are both sharp. One may for example deduce this from \cref{equality case} and plethystic formulas of Langley--Remmel \cite{langley-remmel}, using that $$\Exp(-s_{2,2}) = \sum_{k \geq 0} (-1)^k s_{1^k} \circ s_{2,2}.$$ Indeed, \cite[Theorem 4.5(2)]{langley-remmel} implies that $H_k(H_\infty^{0,1},V_\lambda) \cong \Q$ for $\lambda = (2k,k+1,1^{k-1})$. By \cite[Chapter 1, \S 8, Example 1(a)]{macdonald} the element $\Exp(-s_{2,2})$ is invariant under the involution $\omega$ of the ring of symmetric functions, so we see also that $H_k(H_\infty^{0,1},V_{\lambda'}) \cong \Q$, where $\lambda' = (k+1,2^k,1^{k-1})$.
\end{rem}

\begin{para}Let us now rewrite the generating series \eqref{expression for stable homology} in an alternative form. \end{para}

\begin{prop}\label{alternative form of gen series}Let $i_n(t)=\frac 1n \sum_{d \mid n} \mu(n/d)t^d$ denote the $n$th necklace polynomial. Then 
    \[\Exp(z^{-1}\Log(z+\sum_{k\geq 0}z^k h_{2k})-1) = (1-z) \prod_{n=1}^\infty (1 + \frac{1}{1+z^n} \sum_{k>0}  z^{nk}\psi_n(h_{2k}))^{i_n(z^{-1})} \]
\end{prop}
\begin{proof}Note first that $z^{-1}\Log(z+\sum_{k\geq0} z^kh_{2k})-1 = z^{-1}\Log(1-z^2 + (1-z)\sum_{k>0}z^kh_{2k})$. Indeed, multiply by $z$ and exponentiate, using $\Exp(-z)=1-z$. Using this identity and \cref{lemma:pleth id}, we get
\[\Exp(z^{-1}\Log(z+\sum_{k\geq 0}z^k h_{2k})-1) =\prod_{n=1}^\infty (1-z^{2n} + (1-z^n)\sum_{k>0} {z^{nk}} \psi_n(h_{2k}))^{i_n(z^{-1})}.\]
Combine this with the cyclotomic identity $(1-\alpha s)=\prod_{n=1}^\infty(1-s^n)^{i_n(\alpha)}$, with $\alpha=z^{-1}$ and $s=z^2$, and the result follows.  
\end{proof}

\begin{thm}\label{polynomial growth} For any fixed $\lambda$, the sum $\sum \dim H_k(H_{\infty}^{0,1},S^\lambda(V))(-z)^k$ is a rational function of $z$ with all poles on the unit circle. In particular, the Betti numbers grow as a polynomial in $k$. Moreover, this polynomial is of degree exactly $\vert\lambda\vert/2-1$, unless $\lambda_1 >\vert\lambda\vert/2$, in which case the homology vanishes by \cref{jonas vanishing}. 
\end{thm}

\begin{proof}Expanding the infinite product appearing in \cref{alternative form of gen series} gives a sum of terms, each coefficient of which is a rational function whose denominator is a product of factors $(1+z^n)$. The term with the largest number of such factors in the denominator occuring in arity $2N$ is 
$$ \frac{1}{(1+z)^{N}}(zh_2)^N \binom {z^{-1}}{N}.$$
A Schur polynomial $s_\lambda$ with $\lambda \vdash 2N$ occurs in the expansion of $h_2^N$ if and only if $\length(\lambda)\leq N$.
\end{proof}

\begin{rem}
    The polynomial growth of stable cohomology may be contrasted with the case of $M_g^1$. Indeed, already for constant coefficients we have that $H^\bullet(M_\infty^1,\Q)$ is a polynomial algebra with a generator in each even degree; its Betti numbers grow like the number-theoretic partition function. The stable cohomology of $M_g^1$ with coefficients in any local system $V_\lambda$ is a free module over the cohomology with constant coefficients, and grows similarly.  
\end{rem}

\section{Framed \texorpdfstring{$E_2$}{E2}-structure and logarithmic geometry}

\subsection{Motivic framed \texorpdfstring{$E_2$}{E2}-algebra structure}

\begin{para}Let $q$ be a prime power, and $\ell$ an auxilliary prime not dividing $q$. 
\end{para}\begin{para}
We denote by $\M_g^n$ the algebro-geometric moduli space of genus $g$ curves with $n$ distinct ordered marked points and a nonzero tangent vector at each marking. The space $\M_g^n$ is a smooth stack of dimension $3g-3+2n$. For example, $\M_0^1$ is the Artin stack $B\mathbb G_a$. If $g \geq 1$ and $n \geq 1$, then $\M_g^n \otimes \Q$ is a smooth scheme. It is a scheme, not a stack, since in characteristic zero no automorphism of a smooth curve of positive genus can fix a point and a nonzero tangent vector at that point.  \end{para}

\begin{para}Recall that we defined $M_g^n$ as the classifying space $B\Diff_\partial(S_{g,n})$ of an oriented genus $g$ surface with $n$ boundary components. 
	There is a homotopy equivalence $M_g^n \simeq \M_g^n(\C) $, where by $\M_g^n(\C)$ we mean the analytification of $\M_g^n$. By Artin's comparison theorem, there is therefore an isomorphism 
	$$  H_\bullet^\sing(M_g^n,\Z)\otimes\Q_\ell \cong H_\bullet^\et(\M_g^n \otimes \overline \Q, \Q_\ell)$$ 
	whereby the (co)homology of $M_g^n$ carries the structure of an $\ell$-adic Galois representation. If $2g-2+n>0$, then $\M_g^n$  is the complement of a normal crossing divisor in a smooth proper stack over $\Z$, namely the projective bundle over the Deligne--Mumford compactification $\smash{\overline{\M}_{g,n}}$ compactifying the total space of the direct sum of the $n$ tangent line bundles over $\smash{\overline{\M}_{g,n}}$. Consequently we have comparison isomorphisms 
	$$ H_\bullet^\et(\M_g^n \otimes \overline \Q, \Q_\ell) \cong H_\bullet^\et(\M_g^n \otimes \Fqbar, \Q_\ell) $$
	between cohomology in characteristic zero and characteristic $p$. The result holds also in the exceptional cases where $2g-2+n \leq 0$ by direct calculation; e.g.\ using that the cohomology of $B\mathbb G_a \cong \M_0^1$ is trivial. 
	\end{para}
\begin{para}\label{ag local systems}
	We also have algebro-geometric avatars of the local systems $S^\lambda(V)$ considered on the topological moduli spaces $M_g^n$ (\S\ref{definition of V}, \S\ref{standard coefficient system}). These local systems are naturally pulled back from the space $\M_{g,n}$ of $n$-pointed curves (with no tangent vectors at markings). Suppose that $n>0$. Let $p:\mathcal C \to \M_{g,n}$ be the universal $n$-punctured curve, and define $\V=R^1p_!\Q_\ell$. (We suppress the dependence of $\V$ on $\ell$ for legibility.) Then $\V$ is a local system of rank $2g+n-1$ on $\M_{g,n}$, fitting in a short exact sequence
	$$ 0 \to \Q_\ell^{\oplus (n-1)} \to \V \to \V^0 \to 0 $$
	where $\V^0$ is the pure sheaf of weight $1$ defined as $R^1 \overline p_! \Q_\ell$, where $\overline p : \overline{\mathcal{C}}\to \M_{g,n}$ is the universal curve. The corresponding Betti local system over the complex numbers underlies an admissible variation of mixed Hodge structures, and may be identified with the local system $V$ under the homotopy equivalence $M_g^n \cong \M_g^n(\C)$. There are similar comparison isomorphisms
	$$ H_\bullet^\sing(M_g^1,S^\lambda(V) \otimes \Q_\ell) \cong H_\bullet^\et(\M_g^1 \otimes \overline\Q,S^\lambda(\V)) \cong H_\bullet^\et(\M_g^1 \otimes \Fqbar,S^\lambda(\V))$$
	by the same arguments as in the constant coefficient case, with the additional input into the comparison theorem that the sheaf $\V$ is tamely ramified along the divisor at infinity. (See \cite[Section 3]{bergstromvandergeer} for a detailed argument.) As in \S\ref{standard coefficient system} one may want to modify the above definition of $\V$ in case $n=0$: a definition that works for any $n \geq 0$ is to let $\V$ denote the fiber of the trace map $Rp_!\Q_\ell[1] \to \Q_\ell(-1)[-1]$. However, we will have no reason to consider the case $n=0$ in this section. 
\end{para}

\begin{para}Using that $\coprod_{g \geq 0} M_g^1$ is naturally a framed $E_2$-algebra (\S\ref{desiderata}), there is an induced structure of Batalin--Vilkovisky (BV) algebra on its homology
	$$ \bigoplus_{g \geq 0} H_\bullet(M_g^1,\Q),$$
	since the homology of the framed little disk operad is the operad of BV algebras \cite{getzlerbv}. Similarly, $\coprod_{g \geq 0} M_g^1(KA)$ is a framed $E_2$-algebra.  By the calculation of its homology (\cref{first expression}), or rather the homology of the base component of each space $M_g^1(KA)$, we see that there is a natural BV-algebra structure on 
	$$ \bigoplus_{g \geq 0} \bigoplus_{\lambda} S^\lambda(A) \otimes H_\bullet(M_g^1,S^\lambda(V[-1])).$$
	(Let us suppose as in \cref{first expression} that $A$ is a graded vector space concentrated in degree $2$.) We now want to explain that this BV-structure is ``motivic''. More precisely, we have the following theorem:
\end{para}

\begin{thm}\label{motivic theorem}
	Under the comparison isomorphism
	$$ \bigoplus_{g \geq 0} \bigoplus_{\lambda} S^\lambda(A) \otimes H_\bullet^\sing(M_g^1,S^\lambda(V[-1]) \otimes \Q_\ell) \cong  \bigoplus_{g \geq 0} \bigoplus_{\lambda} S^\lambda(A) \otimes H_\bullet^\et(\M_g^1 \otimes \overline\Q,S^\lambda(\V[-1])) $$
	the natural BV-algebra structure on the left-hand side is compatible with the Galois action on the right-hand side, when the homology of the framed disk operad $FD$ is given the Galois action which on $H_k(FD(n),\Q_\ell)$ is given by the $k$th power of the cyclotomic character, i.e.\ a pure Tate motive of weight $-2k$. 
\end{thm}

\begin{rem}
	There is for each $n$ a homotopy equivalence $FD(n) \cong \M_0^{n+1}(\C)$. On the face of it, this equivalence has no compatibility at all with operadic composition. Nevertheless, it produces the ``correct'' Galois action according to the theorem, as $H_k^\et(\M_0^{n+1},\Q_\ell)$ is a direct sum of copies of $\Q_\ell(k)$. Indeed: $\M_0^{n+1} \cong \mathbb G_m^{n} \times \mathrm{PConf}_n(\mathbb A^1)/\mathbb G_a$ is isomorphic to the complement of an affine hyperplane arrangement, and complements of hyperplane arrangements have this purity property of their cohomology by \cite{lehrerhyperplane,kimhyperplane,shapirohyperplane}.
\end{rem}

\begin{rem}
	A consequence of the theorem is that the transcendentally defined Harer stability map $H_\bullet(M_g^1,S^\lambda(V)) \to H_\bullet(M_{g+1}^1,S^\lambda(V))$ is compatible with Galois actions, since the Harer stabilization is multiplication in this BV-algebra by the fundamental class in $H_0(M_1^1)$. This was previously known by \cite[Section 4]{hainlooijenga}, see also \cite{levinetubular}.
\end{rem}
\begin{rem}
    This result is closely related to conjectures of Cazanave \cite[Section 3.5]{cazanave}. It seems very likely that these conjectures, at least with ``little disks'' replaced by ``framed little disks'', can be proved by similar reasoning.
\end{rem}

\begin{para}\cref{motivic theorem} will be deduced as a special case of a more general construction due to Vaintrob \cite{vaintrob,vaintrob2}. As explained in \S\ref{desiderata}, the framed $E_2$-structure on $\coprod_{g \geq 0} M_g^1$ can be seen as a consequence of the modular operad structure on the collection of spaces $\{M_g^n\}$, and the fact that $\{M_0^n\}$ is a model for the cyclic operad of framed little disks. Vaintrob shows that there is a model for the surface operad $\{M_g^n\}$ in terms of \emph{logarithmic algebraic geometry} over $\Spec \Z$. The model in question was first considered by Kimura--Stasheff--Voronov \cite{kimurastasheffvoronov2} (in genus zero, but the higher genus analogue is an immediate generalization). Kimura--Stasheff--Voronov described their spaces explicitly as real oriented blow-ups of complex algebraic varieties; the insight of Vaintrob is that these real oriented blow-ups may be realized explicitly as analytifications (which in the logarithmic setting are usually referred to as \emph{Kato--Nakayama spaces}) of log schemes over $\Z$ (or log stacks when $g>0$), from which it in particular follows that their homologies carry canonically a ``motivic enhancement'': mixed Hodge structure, $\ell$-adic Galois representation, et cetera. 
\end{para}

\begin{para}\label{salvatore}
    	The construction of Kimura--Stasheff--Voronov discussed in the preceding paragraph is an analogue, for the cyclic operad of two-dimensional framed little disks, of the \emph{Fulton--MacPherson model} of the usual operad of two-dimensional little disks. The latter was introduced by Kontsevich \cite{kontsevichfeynman}. However, a detailed proof that the Fulton--MacPherson model has the homotopy type of the usual little disk operad did not appear until Salvatore \cite[Proposition 3.9]{salvatore}. In particular, Kimura--Stasheff--Voronov do not prove a comparison between their operad and the usual framed disk operad, either, but the result can be proven exactly as in Salvatore, using a variant of the $W$-construction for cyclic operads. In higher genus one needs similarly a version of the $W$-construction for modular operads, which does not seem to have been described in the literature. An example: if $\mathcal O$ is the topological modular operad such that $\mathcal O(g,n)$ is a point for all $(g,n)$, then $W\mathcal O(g,n) \cong {\overline{\mathcal{M}}_{g,n}^{\mathrm{trop}}}$, the moduli space of stable tropical curves. 
\end{para}

\begin{para}
    An easier way to see that the Kimura--Stasheff--Voronov construction models the modular operad of surfaces than the one indicated in \S\ref{salvatore} is to pass through the stack of surface bundles (\S\ref{stack of surface bundles}). The universal curve over Deligne--Mumford space gives rise to a surface bundle over the Kimura--Stasheff--Voronov space by real blow-up, and hence a morphism to the stack $\mathfrak M_g^n$. These maps assemble to a weak equivalence of modular operads. The point is that the stack $\mathfrak M_g^n$ is by construction easy to map into. Similarly, in genus zero, the components of the usual framed little disk space operad are also equipped with tautological surface bundles, defining a map to $\mathfrak M_0^n$, which gives a zig-zag of weak equivalences.  
\end{para}

\begin{para}
	Since we do not expect all readers to be experts in logarithmic geometry, we include here a bare-bones approach to the theory sufficient for our purposes. 
\end{para}

\subsection{Deligne--Faltings log structures}
\begin{para}Foundations for the theory of logarithmic algebraic geometry were first given by Kato \cite{katologarithmic}, building on ideas of Fontaine and Illusie. A closely related formalism was independently developed by Faltings \cite{faltingslogarithmic}, and a set-up equivalent to Faltings's was also proposed in a letter from Deligne to Illusie. The definition of Deligne--Faltings is somewhat simpler than Kato's: a log scheme in the sense of Deligne--Faltings is simply a scheme together with a finite collection of line bundles with sections. Kato refers to this notion as a \emph{DF log scheme} \cite[Complement 1]{katologarithmic}. We will use a strictly more general notion of Deligne--Faltings logarithmic structure, introduced by Borne and Vistoli \cite{bornevistoli}. We now recall their definition; for simplicity, we state it in terms of sheaves of categories instead of fibered categories. 
\end{para}

\begin{defn}\label{div stack}
Let $X$ be a scheme. Let $\mathfrak{Div}_X$ be the following \'etale sheaf of symmetric monoidal groupoids on $X$: its value on $U$ is the groupoid of pairs $(L,s)$ with $L$ a line bundle on $U$, and $s$ a section, with the symmetric monoidal structure given by tensor product of line bundles.   
\end{defn}

\begin{defn}\label{BV log}
    Let $X$ be a scheme. A \emph{pre-DF log structure on $X$} is an \'etale presheaf $A$ of commutative monoids on $X$, considered as a presheaf of symmetric monoidal groupoids with only identity morphisms, together with a symmetric monoidal functor $\phi:A \to \mathfrak{Div}_X$. A pre-DF log structure is said to be a \emph{DF log structure} if $A$ is a sheaf, and $\ker \phi$ is trivial. 
\end{defn}

\begin{para}\label{sheafification}
    Every pre-DF log structure gives rise to a DF log structure, by replacing $A$ with the sheafification of $A/\ker \phi$ \cite[Proposition 3.3]{bornevistoli}. 
\end{para}

\begin{para}
    The collection of all pre-DF log structures on $X$ form a category: a morphism from $(A,\phi)$ to $(A',\phi')$ is a diagram
	$$ \begin{tikzcd}
		 A\arrow[dr,"\phi"']\arrow{rr}&\arrow[d, Rightarrow]& A' \arrow[dl,"\phi'"]\\
		& \mathfrak{Div}_X&
	\end{tikzcd}$$
	of strong symmetric monoidal functors commuting up to a specified natural transformation. DF log structures form a full subcategory of pre-DF log structures.
\end{para}

\begin{para}
A morphism of schemes $f\colon X \to Y$ induces a functor from the category of DF log structures on $Y$ to the category of DF log structures on $X$, which we write as $(A,\phi) \mapsto (f^{\ast}A,f^\ast \phi)$ \cite[Proposition 3.9]{bornevistoli}. The assignment $A \mapsto f^{\ast}A$ is the usual inverse image functor of \'etale sheaves. 
\end{para}\begin{defn}
A \emph{DF log scheme} is a scheme equipped with a DF log structure. We typically write DF log schemes in a sans-serif font, so that $\mathsf X$ might denote a DF log scheme and $X$ its underlying scheme. A morphism of DF log schemes $\mathsf X = (X,A,\phi) \to \mathsf Y = (Y,B,\psi)$ is a morphism of underlying schemes $f: X \to Y$ and a morphism of DF log structures $(f^{\ast}B,f^*\psi) \to (A,\phi)$. A morphism of DF log schemes is said to be \emph{strict} if $(f^{\ast}B,f^*\psi) \to (A,\phi)$ is an isomorphism of DF log structures. 
\end{defn}\begin{rem}
We will freely use also the evident generalization of a \emph{DF log stack}.
\end{rem}

\begin{para}
    A DF log structure in the sense of \cite[Complement 1]{katologarithmic} is a $k$-tuple $(\sigma_i \colon \mathcal O_X \to L_i)_{i=1}^k$ of line bundles on $X$ with sections, for some $k \geq 0$; let us call this a Kato-DF log structure. Every Kato-DF log structure gives rise to a pre-DF log structure in an evident way, for which $A$ is the constant presheaf $\N^k$. A morphism of pre-DF log structures between Kato-DF log structures $(\sigma_i \colon \mathcal O_X \to L_i)_{i=1}^k$ and $(\tau_j \colon \mathcal O_X \to M_j)_{j=1}^l$ then becomes a collection of nonnegative integers $e_{ij}$, $1 \leq i \leq k$, $1 \leq j \leq l$, together with isomorphisms
$$ L_i \cong \bigotimes_{j=1}^l M_j^{\otimes e_{ij}}$$
under which the section $\sigma_i$ is identified with the corresponding tensor product of the sections $\tau_j$. We will sometimes find it convenient to describe a DF log structure simply by specifying a collection of line bundles with sections.
\end{para}

\begin{rem}One reason why \cref{BV log} is preferable to the notion of a Kato-DF log structure is that the functor from Kato-DF log schemes to log schemes in the ordinary sense, is not faithful: for example, the Kato-DF log structure associated to the trivial line bundle and its unit section is not isomorphic to the trivial Kato-DF log structure. They are also not isomorphic as pre-DF log structures, but their associated DF log structures are isomorphic. The category of DF log schemes is equivalent to the category of \emph{quasi-integral} log schemes, considered as a full subcategory of the usual category of log schemes \cite[Theorem 3.6]{bornevistoli}.
\end{rem}

\begin{para}All log schemes in this paper will be naturally presented by Deligne--Faltings data. From now on, we omit the modifier ``DF'', so that the phrase ``log scheme'' or ``pre-log scheme'' will always be taken to mean ``DF log scheme'' or ``pre-DF log scheme''.
\end{para}

\begin{defn} \label{divisorial log structure}Let $X$ be a smooth  variety or Deligne--Mumford stack, and $D \subset X$ a normal crossing divisor. Let $f: D^\nu \to X$ be the natural map from the normalization of $D$, and let $A$ be the \'etale sheaf of commutative monoids on $X$ given by $f_*\N$, where $\N$ is the constant sheaf associated to the natural numbers. Explicitly, $A(U)$ is the set of $\N$-valued functions on the set of irreducible components of $D\times_X U$. There is then a natural functor $A \to \mathfrak{Div}_X$, which takes a generator corresponding to an irreducible component $D_i$ to the line bundle $\mathcal O(D_i)$ with its natural section. We call the associated log structure the \emph{divisorial log structure}. If $D$ has simple normal crossings, the divisorial log structure is naturally presented by a tuple of line bundles with sections, one for each component of the divisor. \end{defn}

\begin{prop}\label{divisorial prop}
	Let $X$ be a smooth Deligne--Mumford stack, and $D \subset X$ a normal crossing divisor. Let $f:D^\nu \to X$ be the natural map from the normalization of $D$, and let $E \subset D^\nu$ be the inverse image of the locus where $D$ has at least $2$ local branches, so that $E$ is a normal crossing divisor in $D^\nu$. The pullback to $D^\nu$ of the divisorial log structure on $X$ coincides with the log structure on $D^\nu$ given by the divisorial log structure associated to $E$, together with the zero-section of the normal bundle of $f$.
\end{prop}

\begin{proof}
	Let $A_D$ and $A_{E}$ be the sheaves of monoids on $X$ and $D^\nu$ defining the respective divisorial log structures. By proper base change, $f^*A_D = f^* f_*\N$ is the pushforward of the constant sheaf $\N$ from $D^\nu \times_X D^\nu \cong D^\nu \sqcup E^\nu$ to $D^\nu$, so that there is an isomorphism
    \[ f^* A_D \cong \N \oplus A_{E}.\] The map $f^*A_D \to \mathfrak{Div}_{D^\nu}$ restricts to the usual divisorial log structure on $A_{E}$. The effect on the summand $\N$ follows because the pullback along $f$ of the line bundle $\mathcal O_X(D)$ with its natural section $\mathcal O_X \to \mathcal O_X(D)$ is precisely the normal bundle of $f$, and the zero-section. 
\end{proof}

\subsection{Kato-Nakayama space}
\begin{para}
Let $\mathsf X$ be a finite type log scheme over $\C$. Kato and Nakayama \cite{katonakayama} have associated a topological space $\mathsf X^{\KN}$ to $\mathsf X$, the \emph{Kato--Nakayama space}, which plays the same role in logarithmic algebraic geometry as the usual analytification $X(\C)$ of a finite type scheme $X$ over $\C$. In particular, Betti cohomology of a log scheme is defined as the cohomology of its Kato--Nakayama space. We will briefly recall the construction and basic properties of the Kato--Nakayama space, following Talpo and Vistoli \cite{talpovistoli}. The treatment of Talpo--Vistoli has the advantage of being formulated directly in the language of Deligne--Faltings log structures on log stacks; the Kato--Nakayama space of a log stack is a topological stack. 
\end{para}
\newcommand{\CC}{\smash{\widehat{\C}}}
\begin{para}
    Let $\CC \to \C$ be the real blow-up of the complex plane at the origin. Explicitly, $\CC = U(1) \times \R_{\geq 0}$, and the map $\CC \to \C$ is given by $(z,t) \mapsto zt$, so the fiber over the origin is a circle and all other fibers are points. The complex torus $\C^\times$ acts on both $\CC$ and $\C$, and the projection is equivariant. 
   \end{para} 

   \begin{para}
       The algebraic stack $[\mathbb A^1 /\mathbb G_m]$ parametrizes pairs $(L,s)$ of a line bundle and a section. In particular, in terms of \cref{div stack}, $\mathfrak{Div}_X$ is the sheaf on $X$ given by $\mathrm{Hom}(-,[\mathbb A^1 /\mathbb G_m])$. Similarly the topological stack $[\C/\C^\times]$ parametrizes topological complex line bundles with a section.
   \end{para}

   \begin{defn}\label{talpo vistoli definition}
       Let $\mathsf X  = (X,A,\phi)$ be a log stack. Its \emph{Kato--Nakayama space} $\mathsf X^\KN$ is the topological stack whose value on a topological space $S$ is the groupoid of all pairs $(f,a)$, where $f : S \to X(\C)$ is a continuous map, and $a$ is a lifting in the following diagram of sheaves on $S$:
       \[ \begin{tikzcd}
       f^* A \arrow[r,"f^*\phi"]\arrow[rd,dashed,"a"']& \mathrm{Hom}(-,[\C/\C^\times]) \\
       & \arrow[u] \mathrm{Hom}(-,[\CC/\C^\times]), \end{tikzcd}\]
       The map $[\CC/\C^\times] \to [\C/\C^\times]$ is induced by the real blow-up $\CC \to \C$.
   \end{defn}

   \begin{para}\label{properties of KN}
       The Kato--Nakayama space has the following basic properties:
       \begin{itemize}
           \item If $\mathsf X$ has underlying stack $X$, then $\mathsf X^\KN$ maps naturally to $X(\C)$. 
           \item If $\mathsf X \to \mathsf Y$ is a \emph{strict} morphism of log stacks, with underlying map of stacks $X\to Y$, then the diagram
           \[ \begin{tikzcd}
               \mathsf X^\KN \arrow[r]\arrow[d]& \arrow[d]\mathsf Y^\KN \\
               X(\C) \arrow[r]& Y(\C)
           \end{tikzcd}\]
           is cartesian.
           \item The Kato--Nakayama functor is a sheaf for the \'etale topology.
       \end{itemize}
   \end{para}

   \begin{construction}\label{blowup construction}
       Consider a complex line bundle $E \to B$, say over a topological space $B$. The structure group of $E$ is $\C^\times$, which acts naturally on $\CC$. Let $\smash{\widehat E}\to B$ denote the associated bundle with fiber $\CC$. The natural map $\CC \to \C$ induces a map $\smash{\widehat E} \to E$, and $\smash{\widehat E}$ is naturally the real blow-up of the total space of the line bundle $E$ along its zero-section. 
   \end{construction}
    
    \begin{para}Let us now give a more concrete description of the Kato--Nakayama space, when the log structure is given by a family of line bundles with sections. Consider first a log stack $\mathsf X$ defined by a pair $(L,s)$ of a line bundle and a section on a complex scheme $X$. We claim that $\mathsf X^{\KN}$ is the inverse image of $s(X)$, thought of as a subset in the total space of the line bundle $L$ over $X(\C)$, under the projection $\smash{\widehat L} \to L$ of \cref{blowup construction}. Indeed, since both the construction in this paragraph and the one of \cref{talpo vistoli definition} satisfies the second property of \S\ref{properties of KN}, it suffices to show this in the universal example of a stack with a line bundle and a section, namely when $\mathsf X$ is $[\mathbb A^1/\mathbb G_m]$ with its tautological log structure. In this case the definition of \cref{talpo vistoli definition} unwinds to $\mathsf X^\KN = [\CC/\C^\times]$. For the definition of this paragraph, we note that total space of the universal line bundle over $[\C/\C^\times]$ is $[\C^2/\C^\times]$, and the universal section is the diagonal embedding of $\C$ in $\C^2$. The inverse image of the diagonal in $[\C \times \CC/\C^\times]$ is $[\CC/\C^\times]$, as claimed.  More generally, given a log stack $\mathsf X$ defined by a finite collection of line-bundles-with-sections $L_1,\dots,L_k$ and $s_1,\dots,s_k$ on an underlying stack $X$, the space $\mathsf X^\KN$ can be identified with the inverse image of $X$, thought of as a subspace of the total space of $L_1 \times_X \dots \times_X L_k$ via the given sections, inside the blow-up $\smash{\widehat L_1 \times_X \dots \times_X \widehat L_k }$.
\end{para}

\begin{examplex}\label{example1}Let $L \to X$ be a  line bundle on a complex variety, and let $\mathsf X$ denote the log scheme defined by the zero section of this bundle. Then $\mathsf X^\KN$ is diffeomorphic to the unit circle bundle over $X(\C)$ with respect to some bundle metric on $L$. \end{examplex}

\begin{examplex}\label{example2} Let $X$ be a smooth complex variety, $D \subset X$ a smooth divisor. Let $\mathsf X$ be the log scheme defined by the divisorial log structure on $X$ (\cref{divisorial log structure}). Then $\mathsf X^\KN$ is a manifold with boundary, diffeomorphic to the complement of a tubular neighborhood of $D$ in $X$. Moreover, by \cref{divisorial prop} there is a map of log schemes $\mathsf D \to \mathsf X$, where $\mathsf D$ is the log structure on $D$ defined by the normal bundle and its zero section. On Kato--Nakayama spaces, $\mathsf D^\KN \to \mathsf X^\KN$ is the inclusion of the boundary in the manifold-with-boundary $\mathsf X^\KN$. \end{examplex}

\begin{examplex} \label{example4}More generally, let $X$ be a smooth complex variety, $D \subset X$ a normal crossing divisor, and $\mathsf X$ the log scheme given by the divisorial log structure on $X$. The Kato--Nakayama space $\mathsf X^\KN$ is naturally a manifold with corners,  diffeomorphic to the complement of an $\epsilon$-neighborhood of $D$. If $x$ is a point where $D$ has $r$ local branches, then the fiber of $\mathsf X^\KN \to X(\C)$ over $x$ is a product of $r$ circles. \end{examplex}

\subsection{(Co)homology of log schemes}

\begin{para}
	Having defined the analytification of a complex log scheme, we in particular have a singular (or ``Betti'') cohomology theory defined by $H^\bullet_\sing(\mathsf X^\KN,\Z)$. It is then natural to ask whether we can also define \'etale, algebraic de Rham, or crystalline cohomology of log schemes, with expected comparison isomorphisms, whether the Betti cohomology underlies a mixed Hodge structure, and so on and so forth. We will not discuss the theory of log motives (although see e.g.\ \cite{shuklin,bindaparkostvaer}); what will be relevant for us in this paper is simply having a well-behaved theory of \'etale cohomology. 
\end{para}

\begin{para}The theory of logarithmic \'etale cohomology is largely parallel to the classical theory. The role of the \'etale site is played by the \emph{Kummer-\'etale} site.\footnote{In fact, there are two common variants of logarithmic \'etale topology: the Kummer-\'etale site, and the \emph{full log \'etale site}. For the purposes of this paper, either one could be used, by \cite[Theorem 9.4]{illusieoverview}.} We recommend \cite{illusieoverview} and the original references \cite{katonakayama,nakayama}. We will not need the precise definition of the Kummer-\'etale site in this paper; the only facts we need are the comparison isomorphisms stated below. Let us now explain how these \'etale cohomology groups behave in the simplest examples. The cases of interest to us will be no more complicated. 
\end{para}

\begin{para}\label{cohomology of log schemes in elementary terms} Let us consider the situations of Examples \ref{example1}--\ref{example4}. In Example \ref{example4}, we may consider the morphism $(X \setminus D) \hookrightarrow \mathsf X$, where $X \setminus D$ has the trivial log structure. This map induces a homotopy equivalence of Kato--Nakayama spaces, and for any reasonable cohomology theory one should have $H^\bullet(X\setminus D) \cong H^\bullet(\mathsf X)$, which uniquely ``pins down'' the motivic structure of $H^\bullet(\mathsf X)$. Similarly, in the setting of Example \ref{example1} we must have $H^\bullet(\mathsf X) \cong H^\bullet(L^\ast)$, where $L^\ast$ denotes the complement of the zero-section in the total space of $L$, equipped with the trivial log structure: indeed, this follows from the previous discussion and the fact that $\mathsf X$ includes into the total space of $L$, equipped with the log structure induced from the divisor given by the zero-section, such that the map on Kato--Nakayama spaces is a homotopy equivalence. More generally, if the log structure on $\mathsf X$ is defined by a normal crossing divisor $D$ and an $n$-tuple of line bundles $L_1,\dots,L_n$ with their zero sections, then $H^\bullet(\mathsf X)$ is isomorphic as a motive to the cohomology of the fibered product $L_1^\ast \times_{X\setminus D} L_2^\ast \times_{X\setminus D} \dots \times_{X\setminus D} L_n^\ast$, with $L^\ast_i$ denoting the complement of the zero-section inside the total space of $L_i$ restricted to $X \setminus D$. 
\end{para}

\begin{para}
	There is also a good theory of local systems on log schemes, and a comparison with locally constant sheaves on the Kummer-\'etale site. When $\mathsf X$ is a log scheme defined by a normal crossing divisor $D$ on a smooth variety $X$, then local systems on $\mathsf X$ are equivalent to local systems on $X \setminus D$ that are tamely ramified along $D$. When $\mathsf X$ is defined by the zero-section of a line bundle $L$, then local systems on $\mathsf X$ are equivalent to local systems on $L^\ast$ that are tamely ramified along the zero-section. 
\end{para}

\begin{para}
    Given \S\ref{cohomology of log schemes in elementary terms}, a reader may reasonably wonder why we bother introducing the formalism of log schemes at all: evidently, all the cohomology groups of log schemes which will concern us can be straightforwardly expressed as cohomologies of ordinary schemes. The answer is: yes, we could in principle have formulated our arguments without ever invoking logarithmic geometry, but only at the expense of expending more effort arguing that certain diagrams commute. Consider for example a log scheme  $\mathsf X$, which is smooth and proper over $\operatorname{Spec} \Z$ in the logarithmic sense. For such $\mathsf X$ we have comparison isomorphisms \[ H^\bullet_{\sing}(\mathsf{X}^{\mathrm{KN}}) \cong H^\bullet_{\et}(\mathsf{X}\otimes \overline\Q) \cong H^\bullet_{\et}(\mathsf{X}\otimes \Fqbar)\]
	relating singular cohomology, \'etale cohomology in characteristic zero, and \'etale cohomology in characteristic $p$ (say with $\Z_\ell$-coefficients), by \cite{katonakayama} and \cite[Theorem 9.9]{illusieoverview}, respectively. The logarithmic formalism ensures moreover that these comparison isomorphisms are functorial with respect to morphisms $\mathsf X \to \mathsf Y$ of log schemes; by contrast, if we express $H^\bullet(\mathsf X)$ and $H^\bullet(\mathsf Y)$ as the cohomologies of ordinary schemes $X'$ and $Y'$ then typically there is no natural morphism $X' \to Y'$ inducing the natural map on cohomology, and hence functoriality is not immediate from general principles. As we explain shortly, the formalism will be applied to certain gluing morphisms between log stacks, from which the relevant operadic structures and homological stability problems arise.
\end{para}

\subsection{Log models of moduli spaces of surfaces with boundary (after D.~Vaintrob)}

\newcommand{\fM}{\mathsf{M}}
 \begin{para}\label{definition of framed moduli}Suppose that $2g-2+n>0$. Let $\fM_{g}^{n}$ be the following log stack: the underlying stack is $\overline \M_{g,n}$, and the log structure is the divisorial log structure associated to the complement of $\M_{g,n}$ in $\overline \M_{g,n}$, {together with} the zero sections of the $n$ tangent line bundles $\mathbb T_i \to \overline \M_{g,n}$, whose fiber over a moduli point $[C,x_1,\ldots,x_n]$ is the tangent space of $C$ at the point $x_i$. (Thus $\mathbb T_i$ is the dual of the line bundle which defines the usual $\psi$-class $\psi_i$.) Following Vaintrob we call $\fM_g^n$ the moduli of \emph{framed log curves}. 
\end{para}

\begin{para}\label{gluing}
Consider the gluing map $f : \overline \M_{g,n+1} \times \overline \M_{g',n'+1} \to \overline \M_{g+g',n+n'}$. The pullback of the divisorial log structure on the target is the product of the divisorial log structure on the source with the log structure associated to the zero-section of the normal bundle of $f$, by \cref{divisorial prop}. Now, the normal bundle of $f$ is the tensor product of the {tangent} line bundles at the two points being glued together \cite[Chapter XIII, \S 3]{acg}. It follows that the gluing is naturally compatible with the log structure defined in \S\ref{definition of framed moduli}, producing a map
$$\fM_{g}^{n+1} \times \fM_{g'}^{n'+1} \to \fM_{g+g'}^{n+n'}.$$
These, and the analogous self-gluings, give the collection $\{\fM_{g}^{n}\}$ the structure of a modular operad in log stacks over $\Z$. 
\end{para}

\begin{para}The Kato--Nakayama space $(\fM_{g}^{n})^\KN$ is homotopy equivalent to $M_g^n \simeq \M_g^n(\C)$. Informally speaking, the part of the log structure given by the boundary of $\overline \M_{g,n}$ has the effect of deleting the boundary, and the part of the log structure given by the $n$ tangent line bundles has the effect of taking the product of the $n$ unit circle bundles associated with these line bundles over $\M_{g,n}$. More is true: the collection $\{(\fM_{g}^{n})^\KN\}$ form a topological modular operad, isomorphic to the Kimura--Stasheff--Voronov model \cite{kimurastasheffvoronov2} for the surface operad $\{M_g^n\}$. Indeed, Kimura--Stasheff--Voronov explicitly describe their model as a real oriented blow-up in a manner similar to how we defined Kato--Nakayama spaces. 
\end{para}

\begin{para}
In \cref{subsection-moduli of surfaces} we explained that there is a framed $E_2$-algebra structure on the disjoint union $\coprod_{g\geq 0} M_g^1$, and that this structure is induced from the fact that the collection of spaces $\{M_g^n\}$ form a topological modular operad, and that $\{M_0^{n+1}\}$ is the framed $E_2$-operad. We have now constructed a modular operad $\{\mathsf M_g^n\}$ in log stacks, whose Kato--Nakayama analytification is homotopy equivalent to $\{M_g^n\}$. It follows in particular that $\coprod_{g\geq 1}\mathsf M_g^1$ is an algebra over the logarithmic framed $E_2$-operad $\{\mathsf M_0^{n+1}\}$. If we define $\mathsf M_0^1$ to be a point, then this structure extends to a framed $E_2$-structure on $\coprod_{g\geq 0}\mathsf M_g^1$.
\end{para}

\begin{rem}
	Having an algebraic interpretation of the framed $E_2$-structure on $\coprod_{g \geq 0} M_g^1$ allows (in a sense) an algebraic interpretation of the Harer stabilization morphisms $M_g^1 \to M_{g+1}^1$; as mentioned in \S\ref{stabilization is E2 multiplication}, stabilization is multiplication by a point of $M_1^1$. More precisely, the framed $E_2$-structure gives a map
	\[ \mathsf M_0^3 \times \mathsf M_1^1 \times \mathsf M_g^1 \to \mathsf M_{g+1}^1.\]
	Hence we obtain a morphism $\mathsf M_g^1 \to \mathsf M_{g+1}^1$ whenever we choose a point of $\mathsf M_0^3$, and a point of $\mathsf M_1^1$. But there is no map $\mathrm{Spec}(\mathbb C) \to \mathsf M_0^3$, and no map $\mathrm{Spec}(\mathbb C) \to \mathsf M_1^1$, either!
	What \emph{is} true is that we obtain a stabilization map on homology (singular or \'etale) $H_\bullet(\mathsf M_g^1)\to H_\bullet(\mathsf M_{g+1}^1)$ after choosing generators of $H_0(\mathsf M_0^3)$ and $H_0(\mathsf M_1^1)$. What is \emph{also} true is that there exist \emph{virtual} morphisms (in the sense of \cite{howellthesis,DPP-log}) of log stacks $\mathrm{Spec}(\mathbb C) \to \mathsf M_0^3$ and $\mathrm{Spec}(\mathbb C) \to \mathsf M_1^1$. So we obtain also a virtual morphism $\mathsf M_g^1 \to \mathsf M_{g+1}^1$.	
	\end{rem}

\begin{para}In any case, we claim now that \cref{motivic theorem} is a consequence of the construction of the logarithmic model $\{\fM_g^n\}$ of the surface operad. Indeed: $\coprod_{g \geq 0} \fM_g^1$ is an algebra over the operad $\{\fM_0^{n+1}\}$ in log stacks. Taking Kato--Nakayama spaces we obtain a model for the algebra $\coprod_{g \geq 0} M_g^1$ over the topological framed $E_2$-operad $\{M_0^n\}$, giving the left-hand side of the isomorphism in \cref{motivic theorem}; taking \'etale homology we obtain the right-hand side, with its natural Galois action. Strictly speaking, what we have said is an argument for the homology with constant coefficients, but the gluing maps for the spaces $\{\fM_g^n\}$ are compatible with the algebro-geometric local systems $\V_\lambda$ in the same way as the corresponding topological local systems on the spaces $\{M_g^n\}$. 
	\end{para}

\begin{rem}
    The argument in the previous paragraph is in some respects unsatisfactory. It would be more appealing to be able to directly construct a log stack $\fM_g^n(KA)$ with a morphism of modular operads $\{\fM_g^n\} \to \{\fM_g^n(KA)\}$, which upon taking analytifications produces a model for $\{M_g^n\} \to \{M_g^n(KA)\}$. Considering moduli spaces of curves with a map to a target space, one is naturally led to consider the Kontsevich space of stable maps. In the simple case of genus zero maps to a \emph{convex} smooth projective variety $X$, we have a smooth stack of stable maps $\overline{\M}_{0,n}(X) = \coprod_\beta \overline{\M}_{0,n}(X,\beta)$. Its boundary is a normal crossing divisor, and we can give it a log structure producing a log stack $\fM_0^n(X)$ with a map of modular operads $\{\fM_0^n\} \to \{\fM_0^n(X)\}$; this accomplishes roughly what we want, except we are considering an algebraic mapping space as opposed to the space of all continuous maps. In order to model the space $M_g^n(KA)$ we would need to work in arbitrary genus, and more critically it seems that we would need to be able to take for the target space $X$ a \emph{higher stack}, more precisely a product of affine stacks $K(\mathbb G_a,n)$, as considered by To\"en \cite{toen}. However, to our knowledge nobody has considered a theory of stable maps into such a ``very stacky'' target in the literature. It would be interesting to be able to formulate the arguments in terms of such a space. Similarly, in the following section it would be most natural to formulate the arguments in terms of ``twisted'' \emph{genus zero} stable maps to a product of quotient stacks $[K(\mathbb G_a,n)/\boldsymbol{\mu}_2]$. 
\end{rem}

\begin{rem}
	The main emphasis of this paper is braid groups, not mapping class groups, and one may similarly ask for a logarithmic interpretation of the framed $E_2$-structure on the disjoint union $\coprod_n \mathrm{Conf}_n(\mathbb D)$ (also in anticipation on the next section of the paper). Let  us briefly explain such a construction. For $n \geq 2$, let $\mathsf C_n$ be the log scheme given by the stack $[\overline{\mathcal M}_{0,n+1}/\Sigma_n]$ and the log structure associated to the boundary divisor, and the zero-section of the tangent line bundle at the $(n+1)$st marked point. We let $\mathsf C_1$ and $\mathsf C_0$ both be a point. Then $\coprod_{n\geq 0} \mathsf C_n$ is an algebra over the operad $\{\mathsf M_0^{n+1}\}$, and on Kato--Nakayama analytifications one recovers the framed $E_2$-structure on $\coprod_n \mathrm{Conf}_n(\mathbb D)$. We will not directly use this construction. 
\end{rem}

\section{Hyperelliptic curves and computation of Galois action}\label{hyperelliptic algebraic section}

\subsection{Motivic framed \texorpdfstring{$E_2$}{E2}-structure in the hyperelliptic case}
\begin{para}
All of what we did in the previous section admits a hyperelliptic analogue. As the results and arguments involve no new ingredients except replacing the spaces $\overline \M_{g,n}$ with the admissible cover spaces of Abramovich--Corti--Vistoli \cite{acv03}, we permit ourselves to be brief. 
\end{para}\begin{para}
We denote by $\H_g^{n,m}$ the moduli stack of smooth hyperelliptic curves of genus $g$, equipped with $n$ distinct ordered marked Weierstrass points, and also $m$ distinct ordered marked non-Weierstrass points, none of which are conjugate under the hyperelliptic involution, and at each marking a nonzero tangent vector. This is a smooth stack over $\Spec \Z[\tfrac 1 2]$. Its analytification $\H_g^{n,m}(\C)$ has the homotopy type of the topological moduli space $H_g^{n,m}$ of hyperelliptic surfaces defined in \S\ref{hyp definition}. 
\end{para}
\begin{para}\label{ACV}
From the work of Abramovich--Corti--Vistoli \cite{acv03} one obtains a smooth compactification of $\H_g^{n,m}$. Explicitly, start with the stack of $(2g+2+n)$-pointed genus zero balanced twisted stable maps to $B(\Z/2)$,  which Abramovich--Corti--Vistoli would denote $\mathcal B_{0,2g+2+n}^{\mathrm{bal}}(\Z/2)$. Take the open and closed substack where the first $2g+2$ marked points are twisted (necessarily with a $\mu_2$ stabilizer) and the last $n$ markings are untwisted. Take the quotient by the action of the symmetric group $\Sigma_{2g+2-m}$, so that only $m$ of the $2g+2$ Weierstrass points are ordered. Take the degree $2^{n+m}$ finite \'etale cover of this stack parametrizing in addition the datum of a trivialization of the $(\Z/2)$-torsor over each of the remaining $n+m$ marked points. We denote this space by $\overline{\mathcal H}_{g,n,m}$. Then $\mathcal H_g^{n,m}$ is Zariski open inside the total space of the vector bundle given by the direct sum of the $(n+m)$ natural tangent line bundles on $\overline{\mathcal H}_{g,n,m}$ (pulled back from $\mathcal B_{0,2g+2+n}^{\mathrm{bal}}(\Z/2)$). The projective bundle compactifying this vector bundle furnishes a smooth modular compactification of $\H_g^{n,m}$ over $\Spec \Z[\tfrac 1 2]$ such that the complement is a normal crossing divisor. 
For any odd $q$ one therefore has comparison isomorphisms
$$ H_\bullet^\sing(H_g^{n,m},\Q_\ell) \cong H_\bullet^\et(\mathcal H_g^{n,m} \otimes \overline\Q, \Q_\ell) \cong H_\bullet^\et(\mathcal H_g^{n,m} \otimes \Fqbar, \Q_\ell),$$
and similarly for coefficients in the local systems $S^\lambda(\V)$. Here $\V$ denotes the local system on $\H_g^{n,m}$ of rank $2g+n+2m-1$ defined as in \S\ref{ag local systems}, given by $R^1p_!\Q_\ell$ where $\mathcal C \to \H_g^{n,m}$ is the universal $(n+2m)$-punctured curve. We will only consider the local system $\V$ in the cases $(n,m) = (1,0)$ and $(n,m)=(0,1)$. 
\end{para}\begin{para}
Now as before there is a natural BV-algebra structure on 
$$ \bigoplus_{g \geq 0} \bigoplus_\lambda S^\lambda(A) \otimes H_\bullet(H_g^{0,1},S^\lambda(V[-1])),$$
since this is the homology of the framed $E_2$-algebra $\coprod_{g \geq 0} H_g^{0,1}(KA)$. Again we claim that the BV-algebra structure is compatible with the Galois action under the natural comparison isomorphism:
\end{para}
\begin{thm}\label{hyperelliptic bv structure}
	Under the comparison isomorphism
	$$ \bigoplus_{g \geq 0} \bigoplus_{\lambda} S^\lambda(A) \otimes H_\bullet^\sing(H_g^{0,1},S^\lambda(V[-1]) \otimes \Q_\ell) \cong  \bigoplus_{g \geq 0} \bigoplus_{\lambda} S^\lambda(A) \otimes H_\bullet^\et(\H_g^{0,1} \otimes \overline\Q,S^\lambda(\V[-1])) $$
	the natural BV-algebra structure on the left-hand side is compatible with the Galois action on the right-hand side, when the homology of the framed disk operad $FD$ is given the Galois action which on $H_k(FD(n),\Q_\ell)$ is given by the $k$th power of the cyclotomic character, i.e.\ a pure Tate motive of weight $-2k$. 
\end{thm}

\begin{para}\label{normal bundle hyp}
	As before, we deduce this from the existence of a suitable  model of the topological cyclic operad $\{H_g^{n,m}\}$ (see \S\ref{hyp operad}) in terms of logarithmic algebraic geometry, and the fact that the framed $E_2$-structure on $\smash{\coprod_{g \geq 0} H_g^{0,1}}$ can be understood as an instance of this operadic structure. Explicitly, we define $\mathsf H_g^{n,m}$ to be the following log stack: the underlying stack is the stack $\smash{\overline{\H}_{g,n,m}}$ from \S\ref{ACV}.  We give it the log structure induced by the normal crossing boundary divisor, and the zero sections of the $(n+m)$ tangent line bundles at the marked points. The operadic structure on the underlying stacks is described in \cite{gcovers}. To see that the operadic structure and logarithmic structure are compatible, we need to know the analogue of \S\ref{gluing}, i.e.\ that the normal bundle of each gluing map   
    \[\overline{\H}_{g,n+1,m} \times \overline{\H}_{g',n'+1,m'} \to \overline{\H}_{g+g',n+n',m+m'}\]
    and
    \[\overline{\H}_{g,n,m+1} \times \overline{\H}_{g',n',m'+1} \to \overline{\H}_{g+g'+1,n+n',m+m'}\]
    is the tensor product of the tangent line bundles at the points on the respective hyperelliptic curves being gluing together. 
	This is well known, but we do not know a reference. One can show it by adapting the arguments of \cite[Chapter XIII, \S3]{acg} based on \cite[\S3.0.4]{acv03}. Alternatively, one may reason as follows:
\end{para}
    \begin{para}\label{stack of all twisted curves}
        
	The forgetful map $\overline{\H}_{g,n,m} \to \mathcal{M}^{\mathrm{tw}}_{0,2g+2+m}/\Sigma_{2g+2-n}$ to the highly nonseparated stack of all twisted stable curves \cite{olsson} is \'etale, and it is finite \'etale over an open substack $\mathcal X$ of the target. Sending a twisted curve to its coarse space defines a morphism $\mathcal X \to \overline{\mathcal{M}}_{0,2g+2+m}/\Sigma_{2g+2-n}$, which exhibits $\mathcal X$ as a root stack (in the sense of Cadman \cite{cadman}) of $\overline{\mathcal{M}}_{0,2g+2+m}/\Sigma_{2g+2-n}$; more specifically, we take a square root along each boundary divisor parametrizing twisted curves with a $(\Z/2)$-stabilizer at the node, i.e.\ those divisors whose generic point parametrizes a curve with two components, such that an odd number of the $(2g+2)$ branch points lie on each component. (More generally, $\mathcal M_{g,n}^{\mathrm{tw}}$ can similarly be described \'etale-locally as a root stack over $\smash{\overline{\mathcal{M}}_{g,n}}$, where we take an $r$th root along each divisor parametrizing an $r$-twisted node.)	
    \end{para}
    \begin{para}
        	Let us now analyze the case of gluing together two Weierstrass points. There is a commutative diagram
\[\begin{tikzcd}
	\overline{\H}_{g,n+1,m} \times \overline{\H}_{g',n'+1,m'} \arrow[r]\arrow[d]&\arrow[d] \mathcal{D}	\arrow[r]\arrow[d]& \overline{\mathcal M}_{0,2g+2+m}/\Sigma_{2g+2-n-1} \times \overline{\mathcal M}_{0,2g'+2+m'}/\Sigma_{2g'+2-n'-1} \arrow[d] \\
	\overline{\H}_{g+g',n+n',m+m'} \arrow[r]& \mathcal X \arrow[r] & \overline{\mathcal M}_{0,2g+2g'+2+m+m'}/\Sigma_{2g+2g'+2-n-n'}
\end{tikzcd}\] 
with $\mathcal X$ as in \S\ref{stack of all twisted curves}. In the left-hand square, the horizontal maps are \'etale, so the normal bundle of the left-hand map is the normal bundle of $\mathcal D$ in $\mathcal X$. In the right-hand square, the bottom map exhibits $\mathcal X$ as a root stack over the target; the right-hand map is the inclusion of one of the boundary divisors along which we take a square root; $\mathcal D$ is the universal square root of this divisor in $\mathcal X$. Thus the normal bundle of $\mathcal D$ in $\mathcal X$ is the canonical square root of the normal bundle of the right-hand map. The normal bundle of the right-hand map is the tensor product of the tangent line bundles at the two points being glued together, and the tangent lines at the corresponding $(\Z/2)$-twisted markings furnish that canonical square root. Finally, the map from the hyperelliptic curve to the genus zero twisted curve is \'etale, so the tangent line at a $(\Z/2)$-twisted marking is canonically identified with the tangent line at the corresponding Weierstrass point. The case of gluing together a conjugate pair of non-Weierstrass points is entirely analogous, but there is no square root involved. This justifies the final assertion of \S\ref{normal bundle hyp}.
    \end{para}

    \begin{para}
    The spaces $\{\mathsf H_g^{n,m}\}$ form a two-colored cyclic operad in log stacks, and taking Kato--Nakayama spaces gives a model for the operad $\{H_g^{n,m}\}$.  
\end{para}

\subsection{Gerstenhaber indecomposables}

\begin{para}
	The goal of this section is to prove the following purity result, from which \cref{thmB} from the introduction follows:
\end{para}
\begin{thm}\label{purity theorem}
	The Galois representation on the stable homology groups $$H_k^\et(\H_\infty^{0,1} \otimes \overline\Q,S^\lambda(\V))$$ is pure Tate of weight $-2k+\vert\lambda\vert$. 
\end{thm}

\begin{para}To this end, we consider the group-completion of the framed $E_2$-algebra $\coprod_{g \geq 0} H_g^{0,1}(KA)$. By the group-completion theorem its homology is given by
	$$ \bigoplus_{g \in \Z} \bigoplus_\lambda S^\lambda (A) \otimes H_\bullet(H_\infty^{0,1},S^\lambda(V[-1])).$$
	On the other hand, we identified in \cref{scanning theorem} this group completion as the loop space $\Omega^2 W(A)$. From the Galois action on the homology of $H_g^{0,1}$ described in the preceding subsection, it follows that we have a natural Galois action on the homology of $\Omega^2 W(A)$, say with $\Q_\ell$-coefficients. From \cref{hyperelliptic bv structure} it follows that this Galois action is compatible with the BV-algebra structure on the homology of $\Omega^2 W(A)$. In fact, the BV-operator will play no role in the argument, and we consider $H_\bullet(\Omega^2 W(A),\Q_\ell)$ as a Gerstenhaber algebra. Since the Galois action is compatible with the Gerstenhaber algebra structure, we see that in order to determine the Galois action on $H_\bullet(\Omega^2 W(A),\Q_\ell)$ it suffices to determine the Galois action on the Gerstenhaber indecomposables. 
\end{para}

\begin{thm}\label{indecomposables}The indecomposables in the Gerstenhaber algebra 
	$$ H_\bullet(\Omega^2 W(A),\Q) \cong \bigoplus_{g \in \Z} \bigoplus_\lambda S^\lambda (A) \otimes H_\bullet(H_\infty^{0,1},S^\lambda(V[-1]))$$
	can be identified with the following three summands: 
	\begin{itemize}
		\item $(g,\lambda) = (1,0)$; the summand $\Q \otimes H_0(H_\infty^{0,1},\Q)$. 
		\item $(g,\lambda) = (-1,0)$; the summand $\Q \otimes  H_0(H_\infty^{0,1},\Q)$. 
		\item $(g,\lambda) = (0,[2])$; the summand $\mathrm{Sym}^2(A) \otimes H_0(H_\infty^{0,1},\wedge^2(V))[-2]$. 
	\end{itemize}
\end{thm}

\begin{para}
	The arithmetic \cref{purity theorem} is a direct consequence of the purely topological \cref{indecomposables}. Indeed, after \cref{indecomposables} it suffices to compute the Galois action on the stable homology in degree $0$. But this is obvious: the action is trivial with trivial coefficients, and pure Tate of weight $2$ on the homology with coefficients in $\wedge^2 \V$, where the nonzero homology arises since the symplectic form\footnote{The word ``symplectic form'' is strictly speaking incorrect here: $\V$ is a symplectic vector space only when $(n,m)=(1,0)$,  and we are considering the case $(n,m)=(0,1)$, which puts us in the situation of Proctor's odd symplectic groups \cite{proctor}.} defines a direct summand $\Q_\ell(-1) \subset \wedge^2 \V$. Then the homology in any other degree and for any other $\lambda$ is determined by the compatibility with Galois action, and that the Gerstenhaber operad has the Galois action which in homological degree $k$ is pure Tate of weight $-2k$. 
\end{para}

\begin{para}
	The proof of \cref{indecomposables} uses Koszul duality theory for Gerstenhaber algebras. Recall the following theorem of Berglund \cite{berglundkoszulspaces}:
\end{para}

\begin{thm}[Berglund] \label{berglund theorem 2}Let $X$ be an $n$-connected Koszul space of finite type. The commutative algebra $H^\bullet(X,\Q)$,  considered as a Gerstenhaber $n$-algebra with trivial bracket, is Koszul as a Gerstenhaber $n$-algebra. Its Koszul dual Gerstenhaber $n$-algebra is $H_\bullet(\Omega^n X, \Q)$. 
\end{thm}

\begin{rem}In our earlier discussions on Koszul duality we formulated it as a duality between algebras and coalgebras, but here we have combined with the linear dual to obtain a duality between algebras and algebras. It is only for this reason we need $X$ to be of finite type. We work with this formulation to be consistent with Berglund's paper, and with Sullivan's work on rational homotopy theory.\end{rem}

\begin{para}
	In formulating this theorem we use that the Gerstenhaber operad is self-dual up to a degree shift; after making this shift by $n$, any Gerstenhaber $n$-algebra has a Koszul dual Gerstenhaber $n$-algebra. We will need a modest generalization of Berglund's result to the case that $X$ is merely $(n-1)$-connected. 
\end{para}

\begin{defn}Let $A_\bullet$ be a nonnegatively graded algebra of the form $A_+ \otimes A_0$, with $A_+$ connected and $A_0$ a polynomial ring concentrated in degree $0$. A \emph{group-completion} of $A$ is a localization of the form $S^{-1}A$, where $S$ is a minimal generating set of $A_0$. 
\end{defn}

\begin{thm}\label{berglund generalization} Let $X$ be an $(n-1)$-connected Koszul space of finite type. The commutative algebra $H^\bullet(X,\Q)$,  considered as a Gerstenhaber $n$-algebra with trivial bracket, is Koszul. The homology $H_\bullet(\Omega^n X, \Q)$ is a group-completion of the Koszul dual Gerstenhaber $n$-algebra.
\end{thm}
\begin{proof}
	This follows from the arguments of Berglund \cite[Section 4]{berglundkoszulspaces}. Berglund shows in general for a simply connected Koszul space $X$ that the Koszul dual Gerstenhaber $n$-algebra of $H^\bullet(X,\Q)$ is the free algebra on the rational homotopy groups of $X$, shifted down $n$ degrees. When $X$ is $n$-connected, this free algebra is precisely $H_\bullet(\Omega^n X,\Q)$. When $X$ is merely $(n-1)$-connected, this fails due to nontriviality of $\pi_0(\Omega^n X)$: $H_0(\Omega^n X,\Q)$ is the Laurent polynomial algebra on $\pi_n^\Q(X)$, not the polynomial algebra. But this is the only difference, and forming the group-completion in the Koszul dual Gerstenhaber $n$-algebra produces $H_\bullet(\Omega^n X,\Q)$. 
\end{proof}

\begin{rem}
	It is instructive to consider the case $n=2$ and $X=S^2$. One may compute $H_\bullet(\Omega^2_0S^2,\Q)$ in two different ways. One way is to use that the $2$-connected cover of $S^2$ is $S^3$, so $H_\bullet(\Omega^2_0S^2,\Q) \cong H_\bullet(\Omega^2 S^3,\Q)$, which by \cref{berglund theorem 2} is a free Gerstenhaber algebra on a single generator $y$ in degree $1$. Alternatively, \cref{berglund generalization} identifies $H_\bullet(\Omega^2 S^2,\Q)$ with the group-completion of the free Gerstenhaber algebra on a single generator $x$ in degree $0$. We may identify $H_\bullet(\Omega_0^2 S^2,\Q)$ with the Gerstenhaber subalgebra generated by the bracket $[x,x]=y$.
\end{rem}

\begin{para}\label{crucial fact}
	The crucial fact is now that Koszul dual Gerstenhaber $n$-algebras have dual quadratic presentations  \cite[Theorem 2.11]{berglundkoszulspaces}. If $D$ is the space of indecomposables in a Koszul Gerstenhaber $n$-algebra, then the indecomposables in the Koszul dual Gerstenhaber $n$-algebra are given by $D^\vee[n]$. 
\end{para}

\begin{proof}[Proof of \cref{indecomposables}] We compute the Gerstenhaber indecomposables in $H_\bullet(\Omega^2 W(A),\Q)$. By \cref{berglund generalization} and \S\ref{crucial fact} this reduces to computing the indecomposables in the algebra $$H^\bullet(W(A),\Q) \cong H^\bullet(S^2\vee KA/\langle\iota\rangle,\Q)$$
	which are evidently $\Q[-2]$ (the fundamental class of the $2$-sphere) and $\mathrm{Sym}^2(A^\vee)$ (which generates the cohomology ring of $KA/\langle\iota\rangle)$. In the homology of $\Omega^2 W(A)$ these two summands correspond to the first and third summand enumerated in \cref{indecomposables}. The second summand of \cref{indecomposables} arises by inverting the first summand in the process of group-completing. \end{proof}

\section{Closed surfaces}\label{closed surfaces}

\begin{para}\label{ses}In this section we will extend the homological stability results to moduli spaces of closed surfaces. It will be convenient at this point to switch from homology to cohomology. Let $\H_g$ be the moduli stack of hyperelliptic curves of genus $g \geq 2$, and let $\mathbb H$ be the standard rank $2g$ local system on $\H_g$, i.e.\ $R^1 p_\ast \Q$, where $p:\mathcal C_g \to \H_g$ is the universal curve. We opt to denote this local system $\mathbb H$ rather than $\V$ for two reasons:
	\begin{enumerate}
		\item In the arguments we will consider the pullbacks of $\mathbb H$ to $\H_g^{1,0}$ and $\H_g^{0,1}$. On $\H_g^{0,1}$ there is a short exact sequence
		$$ 0 \to \Q \to \V \to \mathbb H \to 0$$
		relating the local system that we have so far denoted $\V$ with the pullback of $\mathbb H$, and it will be important to distinguish the two notationally. 
		\item In \S\ref{standard coefficient system} and \S\ref{ag local systems} we gave a definition of a local system\footnote{More precisely, a cohomologically locally constant complex in $D^b(\M_g)$.} $\V$ on $\M_g$, whose restriction to $\H_g$ is not the same as $\mathbb H$. 
		\end{enumerate}
	In this section we calculate the cohomology groups 
$ H^\bullet(\mathcal H_g,S^\lambda(\mathbb H))$
in a stable range, with their Galois action. In particular, the cohomology stabilizes in the sense that there are isomorphisms 
$$ H^k(\mathcal H_g,S^\lambda(\mathbb H)) \cong H^k(\mathcal H_{g+1},S^\lambda(\mathbb H))$$
for $g \gg k$, in spite of the fact that there are no natural stabilization maps when we work with surfaces without boundary. This fact was previously proven by an indirect argument in \cite{millerpatztpetersen}. 
\end{para}\begin{para}\label{gysinpara}
Let first $\H_{g,1,0}$ resp.\ $\H_{g,0,1}$ denote the moduli spaces of hyperelliptic curves with a marked Weierstrass point, resp.\ a marked non-Weierstrass point.  By keeping track of Galois actions in \cref{gysin}, one sees that for any $\ell$-adic sheaf $\mathbb F$ on $\H_{g,1,0}$ resp.\ $\H_{g,0,1}$, there are isomorphisms $$ H^k(\H_g^{1,0},\mathbb F) = H^k(\H_{g,1,0},\mathbb F) \oplus H^{k-1}(\H_{g,1,0},\mathbb F)(-1)$$
	and
$$ H^k(\H_g^{0,1},\mathbb F) = H^k(\H_{g,0,1},\mathbb F) \oplus H^{k-1}(\H_{g,0,1},\mathbb F)(-1)$$
for any $k$. \end{para}

\begin{prop}\label{closed purity}
	If $k\leq \frac{2g-1-\vert\lambda\vert}{2}$, then $H^k(\mathcal H_g,S^\lambda(\mathbb H))$ is pure Tate of weight $2k + \vert \lambda \vert$. 
\end{prop}

\begin{proof}
	Since the forgetful map $\H_{g,1,0} \to \H_g$ is finite, $H^\bullet(\H_{g},S^\lambda(\mathbb H)) \to H^\bullet(\H_{g,1,0},S^\lambda(\mathbb H))$ is injective. (The trace map gives a left inverse.) But $H^\bullet(\H_{g,1,0},S^\lambda(\mathbb H))$ injects into $H^\bullet(\H_{g}^{1,0},S^\lambda(\mathbb H))$ by \cref{gysin}, and we proved in \cref{purity theorem} that the latter is pure Tate of this weight in a stable range. (Note that $\V=\mathbb H$ on $\H_g^{1,0}$, as opposed to on $\H_g^{0,1}$.)
\end{proof}

\begin{para}Once we are armed with this purity result, it suffices to find an expression for the weight-graded Euler characteristic of the homology in a stable range. We set
$$ \Phi_g = \sum_k \sum_i \sum_\lambda (-1)^k z^{i/2} s_\lambda \dim \mathrm{Gr}^W_i H^k(\H_g,S^\lambda(\mathbb H))  $$
and define similarly generating series 
$$ \Phi_g^{1,0} = \sum_k \sum_i \sum_\lambda (-1)^k z^{i/2} s_\lambda 
\dim \mathrm{Gr}^W_i H^k(\H_g^{1,0},S^\lambda(\V))  $$
and 
$$ \Phi_g^{0,1} = \sum_k \sum_i \sum_\lambda (-1)^k z^{i/2} s_\lambda 
\dim \mathrm{Gr}^W_i H^k(\H_g^{0,1},S^\lambda(\V)).  $$
We may also define evident analogues $\Phi_{g,1,0}$ and $\Phi_{g,0,1}$ by replacing $\H_g^{1,0}$ by $\H_{g,1,0}$, etc.\, but from \S\ref{gysinpara} we see that 
$$ \Phi_{g,1,0} = \frac{1}{1-z}\Phi_g^{1,0}, \qquad \Phi_{g,0,1} = \frac{1}{1-z}\Phi_g^{0,1}.$$
We have similarly the stable versions of these generating series when $g=\infty$. We have not yet proven that $\Phi_\infty$ is well-defined, but this will be deduced shortly. \end{para}

\begin{prop}
	For every $g$, there is an identity
	$$ ( 1  - \frac{\partial}{\partial p_1} + z) \Phi_g = z\Phi_{g,1,0} + \sum_{j \geq 0} (-1)^j e_j \Phi_{g,0,1},$$
	where $\frac{\partial}{\partial p_1}$ is the partial derivative with respect to $p_1$, and $e_j$ denotes the $j$th elementary symmetric polynomial.
\end{prop}

\begin{proof}
	We argue that both sides of the identity are given by 
	$$ \sum_k \sum_i \sum_\lambda (-1)^k z^{i/2} s_\lambda \dim \mathrm{Gr}^W_i H^k(\mathcal C_g,S^\lambda(\mathbb H)),  $$
	where $\mathcal C_g \to \H_g$ is the universal curve. 
	
	For the left-hand side, we consider the Leray spectral sequence for $p:\mathcal C_g \to \H_g$. It degenerates e.g.\ by \cite{delignedegeneration}, and $Rp_\ast \Q \cong \Q \oplus \mathbb H[-1] \oplus \Q(-1)[-2].$ Combining this with the projection formula, we see that
	$$ H^k(\mathcal C_g,S^\lambda(\mathbb H)) \cong H^k(\H_g,S^\lambda(\mathbb H)) \oplus H^{k-1}(\H_g,S^\lambda(\mathbb H) \otimes \mathbb H) \oplus H^{k-2}(\H_g,S^\lambda(\mathbb H))(-1).$$
	We should explain why the middle term here leads to $-\frac{\partial}{\partial p_1}$. Indeed, by Schur--Weyl duality the $z^{i/2}$-coefficient of $\Phi_g$ is the class in $\widehat\Lambda$ of the symmetric sequence $n \mapsto \mathrm{Gr}_i^WH^\bullet(\H_g,\mathbb H^{\otimes n})$. But then \cite[Eq.\ (7.5)]{getzlerkapranov} the $z^{i/2}$-coefficient of $\frac{\partial\Phi_g}{\partial p_1}$ is the class in $\widehat\Lambda$ of the symmetric sequence $n \mapsto \mathrm{Gr}_i^WH^\bullet(\H_g,\mathbb H^{\otimes (n+1)})$. 
	
	For the right-hand side, we note that $\H_{g,1,0}$ is a closed substack of $\mathcal C_g$, with open complement $\H_{g,0,1}$. The associated Gysin long exact sequence reads
	$$ \ldots \to H^{k-2}(\H_{g,1,0},S^\lambda(\mathbb H))(-1) \to H^{k}(\mathcal C_g,S^\lambda(\mathbb H)) \to H^{k}( \H_{g,0,1},S^\lambda(\mathbb H)) $$ $$ \hspace{5cm}\to H^{k-1}(\H_{g,1,0},S^\lambda(\mathbb H))(-1) \to \ldots $$
	Taking Euler characteristics in this long exact sequence gives the result. Indeed, the terms of the form $H^\bullet(\H_{g,1,0},S^\lambda(\mathbb H))(-1)$ contribute exactly $z\Phi_{g,1,0}$ to the right-hand side. (Recall that $\mathbb H = \mathbb V$ on $\H_{g,1,0}$.) We claim that the terms of the form $H^\bullet(\H_{g,1,0},S^\lambda(\mathbb H))$ contribute exactly $\sum_{j \geq 0} (-1)^j e_j \Phi_{g,0,1}$. Indeed, once we are taking Euler characteristics we may as well replace $\mathbb H$ with the formal difference $\V - \Q$ (considered in a suitable Grothendieck group), cf.\ \S\ref{ses}, and formally one has 
	$$ S^\lambda(\V - \Q) = \bigoplus c_{\mu,\nu}^\lambda S^\mu(\V) \otimes S^\nu(-\Q),$$
	where $c_{\mu,\nu}^\lambda$ is a Littlewood--Richardson coefficient, cf.\ property (4) of \S\ref{plethysm identities}. The term $S^\nu(-\Q)$ vanishes unless $\nu = (1^j)$, i.e.\ $S^\nu$ is an exterior power, in which case we have $\wedge^j(-\Q) = (-1)^j \Q$. Since $e_j$ is the symmetric function corresponding to $\wedge^j$ the total contribution from these terms is $\sum_{j \geq 0} (-1)^j e_j \Phi_{g,0,1}$.  
\end{proof}

\begin{para}
	Note that every symmetric function $f \in \Lambda$ can be uniquely written as $f_e + f_o$ with $f_e \in \bigoplus_{n} \Lambda_{2n}$ and $f_o \in \bigoplus_n \Lambda_{2n+1}$. We call $f_e$ and $f_o$ the \emph{even} and the \emph{odd} part of $f$. 
\end{para}

\begin{thm}\label{closed theorem}The stable generating series for closed and open surfaces are related by 
	$$\Phi_\infty = \frac{z + \sum_{j \geq 0}e_{2j}}{1-z^2} \Phi_\infty^{0,1}.$$ 
\end{thm}

\begin{proof}
	This identity is obtained from the preceding proposition by letting $g \to \infty$, and then taking the even part of the resulting expression.  
	
	Considering first the left-hand side of the preceding proposition, we note that the generating series $\Phi_g$ is even for any $g$: the hyperelliptic involution implies that $S^\lambda(\mathbb H)$ has no cohomology on $\H_g$ when $\vert\lambda \vert$ is odd. Hence on the left-hand side the even part of the expression is precisely $(1+z)\Phi_g$, and $\frac{\partial \Phi_g}{\partial p_1}$ is purely odd. 
	
	On the right-hand side, we see similarly that $\Phi_{g,1,0}$ is purely even for all $g$ because of the hyperelliptic involution. The same argument does not apply to $\Phi_{g,0,1}$, but by homological stability we have $\Phi_{\infty,1,0}=\Phi_{\infty,0,1}$ and a fortiori it is even as a symmetric function in the stable range. Hence taking the even part just amounts to discarding the terms $e_j$ with odd $j$. Finally $\Phi_{\infty,0,1} = \frac{1}{1-z} \Phi_\infty^{0,1}$.
\end{proof}

\begin{para}
	We combine this formula with the plethystic formula for the stable Poincar\'e series of \cref{plethystic formula for poincare series}, to obtain an analogue of \cref{plethystic formula for poincare series} for closed surfaces.
\end{para}

\begin{cor}\label{closed plethystic formula}
	There is an identity 
	$$\sum_k \sum_\lambda \dim H_k(H_\infty,V_\lambda) (-z)^k s_{\lambda'}  = \frac{z + \sum_{j \geq 0} z^j h_{2j}}{1-z^2}\Exp( z^{-1} \Log(z + \sum_{r \geq 0} z^{r}h_{2r})-1-h_2).$$
\end{cor}

\begin{proof}
	Indeed: \cref{plethystic formula for poincare series} gives a formula for $\Phi_{\infty}^{0,1}$, and we merely need to plug it into \cref{closed theorem}. What needs remarking is firstly that in \cref{closed theorem} the term $S^\lambda(\mathbb H)$ corresponds to the Schur polynomial $s_\lambda$, and in \cref{plethystic formula for poincare series} we find the transpose $s_{\lambda'}$. To relate the formulas, we have to apply the involution $\omega$ of the ring of symmetric functions to \cref{closed theorem}, replacing $e_j$ with $h_j$. Secondly, \cref{closed theorem} is about generating series for weight-graded Euler characteristics, and \cref{plethystic formula for poincare series} is a formula for Poincar\'e series. Since $H^k(\H_g,S^\lambda(\mathbb H))$ is pure of weight $2k+\vert\lambda\vert$, the two are related by a shift of $\vert \lambda\vert$, which results in replacing $h_{2j}$ with $z^j h_{2j}$. 
\end{proof}

\begin{rem}This homological stability formula for the moduli stack of hyperelliptic curves also corresponds to an arithmetic formula, where, in the arithmetic setting, we sum over \emph{all} hyperelliptic curves of genus $g$, or
equivalently, over all quadratic extensions with a given discriminant, including $\infty$. In that version, the arithmetic factor in the moment of $L$-functions is given by an Euler product that incorporates the point at $\infty$, 
and the fact that the same expression $z + \sum_{j \ge 0} z^j h_{2j}$ appears twice corresponds to the fact that the formula for the Euler factor at $\infty$ in the Euler product matches the formula for the Euler factor at 
every other finite prime. We thank one of the anonymous referees for pointing out this fact to us.
\end{rem}

\begin{cor}If $k\leq \tfrac{2g-1-3\vert\lambda\vert} 4$ then there are Galois-equivariant isomorphisms $H^k(\H_g,S^\lambda(\mathbb H)) \cong H^k(\H_{g+1},S^\lambda(\mathbb H)) \cong H^k(\H_{g+2},S^\lambda(\mathbb H)) \ldots $  \end{cor}

\begin{proof}Note first that since $H^j(\H_g^{1,0},S^\lambda(\mathbb V)) \cong H^j(\H_{g+1}^{1,0},S^\lambda(\mathbb V)) $ for $j\leq \tfrac{2g-1-\vert\lambda\vert}{2}$ by the result of Randal-Williams--Wahl, it follows from the Deligne bounds on the weights of the unstable cohomology \cite{Del80} that $\Phi_g^{1,0}$ and $\Phi_{g+1}^{1,0}$ have the same coefficient of $z^i s_\lambda$ for all $i\leq \tfrac{2g-1+\vert\lambda\vert}{4}$. The arguments of this section show that the same holds for $\Phi_g$ and $\Phi_{g+1}$.

We claim now that $\dim H^k(\H_g,S^\lambda(\mathbb H))$ is the coefficient of $z^{k+\vert\lambda\vert/2}$ in $\Phi_g$. Indeed, \cref{closed purity} shows in fact that $H^i(\H_g,S^\lambda(\mathbb H))$ is pure Tate of weight $2i+\vert\lambda\vert$ for all  $i \leq 2k$, and by the Deligne bounds the weight $2k+\vert\lambda\vert$ cannot  occur above this degree.
\end{proof}

\begin{rem}
    By using the results of \cite{MPPRW} instead of the results of \cite{randalwilliamswahl} in the preceding corollary, one obtains a stable range with worse slope but no dependence on $|\lambda|$.  
\end{rem}

\begin{rem}
Armed with \cref{closed plethystic formula} instead of \cref{plethystic formula for poincare series} one can give analogous proofs to show that 
\cref{1/4}, \cref{jonas vanishing} and \cref{1/2} hold with $H_{\infty}^{0,1}$
replaced by \(H_{\infty}\). Similarly, a version of \cref{polynomial growth} holds with
\(H_{\infty}^{0,1}\) replaced by \(H_{\infty}\), except one finds a polynomial of degree $\vert\lambda\vert/2-2$ rather than  $\vert\lambda\vert/2-1$. Let us give the argument. We have by \cref{plethystic formula for poincare series}, \cref{closed plethystic formula} and \cref{alternative form of gen series} that
\begin{gather*}
\sum_{k,\lambda} \dim H_k(H_\infty,V_\lambda) (-z)^k s_{\lambda'} =
\frac{z + \sum_{j \geq 0} z^j h_{2j}}{1-z^2}\sum_{k,\lambda} \dim
H_k(H_\infty^{0,1},V_\lambda)(-z)^k s_{\lambda'} \\
=\frac{1+z + \sum_{j>0} z^j h_{2j}}{1+z}
\prod_{n=1}^\infty (1 + \frac{1}{1+z^n} \sum_{k>0}
z^{nk}\psi_n(h_{2k}))^{i_n(z^{-1})} \\
=(1 + \frac{1}{1+z}\sum_{j>0} z^j h_{2j})^{(1+z^{-1})}
\prod_{n=2}^\infty (1 + \frac{1}{1+z^n} \sum_{k>0}
z^{nk}\psi_n(h_{2k}))^{i_n(z^{-1})}.
\end{gather*}Expanding this infinite product gives a sum of terms, each coefficient
of which is a rational function whose denominator is of the form
$\prod_n (1+z^n)^{k_n}$ for some $k_n\geq 0$.
In arity $2N$, the unique term which maximizes the quantity $\max_n\{k_n\}$ is
$$\frac{1}{(1+z)^{N}}(zh_2)^N \binom
{1+z^{-1}}{N}=\frac{1}{(1+z)^{N-1}}(zh_2)^N \binom
{z^{-1}}{N-1}\frac{1}{Nz},$$
with $k_1=N-1$.
Again, a Schur polynomial $s_\lambda$ with $\lambda \vdash 2N$ occurs in
the expansion of $h_2^N$ if and only if $\ell(\lambda)\leq N$.\end{rem}

\section{Point counting}\label{point counting}

	\begin{para}In this section we study the traces of {F}robenius on the cohomology groups of the moduli space 
	$
	\H_{g}^{1, 0}
	$ 
	with coefficients in the irreducible symplectic representations 
	$
	V_{\lambda}.
	$ 
	In particular, we show that for any fixed finite field of odd characteristic and any $\lambda$ of even weight, the limit as $g \to\infty$ of the traces of {F}robenius on these cohomology groups exists and is a rational function in the cardinality of the finite field. A similar result for the moduli space $\H_{g}$ of closed hyperelliptic curves together with an algorithmic procedure to compute the rational functions occurring in the limit was obtained by the first named author \cite[Section 13]{bergstrom09}. However, rather than trying to compute these rational functions directly, it is more natural to study a certain generating function of the \emph{stable} traces, to be introduced in Subsection \ref{subsec: genfunc}. As we shall see, this function is closely related 
	to the {P}oincar\'e series of the cohomologies $H_{\bullet}(\H_\infty^{1,0},\V_\lambda),$ and is precisely the so-called \emph{arithmetic} factor occurring in the (conjectural) asymptotic formula for the moments of quadratic {D}irichlet $L$-functions over rational function fields. 
 \end{para}

	\subsection{Zeta functions and quadratic \texorpdfstring{$L$}{L}-series} 
	\begin{para}We begin by recalling some basic facts about zeta functions of algebraic curves over finite fields and quadratic $L$-series in the rational function field setting.
	\end{para}
	
	\begin{para}Let $\mathbf{F}_{q}$ be a finite field with $q$ elements, and let $\overline{\mathbf{F}}_{q}$ be a fixed algebraic closure of $\mathbf{F}_{q}.$ Let $C$ be a smooth, projective, and geometrically connected curve of genus $g$ over $\mathbf{F}_{q}.$ The zeta function of $C$ is defined by 
	\[
	Z_{C}(t) = \exp \!\Bigg(\sum_{n \ge 1} \#C(\mathbf{F}_{q^{n}}) \frac{t^n}{n} \Bigg) \qquad 
	(\text{for $|t| < 1\slash q$}). 
	\] 
 \end{para}\begin{para}
	For a prime $\ell$ different from the characteristic of $\mathbf{F}_{q}$ and $n \ge 1,$ one can express the number of points as $\#C(\mathbf{F}_{q^{n}}) = q^{n} + 1 - a_{n}(C),$ where 
	\begin{equation*} 
		a_{n}(C) = \Tr(\mathrm{Frob}_q^n \, \vert \, H_{\text{\'et}}^{1}(\bar{C}, \mathbf{Q}_{\ell})) 
		= \sum_{j = 1}^{2 g} \omega_{j}(C)^{n}. 
	\end{equation*} 
	Here 
	$
	\bar{C}: = C \otimes_{\mathbf{F}_{q}} \overline{\mathbf{F}}_{q}
	$ 
	and $\mathrm{Frob}_q$ is the endomorphism of the $\ell$-adic \'etale cohomology induced by the {F}robenius morphism $\bar{C} \rightarrow \bar{C}.$ Thus the zeta function is a rational function 
	\[
	Z_{C}(t) = \frac{P_{C}(t)}
	{(1 - t)(1 - q t)}
	\] 
	with numerator 
	$
	P_{C}(t) \in \mathbf{Z}[t]
	$ 
	of degree $2g,$ and constant term 
	$
	P_{C}(0) = 1.
	$ 
	Moreover, it satisfies the functional equation 
	\begin{equation*}
		Z_{C}(t) 
		= \big(q t^{2} \big)^{g - 1}  
		Z_{C}(1\slash qt). 
	\end{equation*} 
 \end{para}\begin{para}
	The zeta function satisfies the {R}iemann hypothesis \cite{W}, that is, the \emph{normalized} eigenvalues of {F}robenius 
	$
	q^{-1/2} \omega_{j}(C)
	$ 
	($j = 1, \ldots, 2g$) lie on the unit circle. Moreover, from the functional equation, these eigenvalues have the property that, after an eventual reordering, 
	$
	\omega_{j}(C)\omega_{j + g}(C) = q    
	$ 
	for all $1\le j\le g.$ In other words, there exists a unique conjugacy class $\Theta(C\slash \mathbf{F}_{q})$ in the similitude group $\mathrm{GSp}(2g)$ such that 
	\[
	P_{C}(t)
	= \det(1 - t\Theta(C\slash \mathbf{F}_{q})).
	\] 
	This conjugacy class is called the {F}robenius conjugacy class attached to $C\slash \mathbf{F}_{q}.$\end{para}

	\begin{para}Now assume that $\mathbf{F}_{q}$ has odd characteristic. 
	\end{para}
	\newcommand{\leg}[2]{(\frac {#1} {#2} )}
	\begin{defn} For $d, p \in \mathbf{F}_{q}[x],$ with $p$ irreducible, define the quadratic residue symbol $\leg d p$ to be $0$ if $p \mid d,$ $1$ if $d$ is congruent to a square modulo $p$ and $-1$ otherwise. If $m$ is non-constant, and 
		$
		m = p_{1} \cdots\,  p_{k} 
		$ 
		with each $p_{i}\in \mathbf{F}_{q}[x]$ irreducible, one defines 
		$ \leg d m = \leg d {p_1} \cdots \leg d {p_k}$.
		
	\end{defn}

	\begin{para}It is clear that the quadratic residue symbol is completely multiplicative in both variables, that is, 
 \[
	{\displaystyle \left({\frac {d_{1} d_{2}}{m}}\right)}
	={\displaystyle \left({\frac {d_{1}}{m}}\right)}
	{\displaystyle \left({\frac {d_{2}}{m}}\right)} 
	\;\;\; \text{and} \;\;\; 
	{\displaystyle \left({\frac {d}{m_{1} m_{2}}}\right)} = 
	{\displaystyle \left({\frac {d}{m_{1}}}\right)} 
	{\displaystyle \left({\frac {d}{m_{2}}}\right)} 
	\] 
	for all 
	$
	d, d_{1}, d_{2}, 
	m, m_{1}, m_{2} \in \mathbf{F}_{q}[x],
	$ 
	with 
	$
	m, m_{1}, m_{2}
	$ 
	non-constant. It is also easy to check that (see \cite[Proposition 3.2]{R}) \[
	{\displaystyle \left({\frac {c}{p}}\right)} 
	= c^{\frac{q - 1}{2}\deg p}
	\] 
	for $c \in \mathbf{F}_{q}$ and $p \in \mathbf{F}_{q}[x]$ irreducible. Thus $\leg c p = 1$ if 
	$
	c \in (\mathbf{F}_{q}^{\times})^{2},
	$ 
	and 
	$
	\leg c p = (-1)^{\deg p}
	$ 
	if 
	$
	c \in 
	\mathbf{F}_{q}^{\times}\setminus
	(\mathbf{F}_{q}^{\times})^{2}.
	$ 
	By induction on the irreducible factors, 
	$
	\leg c m = 1
	$ 
	or $(-1)^{\deg m},$ for all $m$ of positive degree, according as $c$ is a square in $\mathbf{F}_{q}^{\times}$ or not.\end{para}
	
	\begin{para}
	The quadratic residue symbol also satisfies the reciprocity law \cite[Theorem 3.5]{R}: for $d, m \in \mathbf{F}_{q}[x]$ monic and non-constant, we have 
	\begin{equation*}
		\begin{split}
			{\displaystyle \left({\frac {d}{m}}\right)} 
			& = (-1)^{\frac{|d| - 1}{2}\frac{|m| - 1}{2}}
			{\displaystyle \left({\frac {m}{d}}\right)}\\
			& = (-1)^{\frac{q - 1}{2}\deg d \deg m}
			\,{\displaystyle \left({\frac {m}{d}}\right)}
		\end{split}
	\end{equation*} 
	where for a non-zero element $a \in \mathbf{F}_{q}[x],$ $|a| = q^{\deg a}$ denotes 	its norm. 
	\end{para}
	
	\begin{para}For $d \in \mathbf{F}_{q}[x]$ square-free, define 
	$
	\chi_{d}(m) 
	= \leg d m,
	$ 
	and extend it to all non-zero polynomials by setting 
	$ 
	\chi_{d}(c) = 1
	$ 
	for all $c \in \mathbf{F}_{q}^{\times}.$ To $\chi_{d}$ there is an associated $L$-function, first considered by E.\ Artin in his thesis \cite{A}, defined for $s \in \mathbf{C}$ with $\Re(s) > 1$ by 
	\[
	L(s, \chi_{d}) \;\,
	= \sum_{\substack{m \, \in \, \mathbf{F}_{q}[x] \\ m \text{ monic}}} \chi_{d}(m)|m|^{-s} 
	= \, \prod_{p} (1 - \chi_{d}(p)|p|^{-s})^{-1}
	\] 
	the product in the right-hand side being over all monic irreducibles. It is easy to see that indeed both the 
	sum and product converge absolutely when $\Re(s) > 1.$\end{para}\begin{para} If $d$ is monic and square-free of degree $D \ge 3,$ the $L$-function and the numerator 
	$
	P_{C_d}(t)
	$ 
	of the zeta function of $C_{d}$ defined by the affine model $y^{2} = d(x)$ are connected by 
	\[
	L(s, \chi_{d}) =  
	(1 - q^{- s})^{\delta_{D}} 
	P_{C_d}(q^{-s}), \;\;\; 
	\text{with $\delta_{D} = (1 + (-1)^{D})\slash 2$.}
	\] 
	This follows by combining Proposition 14.9 and Proposition 17.7 in \cite{R} with the well-known fact that the zeta function of a smooth, projective, and geometrically connected algebraic curve over a finite field is the same as the zeta function of its function field. If $d$ is monic of degree $D = 1, 2,$ then 
	$
	L(s, \chi_{d}) =  
	(1 - q^{- s})^{\delta_{D}},
	$ 
	and for $d \in \mathbf{F}_{q}^{\times}$ we have: 
	\[ 
	L(s, \chi_{d}) 
	= \zeta_{\mathbf{F}_{ 
				q}[x]}(s) 
	= \frac{1}{1 - q^{1 - s}} 
	\;\; \text{if $d \in (\mathbf{F}_{q}^{\times})^{2}$}
	\;\;\; \text{and} \;\;\; L(s, \chi_{d}) 
	= \frac{1}{1 + q^{1 - s}} 
	\;\; \text{if $d \notin (\mathbf{F}_{q}^{\times})^{2}$.}
	\] 
	Thus for $d$ monic and square-free of degree $D > 0,$ the $L$-function is a polynomial in $q^{-s}$ of degree $D - 1,$ and it satisfies the functional equation 
	\[
	(1 - q^{-s})^{-\delta_{D}}L(s, \chi_{d}) = 
	q^{(\delta_{D} + 1)
		(s - \frac{1}{2})}|d|^{\frac{1}{2} - s}
	(1 - q^{s - 1})^{-\delta_{D}} 
	L(1 - s, \chi_{d}). 	
	\] 
 \end{para}\begin{para}
	In what follows, it will be often convenient to substitute 
	$t : = q^{-s},$ and denote the $L$-function by 
	$
	\mathcal{L}(t, \chi_{d}).
	$ 
	  \end{para}
\subsection{A generating function for stable traces}\label{subsec: genfunc}
 \begin{para}Let $\mathbf{F}_{q}$ be a finite field of odd characteristic. 
	For $n > 0,$ let $\mathscr{P}_{n}$ denote the set of monic square-free polynomials of degree $n$ in $\mathbf{F}_{q}[x],$ and for $d \in \mathscr{P}_{n}$, let $\Theta_{d}$ denote the following multiset with $n-1$ elements: if $n$ is odd it is the eigenvalues occurring in the {F}robenius conjugacy class of the hyperelliptic curve $C_{d}\slash \mathbf{F}_{q}$ with affine equation $y^{2} = d(x)$; if $n$ is even then $\Theta_d$ contains in addition the number $1$. With this definition, we have $$\mathcal L(t,\chi_d) = \prod_{\omega \in \Theta_d} (1-\omega t)$$ for all $d$. \end{para}
 \begin{para}
Define for $n >0$ an element $\mathcal Z_n \in \widehat{\Lambda \otimes \mathbf C}$ by the formula
\[\mathcal Z_n = \sum_{d\in \mathscr P_n} \sum_\lambda s_\lambda(\Theta_d) \cdot s_{\lambda'},\] where $s_\lambda(\Theta_d)$ denotes the Schur polynomial in $n-1$ variables evaluated on the elements of $\Theta_d$.
The goal of this subsection is to study the limit $\lim_{n\to\infty} q^{-n}\cdot \mathcal Z_n$. In particular, the limit exists in the sense that its coefficients (say with respect to the basis of Schur polynomials) converge pointwise. We call this limit the generating function of stable traces. It may be written explicitly as follows:
\end{para}
\begin{thm}\label{thm:stable traces}The generating function of the stable traces is given by
\[\lim_{n \to \infty} q^{-n}\cdot \mathcal Z_n =\Exp\Big(q\, \Log
		\Big(q^{-1} + 
		\sum_{k \ge 0} h_{2k}\Big) 
		- 1\Big)
		\] 
		where $\Exp$ and $\Log$ are the plethystic exponential and logarithm, respectively.\end{thm}

\begin{para}The generating series $\mathcal Z_n$ admits a cohomological interpretation, given in \cref{coh formula for Z}. Via this interpretation, \cref{thm:stable traces} follows from \cref{thmA} and \cref{thmB}, together with the bound on the unstable Betti numbers of the braid groups given in \cref{fuks}. However, in this subsection we give a direct proof of \cref{thm:stable traces} by analytic methods, with an explicit error term. While not logically necessary, since it follows from our topological results, we believe that the argument is enlightening and clarifies the difference between  \cref{thm:stable traces} and \cref{MPPRW}/\cref{thmC}.
 \end{para}

 \begin{rem} \cref{thm:stable traces} will be deduced from an asymptotic formula proven in \cref{asympt-stability}. The latter formula is in a sense not new, see, for instance, \cite{Rud, R-G} and \cite{BJ}, where the authors obtained similar asymptotics, but for averages of products of power sum polynomials in the eigenvalues of {F}robenius rather than products of elementary symmetric polynomials. 	However, any of these asymptotic formulas (especially that in \cite{R-G}, which is closely related to ours) is insufficient for our purposes, mainly due to the restriction to partitions whose parts grow at least logarithmically in the genus $g.$ 
	\end{rem} 

\begin{rem}\label{coh formula for Z}
    We may pull back the local system $\mathbb V$ to $\mathrm{Conf}_n(\mathbb A^1)$ from $\H_g^{1,0}$ (if $n=2g+1$ is odd), respectively from $\H_g^{0,1}$ (if $n=2g+2$ is even). Then by the Grothendieck--Lefschetz trace formula, $$\mathcal Z_n = \sum_\lambda \sum_k (-1)^k \mathrm{Tr}(\mathrm{Frob}_q | H^k_{c,\et}(\mathrm{Conf}_n(\mathbb A^1)_{\Fqbar} ,S^\lambda(\mathbb V)) \cdot s_{\lambda'}, $$
     since $\mathscr{P}_n = \mathrm{Conf}_n(\mathbb A^1)(\F_q)$, and the multiset $\Theta_d$ consists precisely of the eigenvalues of $\mathrm{Frob}_q$ on the stalk of $\mathbb V$ at a geometric point above $d$. We will not use this cohomological interpretation of $\mathcal Z_n$ in this subsection. By Poincar\'e duality, $q^{-n}\mathcal Z_n$ is similarly the trace of Frobenius on \'etale homology.  
\end{rem}

\begin{para}
 From the definition of $\mathcal Z_n$ we see that if we specialize $\mathcal Z_n$ to a symmetric function of {finitely} many variables $\mathrm t = t_1,t_2,\ldots, t_r$, then 
$$ \mathcal Z_n(\mathrm t) = \sum_{d \in \mathscr P_n} \prod_{i=1}^r \, \mathcal L(-t_i,\chi_d).$$
Indeed, by the dual Cauchy identity (\S\ref{dual cauchy}) we have for any $d \in \mathscr P_n$ that 
$$ \sum_\lambda s_\lambda(\Theta_d) s_{\lambda'}(t_1,\ldots,t_r) = \prod_{\omega \in \Theta_d} \prod_{i=1}^r (1+\omega t_i) = \prod_{i=1}^r \mathcal L(-t_i,\chi_d).$$
\end{para}

\begin{para}  
Now define
$$ \mathcal Z(\mathrm t,\xi) \,=\, \sum_n \mathcal Z_n(\mathrm t) \, \xi^n \, = \sum_{d \text{ monic \& sq. free}} \,\, \prod_{i=1}^r\, \mathcal L(-t_i,\chi_d) \, \xi^{\deg d}. $$
Write $\mathcal Z(\mathrm t,\xi) = \sum_{\lambda} \mathcal Z_\lambda(\xi)\mathrm t^\lambda$, where $\lambda = (\lambda_1,\ldots,\lambda_r)$ is a tuple of nonnegative integers. \end{para}

\begin{rem}
    If $n$ is odd and $\vert\lambda\vert$ is odd, then the coefficient of $\xi^n$ in $\mathcal Z_\lambda(\mathrm t)$ vanishes. Indeed, this follows by a point-counting argument: if 
	$
	c \in \mathbf{F}_{q}^{\times}\setminus
	(\mathbf{F}_{q}^{\times})^{2}
	$, then 
	$
	\mathcal{L}(t, \chi_{c d}) 
	= \mathcal{L}(- t, \chi_{d})
	$, and by pairing up the terms corresponding to $d(x)$ and $c^{-n}d(cx)$ we see that the sum vanishes.  It can also be seen via  \cref{odd vanishing} and the cohomological interpretation of $\mathcal Z_n$.
\end{rem}

\begin{lem}\label{changing-order-sum}
	Suppose $\lambda = (\lambda_1,\ldots,\lambda_r)$. Then
	$$\mathcal Z_\lambda(\xi) = (-1)^{\vert\lambda\vert}(1-q\xi^2) \sum  \frac{\mathcal L(\xi,\chi_{u})}{\prod_{p \mid u} (1-\xi^{2 \deg p})\prod_{p \mid v}(1+\chi_{u}(p)\xi^{\deg p})}$$
	where the sum is over all monic polynomials 
	$
	m_{1}, \ldots,
	m_{r}
	$
	with
	$
	\deg m_{i} =
	\lambda_{i},
	$
	and for each choice of $m_{i}$, 
	$1\le i\le r$, we write
	$
	m_{1} \cdots\, 
	m_{r}  =  uv^2
	$
	with $u, v$ monic and $u$
	square-free.
\end{lem}
\begin{proof}
     This is merely an application of the quadratic reciprocity law
     \[
			{\displaystyle \left({\frac {d}{m}}\right)} 
			= (-1)^{\frac {q-1} 2 \deg d \deg m }
			\,{\displaystyle \left({\frac {m}{d}}\right)}.
		\]
	To see this, assume $|t_{i}|< 1\slash q$ 
		for $i = 1, \ldots, r,$ and $|\xi| < 1\slash q.$ Then, by expressing each 
		$
		\mathcal{L}(-t_{i}, \chi_{d})
		$ 
		by its defining series, interchanging the order of summation and applying 
		quadratic reciprocity, we can write 
		\[
			\mathcal{Z}(\mathrm{t},\xi)  \;\; 
			= \sum_{m_{1}, \ldots, \, m_{r} \text{ monic}}\;  
			\prod_{i = 1}^r (-t_{i})^{\deg m_{i}} 
			\sum_{d \text{ monic \& sq.\ free}}\, {\displaystyle \left({\frac{\prod_{i = 1}^{r} m_{i}}{d}}\right)} \xi^{\deg d}.
		\]
  Here we used that the sign in the quadratic reciprocity can always be taken to be $1$: if not, then $\deg d$ and $\deg m =\vert\lambda\vert$ are both odd, but in this case the corresponding coefficient of $\mathcal Z(t,\xi)$ vanishes. 
  For fixed 
		$
		m_{1}, \ldots, m_{r}
		$, 
		write 
		$
		m_{1} \cdots\, 
		m_{r} 
		= uv^2
		$ 
		with $u, v$ monic and $u$ square-free. Then we can express 
		the inner sum as 
		\begin{equation*}
			\sum_{d \text{ monic \& sq.\ free} }\, {\displaystyle \left({\frac{\prod_{i = 1}^{r} m_{i}}{d}}\right)} \xi^{\deg d}
				 = \sum_{d \text{ monic \& sq.\ free}}
				\chi_{uv^2}(d)\, \xi^{\deg d} = \prod_{p \, \nmid\,  
					uv}
				\big(1 + \chi_{u}(p)\, \xi^{\deg p}\big) 				
		\end{equation*} 
		the product being over all monic irreducibles not dividing 
		$uv$. 
		For each such $p$, we write 
		\[
		1 + \chi_{u}(p)\, \xi^{\deg p}
		= \frac{1 - \xi^{2\deg p}}{1 - \chi_{u}(p)\xi^{\deg p}}, 
		\] 
		and thus 
		\begin{equation*}
				\sum_{d\text{ monic \& sq.\ free} }\, {\displaystyle \left({\frac{\prod_{i = 1}^{r} m_{i}}{d}}\right)}\xi^{\deg d} 
				 = \big(1 - q \xi^{2}\big)
				\frac{\mathcal{L}(\xi, \chi_{u})}
				{\prod_{p \mid u}\big(1 - \xi^{2\deg p}\big) \prod_{p \mid v}
				\big(1 + \chi_{u}(p)\xi^{\deg p}\big)}.\end{equation*} 
  Here we used that $\prod_p (1-\xi^{2\deg p}) = (1-q\xi^2)$. The result follows.
\end{proof}

\begin{prop} \label{asympt-stability} For $n>0$ and $\lambda= (\lambda_1,\ldots,\lambda_r)$  
  the coefficient of $\xi^n$ in $\mathcal Z_\lambda(\xi)$ is given by \[
		 q^{n}
		\, \cdot 
  \big(1-q^{-1}\big)\sum_{\substack{
				m_{1} \cdots\, 
				m_{r}  =
				v^2 \\
				\deg m_{i} =
				\lambda_{i}}}
		\; \prod_{p \mid v} \bigg(1 + \frac{1}{|p|}\bigg)^{ - 1}
		\; +\;\; O\Big(5^{|\lambda|}		
		q^{|\lambda| + n\slash 2}
		\Big).
		\]
  Note in particular that the sum is empty if $\vert\lambda\vert$ is odd.
	\end{prop}

	\begin{proof} First notice that the coefficient $\mathcal{Z}_{\lambda}(\xi)$ of $\mathrm{t}^{\lambda}$ in $\mathcal{Z}(\mathrm{t}, \xi)$ (see \cref{changing-order-sum}) is a holomorphic function in the open disk $|\xi| < 1,$ except for the factor $\mathcal L(\xi,\chi_{u})$ which picks up a simple pole at $\xi = 1/q$ when $u=1$. Thus, the coefficient of $\xi^n$ in $\mathcal Z_\lambda(\xi)$ is given by \[
		\frac{1}{2 \pi \sqrt{-1}}\;\oint\limits_{|\xi| \, = \, \varepsilon} 
		\frac{\mathcal{Z}_{\lambda}(\xi)}
		{\xi^{n+1}}\, d\xi 
		\] where $0<\varepsilon<1/q$. Consider the integral 
		\[
		\frac{1}{2 \pi \sqrt{-1}}\oint_{\partial \mathscr{A}} 
		\frac{\mathcal{Z}_{\lambda}(\xi)}
		{\xi^{n+1}}\, d\xi
		\] 
		where 
		$
		\mathscr{A} =
		\{\xi \in \mathbf{C} : \varepsilon \le |\xi| \le 
		q^{ -1\slash 2}\}.
		$ 
  Note that 
		\begin{equation} \label{eq: remainder}
			\frac{1}{2 \pi}\left|\;\, 
			\oint\limits_{|\xi| \, = \, q^{
						-1\slash 2}} 
			\frac{\mathcal{Z}_{\lambda}(\xi)}
			{\xi^{n+1}}\, d\xi 
			\, \right|
			\le q^{n/2}
			\cdot \sup_{|\xi| \, = \, q^{
						-1\slash 2}}|\mathcal{Z}_{\lambda}(\xi)|.
		\end{equation} 
		We now bound $\sup_{|\xi| \, = \, q^{-
						1\slash 2}}|\mathcal{Z}_{\lambda}(\xi)|$ using \cref{changing-order-sum}. When $|\xi| = q^{\, - 1\slash 2}$ we have the following simple estimates: {\begin{align*}
	|\mathcal{L}(\xi, \chi_{d})|
	 &= \big|1 - q\xi\big|^{-1} \leq (\sqrt{3} - 1)^{ - 1} \quad (\text{if $d = 1$ and $q \ge 3$}), \\
	|\mathcal{L}(\xi, \chi_{d})|
	 &=  \prod_{\omega \in \Theta_d} \big|1 - \omega\xi\big| \leq 2^{\deg d - 1} \quad (\text{if $d \ne 1$})
\end{align*} 
and 
\[
\prod_{p \mid u} |1-\xi^{2\deg p}| \prod_{p\mid v} |1+\chi_{u}(p)\xi^{\deg p}|  \geq (1-q^{-1/2})^{\deg u + \deg v} \geq (1-q^{-1/2})^{\vert\lambda\vert}. 
\]
Thus, for $|\xi| = q^{\, - 1\slash 2},$ we have \begin{equation*}
	\begin{split}
		|\mathcal{Z}_{\lambda}(\xi)|
		& \le 2 \sum
		\sup_{|\xi| \, = \, q^{\, -
				1\slash 2}} \bigg(\frac{|\mathcal{L}(\xi, \chi_{u})|}
		{\prod_{p \mid u}\big|1 - \xi^{2\deg p}\big|
			\prod_{p \mid v}
			\big|1 + \chi_{u}(p)\xi^{\deg p}\big|}\bigg)\\
		& \leq \bigg(\frac{2}{\sqrt{3} - 1}\bigg)\bigg(\frac{2}{1 - q^{-1/2}}\bigg)^{|\lambda|} 
		\cdot \, \#\{(m_{1}, \ldots,
		m_{r})
		:\text{$\deg m_{i} =
			\lambda_{i}$ 
			for $i=1, \ldots, r$}\}\\
		& \leq \bigg(\frac{2}{\sqrt{3} - 1}\bigg)5^{|\lambda|} q^{|\lambda|}.
	\end{split}
\end{equation*} 
Accordingly the integral \eqref{eq: remainder} is 
$
\ll 
5^{|\lambda|} 
q^{|\lambda| + n\slash 2}
$.} On the other hand, by the residue theorem we have
		\begin{equation*}
			\begin{split}
				\frac{1}{2 \pi \sqrt{-1}}\oint_{\partial \mathscr{A}} 
				\frac{\mathcal{Z}_{\lambda}(\xi)}
				{\xi^{n+1}}\, d\xi
				& = q^{n+1}
				\operatorname{Res}_{\xi=1/q} \mathcal{Z}_{\lambda}(\xi) \\
				& = - \, q^n \,
				\bigg(1 - \frac{1}{q}\bigg)
				\sum_{\substack{
						m_{1} \cdots\, 
						m_{r}  =
						v^2 \\
						\deg m_{i} =
						\lambda_{i}}}
				\; \prod_{p \mid v}
				\bigg(1 + \frac{1}{|p|}\bigg)^{ - 1}
			\end{split}
		\end{equation*}
		and the asymptotic formula follows. This completes the proof.
	\end{proof} 
    
    { \begin{rem} While this does not change the quality of our main result, it is worth noting that the $5^{|\lambda|}$ term in the error term of the asymptotic formula in \cref{asympt-stability} is not optimal. 
	The following argument was kindly pointed out to us by an anonymous referee: 
	the factors depending on $u$ and $v$ can be made subexponential by just observing 
	that the exponent $\deg u+ \deg v$ can be replaced by the number of prime
factors of $u$ plus the number of prime factors of $v$, which can be
bounded using the prime number theorem. For the $L$-values, 
one writes the logarithm as a power series in the variable $\xi$, and bounds 
each coefficient separately, using the Riemann hypothesis to bound the coefficients 
of large powers of $\xi$, and the trivial bound for the Dirichlet character sum to bound the 
coefficients of small powers of $\xi$.
\end{rem}} 

\begin{para}
    We may now prove \cref{thm:stable traces}. 
\end{para}

\begin{proof}[Proof of \cref{thm:stable traces}]
    Let $\mu(n)$ be the usual {M}\"obius function, and let
	$
	i_{n}(q) = n^{-1} \sum_{l\mid n} 
	\mu(n\slash l) q^{l} 
	$
	be the number of monic irreducibles of degree $n$ in $\mathbf{F}_{q}[x].$
	Then we have by \cref{asympt-stability}
	\begin{equation*}
		\begin{split}
			\lim_{n\to\infty}q^{-n} \cdot \mathcal Z_n(\mathrm t) \;\, & = (1-q^{-1})\sum_{m_{1} \cdots\, 
				m_{r}  =
				v^{2}}\;
			\prod_{p \mid v}
			\big(1 + |p|^{-1}\big)^{-1}
			\cdot \, t_{1}^{\deg m_{1}}
			\cdots\,  t_{r}^{\deg
				m_{r}}\\
			&= (1-q^{-1})\prod_{p}\Bigg(1 \, + \,
			\big(1 + |p|^{-1}\big)^{-1}
			\cdot \sum_{k > 0}\sum_{\substack{k_{1}, \ldots,\,
					k_{r} \ge 0\\ 
					k_{1} + \cdots + k_{r} = 2k}}
			\prod_{j = 1}^{r}
			t_{j}^{k_{j}\deg p}\Bigg)\\
			& = (1-q^{-1})\prod_{n > 0} \Big(1 \, + \,
			\big(1 + q^{-n}\big)^{-1}
			\sum_{k > 0}h_{2k}(
			\mathrm{t}^{n})\Big)^{ i_{n}(q)}
		\end{split}
	\end{equation*}
	where
	$
	\mathrm{t}^{n}
	: = (t_{1}^{n}, \ldots,
	t_{r}^{n}).
	$ 
	The result follows by comparing this expression with \cref{alternative form of gen series}.
\end{proof}

{\begin{rem} It is easy to see that \cref{asympt-stability} together with \cref{thmB} implies \cref{thmA}. Indeed, consider the coefficient of $s_{\lambda'}$ in $q^{-n}\mathcal Z_n$, and express it using the Grothendieck--Lefschetz trace formula and Poincar\'e duality in terms of the trace of Frobenius on \'etale homology (see \cref{coh formula for Z}). As $n \to \infty$, the limit on one side of the identity is computed using \cref{asympt-stability}, whereas the other side approaches the coefficient of $s_{\lambda'}$ of the Poincar\'e series of the stable homologies in \cref{thmA} under the substitutions $z \leadsto q^{-1}$ and $s_{\lambda'} \leadsto q^{|\lambda'|/2}s_{\lambda'}$. However, we do not see a way to prove the purity of the stable homology without also computing the stable homology. 
\end{rem}
} 

\subsection{Asymptotics for moments from homological stability}\label{subsec: moments from hom stab}

     \newcommand{\Pg}{\mathscr{P}_{2g+1}}
\newcommand{\uT}{\overline{\Theta}}

\begin{para}We will prefer to think about the representation theory of the usual symplectic groups, rather than the odd symplectic groups. For this reason, we work consistently in this section with the moduli spaces $\H_g^{1,0}$ and $\H_{g,1,0}$ rather than $\H_g^{0,1}$ and $\H_{g,0,1}$. We will also find it convenient to consider the \emph{unitarized} conjugacy class of Frobenius. For $d \in \Pg$ we write $\uT_d$ for $q^{-1/2}\Theta_d$, which we may think of as a conjugacy class in $\operatorname{USp}(2g)$. Then $\uT_d$ consists of the eigenvalues of Frobenius on the stalk above the point $[C_d]$ of $\mathbb V(\tfrac 1 2)$, the half-integral Tate twist  of the standard local system.\end{para}

\begin{defn}
	Introduce the notation 
 $$ \mathrm{tr}_\lambda(g) = \frac{1}{q^{2g+1}}\sum_{d\in \Pg}s_{\langle\lambda\rangle}(\uT_d), $$
 where $s_{\langle\lambda\rangle}$ denotes a symplectic Schur function.  We also set 
 $$ T_\lambda = \lim_{g\to\infty} \mathrm{tr}_\lambda(g).$$
\end{defn}

\begin{para}Applying the Grothendieck--Lefschetz trace formula as in \cref{coh formula for Z} we have 
    $$\mathrm{tr}_\lambda(g) = \sum_k (-1)^k \mathrm{Tr}(\mathrm{Frob}_q \mid H_k(\H_{g}^{1,0} \otimes \overline{\mathbf{F}}_q, \mathbb V_\lambda(\tfrac{\vert\lambda\vert} 2)))$$
    and 
    $$T_\lambda = \sum_k (-1)^k \mathrm{Tr}(\mathrm{Frob}_q \mid H_k(\H_{\infty}^{1,0} \otimes \overline{\mathbf{F}}_q, \mathbb V_\lambda(\tfrac{\vert\lambda\vert} 2))).$$
    Here we use Grothendieck--Lefschetz for stacks \cite[Corollary 6.4.10]{Beh}.
\end{para}


\begin{para}In the previous subsection we used crucially that
$$ \mathcal Z_n(t_1,\ldots,t_r)= \sum_{d\in \mathscr P_n} \prod_{i=1}^r \mathcal L(-t_i,\chi_d);$$
this was a consequence of the dual Cauchy identity. As explained in \S\ref{no symplectic cauchy} there is a perfect symplectic analogue of the dual Cauchy identity when one considers symmetric functions in \emph{infinitely} many variables, but it does not specialize in a useful way to finitely many variables. As an adequate substitute we will use an identity due to Jimbo--Miwa, see 
\cite{jimbo-miwa}, \cite[Theorem 3.8.9.3]{How}, and \cite[Lemma 4]{bump-gamburd}.
\end{para}
\begin{defn}\label{defn lambda-dag}Let $\lambda \subseteq (r^{g})$ be a partition, i.e.\ $\lambda_1 \leq r$ and $\length(\lambda)\leq g$. We set 
	$
	\lambda^{\dag} : = 
	(g - \lambda_{r}' \ge \cdots 
	\ge g - \lambda_{1}'), 
	$ 
	where the non-zero integers among $\lambda_{j}',$ 
	$1 \le j \le r,$ are the parts of the conjugate partition $\lambda'$ of $\lambda.$ If we want to make the dependence of $\lambda^\dag$ on $g$ explicit, we write $\lambda^\dag(g)$.
\end{defn}

\begin{thm}[Jimbo--Miwa]\label{eq: Howe-duality}
	\[
	\prod_{i=1}^g \prod_{j=1}^r (x_i+x_i^{-1}+t_j+t_j^{-1}) = \sum_{\lambda \subseteq (r^g)} s_{\langle \lambda\rangle} (x_1^\pm,\ldots,x_g^\pm) s_{\langle \lambda^\dag\rangle}(t_1^\pm,\ldots,t_r^\pm).
	\]
\end{thm}

\begin{cor}\label{geom-expr-moments}  For complex variables 
	$
	t_{1}, \ldots, t_{r},
	$ 
	we have 
	\[
	\frac{1}{q^{2g+1}}
	\sum_{d \in \Pg} 
	\prod_{j = 1}^{r} \mathcal{L}(q^{-1/2}t_{j}, \chi_{d})
	\, = \, (t_{1}\cdots \,  t_{r})^{g}
	\!\!\sum_{\lambda \subseteq (r^{g})}
	\mathrm{tr}_{\lambda}(g)
	s_{\langle \lambda^{\dag} \rangle}\big( t_{1}^{\pm 1}\!, \ldots,  t_{r}^{\pm 1}\big).
	\] 
	
\end{cor}

\begin{proof}First rewrite \cref{eq: Howe-duality} as$$
	\prod_{i=1}^g \prod_{j=1}^r (1+t_jx_i)(1+t_jx_i^{-1}) = (t_1\cdots t_r)^g\sum_{\lambda \subseteq (r^g)} s_{\langle \lambda\rangle} (x_1^\pm,\ldots,x_g^\pm) s_{\langle \lambda^\dag\rangle}(t_1^\pm,\ldots,t_r^\pm).$$
	Then set $\{x_1^\pm,\ldots,x_g^\pm\}$ equal to $\uT_d$, and observe that 
	$\mathcal L(-q^{-1/2}t_j,\chi_d) = \prod_{\omega \in \uT_d} (1+t_j \omega)$. Finally, take the product over $r$ and sum over $d$, and observe that the result is invariant under $t_j \mapsto -t_j$ since $\mathrm{tr}_\lambda(g)$ vanishes for $\vert\lambda\vert$ odd. 
\end{proof}

\begin{cor}
	\[
	\frac{1}{q^{2g+1}}
	\sum_{d \in \Pg} 
	{L}(\tfrac 1 2, \chi_{d})^r
	\, = \, 
	\sum_{\lambda \subseteq (r^{g})}
	\mathrm{tr}_{\lambda}(g) \, \dim V_{\lambda^\dag},
	\] where $V_{\lambda^\dag}$ denotes an irreducible representation of $\mathrm{Sp}(2r)$. 
\end{cor}
\begin{proof}
	Set $t_1=t_2=\ldots = 1$ in the preceding corollary.
\end{proof}

\begin{para}\label{dimension identity}
	We will in particular make use of the identity
	$$ \sum_{\lambda \subseteq (r^g)} \dim V_\lambda^{\Sp(2g)} \dim V_{\lambda^\dag}^{\Sp(2r)} = 4^{gr}$$
	obtained by setting all $x_i$ and $t_i$ equal to $1$ in \cref{eq: Howe-duality}. Here the superscripts $\Sp(2g)$, resp.\ $\Sp(2r)$, indicate which groups $V_\lambda$, resp.\ $V_{\lambda^\dag}$, are representations of.
\end{para}

\begin{para}\label{defn lambda-dag 2}If $\length(\lambda)>g$ then we can still define $\lambda^\dag$ as in \cref{defn lambda-dag}, but $\lambda^\dag$ is no longer dominant. We extend the definition of $s_{\langle \lambda^\dag\rangle}$ to this case by plugging $\lambda^\dag$ into the Weyl character formula. Let 
$$ \alpha := \{\lambda^\dag + \rho\}-\rho,$$
where $\{\cdot\}$ denotes the unique dominant weight conjugate to a given weight and $\rho$ is the sum of the fundamental weights. Then $s_{\langle\lambda^\dag\rangle}$ vanishes if $\alpha$ is not dominant, and $s_{\langle\lambda^\dag\rangle}$ is $\pm$ the character of $V_{ \alpha}$ if $\alpha$ is dominant. \end{para}

\begin{defn}Let  
$$ \mathcal Q_1(t_1,\ldots,t_r;g,q) = (t_1\cdots t_r)^g\sum_{\lambda_1 \leq r} T_\lambda \, s_{\langle \lambda^\dag(g)\rangle}(t_1^\pm,\ldots,t_r^\pm).$$
We will see shortly that the sum converges. 
\end{defn}
\begin{para}
One can show using the Weyl dimension formula that
$$ \dim V_\lambda^{\Sp(2r)} = \frac{\prod_{1\leq i<j\leq r} (\mu_i^2-\mu_j^2) \prod_{i=1}^r \mu_i}{(2r-1)!!}$$ 
where $\mu_i = \lambda_i+r-i+1$. It is straightforward to see from this expression that for each $\lambda$ with $\lambda_1 \leq r$, the function
$$ g \mapsto \dim V_{\lambda^\dag(g)}$$
is a polynomial in $g$ of degree $r(r+1)/2$. In particular, $\mathcal Q_1(1,\ldots, 1;g,q)$ is a polynomial in $g$ of degree $r(r+1)/2$, for any fixed $q$. 
\end{para}

\begin{lem}\label{fuks}
	Let $V$ be a representation of the braid group $\beta_n$. Then $\dim H_k(\beta_n,V) \leq \binom {n-1} k \dim(V)$.
\end{lem}

\begin{proof}By Poincar\'e duality, it suffices to bound the dimension of $H^{2n-k}_c(\mathrm{Conf}_n(\C),V)$. 
	Use the cellular chain complex associated to the Fuks stratification of $\mathrm{Conf}_n(\C)$ (see \cite{FN62, Fuk70} and \cite[Theorem 4.3]{ETW17}). It has exactly $\binom {n-1} k$ cells of dimension $2n-k$. 
\end{proof}

\begin{lem}\label{dim lemma}Let $t_1,\ldots,t_r$ be on the unit circle, and consider a partition $\lambda$ with $\lambda_1 \leq r$. Then $|s_{\langle \lambda^\dag(g) \rangle}(t_1^\pm,\ldots,t_r^\pm)| \leq 4^{rg} \dim S^{\lambda'}(\C^{2r})$. 
\end{lem}

\begin{proof}We have $\lambda^\dag = (g-\lambda'_r,\ldots,g-\lambda_1')$. Let $\alpha$ be as in \S\ref{defn lambda-dag 2}. Then we can write $\alpha = \beta + \gamma$ with $\beta \subseteq (g^r)$ and $\gamma \subseteq \lambda'$. Thus,
	$$ |s_{\langle \lambda^\dag(g) \rangle}(t_1^\pm,\ldots,t_r^\pm)|\leq \dim V_{ \alpha } \leq \dim V_{ \beta} \dim V_{ \gamma } \leq \dim V_{ (g^r) } \dim V_{ \lambda' }\leq 4^{gr} \dim S^{\lambda'}(\C^{2r})$$
	using in the last step \S\ref{dimension identity} for the first factor, and that $V_{\lambda'}$ is a summand of $S^{\lambda'}(\C^{2r})$ for the second factor. 
\end{proof}

\begin{lem}\label{good bound}For all $\varepsilon > 0$ and $N \in \N$ there is a bound
	$$ \sum_{\vert\lambda\vert=w} \dim H_k(\H_{\infty}^{1,0},\V_\lambda) \dim S^{\lambda'}(\C^{N}) \ll_{N,\varepsilon} 2^{\varepsilon k}.$$
\end{lem}

\begin{proof}
	Consider the generating function of stable Poincar\'e series 
 \[ \Exp\Big(\, z^{-1} \, \Log \Big(z + \sum_{k\geq 0} z^k h_{2k}(\mathrm t)\Big) - h_2(\mathrm t) - 1\Big).\]According to \cref{plethystic formula for poincare series}, the left-hand side of \cref{good bound} is the coefficient of $z^k t^w$ of the function  obtained by setting $t_1=\ldots=t_{N}=t$, $t_i=0$ for $i>N$, in this generating series. By \cref{R1-estimate} and the comparison explained in \S\ref{comparison of Q} the series converges to a holomorphic function on the region\footnote{For our purposes it would be enough to know convergence in the region $\vert z\vert<1$, $\vert tz\vert < 1$, which is simpler to show than \cref{R1-estimate}.} $\vert z\vert<1$, $\vert t^4z\vert < 1$. Now apply Cauchy's inequality, i.e.\ extract the coefficient of $z^k t^w$ using Cauchy's integral formula applied to the circles $\vert z \vert = 2^{-\varepsilon}$, $\vert t \vert = 1$. \end{proof}

\begin{defn}
	Let $\theta : \R \to \R$ be a monotone increasing function. We say that $\theta$ is a \emph{uniform stability bound} if for each $n \in \N$,
	$$ H_k(\beta_n,V_\lambda) \to H_k(\beta_{n+1},V_\lambda)$$
	is an isomorphism for all $k \leq \theta(n)$, for all $\lambda$.	
\end{defn}

\begin{para}\label{theta bound} Let $\theta$ be a uniform stability bound. Then $H_k(\beta_\infty,V_\lambda) \cong H_k(\beta_{n},V_\lambda)$
for all $k \leq \theta(n)$. But observe that $V_\lambda$ is zero as a representation of $\beta_{n}$ if and only if $n\leq 2\length(\lambda)-1$. Thus, $H_k(\beta_\infty,V_\lambda) =0$ for $k\leq \theta(2\length(\lambda)-1)$. In \cref{1/2} we proved that $H_k(\beta_\infty,V_\lambda) =0$ for $k<\length(\lambda)/2$, and this bound is sharp. Thus, a uniform stability bound must satisfy $\theta(4k-1)<k$. Our conjecture is that this bound is close to sharp. 
\end{para}

\begin{conj}\label{stab conj 2}
    There exists a uniform stability bound $\theta(n)=\tfrac 1 4 n - c$, where $c\geq 0$ is an absolute constant.
\end{conj}

\begin{para}The following theorem implies \cref{thmC} from the introduction. \end{para}

\begin{thm}\label{uniform stab implies moments}Suppose that $\theta$ is a uniform stability bound. Then for any $t_1,\ldots,t_r$ on the unit circle
	$$\frac{1}{q^{2g+1}}\sum_{d\in \Pg} \prod_{i=1}^r \mathcal L(q^{-1/2}t_i,\chi_d) = \mathcal Q_1(t;g,q) + O_\theta\Big(4^{g(r+1)}q^{-\theta(2g+1)/2}\Big),$$
	the implied constant depending 	upon $\theta$ in the $O$-symbol being explicitly computable.
\end{thm}

\begin{proof}
	The left-hand side of \cref{uniform stab implies moments} equals 
	$$ (t_1\cdots t_r)^g\sum_{\lambda \subseteq (r^g)}\sum_k (-1)^k \mathrm{Tr}(\mathrm
	{Frob}_q \mid H_k(\H_{g}^{1,0},\mathbb V_\lambda(\tfrac{\vert\lambda\vert}{2}))) s_{\langle\lambda^\dag\rangle}(t_1^\pm,\ldots,t_r^\pm)$$
	and $\mathcal Q_1(t,g)$ equals 
	$$ (t_1\cdots t_r)^g\sum_{\lambda_1 \leq r}\sum_k (-1)^k \mathrm{Tr}(\mathrm
{Frob}_q \mid H_k(\H_{\infty}^{1,0},\mathbb V_\lambda(\tfrac{\vert\lambda\vert}{2}))) s_{\langle\lambda^\dag\rangle}(t_1^\pm,\ldots,t_r^\pm).$$
Subtracting the two, all terms with $k \leq \theta(2g+1)$ cancel by homological stability. To estimate the difference, we first consider the contribution from the {unstable} homology. Using the Deligne bounds on Frobenius eigenvalues from \cite{Del80}, \cref{fuks}, and the identity of \S \ref{dimension identity}, we get  \begin{align*}
	 & \left \vert (t_1\cdots t_r)^g\sum_{\lambda \subseteq (r^g)}\sum_{k >\theta(2g+1)} (-1)^k \mathrm{Tr}(\mathrm
	{Frob}_q \mid H_k(\H_{g}^{1,0},\mathbb V_\lambda(\tfrac{\vert\lambda\vert}{2}))) s_{\langle\lambda^\dag\rangle}(t_1^\pm,\ldots,t_r^\pm)\right\vert \\
	\leq & \sum_{\lambda \subseteq (r^g)}\sum_{k >\theta(2g+1)} q^{-k/2} \dim H_k(\H_{g}^{1,0},\V_\lambda) \dim V_{\lambda^\dag}^{\Sp(2r)} \\
	\leq & \underbrace{\sum_{k > \theta(2g+1)} q^{-k/2} \binom{2g}{k}}_{\ll\, 2^{2g}\, q^{-\theta(2g+1)/2}} \underbrace{\sum_{\lambda \subseteq (r^g)} \dim V_{\lambda}^{\Sp(2g)} \dim V_{\lambda^\dag}^{\Sp(2r)}}_{= \, \,4^{gr}} \ll 4^{g(r+1)}q^{-\theta(2g+1)/2} .
\end{align*}
We now have to estimate the contribution from the stable homology, i.e.\ to bound the quantity
$$ (t_1\cdots t_r)^g\sum (-1)^k \mathrm{Tr}(\mathrm
{Frob}_q \mid H_k(\H_{\infty}^{1,0},\mathbb V_\lambda(\tfrac{\vert\lambda\vert}{2}))) s_{\langle\lambda^\dag\rangle}(t_1^\pm,\ldots,t_r^\pm)$$
where we sum over $k$ and $\lambda$ such that $\lambda_1 \leq r$ and \emph{either} $\length(\lambda)>g$ or $k >\theta(2g+1)$ (or both). 

We split this sum into two: first when $\vert\lambda\vert\leq 2g$, then when $\vert \lambda\vert > 2g$. Recall that the stable homology is pure of weight $-2k$ in degree $k$. We claim that the first sum can be bounded by
$$ \sum_{\substack{\vert\lambda\vert \leq 2g \\ \lambda_1 \leq r}} \sum_{k >\theta(2g+1)} q^{-k} \dim H_k(\H_{\infty}^{1,0},\V_\lambda) s_{\langle \lambda^\dag \rangle}(1^\pm,\ldots,1^\pm).$$
Indeed, if $\lambda \subseteq (r^g)$ it is clear that we can start the summation from $k>\theta(2g+1)$. When $\length(\lambda)>g$ we instead observe that $H_k(\H_{\infty}^{1,0},\V_\lambda)=0$ for $k\leq\theta(2l(\lambda)-1)$ by \S\ref{theta bound}, which is at least $\theta(2g+1)$ by monotonicity of $\theta$. Now using \cref{dim lemma} and \cref{good bound} we find that 
\begin{align*}
	&\sum_{\substack{\vert\lambda\vert \leq 2g \\ \lambda_1 \leq r}} \sum_{k >\theta(2g+1)} q^{-k} \dim H_k(\H_{\infty}^{1,0},\V_\lambda) s_{\langle \lambda^\dag \rangle}(1^\pm,\ldots,1^\pm)  \\
	\leq  4^{gr} & \sum_{w=0}^{2g} \sum_{\vert\lambda\vert=w} \sum_{k >\theta(2g+1)} q^{-k} \dim H_k(\H_{\infty}^{1,0},\V_\lambda) \dim S^{\lambda'}(\C^{2r}) \\
	\ll 4^{gr} &\sum_{w=0}^{2g} \sum_{k>\theta(2g+1)} q^{-k} 2^{\varepsilon k} \ll 4^{gr}(2g+1) 2^{\varepsilon \theta(2g+1)} q^{-\theta(2g+1)} 
\end{align*}
for all $\varepsilon>0$. The second sum is treated in the exact same way, but we use instead that  $H_k(\H_{\infty}^{1,0},\V_\lambda)=0$ for $k<\vert\lambda\vert/4$ (\cref{1/4}). We obtain the estimate
\begin{align*}
	&\sum_{\substack{\vert\lambda\vert > 2g \\ \lambda_1 \leq r}} \sum_{k \geq 
	\vert\lambda\vert/4} q^{-k} \dim H_k(\H_{\infty}^{1,0},\V_\lambda) s_{\langle \lambda^\dag \rangle}(1^\pm,\ldots,1^\pm)  \\
	\leq  4^{gr} & \sum_{w>2g} \sum_{\vert\lambda\vert=w} \sum_{k \geq w/4} q^{-k} \dim H_k(\H_{\infty}^{1,0},\V_\lambda) \dim S^{\lambda'}(\C^{2r}) \\
	\ll 4^{gr} &\sum_{w>2g} \sum_{k\geq w/4} q^{-k} 2^{\varepsilon k} \ll 4^{gr} 2^{\varepsilon g/2} q^{-g/2}. \qedhere
\end{align*}\end{proof}

\begin{rem}
    As stated in \cref{stab conj 2}, we expect $\theta(n)\approx n/4$, which leads to an error term in \cref{uniform stab implies moments} which is roughly $O(q^{-g/4})$. On the other hand, as we explain in the next section of the paper, one expects an asymptotic formula with successive terms $\mathcal T_n$, in which the first term $\mathcal T_1$ accounts for the contribution from stable homology. One also expects the error term in the asymptotic formula of \cref{uniform stab implies moments} to be comparable to the second term $\mathcal T_2$, which is approximately $O(q^{-g/2})$. Looking through the proof of \ref{uniform stab implies moments}, the discrepancy comes from the contributions from the homology of $\beta_n$ which is \emph{just outside} the stable range. In the stable range we know that Frobenius eigenvalues on $H_k$ are $q^{-k}$, but outside the stable range we only have the Deligne bounds, giving the upper bound $q^{-k/2}$. Our expectation is that in the regime just outside the stable range we have \emph{secondary stability} \cite{gkrw-secondary}, and that just as the stable homology is pure of lowest possible weight, also the homology in the range where secondary stability holds has weights (close to) the lowest possible. This would account for the mismatch of error terms in \cref{uniform stab implies moments} and the next section.  
\end{rem}

\begin{rem}
    It seems likely that a version of \cref{uniform stab implies moments} holds also when summing over polynomials of even degree; indeed, our homological stability results are valid for braid groups with both an odd or even number of strands. Pursuing this will require a more detailed study of the representation theory of odd symplectic groups than we want to include here.
\end{rem}
	
	\section{Moments and stability predictions} \label{predictions}

 \begin{para}We now turn our attention to the important problem of determining the asymptotics of the traces $\tr_{\lambda}(g)$ as the genus $g$ grows, for all partitions $\lambda \subseteq (r^{g})$ (i.e., all partitions whose Young diagram fit inside a rectangle of $g$ rows and $r$ columns), where $r$ is a fixed positive integer. Since $\tr_{\lambda}(g) = 0$ when $|\lambda|$ is odd, we can consider only the partitions of even weight. When $|\lambda| \le 2g - 1,$ an asymptotic formula for $\tr_{\lambda}(g)$ can be obtained by combining Proposition \ref{asympt-stability} with the results of Rudnick \cite[Theorem 1]{Rud} and Roditty-Gershon \cite[Theorem 1.1]{R-G}. Obtaining a similar result in a wider range of $\lambda$ seems to be a challenging problem. In fact, 
	to the best of our knowledge, at present, there is not even a prediction as to what the asymptotic behavior of these traces of {F}robenius should be, when $|\lambda|$ is quite large compared with $g.$ The only existing result for general partitions comes from the seminal work of Katz and Sarnak, see \cite[Theorem 10.8.2]{KaS}: 
\end{para}
 \begin{thm}[Katz--Sarnak]
     There exist positive constants $A(g)$ and $C(g)$ such that 
		\[ 
		|\tr_{\lambda}(g)| \le C(g)(\dim\,  V_{\lambda})
		q^{- 1 \slash 2}
		\] 
		for all $\lambda$, as long as 
		$
		q = \#\mathbf{F}_{q} \ge A(g).
		$ 
		Here $V_{\lambda}$ is the irreducible representation of $\mathrm{USp}(2g)$ with highest weight $\lambda.$  
 \end{thm}

\subsection{The general conjecture}
 
 \begin{para}
     To set up some notation, let $A = (A_{ij})_{1 \le i, j \le r + 1}$ be the symmetric matrix whose entries are 
	\[
	A_{ij} = 
	\begin{cases} 
		2 &\mbox{if $i = j$} \\
		-1 &\mbox{if $i \ne j$ and either $i = r + 1$ or $j = r + 1$} \\
		0 & \mbox{otherwise.}
	\end{cases} 
	\] 
	To $A$ one can attach a {K}ac-{M}oody {L}ie algebra $\mathfrak{g}(A)$ with the set 
	of roots $\Delta,$ and 
	$
	\{\alpha_{i}\}_{1 \le i, j \le r + 1} \subset \Delta
	$ 
	a system of simple roots; the reader interested in the general theory of {K}ac-{M}oody algebras can consult \cite{Kac}. We have that $\det A = -\,  2^{r - 1}(r - 4).$ By \cite[Theorem 4.3]{Kac}, one finds at once that $A$ is of finite, affine, or indefinite type when $r \le 3,$ $r = 4,$ 
	or $r \ge 5,$ respectively. For example, the {D}ynkin diagram of $A$ is $D_{4}$ 
	when $r = 3$ and $\tilde{D}_{4}$ when $r = 4;$ for $r \ge 5,$ the {D}ynkin diagram has a star-like shape. Let $W_{r}$ denote the {W}eyl group of $\mathfrak{g}(A),$ and let $\Delta^{\mathrm{re}}_{+}\subset \Delta$ be the set of positive real roots. For 
	$
	\alpha \in \Delta^{\mathrm{re}}_{+},
	\alpha = \sum_{i} k_{i}
	\alpha_{i}
	$ 
	($k_{i} \in \mathbf{N}$), let 
	$d(\alpha) = \sum_{i} k_{i}$ denote the height of $\alpha,$ and for each positive integer $n,$ let $\Phi_{n}$ be the subset of $\Delta^{\mathrm{re}}_{+}$ with $k_{r + 1} = n;$ 
	by \cite[Lemma 2.2]{DT}, $\Phi_{n}$ is a finite set. Finally, for each $\alpha \in \Phi_{n},$ we fix a {W}eyl group element $w_{\aa}$ in reduced form of length 
	$l(w_{\aa})$ such that 
	$
	w_{\aa}(\aa) = \aa_{r + 1}. 
	$ 
	This choice is possible since all real roots are in the orbit 
	$
	W_{r}\aa_{r + 1}, 
	$
	see \cite[Lemma 2.3]{DT}. 
	
 \end{para}
	\begin{para}
	Now consider the $r$th moment-sums of quadratic $L$-functions, that is, the coefficients of the moment generating function 
	$
	\mathcal{Z}(\mathrm{t}, \xi)
	$. Initially, we take the variables to be 
	$
	t_{i} 
	= q^{- s_{i}},
	$ 
	for $i = 1, \ldots, r.$ In \cite{DT}, it has been conjectured that, for 
	$
	\Re(s_{i}) = \frac{1}{2}
	$ 
	and $N \ge 1$ an integer, we have 
	\begin{equation} \label{eq: asympt-moments} 
		\sum_{d \in \Pg}
		\prod_{i = 1}^{r} L(s_{i}, \chi_{d}) 
		\, = \sum_{n \le N} Q_{n, r}(\mathrm{s}'; 2g + 1, q) \, + \, O_{\Theta, \,  q, \, r}
		\left(q^{(2g + 1)(1 + \Theta)\slash 2}\right)
	\end{equation} 
	for any 
	$
	(N + 1)^{-1}  < \Theta  < N^{-1}
    $. 
    {This is the shifted variant of \cref{conj: Moment-conjecture-over-ff} that was mentioned in \S\ref{traces-complete-asympt}.} 
    \end{para}
    \begin{para}
	The terms 
	$
	Q_{n, r}(\mathrm{s}'; 2g + 1, q),
	$ 
	$
	\mathrm{s}' 
	: = (s_{1}, \ldots, 
	s_{r}),
	$ 
	have a somewhat complicated expression of the form 
	\begin{equation*} \label{eq: Qn} 
		Q_{n, r}(\mathrm{s}'; 2g + 1, q)
		= n^{-1}
		\cdot\sum_{\alpha \in \Phi_{n}} \bigg\{\sum_{a\in \{1, \, \theta_{0} \}}\;
		\sum_{\zeta_{a}^{n} = \mathrm{sgn}(a)}
		\frac{\Gamma_{w_{\alpha}}
			(a; \zeta_{a})}{2^{l(w_{\alpha})}}
		S_{\alpha}(\mathrm{s}', \zeta_{a})
		\zeta_{a}^{2g + 1}\bigg\}
		\, q^{\frac{(2g + 1)(d(\alpha) + 1 - 2\alpha(\mathrm{s}
					'))}{2n}}
	\end{equation*} 
	where $\theta_{0}$ is any fixed non-square in $\mathbf{F}_{q}^{\times},$ and $\mathrm{sgn}(a) = 1$ or $-1$ according as 
	$a = 1$ or $\theta_{0}.$ In addition, for 
	$
	\alpha = \!\sum k_{i}
	\alpha_{i} \in \Delta^{\mathrm{re}}_{+},
	$ 
	we set 
	$
	\alpha(\mathrm{s}')
	: = \sum_{i \le r} k_{i}
	s_{i}.
	$ 
	The function 
	$
	S_{\aa}(\mathrm{s}', \zeta_{a})
	$ 
	decomposes as 
	\[ 
	S_{\alpha}(\mathrm{s}', \zeta_{a})
	= \prod_{p} S_{p}^{w_{\alpha}}
	(\mathrm{s}', \zeta_{a})
	\] 
	over all monic irreducibles, with 
	$
	S_{p}^{w_{\aa}}
	(\mathrm{s}', \zeta_{a})
	$ 
	given in \cite[Sec. \!5.2]{DT}. An explicit formula for 
	$
	2^{- l(w_{\aa})}\Gamma_{\!w_{\aa}}\!(a; \zeta_{a})
	$ 
	can be found in \cite[Sec. \!5.1]{DT}. When $r = 1, 2,$ the sets $\Phi_{n}$ 
	($n \ge 2$) are all empty, and when $r = 3$ the same is true, but for $n \ge 3.$ 
	If $r \ge 4,$ then $\Phi_{n}$ is non-empty for all $n \ge 1.$ Consequently, $Q_{n, 1} = Q_{n, 2} = 0$ for $n \ge 2,$ and $Q_{n, 3} = 0$ for $n \ge 3.$ For our purposes here, we shall only need the first two terms of the asymptotic formula of \eqref{eq: asympt-moments}, when $r \ge 3$. \end{para}
	
	\begin{para}
	Prior to \cite{DT}, {A}ndrade and {K}eating \cite[Conjecture 5]{AK} made the weaker conjecture that 
	\[ 
	\sum_{d \in \Pg} L(\tfrac{1}{2}, \chi_{d})^{r} 
	= \sum_{d \in \Pg}Q_{r}(\log_{q} |d|)(1 + o(1))  
	\] 
	where $Q_{r}$ is an explicit polynomial of degree $r(r + 1)\slash 2.$ Note that $Q_{r}(\log_{q} |d|) = Q_{r}(2g + 1)$ for all 
	$d \in \Pg,$ and thus the right-hand side is $\sim \#\Pg \cdot Q_{r}(2g + 1)$ as $g \to \infty.$ One checks the compatibility between this term and the main term 
	$Q_{1, r}$ in \eqref{eq: asympt-moments}. This conjecture has been established only for $r \le 3$ (with explicit power-saving error terms) in \cite{HR} for $r = 1,$ and \cite{Florea1, Florea2} for $r \le 3;$ in \cite{Florea3}, {F}lorea also obtained a weaker version of the asymptotic formula for the fourth moment. The only results towards \eqref{eq: asympt-moments} are the asymptotics established in \cite{Dia}, for the third moment, and in \cite{DPP}, for a weighted fourth moment. 
	\end{para}

	\begin{para}
	The connection between the moment-sums of quadratic $L$-functions and the traces $\tr_{\lambda}(g)$ is given by \cref{geom-expr-moments}. See also \cite[Theorem 6.4]{DV}.	To obtain information about the traces $\tr_{\lambda}(g),$ we shall apply to Equation \eqref{eq: asympt-moments} the {W}eyl integration formula, which, for convenience, we now 	recall, see \cite{Bump}: \end{para}
	
 \begin{thm}[Weyl integration formula]Let $\mathbf{G}$ be a compact connected {L}ie group and $\mathbf{T}$ a maximal torus. 
		Let $du$ and $dt$ be {H}aar measures on these groups, normalized so that their volumes are $1.$ Then, for any $\mathbf{C}$-valued continuous central function $f$ on $\mathbf{G},$ we have 
		\[
		\int_{\mathbf{G}} f(u)\, du 
		= \frac{1}{\#W}\!\int_{\mathbf{T}} f(t) |\delta(t)|^{2}\, dt
		\] 
		where $W = N(\mathbf{T})\slash \mathbf{T}$ is the {W}eyl group of $\mathbf{G},$ and 
		\[
		\delta(t) = \prod_{\alpha > 0} 
		\left(e^{\alpha(t)\slash 2} - e^{-\alpha(t)\slash 2}\right)
		\] 
		the product being over the positive roots of $\mathbf{G}.$ 
	\end{thm}

	\begin{para}We apply {W}eyl's formula when 
	$
	\mathbf{G} = \mathrm{USp}(2r), 
	\mathbf{T} = 
	\left\{\operatorname{diag} \left(e^{\pm \sqrt{-1}\theta_{1}}, \ldots, 
	e^{\pm \sqrt{-1}\theta_{r}} \right) : 
	\theta_{j} \in \mathbf{R}\right\},
	$ 
	and 
	\[
	f(u)  =
	\sum_{\lambda \subseteq (r^{g})}
	\tr_{\lambda}(g)
	s_{\langle \lambda^{\dag} \rangle}(u)
	\] 
	where 
	$
	s_{\langle \lambda^{\dag} \rangle}
	$ 
	is the character of the irreducible representation of $\mathrm{USp}(2r)$ with highest weight $\lambda^{\dag}$. 
    Since each $Q_{n, r}(\mathrm{s}'; 2g + 1, q)$ in 
\S\ref{eq: Qn} is a function of $q^{\pm s_{i}}$, $i = 1, \ldots, r$, it is convenient to substitute 
$
q^{- s_{i}} \! 
= q^{- 1\slash 2}
t_{i},
$ 
hence $|t_{i}| = 1$ for all $i;$ we denote the resulting functions multiplied by the normalizing factor $q^{-2g-1}$ 
by $\mathcal{Q}_{n}(t; g, q)$, that is, 
\[
Q_{n, r}(\mathrm{s}'; 2g + 1, q) = 
q^{2g+1}\mathcal{Q}_{n}(q^{\frac{1}{2} - s_{1}},\dots, 
q^{\frac{1}{2} - s_{r}}; g, q).
\]For $n \ge 1,$ 
	we define 
	\[ 
	\mathcal{T}_{n}(\lambda^{\dag}) = \frac{1}{\#W}\!\int_{\mathbf{T}} \mathcal{Q}_{n}(t; g, q)
	\psi(t)^{-g}
	\overline{s_{\langle \lambda^{\dag} \rangle}(t)}\, |\delta(t)|^{2}\, dt
	\] 
	where $\psi : \mathbf{T} \to \mathbf{C}^{\times}$ is the character defined by 
	$
	\psi(t) = t_{1}
	\cdots \,  t_{r}, 
	$ and $ 
	t = \operatorname{diag}(t_{1}^{\pm 1}, \ldots, t_{r}^{\pm 1}). 
	$ 
	Then \cref{geom-expr-moments} and the conjectural asymptotic formula 
	\eqref{eq: asympt-moments}, combined with the orthonormality of the characters 
	$
	s_{\langle \lambda^{\dag} \rangle},
	$ 
	lead to the following: 
 \end{para}
	\begin{conj} \label{conj-traces} For every $\lambda \subseteq (r^{g})$ we have 
		\[
		\tr_{\lambda}(g)
		= \sum_{n \le N} \mathcal{T}_{n}(\lambda^{\dag})\, + \, O_{\Theta, \,  q, \, r}
		\left(q^{(2g + 1)(\Theta-1)\slash 2}\right)\!.
		\]
	\end{conj}

	\begin{para}While the asymptotic formula \eqref{eq: asympt-moments} has a regular behavior in the sense that 
	$
	Q_{n, r}(\mathrm{s}'; 2g + 1, q)
	$ 
	dominates the term 
	$
	Q_{n+1, r}(\mathrm{s}'; 2g + 1, q)
	$ 
	for all $n \ge 1,$ as either $g$ or $q$ gets large, the asymptotics of 
	$\tr_{\lambda}(g)$ behave quite differently depending upon various regimes of the 
	ratio $|\lambda|\slash g.$ In what follows, we begin to investigate the asymptotic behavior of 
	$\tr_{\lambda}(g)$ by analyzing closely the first two terms 
	$
	\mathcal{T}_{1}(\lambda^{\dag})
	$ 
	and 
	$
	\mathcal{T}_{2}(\lambda^{\dag})
	$ 
	in the above conjecture. 
    {Our analysis of the first term, in particular, also suggests, based on \cref{conj-traces} and \cref{stable-traces-2}, that the family $\{\H_{g}^{1,0}\}$ satisfies homological stability for the twisted coefficients 
$
\mathbb{V}_{\lambda}
$ 
in a range independent of $\lambda$, as in \cref{stab conj 2}.} 
\end{para}
\subsection{The first term}	
	\begin{para}
	To study 
	$
	\mathcal{T}_{1}(\lambda^{\dag}),
	$ 
	we first note that the set 
	$\Phi_{1}$ consists of the roots of the form 
	$
	\alpha = \sum_{i \le r} k_{i}
	\alpha_{i} +
	\alpha_{r + 1},
	$ 
	with $k_{i}  = 0$ or $1$ for all $i = 1, \ldots, r.$ From \cite[Section 7]{DT}, one finds easily that 
	\[  
	\mathcal{Q}_{1}(t; g, q)\psi(t)^{-g} =
	(1-q^{-1})\sum_{\epsilon} \frac{R_{1}(q^{- 1\slash 2}t_{1}^{\epsilon_{1}}, \ldots, q^{- 1\slash 2}t_{r}^{\epsilon_{r}})}{\prod_{i = 1}^{r}
		t_{i}^{\epsilon_{i}g}
		\prod_{1 \le i \le j \le r} (1 - t_{i}^{\epsilon_{i}}
		t_{j}^{\epsilon_{j}})}
	\] 
	the sum being over all tuples 
	$
	\epsilon = (\epsilon_{1}, \ldots, \epsilon_{r}) \in \{-1, 1\}^{r}.
	$ 
	Here $R_{1}(\mathrm{t})$ is given by 
	\[ 
	R_{1}(\mathrm{t}) = 
	\prod_{n \ge 1} B(\mathrm{t}^{n}, q^{n})^{i_{n}\!(q)}
	\] 
	where 
	\[
	B(\mathrm{t}, q) \; = \prod_{1 \le i \le j \le r} (1 - t_{i} t_{j}) \, \cdot \,
	\bigg(1  + \frac{- \, 2 + \prod_{i = 1}^{r}(1 - t_{i})^{-1} 
		+ \prod_{i = 1}^{r}(1 + t_{i})^{-1}}
	{2 (1 + q^{-1})}\bigg).
	\] 
  \end{para}

 \begin{para}\label{comparison of Q} The function $R_1$ is closely related to the generating function of stable traces. Indeed,
 \begin{align*} (1-q^{-1})R_{1}(q^{- 1\slash 2}t_{1}, \ldots, q^{- 1\slash 2}t_{r}) & =  
 \Exp\Big(\, q\, \Log \Big(q^{-1} + \sum_{k\geq 0} q^{-k} h_{2k}(\mathrm t)\Big) - h_2(\mathrm t) - 1\Big)\\ & = \sum_\lambda T_\lambda \, s_{\lambda'}(t_1,\ldots,t_r).     
 \end{align*}The first equality follows by comparing $R_1$  with the expression obtained by setting $z=1/q$ in \cref{alternative form of gen series}, and noting that $\Exp(-h_2(\mathrm t)) = \prod_{i\leq j} (1-t_it_j)$, cf. \S\ref{no symplectic cauchy}. The second equality is \cref{plethystic formula for poincare series}. Combining this last expression with the following lemma, we also see that $\mathcal Q_1$ agrees with the function defined in the previous section.
 \end{para}

 \begin{lem}\label{lem:sympl-sch-id-c} For any partition $\lambda$ with $\lambda_1 \leq r$, any $g$, and indeterminates $t_1,\ldots,t_r$ one has
     \[
	s_{\langle \lambda^{\dag}(g) \rangle}(t_1^\pm,\ldots,t_r^\pm)
	= \sum_{\epsilon} 
	\frac{s_{\lambda'}(t^{\epsilon_1}_1,\ldots, t_r^{\epsilon_r})}
	{\prod_{i = 1}^{r}
		t_{i}^{\epsilon_{i}g}
		\prod_{1 \le i \le j \le r} (1 - t_{i}^{\epsilon_{i}}
		t_{j}^{\epsilon_{j}})},
	\]the sum being over all tuples 
	$
	\epsilon = (\epsilon_{1}, \ldots, \epsilon_{r}) \in \{-1, 1\}^{r}.
	$ 
 \end{lem}
 
\begin{proof}
Recall that 
\[
s_{\langle \lambda^\dag(g) \rangle}(t_1^\pm,\ldots,t_r^\pm)
= \frac{\det \!\left(t_{i}^{g - \lambda_{r + 1 - j}' + r + 1 - j} - t_{i}^{-(g - \lambda_{r + 1 - j}' + r + 1 - j)}\right)_{1\le i, j \le r}}{\mathrm{det}\!\left(t_{i}^{r + 1 - j} - t_{i}^{- (r + 1 - j)} \right)_{1\le i, j \le r}}.
\] 
{ Using the multilinearity of the determinant, write the numerator as} 
\[
\det \!\left(t_{i}^{g - \lambda_{r + 1 - j}' + r + 1 - j} - t_{i}^{-(g - \lambda_{r + 1 - j}' + r + 1 - j)}\right)_{1\le i, j \le r}
= \; \sum_{\epsilon}\, 
\prod_{k = 1}^{r} \epsilon_{k}  
\, \cdot \, \det \!\left(t_{i}^{\epsilon_{i}
	(g - \lambda_{r + 1 - j}' + r + 1 - j)}\right)_{1\le i, j \le r}
\]
and note that the denominator can be expressed as 
\[
\det \!\left(t_{i}^{r + 1 - j} - t_{i}^{- (r + 1 - j)} \right)_{1\le i, j \le r} 
= (-1)^{r(r + 1)\slash 2}
(t_{1} \cdots \,  t_{r})^{- r}
\prod_{1\le i \le j \le  r} 
(1 - t_{i}
t_{j})
\prod_{1\le k < l \le  r} 
(t_{k} - t_{l}).
\] 
In the numerator, replace $t_{i}$ by $1\slash t_{i}$ for all $i$, and note that 
\[
\det \!\left(t_{i}^{r + 1 - j} - t_{i}^{- (r + 1 - j)} \right)_{1\le i, j \le r} 
= \prod_{k = 1}^{r} \epsilon_{k}  \, \cdot\, \det \!\left(t_{i}^{\epsilon_{i}(r + 1 - j)} - t_{i}^{- \epsilon_{i}(r + 1 - j)} \right)_{1\le i, j \le r}. 
\] 
{Thus, we can write 
\begin{align*} 
s_{\langle \lambda^\dag(g) \rangle}(t_1^\pm,\ldots,t_r^\pm)
& = (-1)^{r^2}\sum_{\epsilon}\, 
\frac{\det \!\left(t_{i}^{-\epsilon_{i}
	(g - \lambda_{r + 1 - j}' + r + 1 - j)}\right)_{1\le i, j \le r}}
{\det \!\left(t_{i}^{\epsilon_{i}(r + 1 - j)} - t_{i}^{- \epsilon_{i}(r + 1 - j)} \right)_{1\le i, j \le r}}\\
&= (-1)^{r(r - 1)\slash 2}\sum_{\epsilon}\, 
\frac{\det \!\left(t_{i}^{\epsilon_{i}(\lambda_{r + 1 - j}'  + j - 1)}\right)_{1\le i, j \le r}}
{\prod_{i = 1}^{r}t_{i}^{\epsilon_{i}g}\prod_{1\le i \le j \le  r} 
(1 - t_{i}^{\epsilon_{i}}
t_{j}^{\epsilon_{j}})
\prod_{1\le k < l \le  r} 
(t_{k}^{\epsilon_{k}} - t_{l}^{\epsilon_{l}})}.
\end{align*} 
The lemma now follows from the identity 
\[ 
(-1)^{r(r - 1)\slash 2}
\mathrm{det}\Big(z_{i}^{\lambda_{r + 1 - j}' + j - 1}\Big)_{1\le i, j \le r} 
= \, s_{\lambda'}(z_1,\ldots, z_r)
\, \cdot \prod_{1\le k < l \le  r} 
(z_{k} - z_{l})
\] 
with $z_{i} = t_{i}^{\epsilon_{i}}$ for $i = 1, \ldots, r$.}
\end{proof} 

\begin{lem} \label{R1-estimate} For sufficiently small 
		$\varepsilon > 0,$ the function 
		$R_{1}(\mathrm{t})$ is holomorphic on the polydisk 
		$
		|t_{i}| 
		\le q^{-\frac{1}{4} - \varepsilon},
		$ 
		for $i = 1, \ldots, r.$ Moreover, on this polydisk, we have the upper bound 
		\[
		|R_{1}(\mathrm{t})|
		< \big(1 - q^{-4\varepsilon}\big)^{- 2^{1 + r (r + 1)\slash 2}}. 
		\] 
	\end{lem} 
	
	\begin{proof}{We will use the simple equality 
\[
1  + \frac{x - 2}
{2 (1 + q^{-1})} 
= \frac{x}{2} +  \frac{2 - x}
{2 (q + 1)} 
\] 
with 
$
x = \prod_{i = 1}^{r}(1 - t_{i})^{-1} + 
\prod_{i = 1}^{r}(1 + t_{i})^{-1}.
$ 
By considering the expression
\[ B = \prod_{1 \le i \le j \le r} (1 - t_{i} t_{j}) \, \cdot \,
\bigg(1  + \frac{x-2}
{2 (1 + q^{-1})}\bigg) = \prod_{1 \le i < j \le r} (1 - t_{i} t_{j})
\prod_{i = 1}^{r} (1 - t_{i}^2) \bigg( \frac{x}{2} +  \frac{2 - x}
{2 (q + 1)}  \bigg) \]
we can write 
$
B = B_{1}
+ B_{2},
$ 
where}
		\[
		B_{1}(\mathrm{t}, q) = 
		\frac{1}{2}\prod_{1 \le i < j \le r} (1 - t_{i} 
		t_{j})  
		\, \cdot \, \bigg(\prod_{i = 1}^{r}(1 + t_{i})  
		\, +  \, \prod_{i = 1}^{r}(1 - t_{i})\bigg)
		\] 
		and $B_{2}(\mathrm{t}, q)$ is given by 
		\[
		\frac{2 \prod_{1 \le i \le j \le r}
			(1 - t_{i} 
			t_{j})  \, - \, 
			\prod_{1 \le i < j \le r} (1 - t_{i} 
			t_{j})
			\prod_{i = 1}^{r}(1 + t_{i}) 
			\, - \, \prod_{1 \le i < j \le r} (1 - t_{i} 
			t_{j})
			\prod_{i = 1}^{r}(1 - t_{i})}{2 (q + 1)}.
		\] 
		Note that $B_{2}$ has no constant term and that all monomials of degree two in 
		$
		B_{1}
		$ 
		cancel out, that is to say $B_{1} - 1$ has no monomial of degree smaller than four. Thus, if 
		$
		|t_{i}| \le q^{\,-\frac{1}{4} - \varepsilon}
		$ 
		for $i = 1, \ldots, r,$ we have the estimate 
		\begin{equation*}
			\begin{split}   
				|B(\mathrm{t}, q) - 1| & \le 
				|B_{1}(\mathrm{t}, q) - 1| + |B_{2}(\mathrm{t}, q)| \\
				& \le 2^{- 1 + r (r + 1)\slash 2}
				q^{- 1 - 4 \varepsilon} \, + \, 3 \cdot 2^{- 1 + r (r + 1)\slash 2}
				q^{- \frac{3}{2} - 2\varepsilon}\\
				& < 2^{1 + r (r + 1)\slash 2}
				q^{- 1 - 4 \varepsilon}.
			\end{split}
		\end{equation*} 
		Applying the inequality $i_{n}(q) \le q^{n} \slash n,$ which follows from the well-known formula 
		$
		\sum_{d \mid n} di_{d}(q)  =  q^{n}
		$ 
		($n \ge 1$), and the familiar estimate $\log(1 + y) < y$ for $y > 0,$ we have 
		\[
		\log |R_{1}(\mathrm{t})| < 
		2^{1 + r (r + 1)\slash 2}
		\sum_{n = 1}^{\infty} \frac{q^{n -  n(1 + 4 \varepsilon)}}{n}
		= 2^{1 + r (r + 1)\slash 2}
		\log \big(1 - q^{-4\varepsilon}\big)^{-1}. 
		\] 
		This implies the absolute and uniform convergence of the product $R_{1}(\mathrm{t}).$ The upper bound follows by exponentiating the last inequality. 
	\end{proof} 
	\begin{rem} \label{further-an-cont-R1} By multiplying successively 
		$
		R_{1}(\mathrm{t})
		$ 
		by suitable products of zeta functions, or inverses of such, one can extend 
		$
		R_{1}(\mathrm{t})
		$ 
		even further. For example, the product 
		\begin{equation*} 
			\begin{split}
				\prod_{1 \le i_{1} 
					< i_{2} \le r}
				\!\left(1 - z_{
						i_{1}}^{2} 
				z_{i_{2}}^{2}\right)^{-1} 
				&\prod_{1 \le i_{1} 
					< i_{2} < 
					i_{3}\le r}
				\!\left(1 - z_{
						i_{1}}^{2}
				z_{i_{2}}^{}
				z_{i_{3}}^{}
				\right)^{-1}
				\!\left(1 - z_{
						i_{1}}^{}
				z_{i_{2}}^{2}
				z_{i_{3}}^{}
				\right)^{-1}
				\!\left(1 - z_{
						i_{1}}^{}
				z_{i_{2}}^{}
				z_{i_{3}}^{2}
				\right)^{-1}\\	
				& \cdot 
				\prod_{1 \le i_{1} 
					< i_{2} < 
					i_{3} < 
					i_{4} \le r}
				\!\left(1 - z_{
						i_{1}}^{}
				z_{i_{2}}^{}
				z_{i_3}^{}
				z_{i_4}^{}
				\right)^{-2}
				\cdot \, B(\mathrm{t}, q)
			\end{split}	
		\end{equation*} 
		is $1 + O_{r} \left(q^{-1-6 \varepsilon}\right)$ for 
		$
		|t_{i}| 
		\le q^{-\frac{1}{6} - \varepsilon},
		$ 
		from which it follows at once that 
		$
		R_{1}(\mathrm{t})
		$ 
		is holomorphic on any such polydisk, with small positive $\varepsilon.$ 
		This observation could be used to improve the estimates in Corollary \ref{stable-traces-estimate} and \cref{difference-main-stable} below. However, to simplify the arguments, we shall only make use of \cref{R1-estimate}.
	\end{rem}  
	\begin{lem}\label{average-to-integral-vers1} Let $a_{1}, \ldots, a_{r}$ be distinct non-zero complex numbers such that $a_{i}a_{j} \ne 1$ for all $1\le i, j \le r.$ Suppose $h$ is a symmetric function of $r$ variables, holomorphic on a domain containing 
		$
		\left(a_{1}^{\epsilon_{1}}, \ldots, a_{r}^{\epsilon_{r}}\right)
		$ 
		for all 
		$
		\epsilon = (\epsilon_{1}, \ldots, \epsilon_{r}) \in \{\pm 1\}^{r}.
		$ 
		Then 
		\[
		\sum_{\epsilon} 
		\frac{h(a_{1}^{\epsilon_{1}}, \ldots, a_{r}^{\epsilon_{r}})}
		{\prod_{1 \le i \le j \le r} (1 - a_{i}^{\epsilon_{i}}
			a_{j}^{\epsilon_{j}})}
		= \frac{(-1)^{r(r + 1)\slash 2}}{r!}\frac{1}{\left(2\pi \sqrt{-1}\right)^{r}}\!\oint \cdots \oint h(\mathrm{z}) 
		\frac{\prod_{1\le i < j \le r}(z_{j}^{} 
			- z_{i}^{})^{2}
			(1 - z_{i}^{} z_{j}^{})}
		{\prod_{i, j = 1}^{r} (1 - z_{i}^{}a_{j}^{})
			(1 - z_{i}^{}a_{j}^{-1})}\, 
		\frac{d \mathrm{z}}{\mathrm{z}^{r}}
		\] 
		where each path of integration encloses the $a_{j}^{\pm 1},$ but not 
		$
		z_{j}^{} \! = 0.
		$ 
	\end{lem}

	\begin{proof} This is \cite[Lemma 7.1]{DT}.
	\end{proof}

	\begin{para}We apply Lemma \ref{average-to-integral-vers1} to the function 
	$
	R_{1}
	(q^{- 1\slash 2}\mathrm{z})
	\prod_{i = 1}^{r} z_{i}^{-g}
	$ 
	with $a_{i}  = t_{i}.$ By Lemma \ref{R1-estimate}, the paths of integration can be taken to be the boundaries of 
	$
	\rho_{1} 
	< |z_{i}| 
	< \rho_{2}
	$ 
	with the positive orientations; we are assuming that 
	$
	\rho_{1} < 1 < 
	\rho_{2},
	$ 
	with $\rho_{2} - \rho_{1}$ sufficiently small. \!It follows at once that the average 
	\begin{equation} \label{eq: Av} 
		\sum_{\epsilon} \frac{R_{1}(q^{- 1/2}t_{1}^{\epsilon_{1}}, \ldots, q^{- 1/2}t_{r}^{\epsilon_{r}})}{\prod_{i = 1}^{r}
			t_{i}^{\epsilon_{i}g}
			\prod_{1 \le i \le j \le r} (1 - t_{i}^{\epsilon_{i}}
			t_{j}^{\epsilon_{j}})}
	\end{equation} 
	is holomorphic in a neighborhood of $|t_{i}| = 1$ $(1 \le i \le r).$ In particular, the function $\mathcal{Q}_{1}(t; g, q)\psi(t)^{-g}$ is continuous on $\mathbf{T}.$ \end{para}

	\begin{para}The following proposition provides an explicit formula for the first term 
	$
	\mathcal{T}_{1}(\lambda^{\dag})
	$ 
	in Conjecture \ref{conj-traces}. \end{para}
	
	\begin{prop} \label{explicit-T1} For every partition $\lambda \subseteq (r^{g})$ we have 
		\begin{align} \label{eq: integral-main-term-traces} 
			\mathcal{T}_{1}
				(\lambda^{\dag}) = 
				\frac{(1-q^{-1})}{r!}  
				\cdot\frac{(-1)^{r(r - 1)\slash 2}}{\left(2\pi \sqrt{-1}\right)^{r}}\!
				\oint\limits_{|z_{r}| = 1} 
				\cdots \oint\limits_{|z_{1}| = 1} 
				 R_{1}
				(q^{- 1\slash 2}\mathrm{z}) 
				\prod_{1\le i < j \le  r} 
				(z_{i}^{} - z_{j}^{})\\
				 \cdot \, \mathrm{det}\left(z_{i}^{g - \lambda_{r + 1 - j}' + r + 1 - j} - z_{i}^{-(g - \lambda_{r + 1 - j}' + r + 1 - j)}\right)_{1\le i, j \le r}\, \frac{d \mathrm{z}}{\mathrm{z}^{g + r + 1}}.
		\nonumber
		\end{align}
	\end{prop}

	\begin{proof} By definition 
		\[ 
		\mathcal{T}_{1}(\lambda^{\dag}) = \frac{1}{\#W}\!\int_{\mathbf{T}} 
		\mathcal{Q}_{1}(t; g, q)
		\psi(t)^{-g}
		\overline{s_{\langle \lambda^{\dag} \rangle}(t)}\, |\delta(t)|^{2}\, dt.
		\] 
		The function 
		$
		\mathcal{Q}_{1}(t; g, q)
		\psi(t)^{-g}
		$ 
		is equal to the average \eqref{eq: Av} multiplied by the factor 
		$(1-q^{-1})$. 
		Using Lemma \ref{average-to-integral-vers1}, we express \eqref{eq: Av} as an $r$-fold 
		integral over the boundaries of 
		$
		\rho_{1} 
		< |z_{i}| 
		< \rho_{2}
		$ 
		with the positive orientations. As above, we take 
		$
		\rho_{1} < 1 < 
		\rho_{2},
		$ 
		with 
		$
		\rho_{2} - \rho_{1}
		$ 
		sufficiently small. This integral can be decomposed as a sum of $2^{r}$ 
		$r$-fold integrals over circles 
		$
		|z_{i}| = 
		\rho_{k}
		$ 
		($k = 1, 2$). Pick one of them, say when the first $n$ integrals are over 
		$
		|z_{i}| = 
		\rho_{1}
		$ 
		($i = 1, \ldots, n$), while the remaining ones are over 
		$
		|z_{i}| 
		= \rho_{2}.
		$ 
        
        To the denominator of the expression under the integral, we apply Cauchy's identity for 
		$\mathrm{Sp}(2r)$ (see \S\ref{no symplectic cauchy}), which states that
{ \[ 
\prod_{i, j = 1}^{r} 
(1 - z_{i}t_{j})^{-1} (1 - z_{i}t_{j}^{-1})^{-1} 
\, = \prod_{1\le i < j \le r} (1 - 
z_{i}z_{j})^{-1} 
\cdot \sum_{\substack{\mu \\ l(\mu) \le r}} 
s_{\langle \mu \rangle}
(t_{1}^{\pm 1}\!, \ldots, t_{r}^{\pm 1}) 
s_{\mu }(z_{1}, \ldots, z_{r}).
\] To ensure the absolute convergence of the series involved, we change $z_{i}^{}$ to 
$
z_{i}^{-1}
$ 
for $n + 1 \le  i \le r,$ and apply the simple identity 
\[
\frac{1}{(1 - z_{i}t_{j})
(1 - z_{i}t_{j}^{-1})} 
= \frac{
	z_{i}^{-2}}
{(1 - z_{i}^{-1}t_{j})
(1 - z_{i}^{-1}t_{j}^{-1})}.
\] 
Accordingly, if we set 
\[
\epsilon_{i} = \begin{cases}
	1  & \text{if $1 \le i \le n$}\\ 
	- 1 & \text{if $n + 1 \le i \le r$}  
\end{cases}
\] 
and $\epsilon = (\epsilon_{1}, \ldots, \epsilon_{r})$, then we can write the denominator of the expression under the integral as 
\begin{equation*}
	\begin{split}
&\prod_{i, j = 1}^{r} (1 - z_{i}t_{j})^{-1}
(1 - z_{i}t_{j}^{-1})^{-1} = 
\prod_{i, j = 1}^{r} z_{i}^{-(1-\epsilon_{i})}(1 - z_{i}^{\epsilon_{i}}t_{j})^{-1}
(1 - z_{i}^{\epsilon_{i}}t_{j}^{-1})^{-1}\\
&  = 
\prod_{k = 1}^{r} z_{k}^{-r(1-\epsilon_{k})}
\prod_{1\le i < j \le r} (1 - 
z_{i}^{\epsilon_{i}}z_{j}^{\epsilon_{j}})^{-1} 
\cdot \sum_{\substack{\mu \\ l(\mu) \le r}} 
s_{\langle \mu \rangle}
(t_{1}^{\pm 1}\!, \ldots, t_{r}^{\pm 1}) 
s_{\mu }(\mathrm{z^{\epsilon}}).
\end{split}
\end{equation*} 
By first performing the integral over $\mathbf{T}$ and using the orthonormality of the symplectic Schur functions (which picks out just the partition $\lambda^\dag$ from the sum over $\mu$), we can write the contribution to 
$
\mathcal{T}_{1}(\lambda^{\dag})
$ 
as} 
\begin{equation*}
			\begin{split}
				I_{\epsilon}: = 
				\frac{(1-q^{-1})}{r!}\frac{(-1)^{\frac{r(r + 1)}{2} + n}}{\left(2\pi \sqrt{-1}\right)^{r}} \oint\limits_{|z_{r}| = \rho_{2}} 
				\cdots \oint\limits_{|z_{1}| = \rho_{1}}
				R_{1}
				(q^{- 1\slash 2}\mathrm{z}) 
				&\frac{\prod_{1\le i < j \le r}(z_{j}^{} - z_{i}^{})^{2}
					(1 - z_{i}^{} z_{j}^{})}
				{\prod_{1\le i < j \le r} \big(1 - z_{i}^{\epsilon_{i}}z_{j}^{\epsilon_{j}}\big) 
				}\\ 
				&\cdot \prod_{k = 1}^{r} z_{k}^{- r(1 - \epsilon_{k})}
				\cdot  s_{\lambda^\dag}(\mathrm{z^{\epsilon}})\, 
				\frac{d \mathrm{z}}{\mathrm{z}^{g + r}}.
			\end{split}
		\end{equation*} 
        Notice that 
		$
		(-1)^{n} = (-1)^{-r}\prod_{i = 1}^{r}\epsilon_{i}. 
		$ 
		Using the identities 
		\[
		\frac{(z_{i}^{} 
			- z_{j}^{})
			(1 - z_{i}^{} 
			z_{j}^{})}
		{1 - z_{i}^{\epsilon_{i}}
			z_{j}^{\epsilon_{j}}} 
		= z_{i}^{1 - \epsilon_{i}}
		z_{j}^{1 - \epsilon_{j}}
		\big(z_{i}^{\epsilon_{i}} - z_{j}^{\epsilon_{j}}\big)\;\;\;
		\text{and}\;\;\prod_{1\le i < j \le r}
		z_{i}^{1 - \epsilon_{i}}
		z_{j}^{1 - \epsilon_{j}}
		= \prod_{k = 1}^{r} z_{k}^{(r - 1)(1 - \epsilon_{k})}
		\] 
		we can simplify $I_{\epsilon}$ to: 
		\begin{equation*}
			\begin{split} 
				I_{\epsilon} = 
				\frac{(1-q^{-1})}{r!}\frac{(-1)^{\frac{r(r - 1)}{2}}\prod_{i = 1}^{r}\epsilon_{i}}{\left(2\pi \sqrt{-1}\right)^{r}}
				\oint\limits_{|z_{r}| = 1} 
				\cdots \oint\limits_{|z_{1}| = 1}
				R_{1}
				(q^{- 1\slash 2}\mathrm{z})
				&\prod_{1\le i < j \le r}(z_{i}^{} - z_{j}^{})
				\big(z_{i}^{\epsilon_{i}} - z_{j}^{\epsilon_{j}}\big)\\ 
				&\cdot \prod_{k = 1}^{r} z_{k}^{\epsilon_{k}}
				\cdot  s_{\lambda^\dag}(\mathrm{z^{\epsilon}})\, 
				\frac{d \mathrm{z}}{\mathrm{z}^{g + r + 1}}.
			\end{split}
		\end{equation*} 
		Also, note that the product 
		$
		\prod_{1\le i < j \le r}\big(z_{i}^{\epsilon_{i}} - z_{j}^{\epsilon_{j}}\big)  
		$
		cancels the denominator of 
		$
		s_{\lambda^\dag}(\mathrm{z^{\epsilon}}).
		$ 
		Summing now over all tuples 
		$
		\epsilon = (\epsilon_{1}, \ldots, \epsilon_{r}) \in \{-1, 1\}^{r},
		$ 
		the formula for 
		$
		\mathcal{T}_{1}(\lambda^{\dag})
		$ 
		is an immediate consequence of the simple identity 
		\begin{equation*} 
			\mathrm{det}\!\left(z_{i}^{g - \lambda_{r + 1 - j}' + r + 1 - j} - z_{i}^{-(g - \lambda_{r + 1 - j}' + r + 1 - j)}\right)_{1\le i, j \le r}
			= \; \sum_{\epsilon}\, 
			\prod_{k = 1}^{r} \epsilon_{k}
			z_{k}^{\epsilon_{k}}
			\, \cdot \prod_{1\le i < j \le r}
			\big(z_{i}^{\epsilon_{i}} - z_{j}^{\epsilon_{j}}\big)
			\cdot s_{\lambda^\dag}(\mathrm{z^{\epsilon}})
		\end{equation*} 
		which follows from the multilinearity of the determinant. 
		This completes the proof. 
	\end{proof}

	\begin{para}As a first corollary, we obtain the expected connection between the stable traces $T_{\lambda}$ and 
	$
	\mathcal{T}_{1}(\lambda^{\dag}). $ \end{para}
	
	\begin{cor} \label{stable-traces-2} For any fixed partition 
		$
		\lambda = (\lambda_{1} \ge
		\lambda_{2} \ge \cdots)
		$ with 
		$
		\lambda_{1} \le r,
		$ 
		we have 
		\begin{equation} \label{eq: integral-stable-traces} 
			\begin{split}
				&\lim_{g \to \infty} 
				\mathcal{T}_{1}
				(\lambda^{\dag})
				= 
				T_{\lambda} = \frac{(1-q^{-1})}{r!}  
				\cdot\frac{(-1)^{r(r - 1)\slash 2}}{\left(2\pi \sqrt{-1}\right)^{r}}\!
				\oint\limits_{|z_{r}| = 1} 
				\cdots \oint\limits_{|z_{1}| = 1} 
				R_{1}
				(q^{- 1\slash 2}\mathrm{z}) 
				\prod_{1\le i < j \le  r} 
				(z_{i}^{} - z_{j}^{})\\
				& \hskip140pt \cdot \,\mathrm{det}\Big(z_{i}^{- \lambda_{r + 1 - j}' + r + 1 - j}\Big)_{1\le i, j \le r} \, \frac{d \mathrm{z}}{\mathrm{z}^{r + 1}}.
			\end{split}
		\end{equation}
	\end{cor}
	
	\begin{proof} In formula \eqref{eq: integral-main-term-traces}, divide the determinant by $\mathrm{z}^{g}.$ Applying the {R}iemann-{L}ebesgue lemma, we see at once that the limit and the $r$-fold integral in \eqref{eq: integral-stable-traces} are equal.

		To see that the integral is also equal to 
		$
		T_{\lambda},
		$ 
		divide the determinant by $\mathrm{z}^{r}$, and note, as in the proof of \cref{lem:sympl-sch-id-c}, that 
		\[ 
		(-1)^{r(r - 1)\slash 2}
		\mathrm{det}\Big(z_{i}^{\lambda_{r + 1 - j}' + j - 1}\Big)_{1\le i, j \le r} 
		= \; \mathrm{det}\Big(z_{i}^{\lambda_{j}' + r - j}\Big)_{1\le i, j \le r}
		= \; s_{\lambda'}(\mathrm{z})
		\, \cdot \prod_{1\le i < j \le  r} 
		(z_{i}^{} - z_{j}^{}).
		\] 
		Recalling that $R_{1}(\mathrm{z})$ is the {S}chur-generating function of the stable traces $T_{\lambda},$ it follows at once from the {P}eter-{W}eyl theorem \cite{Bump}, applied to the compact group $\mathrm{U}(r),$ that the integral in \eqref{eq: integral-stable-traces} can be decomposed as 
		\[ \frac{(1-q^{-1})}{r!} 
		\sum_{\substack{\mu \\ l(\mu')\le r}} 	T_{\mu} \cdot 
		\frac{1}{\left(2\pi \sqrt{-1}\right)^{r}}\!
		\oint\limits_{|z_{r}| = 1} 
		\cdots \oint\limits_{|z_{1}| = 1}
		s_{\mu'}(\mathrm{z})
		\overline{s_{\lambda'}(\mathrm{z})} 
		\prod_{1\le i < j \le  r} 
		|z_{i}^{} - z_{j}^{}|^{2}\, \frac{d \mathrm{z}}{\mathrm{z}}.
		\] 
		Letting 
		$ 
		z_{i} 
		= e^{2\pi \sqrt{-1}\theta_{i}},
		$ 
		$ 
		0\le \theta_{i} \le 1,
		$ 
		and using the orthonormality of the {S}chur functions, one finds that the last 
		integral is $1$ or $0$ according as $\mu = \lambda$ or not, 
		which completes the proof.
	\end{proof}

	\begin{para}From the integral representation \eqref{eq: integral-stable-traces} of $T_{\lambda},$ we deduce the following estimate: \end{para}
	 
	\begin{cor} \label{stable-traces-estimate}For any partition $\lambda$ with
		$
		\lambda_{1} = l(\lambda')\le r,
		$ 
		we have the estimate 
		\[
		|T_{\lambda}|
		\ll_{r, \, \varepsilon} q^{\left(-\frac{1}{4} + \varepsilon \right)|\lambda|}.  
		\]
	\end{cor}

	\begin{proof} By Lemma \ref{R1-estimate}, we can integrate in 
		\eqref{eq: integral-stable-traces} over 
		$
		|z_{i}| = q^{\frac{1}{4} - \varepsilon}
		$ 
		(instead of $|z_{i}| = 1$), for 
		$i = 1, \ldots, r,$ and after the simple change of variables 
		$
		z_{i} 
		\to q^{\frac{1}{4} - \varepsilon}z_{i}, 
		$ 
		we can write 
		\begin{equation*}
			\begin{split}
				T_{\lambda} = 
				\frac{(1-q^{-1})q^{\left(-\frac{1}{4} + \varepsilon \right)|\lambda|}}{r!}  
				\cdot\frac{(-1)^{r(r - 1)\slash 2}}{\left(2\pi \sqrt{-1}\right)^{r}}\!
				\oint\limits_{|z_{r}| = 1} 
				\cdots \oint\limits_{|z_{1}| = 1} 
				& R_{1}
				\big(q^{- \frac{1}{4} - \varepsilon}\mathrm{z}\big) 
				\prod_{1\le i < j \le  r} 
				(z_{i}^{} - z_{j}^{})\\
				& \cdot \,\mathrm{det}\Big(z_{i}^{- \lambda_{r + 1 - j}' + r + 1 - j}\Big)_{1\le i, j \le r} \, \frac{d \mathrm{z}}{\mathrm{z}^{r + 1}}.
			\end{split}
		\end{equation*} Applying the upper bound in Lemma \ref{R1-estimate}, and the {H}adamard inequality for determinants, we have 
		\[  
		\left|R_{1}
		\big(q^{- \frac{1}{4} - \varepsilon}\mathrm{z}\big) 
		\prod_{1\le i < j \le  r} 
		(z_{i}^{} - z_{j}^{})
		\cdot \,\mathrm{det}\Big(z_{i}^{- \lambda_{r + 1 - j}' + r + 1 - j}\Big)_{1\le i, j \le r}\right|
		< r^{r}
		\big(1 - q^{-4\varepsilon}\big)^{- 2^{1 + r (r + 1)\slash 2}}
		\] 
		and our assertion follows. 
	\end{proof} 

    \begin{rem}
     This estimate can also be seen from \cref{1/4} and \cref{good bound}, although upon sufficient unwinding the proofs are not so different. Indeed, if $\vert\lambda\vert=w$ and $\lambda_1 \leq r$ then $T_\lambda$ is bounded by 
     \begin{align*}
          \sum_{|\lambda|=w} |T_\lambda| \dim S^{\lambda'}(\C^{r}) & \leq \sum_{|\lambda|=w} \sum_{k \geq w/4} q^{-k} \dim H_k(\H_\infty^{1,0},\V_\lambda) \dim S^{\lambda'}(\C^{r}) \\
           &\leq \sum_{k \geq w/4} q^{-k} 2^{\varepsilon k} 
     \end{align*}
     for any $\varepsilon>0$. Similarly one may show, using \cref{1/2} instead of \cref{1/4}, that $\vert T_\lambda \vert \ll_{r,\varepsilon} q^{(-\frac 1 2 + \varepsilon) \length(\lambda)}$. 
    \end{rem}

	\begin{rem}  Notice that, by Remark \ref{further-an-cont-R1}, one can use the holomorphy of 
		$
		R_{1}(\mathrm{t})
		$ 
		in a larger polydisk to improve the estimate in Corollary \ref{stable-traces-estimate}, and hence the range of $\lambda$ for which $H^{k}(
			\H_{g}^{1,0}, 
			\mathbb{V}_{\lambda}) = 0$ for $k < \frac{g}{2}.$ For instance, using the holomorphy of 
		$ 
		R_{1}
		(q^{- 1\slash 2}\mathrm{z})
		$ 
		on the polydisk 
		$
		|z_{i}| 
		\le q^{\frac{1}{3} - \varepsilon},
		$ 
		$i=1, \ldots, r,$ one can show as in the proof of Corollary \ref{stable-traces-estimate} that 
		\[
		|T_{\lambda}|
		\ll_{r, \, \varepsilon} q^{\left(-\frac{1}{3} + \varepsilon \right)|\lambda| 
			+ \frac{r^{2}(r^{2} - 1)}{36}}
		\] 
		and the corresponding bound on $|\lambda|$ is obtained by imposing the 
		condition 
		\[ 
		q^{\left(-\frac{1}{3} + \varepsilon \right)|\lambda| 
			+ r^{2}(r^{2} - 1)/{36}} < q^{-{g}/{2}}
		\] 
		which, upon replacing $\varepsilon$ by $\varepsilon \slash 4,$ is implied by 
		$|\lambda| > (1 + \varepsilon)\!\left(3 g\slash 2 + r^{2}(r^{2} - 1)\slash 12\right)$.
	\end{rem} 
	\begin{cor} \label{difference-main-stable} For any partition 
		$
		\lambda$ with 
		$
		\lambda_{1} = l(\lambda')\le r,
		$ 
		we have the estimate 
		\[
		|		\mathcal{T}_{1}
		(\lambda^{\dag})
		- T_{\lambda}|
		\ll_{r, \, \varepsilon}
		q^{\left(-\frac{1}{4} + \varepsilon \right)|\lambda| -
			\left(\frac{1}{2} - 2\varepsilon\right)\left(g -\lambda_{1}' + 1\right)}.
		\] 
	\end{cor}

	\begin{proof} For a partition 
		$
		\mu = (\mu_{1} \ge
		\cdots \ge \mu_{r} \ge 0),
		$ 
		we set for notational convenience 
		\[
		a_{\mu}(\mathrm{z}) = \mathrm{det}\Big(z_{i}^{\mu_{j} + r + 1 - j}\Big)_{1\le i, j \le r}.
		\] 
		Then, by the determinant identity in the end of the proof of \cref{explicit-T1} we have that 
		\[
		\left|\, \det\left(z_{i}^{g - \lambda_{r + 1 - j}' + r + 1 - j} - z_{i}^{-(g - \lambda_{r + 1 - j}' + r + 1 - j)}\right)_{1\le i, j \le r} 
		- \; a_{\lambda^{\dag}}(\mathrm{z})\, \right|
		\, \le \, \sum_{\epsilon}\, |a_{\lambda^{\dag}}(\mathrm{z^{\epsilon}})|
		\] 
		the sum being over all tuples 
		$
		\epsilon = (\epsilon_{1}, \ldots, \epsilon_{r}) \in \{-1, 1\}^{r}
		$ 
		with at least one component $\epsilon_{i} = -1.$ 
		Assuming with no loss of generality that $\epsilon_{1} = -1,$ we can estimate 
		$
		|a_{\lambda^{\dag}}(\mathrm{z^{\epsilon}})|,
		$ 
		with 
		$
		|z_{i}| = 
		q^{\frac{1}{4} - \varepsilon},
		$ 
		by applying the {L}aplace expansion along the first row of 
		$
		a_{\lambda^{\dag}}(\mathrm{z^{\epsilon}}).
		$ 
		Notice that 
		$
		\big|z_{1}^{-(g - \lambda_{r + 1 - j}' + r + 1 - j)}\big| 
		\le q^{- \left(\frac{1}{4} - \varepsilon\right)(g - \lambda_{1}' + 1)}
		$ 
		for all $j,$ while by {H}adamard's inequality, the corresponding cofactors are all 
		$
		\ll_{r} 
		q^{\left(\frac{1}{4} - \varepsilon\right)\left((r - 1)g - |\lambda| 
			+ \lambda_{1}' - 1 + r(r + 1)\slash 2\right)}.
		$ 
		Thus we get the estimate 
		\[
		\sum_{\epsilon}\, |a_{\lambda^{\dag}}(\mathrm{z^{\epsilon}})|
		\, \ll_{r}\,  
		q^{\left(\frac{1}{4} - \varepsilon\right)\left((r - 2)g - |\lambda| 
			+ 2\lambda_{1}' - 2 + r(r + 1)\slash 2\right)}.
		\] 
		Now consider the difference between the integrals \eqref{eq: integral-main-term-traces} and \eqref{eq: integral-stable-traces}. Again, by Lemma \ref{R1-estimate}, we can integrate over 
		$
		|z_{i}| = q^{\frac{1}{4} - \varepsilon}\!,
		$ 
		and 
		$
		|R_{1}(q^{- 1\slash 2}\mathrm{z})| < \big(1 - q^{-4\varepsilon}\big)^{- 2^{1 + r (r + 1)\slash 2}} 
		$ 
		\!on these circles. The remaining piece under the integral is easily seen to be 
		$
		\ll_{r} 
		q^{\frac{|\lambda|}{2} - \left(\frac{1}{4} - \varepsilon\right)\left(r g + \frac{r(r + 1)}{2}\right)},
		$ 
		and the estimate in the corollary follows. 
	\end{proof} 
	
	\begin{para}\label{main-term-T1-estimate}
	The estimates in the last two corollaries imply that 
	\begin{equation*} 
		|\mathcal{T}_{1}
		(\lambda^{\dag})|
		\ll_{r, \, \varepsilon}
		q^{
			- \left(\frac{1}{4} - \varepsilon \right)|\lambda|}
	\end{equation*} 
	for any partition 
	$
	\lambda$ with 
	$
	\lambda_{1} = l(\lambda')\le r.
	$ 
	This, combined with Conjecture \ref{conj-traces}, suggests that 
	the traces 
	$
	\tr_{\lambda}(g)
	$ 
	indeed decrease as the weight of $\lambda$ grows with $g$. Again, the estimate in Corollary \ref{difference-main-stable} (and hence also this one) can be improved to 
	\[
	|\mathcal{T}_{1}
	(\lambda^{\dag})
	- T_{\lambda}|
	\ll_{r, \, \varepsilon}
	q^{\left(-\frac{1}{3} + \varepsilon \right)|\lambda| -
		\left(\frac{2}{3} - 2\varepsilon\right)\left(g -\lambda_{1}' + 1\right)+ \frac{r^{2}(r^{2} - 1)}{36}}.
	\] 
\end{para}
 \subsection{The second term}
\begin{para}	To analyze the second term 
	$
	\mathcal{T}_{2}(\lambda^{\dag})
	$ 
	in the asymptotic formula of $\tr_{\lambda}(g),$ one first checks that the set of roots $\Phi_{2}$ is 
	\[
	\Phi_{2} = 
	\Bigg\{\!\sum_{\substack{j = 1 \\ j \neq
			j_{1},\,
			j_{2},\,
			j_{3}}}^{r} \!\!\!
	k_{j}\alpha_{j} 
	+ \alpha_{j_{1}}  \!+
	\alpha_{j_{2}}  
	\! +  \alpha_{j_{3}}  
	\! +  2\alpha_{r +   1} :
	\text{$1 \le j_{1} 
		< j_{2} 
		< j_{3} \le r$, and 
		$k_{j} \in \{0, 2\}$
		for 
		$
		j \neq \! j_{1}, 
		j_{2}, 
		j_{3}
		$}
	\Bigg\}
	\] 
	see \cite{DT}. As in the case of 
	$
	\mathcal{T}_{1}(\lambda^{\dag}),
	$ 
	we shall represent $\mathcal{Q}_{2}(t; g, q)\psi(t)^{-g}$ as an integral. 
	The analogue of the function 
	$
	R_{1}(\mathrm{t})
	$ 
	is defined as follows. For $\zeta$ a $4$th root of $1,$ define 
	\begin{equation*} 
		\begin{split} 
			&\mathrm{d}_{1} 
			= \frac{\left(1 - \frac{\xi_{1}}{\zeta
					q^{3\slash 4}
					\xi_{2}
					\xi_{3}}\right)
				\!\left(1 - \frac{\xi_{2}}{\zeta
					q^{3\slash 4}
					\xi_{1}
					\xi_{3}}\right)
				\!\left(1 - \frac{\xi_{3}}{\zeta
					q^{3\slash 4}
					\xi_{1}
					\xi_{2}}\right)}
			{\left(1 - \frac{\zeta 
					\xi_{2}
					\xi_{3}}
				{q^{1\slash 4}
					\xi_{1}}\right)
				\!\left(1 - \frac{\zeta 
					\xi_{1}
					\xi_{3}}
				{q^{1\slash 4}
					\xi_{2}}\right)
				\!\left(1 - \frac{\zeta 
					\xi_{1}
					\xi_{2}}
				{q^{1\slash 4}
					\xi_{3}}\right)},\;\;\; 
			\mathrm{d}_{2} = \frac{\zeta}
			{q^{3\slash 4}
				\xi_{1}
				\xi_{2}
				\xi_{3}} \\
			&\mathrm{d}_{3} 
			= \frac{\left(1 + \frac{\xi_{1}}{\zeta
					q^{3\slash 4}
					\xi_{2}
					\xi_{3}}\right)
				\!\left(1 + \frac{\xi_{2}}{\zeta
					q^{3\slash 4}
					\xi_{1}
					\xi_{3}}\right)
				\!\left(1 + \frac{\xi_{3}}{\zeta
					q^{3\slash 4}
					\xi_{1}
					\xi_{2}}\right)}
			{\left(1 + \frac{\zeta 
					\xi_{2}
					\xi_{3}}
				{q^{1\slash 4}
					\xi_{1}}\right)
				\!\left(1 + \frac{\zeta 
					\xi_{1}
					\xi_{3}}
				{q^{1\slash 4}
					\xi_{2}}\right)
				\!\left(1 + \frac{\zeta 
					\xi_{1}
					\xi_{2}}
				{q^{1\slash 4}
					\xi_{3}}\right)}
		\end{split}
	\end{equation*} 
	and 
	\begin{equation*}
		\mathrm{e}_{1} =
		\frac{1 - (\zeta 
			q^{3\slash 4}
			\xi_{1}
			\xi_{2}
			\xi_{3})^{-1}}
		{1-\zeta
			q^{\, -1\slash 4}
			\xi_{1}
			\xi_{2}
			\xi_{3}},\;\;\; 
		\mathrm{e}_{2}
		= (\zeta
		q^{1\slash 4}
		\xi_{1}
		\xi_{2}
		\xi_{3})^{-1},\;\;\, \text{and}
		\;\;\;\,
		\mathrm{e}_{3} =
		\frac{1 + (\zeta q^{3\slash 4}
			\xi_{1}
			\xi_{2}
			\xi_{3})^{-1}}
		{1+\zeta
			q^{\,-1\slash 4}
			\xi_{1}
			\xi_{2}
			\xi_{3}}.
	\end{equation*} 
	We set 
	\begin{equation*} 
		\begin{split}
			L_{1} &  = \,
			\tfrac{1}{4}\mathrm{e}_{1}
			(\mathrm{d}_{1} \! +
			\mathrm{d}_{3}
			\! + 2\zeta^{2}\mathrm{d}_{2})\,
			-\, \tfrac{1}{4}\mathrm{e}_{3}
			(\mathrm{d}_{1} \! +
			\mathrm{d}_{3} \! -
			2\zeta^{2}
			\mathrm{d}_{2}) \\
			L_{2}  & = \, \tfrac{1}{8}\mathrm{e}_{1}
			(\mathrm{d}_{1} \! + \mathrm{d}_{3}
			\! + 2\zeta^{2}
			\mathrm{d}_{2})\,
			+\, \tfrac{1}{8}\mathrm{e}_{3}
			(\mathrm{d}_{1} \! + \mathrm{d}_{3}
			\! - 2\zeta^{2}
			\mathrm{d}_{2})
			\, + \, \tfrac{1}{4}\mathrm{e}_{2}
			(\mathrm{d}_{1} \! - \mathrm{d}_{3}) \\
			L_{3} & = \, \tfrac{1}{8}\mathrm{e}_{1}
			(\mathrm{d}_{1} \! + \mathrm{d}_{3}
			\! + 2\zeta^{2}
			\mathrm{d}_{2})\,
			+\, \tfrac{1}{8}\mathrm{e}_{3}
			(\mathrm{d}_{1} \! + \mathrm{d}_{3}
			\! - 2\zeta^{2}
			\mathrm{d}_{2})
			\, - \, \tfrac{1}{4}\mathrm{e}_{2}
			(\mathrm{d}_{1} \! - \mathrm{d}_{3})
		\end{split}
	\end{equation*} 
	and let 
	$
	S = S(\xi_{1}^{2}, \ldots, 
	\xi_{r}^{2}, \zeta^{2}, q^{-1/2})
	$ 
	be defined by 
	\begin{equation*}
		\begin{split}
			S =  & \, \big(1 - q^{-1}\big)
			\!\left\{\frac{L_{1}}
			{q^{3\slash 4}\zeta
				\xi_{1}
				\xi_{2}
				\xi_{3}}
			+\frac{L_{2}}
			{\prod_{i = 1}^{r} (1 - q^{- 1\slash 2}
				\xi_{i}^{2})}
			+ \frac{L_{3}}
			{\prod_{i = 1}^{r} (1 + q^{- 1\slash 2}
				\xi_{i}^{2})}\right\}\\
			& \cdot\, \left(1 - \frac{\zeta^{2}
				\xi_{1}^{2}
				\xi_{2}^{2}
				\xi_{3}^{2}}{\sqrt{q}}\right)
			\prod_{1 \le i \le  3}
			\left(1 - \frac{\zeta^{2}
				\xi_{i}^{2}}
			{\sqrt{q}}\right)\!\left(1 - \frac{\zeta^{2}
				\xi_{1}^{2}
				\xi_{2}^{2}
				\xi_{3}^{2}}
			{\sqrt{q}\, \xi_{i}^{4}}\right)\\
			& \!\cdot \; \prod_{i = 1}^{3}\prod_{j = 4}^{r}
			\left(1 - \frac{
				\xi_{i}^{2}
				\xi_{j}^{2}}{q}\right)
			\!\left(1- \frac{
				\xi_{i}^{-2}
				\xi_{j}^{2}}{q}\right)
			\prod_{4 \le k \le l \le  r}\left(1 -
			\frac{\xi_{k}^{2}
				\xi_{l}^{2}}{q}\right)\!.
		\end{split}
	\end{equation*} 
	This function is symmetric in the variables 
	$
	\xi_{1}^{\pm 1}, \xi_{2}^{\pm 1}, \xi_{3}^{\pm 1}, 
	$ 
	and (separately) in $\xi_{k}$ for $4 \le k \le r,$ see 
	\cite[Section 7]{DT}. As the above notation suggests, it is also even in $\xi_{i}$ ($1 \le i \le r$) and $\zeta,$ and as a function of 
	$q,$ it is a polynomial in $q^{- 1\slash 2}$ of the form 
	\[ 
	S = 1 + O\Big(q^{-\frac{3}{2}}\Big).
	\] 
	Substitute 
	$
	t_{i} =
	\xi_{i}^{2},
	$ 
	and denote the resulting function by 
	$
	S(\mathrm{t}, \zeta^{2}\!, 
	q^{- 1\slash 2}).
	$ 
	We define $R_{2}(\mathrm{t}, \zeta)$ by 
	\[
	R_{2}(\mathrm{t}, \zeta)  
	=  \prod_{n \ge 1} S(\mathrm{t}^{n}\!, 
	\zeta^{2
		n}\!, q^{-n
		\slash 2})^{i_{n}\!(q)}
	\] 
	the product being absolutely and uniformly convergent for 
	$
	t_{1}, 
	t_{2}, 
	t_{3}
	$ 
	in a small neighborhood of the unit circle, and 
	$
	|t_{k}| \le q^{\varepsilon}
	$ 
	($k = 4, \ldots, r$) for sufficiently small positive $\varepsilon.$ Letting 
	\[
	G(t_{1},
	t_{2},
	t_{3}, \zeta) = 
	- \, \frac{1 + \frac{1}{q} + \frac{\zeta^{2}}{\sqrt{q}}\sum_{j = 1}^{3}
		\left(t_{j}^{} + t_{j}^{-1}\right)}
	{\left(1 - \frac{\zeta^{2} 
			t_{1} 
			t_{2} 
			t_{3}}{\sqrt{q}}\right)
		\!\left(1 - \frac{\zeta^{2}}
		{\sqrt{q}\,
			t_{1} 
			t_{2} 
			t_{3}}\right)
		\prod_{i = 1}^{3}
		\left(1 - \frac{\zeta^{2} 
			t_{1} 
			t_{2} 
			t_{3}}
		{\sqrt{q}\, t_{i}^{2}}\right)
		\!\left(1 - \frac{\zeta^{2} t_{i}^{2}}
		{\sqrt{q}\, 
			t_{1} 
			t_{2} 
			t_{3}}\right)}
	\] 
	it follows from the computations in \cite[Section 7]{DT} that 
	\begin{equation*}
		\begin{split}
			&  \mathcal{Q}_{2}(t; g, q)\psi(t)^{-g}
			= \frac{(1-q^{-1})q^{-g/2-2}}{2^{4} 3! (r - 3)!}\\
			& \cdot \sum_{\sigma \in \Sigma_{r}}\, 
			\sum_{\substack{\zeta = 1, \sqrt{-1} 
					\\ \epsilon_{\sigma(i)} = \pm 1}}\, 
			\frac{\zeta^{2g}
				G\!\left(t_{\sigma(1)}^{\epsilon_{\sigma(1)}}, t_{\sigma(2)}^{\epsilon_{\sigma(2)}}, t_{\sigma(3)}^{\epsilon_{\sigma(3)}}, \zeta \right)
				\!R_{2}
				\!\left(t_{\sigma(1)}^{\epsilon_{\sigma(1)}}, 
				\ldots, t_{\sigma(r)}^{\epsilon_{\sigma(r)}}, \zeta \right) 
				\left(t_{\sigma(4)}^{\epsilon_{\sigma(4)}} 
				\cdots \, t_{\sigma(r)}^{\epsilon_{\sigma(r)}}\right)^{-g}}
			{\prod_{i = 1}^{3}\prod_{j = 4}^{r}
				\left(1 - t_{\sigma(i)}^{\epsilon_{\sigma(i)}} 
				t_{\sigma(j)}^{\epsilon_{\sigma(j)}}\right)
				\left(1- t_{\sigma(i)}^{-\epsilon_{\sigma(i)}} 
				t_{\sigma(j)}^{\epsilon_{\sigma(j)}}\right)
				\, \cdot \prod_{4 \le k \le l \le r} \left(1 -
				t_{\sigma(k)}^{\epsilon_{\sigma(k)}} 
				t_{\sigma(l)}^{\epsilon_{\sigma(l)}}\right)}.
		\end{split}
	\end{equation*} 
	As before, using the following slight generalization of Lemma \ref{average-to-integral-vers1}, we can represent $\mathcal{Q}_{2}(t; g, q)\psi(t)^{-g}$ as an integral. \end{para}
	 
	\begin{lem}  Let 
		$
		a_{1}, \ldots, a_{r}
		$ 
		be distinct non-zero complex numbers such that $a_{i}
		a_{j} \ne 1$ for all 
		$1\le i, j \le r.$ Suppose $h$ is a function of $r$ complex variables, holomorphic on a domain containing 
		$
		\left(a_{\sigma(1)}^{\epsilon_{1}}, \ldots, a_{\sigma(r)}^{\epsilon_{r}}\right)
		$ 
		for all $\sigma \in \Sigma_{r}$ and 
		$
		\epsilon = (\epsilon_{1}, \ldots, \epsilon_{r}) \in \{\pm 1\}^{r}.
		$ 
		For $0 \le m < r,$ define $K_{m}(\mathrm{z})$ by 
		\[
		K_{m}(\mathrm{z}) = \frac{h(\mathrm{z})}
		{\prod_{k = 1}^{m}\prod_{l = m + 1}^{r} 
			(1 - z_{k}^{} 
			z_{l}^{})   
			(1 - z_{k}^{-1} 
			z_{l}^{})   
			\prod_{m + 1 \le k \le l \le r} 
			(1 - z_{k}^{} 
			z_{l}^{})}, \quad 
		\mathrm{z}: = (z_{1}, \ldots, z_{r}).
		\] 
		Then we have 
		\begin{equation*}
			\begin{split}
				&\sum_{\sigma \in \Sigma_{r}}\, 
				\sum_{\epsilon} 
				K_{m}\left(a_{\sigma(1)}^{\epsilon_{1}}, \ldots, a_{\sigma(r)}^{\epsilon_{r}} \right)\\ 
				& = \frac{(-1)^{r(r  + 1)\slash 2}}{\left(2\pi \sqrt{-1}\right)^{r}}\!
				\oint \cdots \oint 
				h(\mathrm{z}) 
				\frac{\prod_{1\le i  <  j \le r} 
					(z_{i}^{} - z_{j}^{})^{\mathrm{e}_{ij}} 
					(1 - z_{i}^{} 
					z_{j}^{}) \cdot 
					\prod_{1 \le k \le l \le m} 
					(1 - z_{k}^{} 
					z_{l}^{})
					\prod_{i = 1}^{m} z_{i}^{r - m}} 
				{\prod_{i, j = 1}^{r} (1 - z_{i}^{}
					a_{j}^{})
					(1 - z_{i}^{}
					a_{j}^{-1})} 
				\frac{d \mathrm{z}}{\mathrm{z}^{r}}	
			\end{split} 
		\end{equation*} 
		where $\mathrm{e}_{ij} = 1$ or $2$ according as 
		$i\le m$ and $j\ge m+1$ or not, 
		$
		\mathrm{z}^{r} 
		: = z_{1}^{r} \cdots\, z_{r}^{r},$ 
		and where each path of integration encloses the 
		$a_{j}^{\pm 1},$ but not 
		$z_{j}^{} \! = 0.$ 
	\end{lem}

	\begin{proof} The same idea as in the proof of \cite[Lemma 7.1]{DT}. Alternatively, the lemma can be easily deduced from \cite[Lemma 7.5]{DT}. 
	\end{proof}

	\begin{para}We apply this lemma with $m = 3$ and 
	\begin{equation*}
		h(z_{1}, \ldots, 
		z_{r}) \; =  \sum_{\zeta = 1, \sqrt{-1}}\, 
		\zeta^{2g}
		G(z_{1}, 
		z_{2}, 
		z_{3}, \zeta)
		R_{2}
		(z_{1}, 
		\ldots, z_{r}, \zeta) 
		(z_{4} 
		\, \cdots \, z_{r})^{-g}.
	\end{equation*} 
	Following the same idea as in the proof of Proposition \ref{explicit-T1}, one obtains: \end{para}
	
	\begin{prop} \label{explicit-T2}  For every partition $\lambda \subseteq (r^{g})$, we have 
		\begin{align*}
				\mathcal{T}_{2}
				(\lambda^{\dag}) = 
				\frac{(1-q^{-1})q^{-g/2-2}}{2^{4} 3! (r - 3)!}  
				\cdot \frac{(-1)^{r(r - 1)\slash 2}}{\left(2\pi \sqrt{-1}\right)^{r}}\!
				\oint\limits_{|z_{r}| = 1} 
				 \cdots \oint\limits_{|z_{1}| = 1} 
				h(\mathrm{z}) 
				\prod_{1\le i < j \le  r} 
				(z_{i}^{} - z_{j}^{})^{\mathrm{e}_{ij}  - 1} \\ 
				\cdot \prod_{1 \le k \le l \le 3} 
				(1 - z_{k}^{} 
				z_{l}^{})
				\, \prod_{i = 1}^{3} z_{i}^{r -3}
				  \cdot \, \det\left(z_{i}^{g - \lambda_{r + 1 - j}' + r + 1 - j} - z_{i}^{-(g - \lambda_{r + 1 - j}' + r + 1 - j)}\right)_{1\le i, j \le r}\, \frac{d \mathrm{z}}{\mathrm{z}^{r+1}}.
			\end{align*}
	\end{prop} \begin{para}
	For large $g,$ the value of 
	$
	|\mathcal{T}_{2}
	(\lambda^{\dag})|
	$ 
	is as large as possible for partitions $\lambda$ such that 
	$
	g - \lambda_{i}' = O(1)
	$ 
	for $i = 1$, $2, 3$, and 
	$
	\lambda_{i}' = O(1)
	$ 
	for $i \ge 4.$ For instance, take $r = 3$ and 
	$
	\lambda' = (g, g, g)
	$ 
	with $g$ even, that is, $\lambda = (3, 3, 3, \ldots)$. In this case, 
	\begin{align*} 
			\mathcal{T}_{2}
			(\lambda^{\dag}) = 
			- \, \frac{(1-q^{-1})q^{-g/2-2}}{2^{4} 3!}  
			\cdot \frac{1}{\left(2\pi \sqrt{-1}\right)^{3}}\!
			\oint\limits_{|z_{3}| = 1}\, 
			\oint\limits_{|z_{2}| = 1}\,  \oint\limits_{|z_{1}| = 1}
			h(z_{1}, 
			z_{2}, 
			z_{3}) \\
			\prod_{1\le i < j \le  3} 
			(z_{i} - z_{j}) 
			\, \cdot \prod_{1 \le k \le l \le 3} 
			(1 - z_{k} 
			z_{l})
			 \cdot \, 
			\mathrm{det}\!\left(z_{i}^{4 - j} - z_{i}^{- 4 + j}\right)_{1\le i, j \le 3}\, 
			\frac{d z_{1}
				d z_{2}
				d z_{3}}
			{z_{1}^{4}
				z_{2}^{4}
				z_{3}^{4}}
		\end{align*} 
	with 
	\[
	h(z_{1}, 
	z_{2}, 
	z_{3}) \; =  \sum_{\zeta = 1, \sqrt{-1}}
	G(z_{1}, 
	z_{2}, 
	z_{3}, \zeta)
	R_{2}
	(z_{1}, 
	z_{2}, 
	z_{3}, \zeta).
	\] 
	We have 
	\[
	\mathcal{T}_{2}
	(\lambda^{\dag}) =    
	- \, (1-q^{-1})q^{-g/2-2}\left(1 - q^{- 2} + 2 q^{ - 3} - 9 q^{ - 4} - 32 q^{ - 5} + \cdots \right)
	\] 
	which can be easily checked as follows. The function 
	$
	S(z_{1}, 
	z_{2}, 
	z_{3}, 
	\zeta^{2}\!, x)
	$ 
	is a polynomial in $x$ of degree $14$ of the form 
	$ 
	S(z_{1}, 
	z_{2}, 
	z_{3}, 
	\zeta^{2}\!, x)
	= 1 + O\!\left(x^{3}\right)\!,
	$ 
	and thus 
	\[ 
	R_{2}(
	z_{1}, 
	z_{2}, 
	z_{3}, \zeta)  
	\, =  \prod_{n \le 10} S(
	z_{1}^{n}, 
	z_{2}^{n}, 
	z_{3}^{n}, 
	\zeta^{2n}\!, q^{-n
		\slash 2})^{i_{n}\!(q)}\, \cdot \, 
	\left(1 + O\!\left(q^{-\frac{11}{2}}\right)\right)\!. 
	\] 
	Use this approximation, and truncate the power series expansion of 
	$
	G(z_{1},
	z_{2},
	z_{3}, \zeta)
	$ 
	to its $q^{- 5}$th term. Replacing $G$ and $R_{2}$ under the integral of 
	$
	\mathcal{T}_{2}
	(\lambda^{\dag})
	$ 
	by their approximations and truncating again at the $q^{- 5}$th term, 
	our assertion follows by direct computation. 
	\end{para}
	\begin{para}
	On the other hand, \S\ref{main-term-T1-estimate} implies that 
	$
	\mathcal{T}_{1}
	(\lambda^{\dag})
	\ll_{r, \, \varepsilon} 
	q^{\left( - \frac{\rho}{4} + \varepsilon \right)g}
	$ 
	when $(|\lambda|\slash g) \ge \rho,$ for a positive $\rho$. In particular, if $\lambda_{i} = 3$ for all $i,$ then 
	\begin{equation*} 
		\tr_{\lambda}(g) \sim 
		\mathcal{T}_{2}
		(\lambda^{\dag}) 
	\end{equation*}
	when $g \to \infty,$ which shows that 
	$
	\mathcal{T}_{2}
	(\lambda^{\dag})
	$ 
	dominates 
	$
	\mathcal{T}_{1}
	(\lambda^{\dag})
	$ 
	in this case. In fact $\tr_{\lambda}(g) \sim 
		\mathcal{T}_{2}
		(\lambda^{\dag}) $ can be proven unconditionally, 
	since the asymptotic formula \eqref{eq: asympt-moments} when $r = 3$ with an error term on the order of 
	$q^{\left(-\frac{2}{3} + \varepsilon \right)g}$ can be established by a similar technique to that used in \cite{Dia} for the cubic moment of central values of quadratic $L$-functions. \end{para}
 
 \begin{para}The regular behavior of the second term 
	$
	\mathcal{T}_{2}
	(\lambda^{\dag})
	$ 
	when 
	$
	g - \lambda_{i}' = O(1)
	$ 
	for $i = 1, 2, 3$ and 
	$
	\lambda_{i}' = O(1)
	$ 
	for $i \ge 4$ suggests an explanation in terms of \emph{secondary stability} in the sense of Kupers--Galatius--Randal-Williams \cite{gkrw-secondary} for the coefficients 
	$
	\mathbb{V}_{\lambda}
	$. It will be interesting to further investigate this phenomenon in the different regimes of $\lambda,$ dictated by significant changes in the behavior of the traces $\tr_{\lambda}(g).$ 
	
	\end{para}

\bibliographystyle{alpha}
\bibliography{database}

@article {daspetersen,
    AUTHOR = {Das, Ronno and Petersen, Dan},
     TITLE = {The {M}umford conjecture (after {B}ianchi)},
   JOURNAL = {J. Topol.},
  FJOURNAL = {Journal of Topology},
    VOLUME = {18},
      YEAR = {2025},
    NUMBER = {1},
      note = {Paper no. e70016. 29 pages},
   MRCLASS = {14H10 (55P48 55R40 57K20)},
  MRNUMBER = {4878403},
       DOI = {10.1112/topo.70016},
       URL = {https://doi.org/10.1112/topo.70016},
}

@misc{tablesjonas,
  author = {Bergstr\"om, Jonas},
   year = {2025},
  note = {GitHub repository available at  \url{https://github.com/jonasbergstroem/Stable-homology-of-braid-groups-and-moduli-of-hyperelliptic-curves}}
}

@article {bindaparkostvaer,
	AUTHOR = {Binda, Federico and Park, Doosung and {\O}stv{\ae}r, Paul Arne},
	TITLE = {Triangulated categories of logarithmic motives over a field},
	JOURNAL = {Ast\'{e}risque},
	FJOURNAL = {Ast\'{e}risque},
	NUMBER = {433},
	YEAR = {2022},
	PAGES = {ix+267},
	ISSN = {0303-1179},
	ISBN = {978-2-85629-957-9},
	MRCLASS = {14A21 (14A30 14F42 18F10 18G35 18G80 18Nxx 19E15)},
	MRNUMBER = {4439182},
	DOI = {10.24033/ast},
	URL = {https://doi-org.ezp.sub.su.se/10.24033/ast},
}

@article {nakayama,
	AUTHOR = {Nakayama, Chikara},
	TITLE = {Logarithmic \'{e}tale cohomology},
	JOURNAL = {Math. Ann.},
	FJOURNAL = {Mathematische Annalen},
	VOLUME = {308},
	YEAR = {1997},
	NUMBER = {3},
	PAGES = {365--404},
	ISSN = {0025-5831},
	MRCLASS = {14F20},
	MRNUMBER = {1457738},
	MRREVIEWER = {Claudio Pedrini},
	DOI = {10.1007/s002080050081},
	URL = {https://doi-org.ezp.sub.su.se/10.1007/s002080050081},
}

@book {lothaire,
	AUTHOR = {Lothaire, M.},
	TITLE = {Combinatorics on words},
	SERIES = {Encyclopedia of Mathematics and its Applications},
	VOLUME = {17},
	NOTE = {A collective work by Dominique Perrin, Jean Berstel, Christian
	Choffrut, Robert Cori, Dominique Foata, Jean Eric Pin,
	Guiseppe Pirillo, Christophe Reutenauer, Marcel-P.
	Sch\"{u}tzenberger, Jacques Sakarovitch and Imre Simon.
	With a foreword by Roger Lyndon.
	Edited and with a preface by Perrin},
	PUBLISHER = {Addison-Wesley Publishing Co., Reading, Mass.},
	YEAR = {1983},
	PAGES = {xix+238},
	ISBN = {0-201-13516-7},
	MRCLASS = {05-02 (03D03 03D40 20M05 68F99)},
	MRNUMBER = {675953},
	MRREVIEWER = {Heinrich Seidel},
}

@article {noohi-mapping,
	AUTHOR = {Noohi, Behrang},
	TITLE = {Mapping stacks of topological stacks},
	JOURNAL = {J. Reine Angew. Math.},
	FJOURNAL = {Journal f\"{u}r die Reine und Angewandte Mathematik. [Crelle's
	Journal]},
	VOLUME = {646},
	YEAR = {2010},
	PAGES = {117--133},
	ISSN = {0075-4102},
	MRCLASS = {57R19 (55P35 55P50)},
	MRNUMBER = {2719557},
	MRREVIEWER = {Frank Neumann},
	DOI = {10.1515/CRELLE.2010.067},
	URL = {https://doi-org.ezp.sub.su.se/10.1515/CRELLE.2010.067},
}

@article {gmtw,
	AUTHOR = {Galatius, S{\o}ren and Madsen, Ib and Tillmann, Ulrike and Weiss,
	Michael},
	TITLE = {The homotopy type of the cobordism category},
	JOURNAL = {Acta Math.},
	FJOURNAL = {Acta Mathematica},
	VOLUME = {202},
	YEAR = {2009},
	NUMBER = {2},
	PAGES = {195--239},
	ISSN = {0001-5962},
	MRCLASS = {55P47 (55P15 55P42 57R75)},
	MRNUMBER = {2506750},
	MRREVIEWER = {Richard John Steiner},
	DOI = {10.1007/s11511-009-0036-9},
	URL = {https://doi.org/10.1007/s11511-009-0036-9},
}

@article {segalconfiguration,
	AUTHOR = {Segal, Graeme},
	TITLE = {Configuration-spaces and iterated loop-spaces},
	JOURNAL = {Invent. Math.},
	FJOURNAL = {Inventiones Mathematicae},
	VOLUME = {21},
	YEAR = {1973},
	PAGES = {213--221},
	ISSN = {0020-9910},
	MRCLASS = {55D35},
	MRNUMBER = {331377},
	MRREVIEWER = {J. P. May},
	DOI = {10.1007/BF01390197},
	URL = {https://doi.org/10.1007/BF01390197},
}

@article {adams,
	AUTHOR = {Adams, J. Frank},
	TITLE = {On the cobar construction},
	JOURNAL = {Proc. Nat. Acad. Sci. U.S.A.},
	FJOURNAL = {Proceedings of the National Academy of Sciences of the United
	States of America},
	VOLUME = {42},
	YEAR = {1956},
	PAGES = {409--412},
	ISSN = {0027-8424},
	MRCLASS = {55.0X},
	MRNUMBER = {79266},
	MRREVIEWER = {W. S. Massey},
	DOI = {10.1073/pnas.42.7.409},
	URL = {https://doi-org.ezp.sub.su.se/10.1073/pnas.42.7.409},
}

@article {riverazeinalian,
	AUTHOR = {Rivera, Manuel and Zeinalian, Mahmoud},
	TITLE = {Cubical rigidification, the cobar construction and the based
	loop space},
	JOURNAL = {Algebr. Geom. Topol.},
	FJOURNAL = {Algebraic \& Geometric Topology},
	VOLUME = {18},
	YEAR = {2018},
	NUMBER = {7},
	PAGES = {3789--3820},
	ISSN = {1472-2747},
	MRCLASS = {55U35 (18D20 18G30 55P35 55U10 57T30)},
	MRNUMBER = {3892231},
	MRREVIEWER = {Timothy Porter},
	DOI = {10.2140/agt.2018.18.3789},
	URL = {https://doi-org.ezp.sub.su.se/10.2140/agt.2018.18.3789},
}

@incollection {faltingslogarithmic,
    AUTHOR = {Faltings, Gerd},
     TITLE = {{$F$}-isocrystals on open varieties: results and conjectures},
 BOOKTITLE = {The {G}rothendieck {F}estschrift, {V}ol. {II}},
    SERIES = {Progr. Math.},
    VOLUME = {87},
     PAGES = {219--248},
 PUBLISHER = {Birkh\"{a}user Boston, Boston, MA},
      YEAR = {1990},
   MRCLASS = {14F30},
  MRNUMBER = {1106900},
MRREVIEWER = {Min Ho Lee},
}

@incollection {katologarithmic,
    AUTHOR = {Kato, Kazuya},
     TITLE = {Logarithmic structures of {F}ontaine-{I}llusie},
 BOOKTITLE = {Algebraic analysis, geometry, and number theory ({B}altimore,
              {MD}, 1988)},
     PAGES = {191--224},
 PUBLISHER = {Johns Hopkins Univ. Press, Baltimore, MD},
      YEAR = {1989},
   MRCLASS = {14F30 (14G20)},
  MRNUMBER = {1463703},
MRREVIEWER = {Adolfo Quir\'{o}s},
}

@article {levinetubular,
    AUTHOR = {Levine, Marc},
     TITLE = {Motivic tubular neighborhoods},
   JOURNAL = {Doc. Math.},
  FJOURNAL = {Documenta Mathematica},
    VOLUME = {12},
      YEAR = {2007},
     PAGES = {71--146},
      ISSN = {1431-0635},
   MRCLASS = {14F42 (14C25 18F20 55P42)},
  MRNUMBER = {2302525},
MRREVIEWER = {Oliver R\"{o}ndigs},
}

@book {polishchukpositselski,
    AUTHOR = {Polishchuk, Alexander and Positselski, Leonid},
     TITLE = {Quadratic algebras},
    SERIES = {University Lecture Series},
    VOLUME = {37},
 PUBLISHER = {American Mathematical Society, Providence, RI},
      YEAR = {2005},
     PAGES = {xii+159},
      ISBN = {0-8218-3834-2},
   MRCLASS = {16S37 (16E30)},
  MRNUMBER = {2177131},
MRREVIEWER = {Ralf Fr\"{o}berg},
       DOI = {10.1090/ulect/037},
       URL = {https://doi.org/10.1090/ulect/037},
}

@article {backelinfroberg,
    AUTHOR = {Backelin, J\"{o}rgen and Fr\"{o}berg, Ralf},
     TITLE = {Koszul algebras, {V}eronese subrings and rings with linear
              resolutions},
   JOURNAL = {Rev. Roumaine Math. Pures Appl.},
  FJOURNAL = {Acad\'{e}mie de la R\'{e}publique Populaire Roumaine. Revue Roumaine
              de Math\'{e}matiques Pures et Appliqu\'{e}es},
    VOLUME = {30},
      YEAR = {1985},
    NUMBER = {2},
     PAGES = {85--97},
      ISSN = {0035-3965},
   MRCLASS = {16A03},
  MRNUMBER = {789425},
MRREVIEWER = {Freddy M. J. Van Oystaeyen},
}

@article {saw3,
    AUTHOR = {Sawin, Will},
     TITLE = {Square-root cancellation for sums of factorization functions
              over squarefree progressions in {$\Bbb F_q[t]$}},
   JOURNAL = {Acta Math.},
  FJOURNAL = {Acta Mathematica},
    VOLUME = {233},
      YEAR = {2024},
    NUMBER = {2},
     PAGES = {285--418},
      ISSN = {0001-5962,1871-2509},
   MRCLASS = {11N37 (11T55)},
  MRNUMBER = {4827656},
       DOI = {10.4310/acta.2024.v233.n2.a3},
       URL = {https://doi.org/10.4310/acta.2024.v233.n2.a3},
}

@unpublished{EL,
    AUTHOR = {Ellenberg, Jordan and Landesman, Aaron},
     TITLE = {Homological stability for generalized {H}urwitz spaces and {S}elmer groups in quadratic twist families over function fields},
 	year = {2023},
	note = "Preprint available at arXiv:2310.16286",
 }

@unpublished{SheSt24,
    AUTHOR = {Shen, Quanli and Stucky, Joshua},
     TITLE = {The fourth moment of quadratic {D}irichlet ${L}$-Functions {II}},
 	year = {2024},
	note = "Preprint available at arXiv:2402.01497",
 }

@article {EVW,
    AUTHOR = {Ellenberg, Jordan S. and Venkatesh, Akshay and Westerland,
              Craig},
     TITLE = {Homological stability for {H}urwitz spaces and the
              {C}ohen-{L}enstra conjecture over function fields},
   JOURNAL = {Ann. of Math. (2)},
  FJOURNAL = {Annals of Mathematics. Second Series},
    VOLUME = {183},
      YEAR = {2016},
    NUMBER = {3},
     PAGES = {729--786},
      ISSN = {0003-486X,1939-8980},
   MRCLASS = {14H10 (11G25 11R29 14H30 55R80)},
  MRNUMBER = {3488737},
MRREVIEWER = {Benjamin\ Collas},
       DOI = {10.4007/annals.2016.183.3.1},
       URL = {https://doi.org/10.4007/annals.2016.183.3.1},
}

@article {saw1,
    AUTHOR = {Sawin, Will},
     TITLE = {Square-root cancellation for sums of factorization functions
              over short intervals in function fields},
   JOURNAL = {Duke Math. J.},
  FJOURNAL = {Duke Mathematical Journal},
    VOLUME = {170},
      YEAR = {2021},
    NUMBER = {5},
     PAGES = {997--1026},
      ISSN = {0012-7094,1547-7398},
   MRCLASS = {11T55 (11M38 11N37)},
  MRNUMBER = {4255048},
MRREVIEWER = {R\'{e}gis\ Blache},
       DOI = {10.1215/00127094-2020-0060},
       URL = {https://doi.org/10.1215/00127094-2020-0060},
}

@article {saw2,
    AUTHOR = {Sawin, Will},
     TITLE = {A representation theory approach to integral moments of
              {$L$}-functions over function fields},
   JOURNAL = {Algebra Number Theory},
  FJOURNAL = {Algebra \& Number Theory},
    VOLUME = {14},
      YEAR = {2020},
    NUMBER = {4},
     PAGES = {867--906},
      ISSN = {1937-0652,1944-7833},
   MRCLASS = {11M50 (11M38 11T55 14F20 20G05)},
  MRNUMBER = {4114059},
MRREVIEWER = {R\'{e}gis\ Blache},
       DOI = {10.2140/ant.2020.14.867},
       URL = {https://doi.org/10.2140/ant.2020.14.867},
}

@article {arnoldbraid2,
	AUTHOR = {Arnold, Vladimir I.},
	TITLE = {Certain topological invariants of algebrac functions},
	JOURNAL = {Trudy Moskov. Mat. Ob\v{s}\v{c}.},
	FJOURNAL = {Trudy Moskovskogo Matemati\v{c}eskogo Ob\v{s}\v{c}estva},
	VOLUME = {21},
	YEAR = {1970},
	PAGES = {27--46},
	ISSN = {0134-8663},
	MRCLASS = {14.55},
	MRNUMBER = {0274462},
	MRREVIEWER = {N. Popescu},
}

@article {chen-burau,
	AUTHOR = {Chen, Weiyan},
	TITLE = {Homology of braid groups, the {B}urau representation, and
	points on superelliptic curves over finite fields},
	JOURNAL = {Israel J. Math.},
	FJOURNAL = {Israel Journal of Mathematics},
	VOLUME = {220},
	YEAR = {2017},
	NUMBER = {2},
	PAGES = {739--762},
	ISSN = {0021-2172},
	MRCLASS = {20F36 (11G05 14H20 14H50)},
	MRNUMBER = {3666443},
	MRREVIEWER = {Markus Szymik},
	DOI = {10.1007/s11856-017-1534-7},
	URL = {https://doi.org/10.1007/s11856-017-1534-7},
}

@article {birmanhilden,
	AUTHOR = {Birman, Joan S. and Hilden, Hugh M.},
	TITLE = {On isotopies of homeomorphisms of {R}iemann surfaces},
	JOURNAL = {Ann. of Math. (2)},
	FJOURNAL = {Annals of Mathematics. Second Series},
	VOLUME = {97},
	YEAR = {1973},
	PAGES = {424--439},
	ISSN = {0003-486X},
	MRCLASS = {30A48},
	MRNUMBER = {325959},
	MRREVIEWER = {William Harvey},
	DOI = {10.2307/1970830},
	URL = {https://doi.org/10.2307/1970830},
}

@article {songtillmann,
	AUTHOR = {Song, Yongjin and Tillmann, Ulrike},
	TITLE = {Braids, mapping class groups, and categorical delooping},
	JOURNAL = {Math. Ann.},
	FJOURNAL = {Mathematische Annalen},
	VOLUME = {339},
	YEAR = {2007},
	NUMBER = {2},
	PAGES = {377--393},
	ISSN = {0025-5831},
	MRCLASS = {55P48 (55R37 57M50)},
	MRNUMBER = {2324724},
	MRREVIEWER = {Paolo Salvatore},
	DOI = {10.1007/s00208-007-0117-z},
	URL = {https://doi.org/10.1007/s00208-007-0117-z},
}

@incollection {segaltillmann,
	AUTHOR = {Segal, Graeme and Tillmann, Ulrike},
	TITLE = {Mapping configuration spaces to moduli spaces},
	BOOKTITLE = {Groups of diffeomorphisms},
	SERIES = {Adv. Stud. Pure Math.},
	VOLUME = {52},
	PAGES = {469--477},
	PUBLISHER = {Math. Soc. Japan, Tokyo},
	YEAR = {2008},
	MRCLASS = {55R80 (32G15)},
	MRNUMBER = {2509721},
	MRREVIEWER = {John McCleary},
	DOI = {10.2969/aspm/05210469},
	URL = {https://doi.org/10.2969/aspm/05210469},
}

@phdthesis{cazanave,
	author       = {Christophe Cazanave}, 
	title        = {Th\'eorie homotopique des sch\'emas d'{A}tiyah et {H}itchins},
	school       = {Universit\'e Paris 13},
	year         = "2009"
}

@article {Beh,
	AUTHOR = {Behrend, Kai A.},
	TITLE = {Derived {$l$}-adic categories for algebraic stacks},
	JOURNAL = {Mem. Amer. Math. Soc.},
	FJOURNAL = {Memoirs of the American Mathematical Society},
	VOLUME = {163},
	YEAR = {2003},
	NUMBER = {774},
	PAGES = {viii+93},
	ISSN = {0065-9266},
	MRCLASS = {14A20 (11F25 11G05 14F99 18E30)},
	MRNUMBER = {1963494},
	MRREVIEWER = {Burt Totaro},
	DOI = {10.1090/memo/0774},
	URL = {https://doi.org/10.1090/memo/0774},
}

@book {Kac,
    AUTHOR = {Kac, Victor G.},
     TITLE = {Infinite-dimensional {L}ie algebras},
   EDITION = {Third},
 PUBLISHER = {Cambridge University Press, Cambridge},
      YEAR = {1990},
     PAGES = {xxii+400},
      ISBN = {0-521-37215-1; 0-521-46693-8},
   MRCLASS = {17B65 (17B67 17B68 58F07)},
  MRNUMBER = {1104219},
       DOI = {10.1017/CBO9780511626234},
       URL = {https://doi-org.ezp.sub.su.se/10.1017/CBO9780511626234},
}

@book {R,
    AUTHOR = {Rosen, Michael},
     TITLE = {Number theory in function fields},
    SERIES = {Graduate Texts in Mathematics},
    VOLUME = {210},
 PUBLISHER = {Springer-Verlag, New York},
      YEAR = {2002},
     PAGES = {xii+358},
      ISBN = {0-387-95335-3},
   MRCLASS = {11R58 (11R60 11T55)},
  MRNUMBER = {1876657},
MRREVIEWER = {Ernst-Ulrich Gekeler},
       DOI = {10.1007/978-1-4757-6046-0},
       URL = {https://doi-org.ezp.sub.su.se/10.1007/978-1-4757-6046-0},
}

@article {R-G,
    AUTHOR = {Roditty-Gershon, Edva},
     TITLE = {Statistics for products of traces of high powers of the
              {F}robenius class of hyperelliptic curves},
   JOURNAL = {J. Number Theory},
  FJOURNAL = {Journal of Number Theory},
    VOLUME = {132},
      YEAR = {2012},
    NUMBER = {3},
     PAGES = {467--484},
      ISSN = {0022-314X},
   MRCLASS = {11M50 (11G20)},
  MRNUMBER = {2875350},
MRREVIEWER = {Niko Naumann},
       DOI = {10.1016/j.jnt.2011.09.008},
       URL = {https://doi-org.ezp.sub.su.se/10.1016/j.jnt.2011.09.008},
}

@article {bump-gamburd,
    AUTHOR = {Bump, Daniel and Gamburd, Alex},
     TITLE = {On the averages of characteristic polynomials from classical
              groups},
   JOURNAL = {Comm. Math. Phys.},
  FJOURNAL = {Communications in Mathematical Physics},
    VOLUME = {265},
      YEAR = {2006},
    NUMBER = {1},
     PAGES = {227--274},
      ISSN = {0010-3616},
   MRCLASS = {60B15 (05E10 15A52 20G05 81R05)},
  MRNUMBER = {2217304},
MRREVIEWER = {Uwe Franz},
       DOI = {10.1007/s00220-006-1503-1},
       URL = {https://doi-org.ezp.sub.su.se/10.1007/s00220-006-1503-1},
}

@incollection {How,
    AUTHOR = {Howe, Roger},
     TITLE = {Perspectives on invariant theory: {S}chur duality,
              multiplicity-free actions and beyond},
 BOOKTITLE = {The {S}chur lectures (1992) ({T}el {A}viv)},
    SERIES = {Israel Math. Conf. Proc.},
    VOLUME = {8},
     PAGES = {1--182},
 PUBLISHER = {Bar-Ilan Univ., Ramat Gan},
      YEAR = {1995},
   MRCLASS = {13A50 (15A72 20G05 22E46)},
  MRNUMBER = {1321638},
MRREVIEWER = {Frank D. Grosshans},
}

@book {Bump,
    AUTHOR = {Bump, Daniel},
     TITLE = {Lie groups},
    SERIES = {Graduate Texts in Mathematics},
    VOLUME = {225},
 PUBLISHER = {Springer-Verlag, New York},
      YEAR = {2004},
     PAGES = {xii+451},
      ISBN = {0-387-21154-3},
   MRCLASS = {22-01 (22C05 22E15 22E46)},
  MRNUMBER = {2062813},
MRREVIEWER = {Karl-Hermann Neeb},
       DOI = {10.1007/978-1-4757-4094-3},
       URL = {https://doi-org.ezp.sub.su.se/10.1007/978-1-4757-4094-3},
}

@article {A,
    AUTHOR = {Artin, Emil},
     TITLE = {Quadratische {K}\"{o}rper im {G}ebiete der h\"{o}heren {K}ongruenzen.
              {I}},
   JOURNAL = {Math. Z.},
  FJOURNAL = {Mathematische Zeitschrift},
    VOLUME = {19},
      YEAR = {1924},
    NUMBER = {1},
     PAGES = {153--206},
      ISSN = {0025-5874},
   MRCLASS = {DML},
  MRNUMBER = {1544651},
       DOI = {10.1007/BF01181074},
       URL = {https://doi-org.ezp.sub.su.se/10.1007/BF01181074},
}

@article {tillmann-stable,
	AUTHOR = {Tillmann, Ulrike},
	TITLE = {On the homotopy of the stable mapping class group},
	JOURNAL = {Invent. Math.},
	FJOURNAL = {Inventiones Mathematicae},
	VOLUME = {130},
	YEAR = {1997},
	NUMBER = {2},
	PAGES = {257--275},
	ISSN = {0020-9910},
	MRCLASS = {57M99 (19L20 55P47 55Q50 55S12 55U99)},
	MRNUMBER = {1474157},
	MRREVIEWER = {Darryl McCullough},
	DOI = {10.1007/s002220050184},
	URL = {https://doi-org.ezp.sub.su.se/10.1007/s002220050184},
}

@article {hamstrom,
	AUTHOR = {Hamstrom, Mary-Elizabeth},
	TITLE = {Homotopy groups of the space of homeomorphisms on a
	{$2$}-manifold},
	JOURNAL = {Illinois J. Math.},
	FJOURNAL = {Illinois Journal of Mathematics},
	VOLUME = {10},
	YEAR = {1966},
	PAGES = {563--573},
	ISSN = {0019-2082},
	MRCLASS = {55.45 (57.55)},
	MRNUMBER = {202140},
	MRREVIEWER = {George McCarty},
	URL = {http://projecteuclid.org.ezp.sub.su.se/euclid.ijm/1256054895},
}

@article {teichmuller,
	AUTHOR = {Teichm\"{u}ller, Oswald},
	TITLE = {Extremale quasikonforme {A}bbildungen und quadratische
	{D}ifferentiale},
	JOURNAL = {Abh. Preuss. Akad. Wiss. Math.-Nat. Kl.},
	VOLUME = {1939},
	YEAR = {1940},
	NUMBER = {22},
	PAGES = {197},
	MRCLASS = {27.0X},
	MRNUMBER = {0003242},
	MRREVIEWER = {R. P. Boas, Jr.},
}

@article {earle-eells,
	AUTHOR = {Earle, Clifford J. and Eells, James},
	TITLE = {A fibre bundle description of {T}eichm\"{u}ller theory},
	JOURNAL = {J. Differential Geometry},
	FJOURNAL = {Journal of Differential Geometry},
	VOLUME = {3},
	YEAR = {1969},
	PAGES = {19--43},
	ISSN = {0022-040X},
	MRCLASS = {57.47},
	MRNUMBER = {276999},
	URL = {http://projecteuclid.org.ezp.sub.su.se/euclid.jdg/1214428816},
}

@article {dwyertwisted,
	AUTHOR = {Dwyer, William G.},
	TITLE = {Twisted homological stability for general linear groups},
	JOURNAL = {Ann. of Math. (2)},
	FJOURNAL = {Annals of Mathematics. Second Series},
	VOLUME = {111},
	YEAR = {1980},
	NUMBER = {2},
	PAGES = {239--251},
	ISSN = {0003-486X},
	MRCLASS = {18G99 (57R99)},
	MRNUMBER = {569072},
	MRREVIEWER = {J. F. Adams},
	DOI = {10.2307/1971200},
	URL = {https://doi.org/10.2307/1971200},
}

@article {krannichtopologicalmoduli,
	AUTHOR = {Krannich, Manuel},
	TITLE = {Homological stability of topological moduli spaces},
	JOURNAL = {Geom. Topol.},
	FJOURNAL = {Geometry \& Topology},
	VOLUME = {23},
	YEAR = {2019},
	NUMBER = {5},
	PAGES = {2397--2474},
	ISSN = {1465-3060},
	MRCLASS = {55P48 (55R40 55R80 57R19 57R50)},
	MRNUMBER = {4019896},
	MRREVIEWER = {Maria Basterra},
	DOI = {10.2140/gt.2019.23.2397},
	URL = {https://doi.org/10.2140/gt.2019.23.2397},
}

@article {Jut,
    AUTHOR = {Jutila, Matti},
     TITLE = {On the mean value of {$L({\tfrac{1}{2}},\,\chi )$} for real
              characters},
   JOURNAL = {Analysis},
  FJOURNAL = {Analysis. International Journal of Analysis and its
              Application},
    VOLUME = {1},
      YEAR = {1981},
    NUMBER = {2},
     PAGES = {149--161},
      ISSN = {0174-4747},
   MRCLASS = {10H08},
  MRNUMBER = {632705},
MRREVIEWER = {D. R. Heath-Brown},
       DOI = {10.1524/anly.1981.1.2.149},
       URL = {https://doi-org.ezp.sub.su.se/10.1524/anly.1981.1.2.149},
}

@article {Young,
    AUTHOR = {Young, Matthew P.},
     TITLE = {The third moment of quadratic {D}irichlet {L}-functions},
   JOURNAL = {Selecta Math. (N.S.)},
  FJOURNAL = {Selecta Mathematica. New Series},
    VOLUME = {19},
      YEAR = {2013},
    NUMBER = {2},
     PAGES = {509--543},
      ISSN = {1022-1824},
   MRCLASS = {11M06 (11A25 11N37)},
  MRNUMBER = {3090236},
MRREVIEWER = {Olivier Bordell\`es},
       DOI = {10.1007/s00029-012-0104-4},
       URL = {https://doi-org.ezp.sub.su.se/10.1007/s00029-012-0104-4},
}

@article {SY,
    AUTHOR = {Soundararajan, K. and Young, Matthew P.},
     TITLE = {The second moment of quadratic twists of modular
              {$L$}-functions},
   JOURNAL = {J. Eur. Math. Soc. (JEMS)},
  FJOURNAL = {Journal of the European Mathematical Society (JEMS)},
    VOLUME = {12},
      YEAR = {2010},
    NUMBER = {5},
     PAGES = {1097--1116},
      ISSN = {1435-9855},
   MRCLASS = {11F66 (11F67)},
  MRNUMBER = {2677611},
MRREVIEWER = {D. R. Heath-Brown},
       DOI = {10.4171/JEMS/224},
       URL = {https://doi-org.ezp.sub.su.se/10.4171/JEMS/224},
}

@article {Ing,
    AUTHOR = {Ingham, Albert E.},
     TITLE = {Mean-value theorems in the theory of the {R}iemann
              zeta-function},
   JOURNAL = {Proc. London Math. Soc. (2)},
  FJOURNAL = {Proceedings of the London Mathematical Society. Second Series},
    VOLUME = {27},
      YEAR = {1927},
    NUMBER = {4},
     PAGES = {273--300},
      ISSN = {0024-6115},
   MRCLASS = {DML},
  MRNUMBER = {1575391},
       DOI = {10.1112/plms/s2-27.1.273},
       URL = {https://doi-org.ezp.sub.su.se/10.1112/plms/s2-27.1.273},
}

@article {Ha-Li,
    AUTHOR = {Hardy, G. H. and Littlewood, J. E.},
     TITLE = {Contributions to the theory of the {R}iemann zeta-function and
              the theory of the distribution of primes},
   JOURNAL = {Acta Math.},
  FJOURNAL = {Acta Mathematica},
    VOLUME = {41},
      YEAR = {1916},
    NUMBER = {1},
     PAGES = {119--196},
      ISSN = {0001-5962},
   MRCLASS = {DML},
  MRNUMBER = {1555148},
       DOI = {10.1007/BF02422942},
       URL = {https://doi-org.ezp.sub.su.se/10.1007/BF02422942},
}

@article {GHRR,
    AUTHOR = {Goulden, Ian P. and Huynh, Duc Khiem and Rishikesh and
              Rubinstein, Michael O.},
     TITLE = {Lower order terms for the moments of symplectic and orthogonal
              families of {$L$}-functions},
   JOURNAL = {J. Number Theory},
  FJOURNAL = {Journal of Number Theory},
    VOLUME = {133},
      YEAR = {2013},
    NUMBER = {2},
     PAGES = {639--674},
      ISSN = {0022-314X},
   MRCLASS = {11A25 (11M06 11N37)},
  MRNUMBER = {2994379},
MRREVIEWER = {Olivier Bordell\`es},
       DOI = {10.1016/j.jnt.2012.08.009},
       URL = {https://doi-org.ezp.sub.su.se/10.1016/j.jnt.2012.08.009},
}

@article {Sound,
    AUTHOR = {Soundararajan, K.},
     TITLE = {Nonvanishing of quadratic {D}irichlet {$L$}-functions at
              {$s=\frac12$}},
   JOURNAL = {Ann. of Math. (2)},
  FJOURNAL = {Annals of Mathematics. Second Series},
    VOLUME = {152},
      YEAR = {2000},
    NUMBER = {2},
     PAGES = {447--488},
      ISSN = {0003-486X},
   MRCLASS = {11M20 (11R42)},
  MRNUMBER = {1804529},
MRREVIEWER = {J. B. Conrey},
       DOI = {10.2307/2661390},
       URL = {https://doi-org.ezp.sub.su.se/10.2307/2661390},
}

@article {CFKRS,
    AUTHOR = {Conrey, J. B. and Farmer, D. W. and Keating, J. P. and
              Rubinstein, M. O. and Snaith, N. C.},
     TITLE = {Integral moments of {$L$}-functions},
   JOURNAL = {Proc. London Math. Soc. (3)},
  FJOURNAL = {Proceedings of the London Mathematical Society. Third Series},
    VOLUME = {91},
      YEAR = {2005},
    NUMBER = {1},
     PAGES = {33--104},
      ISSN = {0024-6115},
   MRCLASS = {11M26},
  MRNUMBER = {2149530},
MRREVIEWER = {K. Soundararajan},
       DOI = {10.1112/S0024611504015175},
       URL = {https://doi-org.ezp.sub.su.se/10.1112/S0024611504015175},
}

@article {florea3,
    AUTHOR = {Florea, Alexandra},
     TITLE = {The fourth moment of quadratic {D}irichlet {$L$}-functions
              over function fields},
   JOURNAL = {Geom. Funct. Anal.},
  FJOURNAL = {Geometric and Functional Analysis},
    VOLUME = {27},
      YEAR = {2017},
    NUMBER = {3},
     PAGES = {541--595},
      ISSN = {1016-443X},
   MRCLASS = {11M06 (11R58)},
  MRNUMBER = {3655956},
MRREVIEWER = {Caroline L. Turnage-Butterbaugh},
       DOI = {10.1007/s00039-017-0409-8},
       URL = {https://doi-org.ezp.sub.su.se/10.1007/s00039-017-0409-8},
}

@article {florea2,
    AUTHOR = {Florea, Alexandra},
     TITLE = {The second and third moment of {$L(1/2,\chi)$} in the
              hyperelliptic ensemble},
   JOURNAL = {Forum Math.},
  FJOURNAL = {Forum Mathematicum},
    VOLUME = {29},
      YEAR = {2017},
    NUMBER = {4},
     PAGES = {873--892},
      ISSN = {0933-7741},
   MRCLASS = {11M06 (11M38)},
  MRNUMBER = {3669007},
MRREVIEWER = {Biswajyoti Saha},
       DOI = {10.1515/forum-2015-0152},
       URL = {https://doi-org.ezp.sub.su.se/10.1515/forum-2015-0152},
}

@article {florea1,
    AUTHOR = {Florea, Alexandra M.},
     TITLE = {Improving the error term in the mean value of
              {$L(\frac12,\chi)$} in the hyperelliptic ensemble},
   JOURNAL = {Int. Math. Res. Not. IMRN},
  FJOURNAL = {International Mathematics Research Notices. IMRN},
      YEAR = {2017},
    NUMBER = {20},
     PAGES = {6119--6148},
      ISSN = {1073-7928},
   MRCLASS = {11M06 (11M50)},
  MRNUMBER = {3712193},
MRREVIEWER = {Steven Joel Miller},
       DOI = {10.1093/imrn/rnv387},
       URL = {https://doi-org.ezp.sub.su.se/10.1093/imrn/rnv387},
}

@article {randalwilliamstwisted,
    AUTHOR = {Randal-Williams, Oscar},
     TITLE = {Cohomology of automorphism groups of free groups with twisted
              coefficients},
   JOURNAL = {Selecta Math. (N.S.)},
  FJOURNAL = {Selecta Mathematica. New Series},
    VOLUME = {24},
      YEAR = {2018},
    NUMBER = {2},
     PAGES = {1453--1478},
      ISSN = {1022-1824},
   MRCLASS = {20F28 (20J06 57R20)},
  MRNUMBER = {3782426},
MRREVIEWER = {Valeriy G. Bardakov},
       DOI = {10.1007/s00029-017-0311-0},
       URL = {https://doi.org/10.1007/s00029-017-0311-0},
}

@article {DGH,
    AUTHOR = {Diaconu, Adrian and Goldfeld, Dorian and Hoffstein, Jeffrey},
     TITLE = {Multiple {D}irichlet series and moments of zeta and
              {$L$}-functions},
   JOURNAL = {Compositio Math.},
  FJOURNAL = {Compositio Mathematica},
    VOLUME = {139},
      YEAR = {2003},
    NUMBER = {3},
     PAGES = {297--360},
      ISSN = {0010-437X},
   MRCLASS = {11M06 (11F66 11M41)},
  MRNUMBER = {2041614},
MRREVIEWER = {Emmanuel P. Royer},
       DOI = {10.1023/B:COMP.0000018137.38458.68},
       URL = {https://doi-org.ezp.sub.su.se/10.1023/B:COMP.0000018137.38458.68},
}

@article {Dia,
    AUTHOR = {Diaconu, Adrian},
     TITLE = {On the third moment of {$L(\frac{1}{2},\chi_d)$} {I}: {T}he
              rational function field case},
   JOURNAL = {J. Number Theory},
  FJOURNAL = {Journal of Number Theory},
    VOLUME = {198},
      YEAR = {2019},
     PAGES = {1--42},
      ISSN = {0022-314X},
   MRCLASS = {11M06 (11R58)},
  MRNUMBER = {3912928},
MRREVIEWER = {Sandro Bettin},
       DOI = {10.1016/j.jnt.2018.09.023},
       URL = {https://doi-org.ezp.sub.su.se/10.1016/j.jnt.2018.09.023},
}

@article {gkrw-secondary,
    AUTHOR = {Galatius, S{\o}ren and Kupers, Alexander and Randal-Williams,
              Oscar},
     TITLE = {{$E_2$}-cells and mapping class groups},
   JOURNAL = {Publ. Math. Inst. Hautes \'{E}tudes Sci.},
  FJOURNAL = {Publications Math\'{e}matiques. Institut de Hautes \'{E}tudes
              Scientifiques},
    VOLUME = {130},
      YEAR = {2019},
     PAGES = {1--61},
      ISSN = {0073-8301},
   MRCLASS = {20J05 (20F38 57K20)},
  MRNUMBER = {4028513},
MRREVIEWER = {Nick Salter},
       DOI = {10.1007/s10240-019-00107-8},
       URL = {https://doi-org.ezp.sub.su.se/10.1007/s10240-019-00107-8},
}

@article {randalwilliamswahl,
	AUTHOR = {Randal-Williams, Oscar and Wahl, Nathalie},
	TITLE = {Homological stability for automorphism groups},
	JOURNAL = {Adv. Math.},
	FJOURNAL = {Advances in Mathematics},
	VOLUME = {318},
	YEAR = {2017},
	PAGES = {534--626},
	ISSN = {0001-8708},
	MRCLASS = {20J05},
	MRNUMBER = {3689750},
	MRREVIEWER = {Jason Philip Gino Semeraro},
	DOI = {10.1016/j.aim.2017.07.022},
	URL = {https://doi.org/10.1016/j.aim.2017.07.022},
}

@article {noohi-homotopy,
	AUTHOR = {Noohi, Behrang},
	TITLE = {Homotopy types of topological stacks},
	JOURNAL = {Adv. Math.},
	FJOURNAL = {Advances in Mathematics},
	VOLUME = {230},
	YEAR = {2012},
	NUMBER = {4-6},
	PAGES = {2014--2047},
	ISSN = {0001-8708},
	MRCLASS = {55U40 (55P65 55R65)},
	MRNUMBER = {2927363},
	MRREVIEWER = {Frank Neumann},
	DOI = {10.1016/j.aim.2012.04.001},
	URL = {https://doi-org.ezp.sub.su.se/10.1016/j.aim.2012.04.001},
}

@article {bauerformal,
	AUTHOR = {Bauer, Tilman},
	TITLE = {Formal plethories},
	JOURNAL = {Adv. Math.},
	FJOURNAL = {Advances in Mathematics},
	VOLUME = {254},
	YEAR = {2014},
	PAGES = {497--569},
	ISSN = {0001-8708},
	MRCLASS = {55S25 (16W99 18D20)},
	MRNUMBER = {3161106},
	MRREVIEWER = {J. P. C. Greenlees},
	DOI = {10.1016/j.aim.2013.12.023},
	URL = {https://doi-org.ezp.sub.su.se/10.1016/j.aim.2013.12.023},
}

@article {AK,
    AUTHOR = {Andrade, J. C. and Keating, J. P.},
     TITLE = {Conjectures for the integral moments and ratios of
              {$L$}-functions over function fields},
   JOURNAL = {J. Number Theory},
  FJOURNAL = {Journal of Number Theory},
    VOLUME = {142},
      YEAR = {2014},
     PAGES = {102--148},
      ISSN = {0022-314X},
   MRCLASS = {11G20 (11M50 14G10)},
  MRNUMBER = {3208396},
MRREVIEWER = {Steven Joel Miller},
       DOI = {10.1016/j.jnt.2014.02.019},
       URL = {https://doi-org.ezp.sub.su.se/10.1016/j.jnt.2014.02.019},
}

@unpublished{shuklin,
	author = {Shuklin, Georgii},
	title = "{A {V}oevodsky motive associated to a log scheme}",
	note = {Available at arXiv:2209.03720},
	year = {2022},
}

@article {XLi,
    AUTHOR = {Li, Xiannan},
     TITLE = {Moments of quadratic twists of modular {$L$}-functions},
   JOURNAL = {Invent. Math.},
  FJOURNAL = {Inventiones Mathematicae},
    VOLUME = {237},
      YEAR = {2024},
    NUMBER = {2},
     PAGES = {697--733},
      ISSN = {0020-9910,1432-1297},
   MRCLASS = {11F67 (11F11)},
  MRNUMBER = {4768632},
       DOI = {10.1007/s00222-024-01265-1},
       URL = {https://doi.org/10.1007/s00222-024-01265-1},
}

@article {harrvistrupwahl,
    AUTHOR = {Harr, Oscar and Vistrup, Max and Wahl, Nathalie},
     TITLE = {Disordered arcs and {H}arer stability},
   JOURNAL = {High. Struct.},
  FJOURNAL = {Higher Structures},
    VOLUME = {8},
      YEAR = {2024},
    NUMBER = {1},
     PAGES = {193--223},
      ISSN = {2209-0606},
   MRCLASS = {57K20 (57M07 57R50)},
  MRNUMBER = {4752520},
}

@article {bianchi3,
    AUTHOR = {Bianchi, Andrea},
     TITLE = {Deloopings of {H}urwitz spaces},
   JOURNAL = {Compos. Math.},
  FJOURNAL = {Compositio Mathematica},
    VOLUME = {160},
      YEAR = {2024},
    NUMBER = {7},
     PAGES = {1651--1714},
      ISSN = {0010-437X,1570-5846},
   MRCLASS = {55P35 (55N45 55P62 55R80 57T25)},
  MRNUMBER = {4797109},
       DOI = {10.1112/S0010437X2400719X},
       URL = {https://doi.org/10.1112/S0010437X2400719X},
}

@article {bianchi2,
    AUTHOR = {Bianchi, Andrea},
     TITLE = {Hurwitz-{R}an spaces},
   JOURNAL = {Geom. Dedicata},
  FJOURNAL = {Geometriae Dedicata},
    VOLUME = {217},
      YEAR = {2023},
    NUMBER = {5},
     PAGES = {1--56, Article ID 84},
      ISSN = {0046-5755,1572-9168},
   MRCLASS = {55R80 (18F60 54B15)},
  MRNUMBER = {4622111},
       DOI = {10.1007/s10711-023-00820-z},
       URL = {https://doi.org/10.1007/s10711-023-00820-z},
}

@article {bianchi4,
    AUTHOR = {Bianchi, Andrea},
     TITLE = {Moduli spaces of {R}iemann surfaces as {H}urwitz spaces},
   JOURNAL = {Adv. Math.},
  FJOURNAL = {Advances in Mathematics},
    VOLUME = {430},
      YEAR = {2023},
     PAGES = {1--62, Article ID 109217},
      ISSN = {0001-8708,1090-2082},
   MRCLASS = {55P62 (55P35 55R80 57T25)},
  MRNUMBER = {4621956},
       DOI = {10.1016/j.aim.2023.109217},
       URL = {https://doi.org/10.1016/j.aim.2023.109217},
}

@incollection {luriecobordism,
    AUTHOR = {Lurie, Jacob},
     TITLE = {On the classification of topological field theories},
 BOOKTITLE = {Current developments in mathematics, 2008},
     PAGES = {129--280},
 PUBLISHER = {Int. Press, Somerville, MA},
      YEAR = {2009},
      ISBN = {978-1-57146-139-1},
   MRCLASS = {57R56 (18D10 18G30 57R15 57R75)},
  MRNUMBER = {2555928},
MRREVIEWER = {Julia\ Bergner},
}

@unpublished{sierrawahl,
	author = {Sierra, Ismael and Wahl, Nathalie},
	title = "{Homological stability for symplectic groups via algebraic arc complexes}",
    note = {arXiv:2411.07895},
	year = {2024},
}

@unpublished{MPPRW,
	author = {Miller, Jeremy and Patzt, Peter and Petersen, Dan and Randal-Williams, Oscar},
	title = "{Uniform twisted homological stability}",
	note = {Available at arXiv:2402.00354},
	year = {2024},
}

@article {bianchi1,
    AUTHOR = {Bianchi, Andrea},
     TITLE = {Partially multiplicative quandles and simplicial {H}urwitz
              spaces},
   JOURNAL = {Doc. Math.},
  FJOURNAL = {Documenta Mathematica},
    VOLUME = {30},
      YEAR = {2025},
    NUMBER = {3},
     PAGES = {611--672},
      ISSN = {1431-0635,1431-0643},
   MRCLASS = {55R80 (08A05 08A35 18M15 20B05 20M05)},
  MRNUMBER = {4916105},
       DOI = {10.4171/dm/996},
       URL = {https://doi.org/10.4171/dm/996},
}

@article {langley-remmel,
    AUTHOR = {Langley, Thomas and Remmel, Jeffrey},
     TITLE = {The plethysm {$s_\lambda[s_\mu]$} at hook and near-hook
              shapes},
   JOURNAL = {Electron. J. Combin.},
  FJOURNAL = {Electronic Journal of Combinatorics},
    VOLUME = {11},
      YEAR = {2004},
    NUMBER = {1},
     PAGES = {Research Paper 11, 26},
      ISSN = {1077-8926},
   MRCLASS = {05E10 (05E05)},
  MRNUMBER = {2035305},
MRREVIEWER = {Grant\ Walker},
       DOI = {10.37236/1764},
       URL = {https://doi.org/10.37236/1764},
}

@article {CFZ,
    AUTHOR = {Conrey, Brian and Farmer, David W. and Zirnbauer, Martin R.},
     TITLE = {Autocorrelation of ratios of {$L$}-functions},
   JOURNAL = {Commun. Number Theory Phys.},
  FJOURNAL = {Communications in Number Theory and Physics},
    VOLUME = {2},
      YEAR = {2008},
    NUMBER = {3},
     PAGES = {593--636},
      ISSN = {1931-4523,1931-4531},
   MRCLASS = {11M26},
  MRNUMBER = {2482944},
MRREVIEWER = {Steven\ Joel\ Miller},
       DOI = {10.4310/CNTP.2008.v2.n3.a4},
       URL = {https://doi.org/10.4310/CNTP.2008.v2.n3.a4},
}

@unpublished{Wang,
	author = {Wang, Victor Y.},
	title = "{Notes on zeta ratio stabilization}",
	note = {Available at arXiv:2402.01214},
	year = {2024},
}

@article {baezmoellertrimble,
    AUTHOR = {Baez, John C. and Moeller, Joe and Trimble, Todd},
     TITLE = {Schur functors and categorified plethysm},
   JOURNAL = {High. Struct.},
  FJOURNAL = {Higher Structures},
    VOLUME = {8},
      YEAR = {2024},
    NUMBER = {1},
     PAGES = {1--53},
      ISSN = {2209-0606},
   MRCLASS = {18D20 (05E05 18A35 18F30 18Mxx 19A22 20C30)},
  MRNUMBER = {4752517},
MRREVIEWER = {Primo\v z\ Moravec},
}

@unpublished{hatcher-madsenweiss,
	author = {Hatcher, Allen},
	title = "{A short exposition of the {M}adsen--{W}eiss theorem}",
	note = {Available at arXiv:1103.5223},
	year = {2011},
}

@article {petersentavakolyin,
    AUTHOR = {Petersen, Dan and Tavakol, Mehdi and Yin, Qizheng},
     TITLE = {Tautological classes with twisted coefficients},
   JOURNAL = {Ann. Sci. \'{E}c. Norm. Sup\'{e}r. (4)},
  FJOURNAL = {Annales Scientifiques de l'\'{E}cole Normale Sup\'{e}rieure. Quatri\`eme
              S\'{e}rie},
    VOLUME = {54},
      YEAR = {2021},
    NUMBER = {5},
     PAGES = {1179--1236},
      ISSN = {0012-9593},
   MRCLASS = {14H10 (14C25)},
  MRNUMBER = {4363247},
       DOI = {10.24033/asens.2479},
       URL = {https://doi.org/10.24033/asens.2479},
}

@article {bergstromminabe,
	AUTHOR = {Bergstr\"{o}m, Jonas and Minabe, Satoshi},
	TITLE = {On the cohomology of moduli spaces of (weighted) stable
	rational curves},
	JOURNAL = {Math. Z.},
	FJOURNAL = {Mathematische Zeitschrift},
	VOLUME = {275},
	YEAR = {2013},
	NUMBER = {3-4},
	PAGES = {1095--1108},
	ISSN = {0025-5874},
	MRCLASS = {14H10 (20C30)},
	MRNUMBER = {3127048},
	MRREVIEWER = {Montserrat Teixidor i Bigas},
	DOI = {10.1007/s00209-013-1171-8},
	URL = {https://doi-org.ezp.sub.su.se/10.1007/s00209-013-1171-8},
}

@unpublished{millerpatztpetersen,
	author = {Miller, Jeremy and Patzt, Peter and Petersen, Dan},
	title = "{Representation stability, secondary stability, and polynomial functors}",
	note = {Preprint available at arXiv:1910.05574},
	year = {2019},
}

@article {gelfandzelevinskyodd,
    AUTHOR = {Gelfand, I. M. and Zelevinsky, A. V.},
     TITLE = {Representation models for classical groups and their higher
              symmetries},
      NOTE = {The mathematical heritage of \'{E}lie Cartan (Lyon, 1984)},
   JOURNAL = {Ast\'{e}risque},
  FJOURNAL = {Ast\'{e}risque},
      YEAR = {1985},
    NUMBER = {Num\'{e}ro Hors S\'{e}rie},
     PAGES = {117--128},
      ISSN = {0303-1179},
   MRCLASS = {22E55 (20G05)},
  MRNUMBER = {837197},
MRREVIEWER = {J. L\~{o}hmus},
}

@article {proctor,
    AUTHOR = {Proctor, Robert A.},
     TITLE = {Odd symplectic groups},
   JOURNAL = {Invent. Math.},
  FJOURNAL = {Inventiones Mathematicae},
    VOLUME = {92},
      YEAR = {1988},
    NUMBER = {2},
     PAGES = {307--332},
      ISSN = {0020-9910},
   MRCLASS = {20G15 (11E57)},
  MRNUMBER = {936084},
MRREVIEWER = {S. I. Gel\cprime fand},
       DOI = {10.1007/BF01404455},
       URL = {https://doi-org.ezp.sub.su.se/10.1007/BF01404455},
}

@article {bianchi-braid,
    AUTHOR = {Bianchi, Andrea},
     TITLE = {Braid groups, mapping class groups and their homology with
              twisted coefficients},
   JOURNAL = {Math. Proc. Cambridge Philos. Soc.},
  FJOURNAL = {Mathematical Proceedings of the Cambridge Philosophical
              Society},
    VOLUME = {172},
      YEAR = {2022},
    NUMBER = {2},
     PAGES = {249--266},
      ISSN = {0305-0041},
   MRCLASS = {55R35 (20F36 55R40 55R80)},
  MRNUMBER = {4379067},
       DOI = {10.1017/s0305004121000219},
       URL = {https://doi-org.ezp.sub.su.se/10.1017/s0305004121000219},
}

@article {charney,
    AUTHOR = {Charney, Ruth},
     TITLE = {A generalization of a theorem of {V}ogtmann},
 BOOKTITLE = {Proceedings of the {N}orthwestern conference on cohomology of
              groups ({E}vanston, {I}ll., 1985)},
   JOURNAL = {J. Pure Appl. Algebra},
  FJOURNAL = {Journal of Pure and Applied Algebra},
    VOLUME = {44},
      YEAR = {1987},
    NUMBER = {1-3},
     PAGES = {107--125},
      ISSN = {0022-4049},
   MRCLASS = {18F25 (11E70 19D55 19G99 20J05)},
  MRNUMBER = {885099},
MRREVIEWER = {Ross Staffeldt},
       DOI = {10.1016/0022-4049(87)90019-3},
       URL = {https://doi-org.ezp.sub.su.se/10.1016/0022-4049(87)90019-3},
}

@article {olsson,
    AUTHOR = {Olsson, Martin C.},
     TITLE = {({L}og) twisted curves},
   JOURNAL = {Compos. Math.},
  FJOURNAL = {Compositio Mathematica},
    VOLUME = {143},
      YEAR = {2007},
    NUMBER = {2},
     PAGES = {476--494},
      ISSN = {0010-437X,1570-5846},
   MRCLASS = {14D22 (14A20 14N35)},
  MRNUMBER = {2309994},
MRREVIEWER = {Charles\ D.\ Cadman},
       DOI = {10.1112/S0010437X06002442},
       URL = {https://doi.org/10.1112/S0010437X06002442},
}

@article {cadman,
    AUTHOR = {Cadman, Charles},
     TITLE = {Using stacks to impose tangency conditions on curves},
   JOURNAL = {Amer. J. Math.},
  FJOURNAL = {American Journal of Mathematics},
    VOLUME = {129},
      YEAR = {2007},
    NUMBER = {2},
     PAGES = {405--427},
      ISSN = {0002-9327,1080-6377},
   MRCLASS = {14D20 (14A20 14N35)},
  MRNUMBER = {2306040},
MRREVIEWER = {Michael\ A.\ Rose},
       DOI = {10.1353/ajm.2007.0007},
       URL = {https://doi.org/10.1353/ajm.2007.0007},
}

@article {bandboyland,
    AUTHOR = {Band, Gavin and Boyland, Philip},
     TITLE = {The {B}urau estimate for the entropy of a braid},
   JOURNAL = {Algebr. Geom. Topol.},
  FJOURNAL = {Algebraic \& Geometric Topology},
    VOLUME = {7},
      YEAR = {2007},
     PAGES = {1345--1378},
      ISSN = {1472-2747},
   MRCLASS = {37E30 (20F36 37B40 57M25)},
  MRNUMBER = {2350285},
MRREVIEWER = {Darryl McCullough},
       DOI = {10.2140/agt.2007.7.1345},
       URL = {https://doi-org.ezp.sub.su.se/10.2140/agt.2007.7.1345},
}

@article {monoidalgrothendieck,
    AUTHOR = {Moeller, Joe and Vasilakopoulou, Christina},
     TITLE = {Monoidal {G}rothendieck construction},
   JOURNAL = {Theory Appl. Categ.},
  FJOURNAL = {Theory and Applications of Categories},
    VOLUME = {35},
      YEAR = {2020},
     PAGES = {Paper No. 31, 1159--1207},
      ISSN = {1201-561X},
   MRCLASS = {18D30 (18M05)},
  MRNUMBER = {4127726},
MRREVIEWER = {Laurent\ Poinsot},
}

@incollection {salvatore,
    AUTHOR = {Salvatore, Paolo},
     TITLE = {Configuration spaces with summable labels},
 BOOKTITLE = {Cohomological methods in homotopy theory ({B}ellaterra, 1998)},
    SERIES = {Progr. Math.},
    VOLUME = {196},
     PAGES = {375--395},
 PUBLISHER = {Birkh\"auser, Basel},
      YEAR = {2001},
      ISBN = {3-7643-6588-9},
   MRCLASS = {55R80 (18D50 55P43)},
  MRNUMBER = {1851264},
MRREVIEWER = {Pilar\ C.\ Carrasco},
}

@article {bloomquist,
    AUTHOR = {Bloomquist, Wade and Patzt, Peter and Scherich, Nancy},
     TITLE = {Quotients of braid groups by their congruence subgroups},
   JOURNAL = {Proc. Amer. Math. Soc. Ser. B},
  FJOURNAL = {Proceedings of the American Mathematical Society. Series B},
    VOLUME = {11},
      YEAR = {2024},
     PAGES = {508--524},
      ISSN = {2330-1511},
   MRCLASS = {20F36 (20H05 57K20)},
  MRNUMBER = {4811535},
MRREVIEWER = {Ioannis\ Diamantis},
       DOI = {10.1090/bproc/200},
       URL = {https://doi.org/10.1090/bproc/200},
}

@unpublished{vaintrob,
	author = {Vaintrob, Dmitry},
	title = "{Moduli of framed formal curves}",
	note = {arXiv:1910.11550 },
	year = 2019,
}

@unpublished{vaintrob2,
	author = {Vaintrob, Dmitry},
	title = "{Formality of little disks and algebraic geometry}",
	note = {arXiv:2103.15054},
	year = 2021,
}

@article {kawazumimorita,
	AUTHOR = {Kawazumi, Nariya and Morita, Shigeyuki},
	TITLE = {The primary approximation to the cohomology of the moduli
	space of curves and cocycles for the stable characteristic
	classes},
	JOURNAL = {Math. Res. Lett.},
	FJOURNAL = {Mathematical Research Letters},
	VOLUME = {3},
	YEAR = {1996},
	NUMBER = {5},
	PAGES = {629--641},
	ISSN = {1073-2780},
	MRCLASS = {14H10 (14F99 57R20)},
	MRNUMBER = {1418577},
	MRREVIEWER = {Gheorghe Ionesei},
	DOI = {10.4310/MRL.1996.v3.n5.a6},
	URL = {http://dx.doi.org/10.4310/MRL.1996.v3.n5.a6},
}

@incollection {hainlooijenga,
	AUTHOR = {Hain, Richard and Looijenga, Eduard},
	TITLE = {Mapping class groups and moduli spaces of curves},
	BOOKTITLE = {Algebraic geometry---{S}anta {C}ruz 1995},
	SERIES = {Proc. Sympos. Pure Math.},
	VOLUME = {62},
	PAGES = {97--142},
	PUBLISHER = {Amer. Math. Soc., Providence, RI},
	YEAR = {1997},
	MRCLASS = {14H10 (14C30 14H15 32G15 32G20)},
	MRNUMBER = {1492535 (99a:14032)},
	MRREVIEWER = {Takashi Ichikawa},
}

@article {tallwraith,
    AUTHOR = {Tall, David O. and Wraith, Gavin C.},
     TITLE = {Representable functors and operations on rings},
   JOURNAL = {Proc. London Math. Soc. (3)},
  FJOURNAL = {Proceedings of the London Mathematical Society. Third Series},
    VOLUME = {20},
      YEAR = {1970},
     PAGES = {619--643},
      ISSN = {0024-6115},
   MRCLASS = {13.90 (08.00)},
  MRNUMBER = {265348},
MRREVIEWER = {J. Sonner},
       DOI = {10.1112/plms/s3-20.4.619},
       URL = {https://doi.org/10.1112/plms/s3-20.4.619},
}

@article {borgerwieland,
    AUTHOR = {Borger, James and Wieland, Ben},
     TITLE = {Plethystic algebra},
   JOURNAL = {Adv. Math.},
  FJOURNAL = {Advances in Mathematics},
    VOLUME = {194},
      YEAR = {2005},
    NUMBER = {2},
     PAGES = {246--283},
      ISSN = {0001-8708},
   MRCLASS = {13K05 (13A99 16W30)},
  MRNUMBER = {2139914},
MRREVIEWER = {Stefaan Caenepeel},
       DOI = {10.1016/j.aim.2004.06.006},
       URL = {https://doi.org/10.1016/j.aim.2004.06.006},
}

@article {kimurastasheffvoronov2,
    AUTHOR = {Kimura, Takashi and Stasheff, Jim and Voronov, Alexander A.},
     TITLE = {On operad structures of moduli spaces and string theory},
   JOURNAL = {Comm. Math. Phys.},
  FJOURNAL = {Communications in Mathematical Physics},
    VOLUME = {171},
      YEAR = {1995},
    NUMBER = {1},
     PAGES = {1--25},
      ISSN = {0010-3616},
   MRCLASS = {14H10 (18G35 55P35 81T30 81T40)},
  MRNUMBER = {1341693},
MRREVIEWER = {Reinhold W. Gebert},
       URL = {http://projecteuclid.org/euclid.cmp/1104273401},
}

@article {katonakayama,
    AUTHOR = {Kato, Kazuya and Nakayama, Chikara},
     TITLE = {Log {B}etti cohomology, log \'{e}tale cohomology, and log de
              {R}ham cohomology of log schemes over {${\bf C}$}},
   JOURNAL = {Kodai Math. J.},
  FJOURNAL = {Kodai Mathematical Journal},
    VOLUME = {22},
      YEAR = {1999},
    NUMBER = {2},
     PAGES = {161--186},
      ISSN = {0386-5991},
   MRCLASS = {14F20 (14F25 14F40)},
  MRNUMBER = {1700591},
MRREVIEWER = {Claudio Pedrini},
       DOI = {10.2996/kmj/1138044041},
       URL = {https://doi.org/10.2996/kmj/1138044041},
}

@article {illusieoverview,
    AUTHOR = {Illusie, Luc},
     TITLE = {An overview of the work of {K}. {F}ujiwara, {K}. {K}ato, and
              {C}. {N}akayama on logarithmic \'{e}tale cohomology},
      NOTE = {Cohomologies $p$-adiques et applications arithm\'{e}tiques, II},
   JOURNAL = {Ast\'{e}risque},
  FJOURNAL = {Ast\'{e}risque},
    NUMBER = {279},
      YEAR = {2002},
     PAGES = {271--322},
      ISSN = {0303-1179},
   MRCLASS = {14F20 (14C25 14F30)},
  MRNUMBER = {1922832},
MRREVIEWER = {Mark Kisin},
}

@incollection {borelstablereal,
    AUTHOR = {Borel, Armand},
     TITLE = {Stable real cohomology of arithmetic groups. {II}},
 BOOKTITLE = {Manifolds and {L}ie groups ({N}otre {D}ame, {I}nd., 1980)},
    SERIES = {Progr. Math.},
    VOLUME = {14},
     PAGES = {21--55},
 PUBLISHER = {Birkh\"auser, Boston, Mass.},
      YEAR = {1981},
   MRCLASS = {22E41},
  MRNUMBER = {642850 (83h:22023)},
MRREVIEWER = {Avner Ash},
}

@article {armstrongorbit,
    AUTHOR = {Armstrong, Mark A.},
     TITLE = {On the fundamental group of an orbit space},
   JOURNAL = {Proc. Cambridge Philos. Soc.},
  FJOURNAL = {Proceedings of the Cambridge Philosophical Society},
    VOLUME = {61},
      YEAR = {1965},
     PAGES = {639--646},
      ISSN = {0008-1981},
   MRCLASS = {55.40},
  MRNUMBER = {187244},
MRREVIEWER = {R. F. Williams},
       DOI = {10.1017/s0305004100038974},
       URL = {https://doi.org/10.1017/s0305004100038974},
}

@article {delignedegeneration,
    AUTHOR = {Deligne, Pierre},
     TITLE = {Th\'eor\`eme de {L}efschetz et crit\`eres de
              d\'eg\'en\'erescence de suites spectrales},
   JOURNAL = {Inst. Hautes \'Etudes Sci. Publ. Math.},
  FJOURNAL = {Institut des Hautes \'Etudes Scientifiques. Publications
              Math\'ematiques},
    NUMBER = {35},
      YEAR = {1968},
     PAGES = {259--278},
      ISSN = {0073-8301},
   MRCLASS = {14.52},
  MRNUMBER = {0244265 (39 \#5582)},
MRREVIEWER = {D. I. Lieberman},
}

@article {bmp,
    AUTHOR = {Brendle, Tara and Margalit, Dan and Putman, Andrew},
     TITLE = {Generators for the hyperelliptic {T}orelli group and the
              kernel of the {B}urau representation at {$t=-1$}},
   JOURNAL = {Invent. Math.},
  FJOURNAL = {Inventiones Mathematicae},
    VOLUME = {200},
      YEAR = {2015},
    NUMBER = {1},
     PAGES = {263--310},
      ISSN = {0020-9910},
   MRCLASS = {20F65 (20F36)},
  MRNUMBER = {3323579},
MRREVIEWER = {Dawid Kielak},
       DOI = {10.1007/s00222-014-0537-9},
       URL = {https://doi-org.ezp.sub.su.se/10.1007/s00222-014-0537-9},
}

@book {cohenladamay,
    AUTHOR = {Cohen, Frederick R. and Lada, Thomas J. and May, J. Peter},
     TITLE = {The homology of iterated loop spaces},
    SERIES = {Lecture Notes in Mathematics, Vol. 533},
 PUBLISHER = {Springer-Verlag},
   ADDRESS = {Berlin},
      YEAR = {1976},
     PAGES = {vii+490},
   MRCLASS = {55G25 (55D35)},
  MRNUMBER = {0436146 (55 \#9096)},
MRREVIEWER = {Peter J. Eccles},
}

@book {lodayvallette,
    AUTHOR = {Loday, Jean-Louis and Vallette, Bruno},
     TITLE = {Algebraic operads},
    SERIES = {Grundlehren der Mathematischen Wissenschaften},
    VOLUME = {346},
 PUBLISHER = {Springer},
   ADDRESS = {Heidelberg},
      YEAR = {2012},
     PAGES = {xxiv+634},
      ISBN = {978-3-642-30361-6},
   MRCLASS = {18D50 (16E99)},
  MRNUMBER = {2954392},
MRREVIEWER = {Andrey Yu. Lazarev},
       DOI = {10.1007/978-3-642-30362-3},
       URL = {http://dx.doi.org/10.1007/978-3-642-30362-3},
}

@article {Ke-Sn1,
    AUTHOR = {Keating, J. P. and Snaith, N. C.},
     TITLE = {Random matrix theory and {$\zeta(1/2+it)$}},
   JOURNAL = {Comm. Math. Phys.},
  FJOURNAL = {Communications in Mathematical Physics},
    VOLUME = {214},
      YEAR = {2000},
    NUMBER = {1},
     PAGES = {57--89},
      ISSN = {0010-3616},
   MRCLASS = {11M26 (15A52 82B41)},
  MRNUMBER = {1794265},
MRREVIEWER = {Ze\'{e}v Rudnick},
       DOI = {10.1007/s002200000261},
       URL = {https://doi-org.ezp.sub.su.se/10.1007/s002200000261},
}

@article {Ke-Sn2,
    AUTHOR = {Keating, J. P. and Snaith, N. C.},
     TITLE = {Random matrix theory and {$L$}-functions at {$s=1/2$}},
   JOURNAL = {Comm. Math. Phys.},
  FJOURNAL = {Communications in Mathematical Physics},
    VOLUME = {214},
      YEAR = {2000},
    NUMBER = {1},
     PAGES = {91--110},
      ISSN = {0010-3616},
   MRCLASS = {11M26 (15A52 82B41)},
  MRNUMBER = {1794267},
MRREVIEWER = {Ze\'{e}v Rudnick},
       DOI = {10.1007/s002200000262},
       URL = {https://doi-org.ezp.sub.su.se/10.1007/s002200000262},
}

@book {KaS,
    AUTHOR = {Katz, Nicholas M. and Sarnak, Peter},
     TITLE = {Random matrices, {F}robenius eigenvalues, and monodromy},
    SERIES = {American Mathematical Society Colloquium Publications},
    VOLUME = {45},
 PUBLISHER = {American Mathematical Society, Providence, RI},
      YEAR = {1999},
     PAGES = {xii+419},
      ISBN = {0-8218-1017-0},
   MRCLASS = {11G25 (11M06 11Y35 14D05 14G10 60F99 82B44)},
  MRNUMBER = {1659828},
MRREVIEWER = {Philippe G. Michel},
       DOI = {10.1090/coll/045},
       URL = {https://doi-org.ezp.sub.su.se/10.1090/coll/045},
}

@article {Rud,
    AUTHOR = {Rudnick, Ze\'{e}v},
     TITLE = {Traces of high powers of the {F}robenius class in the
              hyperelliptic ensemble},
   JOURNAL = {Acta Arith.},
  FJOURNAL = {Acta Arithmetica},
    VOLUME = {143},
      YEAR = {2010},
    NUMBER = {1},
     PAGES = {81--99},
      ISSN = {0065-1036},
   MRCLASS = {11M50 (11G20)},
  MRNUMBER = {2640060},
MRREVIEWER = {Steven Joel Miller},
       DOI = {10.4064/aa143-1-5},
       URL = {https://doi-org.ezp.sub.su.se/10.4064/aa143-1-5},
}

@article {Su90,
    AUTHOR = {Sundaram, Sheila},
     TITLE = {The {C}auchy identity for {${\rm Sp}(2n)$}},
   JOURNAL = {J. Combin. Theory Ser. A},
  FJOURNAL = {Journal of Combinatorial Theory. Series A},
    VOLUME = {53},
      YEAR = {1990},
    NUMBER = {2},
     PAGES = {209--238},
      ISSN = {0097-3165},
   MRCLASS = {05E10 (05A19 20C33 22E45)},
  MRNUMBER = {1041446},
       DOI = {10.1016/0097-3165(90)90058-5},
       URL = {https://doi-org.ezp.sub.su.se/10.1016/0097-3165(90)90058-5},
}

@article {BJ,
    AUTHOR = {Bae, Sunghan and Jung, Hwanyup},
     TITLE = {Statistics for products of traces of high powers of the
              {F}robenius class of hyperelliptic curves in even
              characteristic},
   JOURNAL = {Int. J. Number Theory},
  FJOURNAL = {International Journal of Number Theory},
    VOLUME = {15},
      YEAR = {2019},
    NUMBER = {7},
     PAGES = {1519--1530},
      ISSN = {1793-0421},
   MRCLASS = {11G20 (11M38 11T55)},
  MRNUMBER = {3982825},
MRREVIEWER = {John T. Cullinan},
       DOI = {10.1142/S1793042119500878},
       URL = {https://doi-org.ezp.sub.su.se/10.1142/S1793042119500878},
}

@incollection {jimbo-miwa,
	AUTHOR = {Jimbo, Michio and Miwa, Tetsuji},
	TITLE = {On a duality of branching rules for affine {L}ie algebras},
	BOOKTITLE = {Algebraic groups and related topics ({K}yoto/{N}agoya, 1983)},
	SERIES = {Adv. Stud. Pure Math.},
	VOLUME = {6},
	PAGES = {17--65},
	PUBLISHER = {North-Holland, Amsterdam},
	YEAR = {1985},
	MRCLASS = {17B67},
	MRNUMBER = {803329},
	MRREVIEWER = {Nigel B. Backhouse},
	DOI = {10.2969/aspm/00610017},
	URL = {https://doi-org.ezp.sub.su.se/10.2969/aspm/00610017},
}

@article {chr,
    AUTHOR = {Calderbank, A. Robert and Hanlon, Philip and Robinson, Robert W.},
     TITLE = {Partitions into even and odd block size and some unusual
              characters of the symmetric groups},
   JOURNAL = {Proc. London Math. Soc. (3)},
  FJOURNAL = {Proceedings of the London Mathematical Society. Third Series},
    VOLUME = {53},
      YEAR = {1986},
    NUMBER = {2},
     PAGES = {288--320},
      ISSN = {0024-6115},
   MRCLASS = {20C30 (05A15 05A17 06A10)},
  MRNUMBER = {850222},
MRREVIEWER = {Bruce Sagan},
       DOI = {10.1112/plms/s3-53.2.288},
       URL = {https://doi-org.ezp.sub.su.se/10.1112/plms/s3-53.2.288},
}

@article {Del80,
    AUTHOR = {Deligne, Pierre},
     TITLE = {La conjecture de {W}eil. {II}},
   JOURNAL = {Inst. Hautes \'Etudes Sci. Publ. Math.},
  FJOURNAL = {Institut des Hautes \'Etudes Scientifiques. Publications
              Math\'ematiques},
    NUMBER = {52},
      YEAR = {1980},
     PAGES = {137--252},
      ISSN = {0073-8301},
     CODEN = {PMIHA6},
   MRCLASS = {14G13 (10H10)},
  MRNUMBER = {601520 (83c:14017)},
MRREVIEWER = {Spencer J. Bloch},
       URL = {http://www.numdam.org/item?id=PMIHES_1980__52__137_0},
}

@article {sullivaninfinitesimal,
    AUTHOR = {Sullivan, Dennis},
     TITLE = {Infinitesimal computations in topology},
   JOURNAL = {Inst. Hautes \'Etudes Sci. Publ. Math.},
  FJOURNAL = {Institut des Hautes \'Etudes Scientifiques. Publications
              Math\'ematiques},
    VOLUME = {47},
      YEAR = {1977},
     PAGES = {269--331},
      ISSN = {0073-8301},
   MRCLASS = {57D99 (55D99 58A10)},
  MRNUMBER = {0646078 (58 \#31119)},
MRREVIEWER = {J. F. Adams},
}

@article {kimhyperplane,
    AUTHOR = {Kim, Minhyong},
     TITLE = {Weights in cohomology groups arising from hyperplane
              arrangements},
   JOURNAL = {Proc. Amer. Math. Soc.},
  FJOURNAL = {Proceedings of the American Mathematical Society},
    VOLUME = {120},
      YEAR = {1994},
    NUMBER = {3},
     PAGES = {697--703},
      ISSN = {0002-9939},
     CODEN = {PAMYAR},
   MRCLASS = {14F20 (14C30 52B30)},
  MRNUMBER = {1179589 (94e:14027)},
MRREVIEWER = {Andrzej Kozlowski},
       DOI = {10.2307/2160458},
       URL = {http://dx.doi.org/10.2307/2160458},
}

@article {lehrerhyperplane,
    AUTHOR = {Lehrer, Gus I.},
     TITLE = {The {$l$}-adic cohomology of hyperplane complements},
   JOURNAL = {Bull. London Math. Soc.},
  FJOURNAL = {The Bulletin of the London Mathematical Society},
    VOLUME = {24},
      YEAR = {1992},
    NUMBER = {1},
     PAGES = {76--82},
      ISSN = {0024-6093},
     CODEN = {LMSBBT},
   MRCLASS = {14F20 (52B30)},
  MRNUMBER = {1139062 (92j:14022)},
MRREVIEWER = {P. Orlik},
       DOI = {10.1112/blms/24.1.76},
       URL = {http://dx.doi.org/10.1112/blms/24.1.76},
}

@unpublished{mixedhodge,
   author = {{Getzler}, Ezra},
    title = "{Mixed {H}odge structures of configuration spaces}",
archivePrefix = "arXiv",
   note = {Preprint 96-61, Max-Planck-Institut f\"ur Mathematik, Bonn. arXiv:alg-geom/9510018},
     year = 1995,
}

@article {grw,
    AUTHOR = {Galatius, S{\o}ren and Randal-Williams, Oscar},
     TITLE = {Monoids of moduli spaces of manifolds},
   JOURNAL = {Geom. Topol.},
  FJOURNAL = {Geometry \& Topology},
    VOLUME = {14},
      YEAR = {2010},
    NUMBER = {3},
     PAGES = {1243--1302},
      ISSN = {1465-3060},
   MRCLASS = {57R90 (55P47 57R56)},
  MRNUMBER = {2653727},
MRREVIEWER = {R. M. Vogt},
       DOI = {10.2140/gt.2010.14.1243},
       URL = {https://doi-org.ezp.sub.su.se/10.2140/gt.2010.14.1243},
}

@unpublished{evw2,
   author = {Ellenberg, Jordan and Venkatesh, Akshay and Westerland, Craig},
    title = "{Homological stability for {H}urwitz spaces and
the {C}ohen--{L}enstra conjecture over function fields, II}",
note = "Available at arXiv:1212.0923",
     year = 2012,
}

@article {priddy,
    AUTHOR = {Priddy, Stewart B.},
     TITLE = {Koszul resolutions},
   JOURNAL = {Trans. Amer. Math. Soc.},
  FJOURNAL = {Transactions of the American Mathematical Society},
    VOLUME = {152},
      YEAR = {1970},
     PAGES = {39--60},
      ISSN = {0002-9947},
   MRCLASS = {18.20},
  MRNUMBER = {265437},
MRREVIEWER = {A. K. Bousfield},
       DOI = {10.2307/1995637},
       URL = {https://doi-org.ezp.sub.su.se/10.2307/1995637},
}

@article {berglundkoszulspaces,
    AUTHOR = {Berglund, Alexander},
     TITLE = {Koszul spaces},
   JOURNAL = {Trans. Amer. Math. Soc.},
  FJOURNAL = {Transactions of the American Mathematical Society},
    VOLUME = {366},
      YEAR = {2014},
    NUMBER = {9},
     PAGES = {4551--4569},
      ISSN = {0002-9947},
   MRCLASS = {55P62 (16S37 18D50)},
  MRNUMBER = {3217692},
MRREVIEWER = {Samuel B. Smith},
       DOI = {10.1090/S0002-9947-2014-05935-7},
       URL = {https://doi.org/10.1090/S0002-9947-2014-05935-7},
}

@article {bornevistoli,
    AUTHOR = {Borne, Niels and Vistoli, Angelo},
     TITLE = {Parabolic sheaves on logarithmic schemes},
   JOURNAL = {Adv. Math.},
  FJOURNAL = {Advances in Mathematics},
    VOLUME = {231},
      YEAR = {2012},
    NUMBER = {3-4},
     PAGES = {1327--1363},
      ISSN = {0001-8708},
   MRCLASS = {14F05 (14A20 18D10)},
  MRNUMBER = {2964607},
MRREVIEWER = {Jon Eivind Vatne},
       DOI = {10.1016/j.aim.2012.06.015},
       URL = {https://doi-org.ezp.sub.su.se/10.1016/j.aim.2012.06.015},
}

@incollection {ivanovtwisted,
	AUTHOR = {Ivanov, Nikolai V.},
	TITLE = {On the homology stability for {T}eichm\"uller modular groups:
	closed surfaces and twisted coefficients},
	BOOKTITLE = {Mapping class groups and moduli spaces of {R}iemann surfaces
	({G}\"ottingen, 1991/{S}eattle, {WA}, 1991)},
	SERIES = {Contemp. Math.},
	VOLUME = {150},
	PAGES = {149--194},
	PUBLISHER = {Amer. Math. Soc., Providence, RI},
	YEAR = {1993},
	MRCLASS = {57N05 (20F38 30F60 32G15 57M99)},
	MRNUMBER = {1234264 (94h:57022)},
	MRREVIEWER = {Darryl McCullough},
	DOI = {10.1090/conm/150/01290},
	URL = {http://dx.doi.org/10.1090/conm/150/01290},
}

@article {Di-Wh,
    AUTHOR = {Diaconu, Adrian and Whitehead, Ian},
     TITLE = {On the third moment of {$L(\frac{1}{2}, \chi_d)$} {II}: the
              number field case},
   JOURNAL = {J. Eur. Math. Soc. (JEMS)},
  FJOURNAL = {Journal of the European Mathematical Society (JEMS)},
    VOLUME = {23},
      YEAR = {2021},
    NUMBER = {6},
     PAGES = {2051--2070},
      ISSN = {1435-9855},
   MRCLASS = {11M06 (11F68)},
  MRNUMBER = {4244522},
MRREVIEWER = {Sandro Bettin},
       DOI = {10.4171/JEMS/1049},
       URL = {https://doi-org.ezp.sub.su.se/10.4171/JEMS/1049},
}

@article {borelstablereal1,
    AUTHOR = {Borel, Armand},
     TITLE = {Stable real cohomology of arithmetic groups},
   JOURNAL = {Ann. Sci. \'{E}cole Norm. Sup. (4)},
  FJOURNAL = {Annales Scientifiques de l'\'{E}cole Normale Sup\'{e}rieure. Quatri\`eme
              S\'{e}rie},
    VOLUME = {7},
      YEAR = {1974},
     PAGES = {235--272 (1975)},
      ISSN = {0012-9593},
   MRCLASS = {22E40 (20G10)},
  MRNUMBER = {387496},
MRREVIEWER = {H. Garland},
       URL = {http://www.numdam.org/item?id=ASENS_1974_4_7_2_235_0},
}

@book {W,
    AUTHOR = {Weil, Andr\'{e}},
     TITLE = {Sur les courbes alg\'{e}briques et les vari\'{e}t\'{e}s qui s'en
              d\'{e}duisent},
    SERIES = {Publ. Inst. Math. Univ. Strasbourg},
    VOLUME = {7},
 PUBLISHER = {Hermann \& Cie, Paris},
      YEAR = {1948},
     PAGES = {iv+85},
   MRCLASS = {14.0X},
  MRNUMBER = {0027151},
MRREVIEWER = {O. F. G. Schilling},
}

@article {RW,
    AUTHOR = {Rubinstein, Michael O. and Wu, Kaiyu},
     TITLE = {Moments of zeta functions associated to hyperelliptic curves
              over finite fields},
   JOURNAL = {Philos. Trans. Roy. Soc. A},
  FJOURNAL = {Philosophical Transactions of the Royal Society A.
              Mathematical, Physical and Engineering Sciences},
    VOLUME = {373},
      YEAR = {2015},
    NUMBER = {2040},
     PAGES = {20140307, 37},
      ISSN = {1364-503X},
   MRCLASS = {11M06 (11G20 11M50)},
  MRNUMBER = {3338121},
MRREVIEWER = {Adam J. Harper},
       DOI = {10.1098/rsta.2014.0307},
       URL = {https://doi-org.ezp.sub.su.se/10.1098/rsta.2014.0307},
}

@article {HR,
    AUTHOR = {Hoffstein, Jeffrey and Rosen, Michael},
     TITLE = {Average values of {$L$}-series in function fields},
   JOURNAL = {J. Reine Angew. Math.},
  FJOURNAL = {Journal f\"{u}r die Reine und Angewandte Mathematik. [Crelle's
              Journal]},
    VOLUME = {426},
      YEAR = {1992},
     PAGES = {117--150},
      ISSN = {0075-4102},
   MRCLASS = {11E41 (11R29 11R58)},
  MRNUMBER = {1155750},
MRREVIEWER = {Kazuyuki Hatada},
       DOI = {10.1515/crll.1992.426.117},
       URL = {https://doi-org.ezp.sub.su.se/10.1515/crll.1992.426.117},
}

@unpublished{DPP,
	author = {Diaconu, Adrian and Pa\c{s}ol, Vicen\c{t}iu and Popa, Alexandru A.},
	title = "{Quadratic {W}eyl group multiple {D}irichlet series of type $D_{\scriptscriptstyle 4}^{\scriptscriptstyle (1)}$}",
	note = {Preprint available at arXiv:2111.11062},
	year = 2021,
}

@unpublished{DV,
	author = {Diaconu, Adrian and Pa\c{s}ol, Vicen\c{t}iu},
	title = "{Moduli of hyperelliptic curves and multiple {D}irichlet series}",
	note = {Preprint available at arXiv:1808.09667},
	year = 2018,
}

@article {DT,
    AUTHOR = {Diaconu, Adrian and Twiss, Henry},
     TITLE = {Secondary terms in the asymptotics of moments of
              {$L$}-functions},
   JOURNAL = {J. Number Theory},
  FJOURNAL = {Journal of Number Theory},
    VOLUME = {252},
      YEAR = {2023},
     PAGES = {243--297},
      ISSN = {0022-314X,1096-1658},
   MRCLASS = {11R58 (11F68 11M06 11M32)},
  MRNUMBER = {4618015},
       DOI = {10.1016/j.jnt.2023.04.010},
       URL = {https://doi.org/10.1016/j.jnt.2023.04.010},
}

@article {Sh,
    AUTHOR = {Shen, Quanli},
     TITLE = {The fourth moment of quadratic {D}irichlet {$L$}-functions},
   JOURNAL = {Math. Z.},
  FJOURNAL = {Mathematische Zeitschrift},
    VOLUME = {298},
      YEAR = {2021},
    NUMBER = {1-2},
     PAGES = {713--745},
      ISSN = {0025-5874},
   MRCLASS = {11M06 (11M50)},
  MRNUMBER = {4257106},
MRREVIEWER = {D. R. Heath-Brown},
       DOI = {10.1007/s00209-020-02609-2},
       URL = {https://doi-org.ezp.sub.su.se/10.1007/s00209-020-02609-2},
}

@incollection {kontsevichfeynman,
    AUTHOR = {Kontsevich, Maxim},
     TITLE = {Feynman diagrams and low-dimensional topology},
 BOOKTITLE = {First {E}uropean {C}ongress of {M}athematics, {V}ol.\ {II}
              ({P}aris, 1992)},
    SERIES = {Progr. Math.},
    VOLUME = {120},
     PAGES = {97--121},
 PUBLISHER = {Birkh\"auser, Basel},
      YEAR = {1994},
   MRCLASS = {57R57 (14H15 32G15 57M25)},
  MRNUMBER = {1341841 (96h:57027)},
MRREVIEWER = {Anatoly Libgober},
}

@incollection {joyalanalyticfunctors,
	AUTHOR = {Joyal, Andr{\'e}},
	TITLE = {Foncteurs analytiques et esp\`eces de structures},
	BOOKTITLE = {Combinatoire \'enum\'erative ({M}ontreal, {Q}ue.,
	1985)},
	SERIES = {Lecture Notes in Math.},
	VOLUME = {1234},
	PAGES = {126--159},
	PUBLISHER = {Springer, Berlin},
	YEAR = {1986},
	MRCLASS = {05A15 (05A19 18B99)},
	MRNUMBER = {927763 (89b:05014)},
	MRREVIEWER = {Joseph Kung},
	DOI = {10.1007/BFb0072514},
	URL = {http://dx.doi.org/10.1007/BFb0072514},
}

@article {segalrationalfunctions,
	AUTHOR = {Segal, Graeme},
	TITLE = {The topology of spaces of rational functions},
	JOURNAL = {Acta Math.},
	FJOURNAL = {Acta Mathematica},
	VOLUME = {143},
	YEAR = {1979},
	NUMBER = {1-2},
	PAGES = {39--72},
	ISSN = {0001-5962},
	CODEN = {ACMAA8},
	MRCLASS = {55P10 (32C42 81E10)},
	MRNUMBER = {533892 (81c:55013)},
	MRREVIEWER = {D. B. Fuks},
	DOI = {10.1007/BF02392088},
	URL = {http://dx.doi.org/10.1007/BF02392088},
}

@article {Fuk70,
    AUTHOR = {Fuks, D. B.},
     TITLE = {Cohomology of the braid group {${\rm mod}\ 2$}},
   JOURNAL = {Funkcional. Anal. i Prilo\v{z}en.},
  FJOURNAL = {Akademija Nauk SSSR. Funkcional\cprime nyi Analiz i ego Prilo\v{z}enija},
    VOLUME = {4},
      YEAR = {1970},
    NUMBER = {2},
     PAGES = {62--73},
      ISSN = {0374-1990},
   MRCLASS = {14.55},
  MRNUMBER = {0274463},
MRREVIEWER = {N. Popescu},
}

@article {FN62,
    AUTHOR = {Fox, Ralph and Neuwirth, Lee},
     TITLE = {The braid groups},
   JOURNAL = {Math. Scand.},
  FJOURNAL = {Mathematica Scandinavica},
    VOLUME = {10},
      YEAR = {1962},
     PAGES = {119--126},
      ISSN = {0025-5521},
   MRCLASS = {55.20},
  MRNUMBER = {150755},
MRREVIEWER = {H. R. Gluck},
       DOI = {10.7146/math.scand.a-10518},
       URL = {https://doi-org.ezp.sub.su.se/10.7146/math.scand.a-10518},
}

@article {tillmann-detects,
    AUTHOR = {Tillmann, Ulrike},
     TITLE = {Higher genus surface operad detects infinite loop spaces},
   JOURNAL = {Math. Ann.},
  FJOURNAL = {Mathematische Annalen},
    VOLUME = {317},
      YEAR = {2000},
    NUMBER = {3},
     PAGES = {613--628},
      ISSN = {0025-5831},
   MRCLASS = {55P47 (55P48 55R35 57R56 81T40)},
  MRNUMBER = {1776120},
MRREVIEWER = {Kathryn P. Hess},
       DOI = {10.1007/PL00004416},
       URL = {https://doi-org.ezp.sub.su.se/10.1007/PL00004416},
}

@article {Sundaram,
    AUTHOR = {Sundaram, Sheila},
     TITLE = {The homology representations of the symmetric group on
              {C}ohen-{M}acaulay subposets of the partition lattice},
   JOURNAL = {Adv. Math.},
  FJOURNAL = {Advances in Mathematics},
    VOLUME = {104},
      YEAR = {1994},
    NUMBER = {2},
     PAGES = {225--296},
      ISSN = {0001-8708},
   MRCLASS = {05E25 (20C30)},
  MRNUMBER = {1273390},
       DOI = {10.1006/aima.1994.1030},
       URL = {https://doi-org.ezp.sub.su.se/10.1006/aima.1994.1030},
}

@article {toen,
    AUTHOR = {To{\"{e}}n, Bertrand},
     TITLE = {Champs affines},
   JOURNAL = {Selecta Math. (N.S.)},
  FJOURNAL = {Selecta Mathematica. New Series},
    VOLUME = {12},
      YEAR = {2006},
    NUMBER = {1},
     PAGES = {39--135},
      ISSN = {1022-1824},
   MRCLASS = {14F35 (14A20 18F10 55U35)},
  MRNUMBER = {2244263},
MRREVIEWER = {Mark Hovey},
       DOI = {10.1007/s00029-006-0019-z},
       URL = {https://doi-org.ezp.sub.su.se/10.1007/s00029-006-0019-z},
}

@unpublished{ETW17,
    AUTHOR = {Ellenberg, Jordan and Tran, TriThang and Westerland, Craig},
     TITLE = {Fox--{N}euwirth--{F}uks cells, quantum shuffle algebras, and {M}alle's conjecture for function fields},
 	year = {2017},
	note = "Preprint. arXiv:1701.04541",
 }

@article {talpovistoli,
    AUTHOR = {Talpo, Mattia and Vistoli, Angelo},
     TITLE = {The {K}ato-{N}akayama space as a transcendental root stack},
   JOURNAL = {Int. Math. Res. Not. IMRN},
  FJOURNAL = {International Mathematics Research Notices. IMRN},
      YEAR = {2018},
    NUMBER = {19},
     PAGES = {6145--6176},
      ISSN = {1073-7928,1687-0247},
   MRCLASS = {14A20 (18D10)},
  MRNUMBER = {3867403},
MRREVIEWER = {Hsian-Hua\ Tseng},
       DOI = {10.1093/imrn/rnx079},
       URL = {https://doi-org.ezp.sub.su.se/10.1093/imrn/rnx079},
}

@unpublished{DPP-log,
    AUTHOR = {Dupont, Cl\'ement and Panzer, Erik and Pym, Brent},
     TITLE = {Logarithmic morphisms, tangential basepoints, and little disks},
 	year = {2024},
	note = "Preprint. arXiv:2408.13108",
 }

@incollection{wahlstability,
    AUTHOR = {Wahl, Nathalie},
     TITLE = {Homological stability for mapping class groups of surfaces},
     BOOKTITLE = {Handbook of Moduli, Vol. III},
     EDITOR = {Farkas, Gavril and Morrison, Ian},
     pages = {547--583},
         SERIES = {Adv. Lect. Math. (ALM)},
    VOLUME = {24},
 PUBLISHER = {Int. Press},
  ADDRESS = {Somerville, MA},
 	year = {2013},
 }

@article {bergstromvandergeer,
    AUTHOR = {Bergstr{\"o}m, Jonas and van der Geer, Gerard},
     TITLE = {The {E}uler characteristic of local systems on the moduli of
              curves and abelian varieties of genus three},
   JOURNAL = {J. Topol.},
  FJOURNAL = {Journal of Topology},
    VOLUME = {1},
      YEAR = {2008},
    NUMBER = {3},
     PAGES = {651--662},
      ISSN = {1753-8416},
   MRCLASS = {14K10 (14H15)},
  MRNUMBER = {2417447 (2009g:14053)},
MRREVIEWER = {Montserrat Teixidor i Bigas},
       DOI = {10.1112/jtopol/jtn015},
       URL = {http://dx.doi.org/10.1112/jtopol/jtn015},
}

@incollection {segalCFT,
	AUTHOR = {Segal, Graeme},
	TITLE = {The definition of conformal field theory},
	BOOKTITLE = {Differential geometrical methods in theoretical physics
	({C}omo, 1987)},
	SERIES = {NATO Adv. Sci. Inst. Ser. C: Math. Phys. Sci.},
	VOLUME = {250},
	PAGES = {165--171},
	PUBLISHER = {Kluwer Acad. Publ., Dordrecht},
	YEAR = {1988},
	MRCLASS = {58D30 (14H15 32G15 81E05 81E40)},
	MRNUMBER = {981378},
	MRREVIEWER = {Yukihiko Namikawa},
}

@article {gcovers,
	AUTHOR = {Petersen, Dan},
	TITLE = {The operad structure of admissible {$G$}-covers},
	JOURNAL = {Algebra Number Theory},
	FJOURNAL = {Algebra \& Number Theory},
	VOLUME = {7},
	YEAR = {2013},
	NUMBER = {8},
	PAGES = {1953--1975},
	ISSN = {1937-0652},
	MRCLASS = {18D50 (14D21 14H10 14N35)},
	MRNUMBER = {3134040},
	MRREVIEWER = {{\^A}ngela Mestre},
	DOI = {10.2140/ant.2013.7.1953},
	URL = {http://dx.doi.org/10.2140/ant.2013.7.1953},
}

@article {getzlerbv,
    AUTHOR = {Getzler, Ezra},
     TITLE = {Batalin-{V}ilkovisky algebras and two-dimensional topological
              field theories},
   JOURNAL = {Comm. Math. Phys.},
  FJOURNAL = {Communications in Mathematical Physics},
    VOLUME = {159},
      YEAR = {1994},
    NUMBER = {2},
     PAGES = {265--285},
      ISSN = {0010-3616},
     CODEN = {CMPHAY},
   MRCLASS = {81T70 (17B81 55Q99 58Z05 81T40)},
  MRNUMBER = {1256989 (95h:81099)},
MRREVIEWER = {J. Stasheff},
       URL = {http://projecteuclid.org/getRecord?id=euclid.cmp/1104254599},
}

@unpublished{costello-moduli,
	author = {Costello, Kevin},
	title = {The ${A}_\infty$ operad and the moduli space of curves
	},
	note  =  {Preprint available at arXiv:math/0402015},
	year = {2004},
}

@book {macdonald,
    AUTHOR = {Macdonald, Ian G.},
     TITLE = {Symmetric functions and {H}all polynomials},
    SERIES = {Oxford Mathematical Monographs},
   EDITION = {Second},
      NOTE = {With contributions by A. Zelevinsky,
              Oxford Science Publications},
 PUBLISHER = {The Clarendon Press Oxford University Press},
   ADDRESS = {New York},
      YEAR = {1995},
     PAGES = {x+475},
      ISBN = {0-19-853489-2},
   MRCLASS = {05E05 (05-02 20C30 20C33 20K01 33C80 33D80)},
  MRNUMBER = {1354144 (96h:05207)},
MRREVIEWER = {John R. Stembridge},
}

@book {farbmargalit,
    AUTHOR = {Farb, Benson and Margalit, Dan},
     TITLE = {A primer on mapping class groups},
    SERIES = {Princeton Mathematical Series},
    VOLUME = {49},
 PUBLISHER = {Princeton University Press},
   ADDRESS = {Princeton, NJ},
      YEAR = {2012},
     PAGES = {xiv+472},
      ISBN = {978-0-691-14794-9},
   MRCLASS = {57M50 (20F36 20F65 57M07 57N05)},
  MRNUMBER = {2850125 (2012h:57032)},
MRREVIEWER = {Stephen P. Humphries},
}

@article {mcduff,
    AUTHOR = {McDuff, Dusa},
     TITLE = {Configuration spaces of positive and negative particles},
   JOURNAL = {Topology},
  FJOURNAL = {Topology. An International Journal of Mathematics},
    VOLUME = {14},
      YEAR = {1975},
     PAGES = {91--107},
      ISSN = {0040-9383},
   MRCLASS = {55D35 (57D25)},
  MRNUMBER = {0358766 (50 \#11225)},
MRREVIEWER = {D. B. Fuks},
}

@article {mcduffsegal,
	AUTHOR = {McDuff, Dusa and Segal, Graeme},
	TITLE = {Homology fibrations and the ``group-completion'' theorem},
	JOURNAL = {Invent. Math.},
	FJOURNAL = {Inventiones Mathematicae},
	VOLUME = {31},
	YEAR = {1975/76},
	NUMBER = {3},
	PAGES = {279--284},
	ISSN = {0020-9910},
	MRCLASS = {55D45},
	MRNUMBER = {402733},
	MRREVIEWER = {Harold Hastings},
	DOI = {10.1007/BF01403148},
	URL = {https://doi-org.ezp.sub.su.se/10.1007/BF01403148},
}

@article {kupersrandalwilliams-torelli,
	AUTHOR = {Kupers, Alexander and Randal-Williams, Oscar},
	TITLE = {On the cohomology of {T}orelli groups},
	JOURNAL = {Forum Math. Pi},
	FJOURNAL = {Forum of Mathematics. Pi},
	VOLUME = {8},
	YEAR = {2020},
	PAGES = {e7, 83},
	MRCLASS = {57S05 (11F75 18M05 20G05 55R40)},
	MRNUMBER = {4089394},
	MRREVIEWER = {Nick Salter},
	DOI = {10.1017/fmp.2020.5},
	URL = {https://doi-org.ezp.sub.su.se/10.1017/fmp.2020.5},
}

@incollection {mumfordtowards,
    AUTHOR = {Mumford, David},
     TITLE = {Towards an enumerative geometry of the moduli space of curves},
 BOOKTITLE = {Arithmetic and geometry, {V}ol. {II}},
    SERIES = {Progr. Math.},
    VOLUME = {36},
     PAGES = {271--328},
 PUBLISHER = {Birkh\"auser Boston},
   ADDRESS = {Boston, MA},
      YEAR = {1983},
   MRCLASS = {14H10 (14C15)},
  MRNUMBER = {717614 (85j:14046)},
MRREVIEWER = {Werner Kleinert},
}

@article {madsenweiss,
    AUTHOR = {Madsen, Ib and Weiss, Michael},
     TITLE = {The stable moduli space of {R}iemann surfaces: {M}umford's
              conjecture},
   JOURNAL = {Ann. of Math. (2)},
  FJOURNAL = {Annals of Mathematics. Second Series},
    VOLUME = {165},
      YEAR = {2007},
    NUMBER = {3},
     PAGES = {843--941},
      ISSN = {0003-486X},
     CODEN = {ANMAAH},
   MRCLASS = {14H10 (14F43 19D06 55P47)},
  MRNUMBER = {2335797 (2009b:14051)},
MRREVIEWER = {Ulrike Tillmann},
       DOI = {10.4007/annals.2007.165.843},
       URL = {http://dx.doi.org/10.4007/annals.2007.165.843},
}

@article {harerstability,
    AUTHOR = {Harer, John L.},
     TITLE = {Stability of the homology of the mapping class groups of
              orientable surfaces},
   JOURNAL = {Ann. of Math. (2)},
  FJOURNAL = {Annals of Mathematics. Second Series},
    VOLUME = {121},
      YEAR = {1985},
    NUMBER = {2},
     PAGES = {215--249},
      ISSN = {0003-486X},
     CODEN = {ANMAAH},
   MRCLASS = {57M99 (20F34)},
  MRNUMBER = {786348 (87f:57009)},
MRREVIEWER = {K. Vogtmann},
       DOI = {10.2307/1971172},
       URL = {http://dx.doi.org/10.2307/1971172},
}

@incollection {backgroundspace,
    AUTHOR = {Cohen, Ralph L. and Madsen, Ib},
     TITLE = {Surfaces in a background space and the homology of mapping
              class groups},
 BOOKTITLE = {Algebraic geometry---{S}eattle 2005. {P}art 1},
    SERIES = {Proc. Sympos. Pure Math.},
    VOLUME = {80},
     PAGES = {43--76},
 PUBLISHER = {Amer. Math. Soc.},
   ADDRESS = {Providence, RI},
      YEAR = {2009},
   MRCLASS = {57M07 (55P47)},
  MRNUMBER = {2483932 (2010d:57002)},
MRREVIEWER = {Laurence R. Taylor},
}

@article {looijengastable,
    AUTHOR = {Looijenga, Eduard},
     TITLE = {Stable cohomology of the mapping class group with symplectic
              coefficients and of the universal {A}bel-{J}acobi map},
   JOURNAL = {J. Algebraic Geom.},
  FJOURNAL = {Journal of Algebraic Geometry},
    VOLUME = {5},
      YEAR = {1996},
    NUMBER = {1},
     PAGES = {135--150},
      ISSN = {1056-3911},
   MRCLASS = {14H15 (14C30 14F25)},
  MRNUMBER = {1358038 (97g:14026)},
}

@article {miller,
    AUTHOR = {Miller, Edward Y.},
     TITLE = {The homology of the mapping class group},
   JOURNAL = {J. Differential Geom.},
  FJOURNAL = {Journal of Differential Geometry},
    VOLUME = {24},
      YEAR = {1986},
    NUMBER = {1},
     PAGES = {1--14},
      ISSN = {0022-040X},
     CODEN = {JDGEAS},
   MRCLASS = {32G15 (57N05)},
  MRNUMBER = {857372 (88b:32051)},
MRREVIEWER = {Ronnie Lee},
       URL = {http://projecteuclid.org/getRecord?id=euclid.jdg/1214440254},
}

@article {shapirohyperplane,
	AUTHOR = {Shapiro, Boris Z.},
	TITLE = {The mixed {H}odge structure of the complement to an arbitrary
	arrangement of affine complex hyperplanes is pure},
	JOURNAL = {Proc. Amer. Math. Soc.},
	FJOURNAL = {Proceedings of the American Mathematical Society},
	VOLUME = {117},
	YEAR = {1993},
	NUMBER = {4},
	PAGES = {931--933},
	ISSN = {0002-9939},
	CODEN = {PAMYAR},
	MRCLASS = {32S35 (52B30)},
	MRNUMBER = {1131042 (93e:32045)},
	MRREVIEWER = {P. Orlik},
	DOI = {10.2307/2159517},
	URL = {http://dx.doi.org/10.2307/2159517},
}

@phdthesis{howellthesis,
	author = "Howell, Nicholas L.",
	title = "Motives of log schemes",
	school = "University of Oregon",
	year = 2017}

@article {boldsen,
    AUTHOR = {Boldsen, S{\o}ren K.},
     TITLE = {Improved homological stability for the mapping class group
              with integral or twisted coefficients},
   JOURNAL = {Math. Z.},
  FJOURNAL = {Mathematische Zeitschrift},
    VOLUME = {270},
      YEAR = {2012},
    NUMBER = {1-2},
     PAGES = {297--329},
      ISSN = {0025-5874},
     CODEN = {MAZEAX},
   MRCLASS = {57M50 (55T05 57N05)},
  MRNUMBER = {2875835 (2012k:57026)},
MRREVIEWER = {Thomas Koberda},
       DOI = {10.1007/s00209-010-0798-y},
       URL = {http://dx.doi.org/10.1007/s00209-010-0798-y},
}

@book {acg,
    AUTHOR = {Arbarello, Enrico and Cornalba, Maurizio and Griffiths, Phillip
              A.},
     TITLE = {Geometry of algebraic curves. {V}olume {II}},
    SERIES = {Grundlehren der Mathematischen Wissenschaften [Fundamental
              Principles of Mathematical Sciences]},
    VOLUME = {268},
      NOTE = {With a contribution by Joseph Daniel Harris},
 PUBLISHER = {Springer},
   ADDRESS = {Heidelberg},
      YEAR = {2011},
     PAGES = {xxx+963},
      ISBN = {978-3-540-42688-2},
   MRCLASS = {14H10},
  MRNUMBER = {2807457},
}

@article {bergstrom09,
    AUTHOR = {Bergstr{\"o}m, Jonas},
     TITLE = {Equivariant counts of points of the moduli spaces of pointed
              hyperelliptic curves},
   JOURNAL = {Doc. Math.},
  FJOURNAL = {Documenta Mathematica},
    VOLUME = {14},
      YEAR = {2009},
     PAGES = {259--296},
      ISSN = {1431-0635},
   MRCLASS = {14H10 (11G20)},
  MRNUMBER = {MR2538614},
}

@article {getzlerkapranov,
    AUTHOR = {Getzler, Ezra and Kapranov, Mikhail M.},
     TITLE = {Modular operads},
   JOURNAL = {Compositio Math.},
  FJOURNAL = {Compositio Mathematica},
    VOLUME = {110},
      YEAR = {1998},
    NUMBER = {1},
     PAGES = {65--126},
      ISSN = {0010-437X},
     CODEN = {CMPMAF},
   MRCLASS = {18C15 (08A02 14H10 57M50)},
  MRNUMBER = {MR1601666 (99f:18009)},
MRREVIEWER = {Alexandre I. Kabanov},
       DOI = {10.1023/A:1000245600345},
}

@article {acv03,
    AUTHOR = {Abramovich, Dan and Corti, Alessio and Vistoli, Angelo},
     TITLE = {Twisted bundles and admissible covers},
      NOTE = {Special issue in honor of Steven L. Kleiman},
   JOURNAL = {Comm. Algebra},
  FJOURNAL = {Communications in Algebra},
    VOLUME = {31},
      YEAR = {2003},
    NUMBER = {8},
     PAGES = {3547--3618},
      ISSN = {0092-7872},
     CODEN = {COALDM},
   MRCLASS = {14H10 (14A20 14H30)},
  MRNUMBER = {MR2007376 (2005b:14049)},
MRREVIEWER = {Andrew Kresch},
       DOI = {10.1081/AGB-120022434},
       URL = {http://dx.doi.org/10.1081/AGB-120022434},
}

\end{document}